\documentclass[english]{article}

\usepackage{geometry}
\usepackage[T1]{fontenc}
\usepackage{amsmath,amssymb,amsthm,graphicx,bm,dsfont}
\usepackage{epsfig, fontawesome}
\usepackage[authoryear]{natbib} 
\usepackage{psfrag}
\usepackage{enumerate} \usepackage{setspace}
\usepackage{float} 
\usepackage{color}
\usepackage{array}
\usepackage{microtype, wasysym, comment,mathtools}
\usepackage{subfigure}
\usepackage{booktabs} 
\usepackage{enumitem}
\usepackage[colorlinks=true,allcolors=blue]{hyperref}%
\usepackage{graphicx} 
\usepackage{epstopdf}
\usepackage[ruled,vlined]{algorithm2e}
\usepackage{algorithmic}

\definecolor{yuting}{RGB}{255,69,0}
\definecolor{gen}{RGB}{199,21,133}

\DeclareMathOperator{\ind}{\mathds{1}}

\newcommand{\second}{\prime\prime}
\newcommand{\third}{\prime\prime\prime}
\newcommand{\vstar}{v^\star}
\newcommand{\lt}{\left}
\newcommand{\rt}{\right}
\newcommand{\dx}{\mathrm{d} x}

\newcommand{\myE}{\mathbb{E}}

\newcommand{\Xbayes}{\widehat{X}^{\textrm{bayes}}}
\newcommand{\power}{s}
\newcommand{\vseed}{\widetilde{v}}

\newtheorem{assumption}{\textbf{Assumption}}
\newtheorem{claim}{\textbf{Claim}}
\newtheorem{remark}{\textbf{Remark}}

\theoremstyle{plain}

\newtheorem{theo}{Theorem}[section]

\newtheorem{lem}{Lemma}[section]
\newtheorem{prop}{Proposition}[section]
\newtheorem{cor}{Corollary}[section]

\theoremstyle{definition} 

\newtheorem{nota}{Notation}[section]
\newtheorem{de}{Definition}[section]
\newtheorem{exa}{Example}[section]
\newtheorem{as}{Assumption}[section]
\newtheorem{alg}{Algorithm}[section]

\newcommand{\btheo}{\begin{theo}}
\newcommand{\bde}{\begin{de}}
\newcommand{\ble}{\begin{lem}}
\newcommand{\bpr}{\begin{prop}}
\newcommand{\bno}{\begin{nota}}
\newcommand{\bex}{\begin{exa}}
\newcommand{\bcor}{\begin{cor}}
\newcommand{\spro}{\begin{proof}}
\newcommand{\bas}{\begin{as}}
\newcommand{\balg}{\begin{alg}}

\newcommand{\etheo}{\end{theo}}
\newcommand{\ede}{\end{de}}
\newcommand{\ele}{\end{lem}}
\newcommand{\epr}{\end{prop}}
\newcommand{\eno}{\end{nota}}
\newcommand{\eex}{\end{exa}}
\newcommand{\ecor}{\end{cor}}
\newcommand{\fpro}{\end{proof}}
\newcommand{\eas}{\end{as}}
\newcommand{\ealg}{\end{alg}}

\theoremstyle{plain}

\newtheorem{theos}{Theorem}
\newtheorem{props}{Proposition}
\newtheorem{lems}{Lemma}
\newtheorem{cors}{Corollary}

\theoremstyle{definition}
\newtheorem{exas}{Example}
\newtheorem{algs}{Algorithm}
\newtheorem{asss}{Assumption}
\newtheorem{defns}{Definition}

\newcommand{\btheos}{\begin{theos}}
\newcommand{\etheos}{\end{theos}}
\newcommand{\bprops}{\begin{props}}
\newcommand{\eprops}{\end{props}}
\newcommand{\bdes}{\begin{defns}}
\newcommand{\edes}{\end{defns}}
\newcommand{\blems}{\begin{lems}}
\newcommand{\elems}{\end{lems}}
\newcommand{\bcors}{\begin{cors}}
\newcommand{\ecors}{\end{cors}}
\newcommand{\bexs}{\begin{exas}}
\newcommand{\eexs}{\end{exas}}
\newcommand{\balgs}{\begin{algs}}
\newcommand{\ealgs}{\end{algs}}
\newcommand{\bass}{\begin{asss}}
\newcommand{\eass}{\end{asss}}

\newcommand{\ltwo}[1]{\|#1\|_2}
\newcommand{\red}[1]{\textcolor{red}{#1}}

\newcommand{\taustar}{{\tau^\star}}

\newcommand{\real}{\ensuremath{\mathbb{R}}}

\newcommand{\inprod}[2]{\ensuremath{\langle #1 , \, #2 \rangle}}

\newcommand{\thetahat}{\ensuremath{\widehat{\theta}}}

\newcommand{\mprob}{\ensuremath{\mathbb{P}}}

\newcommand{\defn}{\coloneqq}

\newcommand{\argmax}{\arg\!\max}

\newcommand{\Exs}{\ensuremath{\mathbb{E}}}

\long\def\comment#1{}

\newcommand{\HACKPROOF}{\begin{proof}}
\newcommand{\HACKENDPROOF}{\end{proof}}

\newlength{\widebarargwidth}
\newlength{\widebarargheight}
\newlength{\widebarargdepth}

\makeatletter
\long\def\@makecaption#1#2{
        \vskip 0.8ex
        \setbox\@tempboxa\hbox{\small {\bf #1:} #2}
        \parindent 1.5em  
        \dimen0=\hsize
        \advance\dimen0 by -3em
        \ifdim \wd\@tempboxa >\dimen0
                \hbox to \hsize{
                        \parindent 0em
                        \hfil 
                        \parbox{\dimen0}{\def\baselinestretch{0.96}\small
                                {\bf #1.} #2
                                } 
                        \hfil}
        \else \hbox to \hsize{\hfil \box\@tempboxa \hfil}
        \fi
        }
\makeatother

\definecolor{myred}{RGB}{0,0,0} 
\renewcommand{\red}[1]{\textcolor{myred}{#1}}

\allowdisplaybreaks

\begin{document}

\title{A Non-Asymptotic Framework for Approximate Message Passing in Spiked Models\footnotetext{Corresponding author: Yuting Wei (email: \texttt{ytwei@wharton.upenn.edu}).}}

\author{Gen Li \hspace*{0.8cm} Yuting Wei  \\[.2in]
	Department of Statistics and Data Science, the Wharton School \\
	University of Pennsylvania, Philadelphia, PA
}
\maketitle

\begin{abstract}
Approximate message passing (AMP) emerges as an effective iterative paradigm for solving high-dimensional statistical problems. However, prior AMP theory --- which focused mostly on high-dimensional asymptotics --- fell short of predicting the AMP dynamics when the number of iterations surpasses $o\big( \frac{\log n}{\log\log n}\big)$ (with $n$ the problem dimension). To address this inadequacy, this paper develops a non-asymptotic framework for understanding AMP in spiked matrix estimation. Built upon new decomposition of AMP updates and controllable residual terms, we lay out an analysis recipe to characterize the finite-sample behavior of AMP in the presence of an independent initialization, which is further generalized to allow for spectral initialization. As two concrete consequences of the proposed analysis recipe: (i) when solving $\mathbb{Z}_2$ synchronization,  we predict the behavior of spectrally initialized AMP for up to $O\big( \frac{n}{\mathrm{poly}\log n}\big)$ iterations, showing that the algorithm succeeds without the need of a subsequent refinement stage (as conjectured recently by \citet{celentano2021local}); (ii) we characterize the non-asymptotic behavior of AMP in sparse PCA (in the spiked Wigner model) for a broad range of signal-to-noise ratio. 
\end{abstract}


\medskip


\setcounter{tocdepth}{2}
\tableofcontents

\section{Introduction}


Approximate Message Passing (AMP) refers to a class of iterative algorithms that has received 
considerable attention over the past two decades,  
partly due to its versatility in solving a diverse array of science and engineering problems (\cite{schniter2011message,fletcher2014scalable,rush2017capacity,borgerding2016onsager}) as well as its capability in approaching the theoretical limits of many of these problems. 
Originally introduced in the context of compressed sensing as a family of low-complexity iterative algorithms \citep{donoho2009message}, AMP lends itself well to a wide spectrum of high-dimensional statistical problems, both as a class of efficient estimation algorithms and as a powerful theoretical machinery. 
Examples of this kind abound, including robust M-estimators \citep{donoho2016high,donoho2015variance}, 
sparse linear regression \citep{bayati2011lasso,donoho2013information,bu2020algorithmic,li2021minimum}, generalized linear models \citep{sur2019likelihood,sur2019modern,venkataramanan2021estimation,barbier2019optimal},  phase retrieval \citep{ma2018optimization,schniter2014compressive,aubin2020exact}, 
community detection \citep{deshpande2017asymptotic,ma2021community}, 
structured matrix estimation and principal component analysis (PCA) \citep{rangan2012iterative,montanari2021estimation,deshpande2014information,mondelli2021pca}, 
mean-field spin glass models \citep{sellke2021optimizing,fan2022tap,fan2021replica}, to name just a few.     
The interested reader is referred to \cite{feng2021unifying} for a recent overview of AMP and its wide applicability.

\subsection{Asymptotic vs.~non-asymptotic AMP theory}

\paragraph{High-dimensional asymptotics and state evolution.}

A key appealing feature of AMP lies in its effectiveness in analyzing estimators under high-dimensional asymptotics or large-system limits (for instance, in robust M-estimation, this might refer to the regime where the number of observations scales proportionally with the number of unknowns \citep{bayati2011dynamics,javanmard2013state}).  
In such challenging regimes, the limiting behavior of AMP (as the problem dimension diverges) can often be accurately predicted by the so-called \emph{state evolution (SE)},  a recurrence formula that tracks how a small number of key parameters evolve from one iteration to the next. 
For various estimation problems, 
an algorithmic design paradigm is to construct a general class of AMP instances, 
and then identify the optimal choice by inspecting their state-evolution characterizations (which can often be done given that state evolution might only involve very few (e.g., 2) key parameters).


\paragraph{Non-asymptotic theory for AMP?} 

Despite the predicting power of state evolution in high-dimensional asymptotics, 
most existing AMP theory exhibited an asymptotic flavor (often stated in a weak convergence sense as problem dimension tends to infinity), 
which fell short if the number of iterations grows with the problem dimension. 
In light of this, there are two main limitations that are pronounced in current understanding of AMP:
\begin{itemize}
	\item[(i)]
When AMP is deployed as an analysis device, the theoretical guarantees obtained based on existing state-evolution predictions are asymptotic in nature. 
		For this reason, it might sometimes lose advantages over alternative machineries such as the convex Gaussian min-max theorem \citep{thrampoulidis2018precise,celentano2020lasso} and the leave-one-out analysis framework \citep{el2018impact} when the goal is to understand non-asymptotic fine-grained statistical behavior of the estimators; 

	\item[(ii)] When AMP is employed as an optimization algorithm of its own,  most prior AMP theory could only accommodate a non-growing number of iterations, thereby significantly limiting the optimization accuracy AMP can achieve (e.g., such asymptotic AMP theory cannot yield an optimization error that is $o_n(1)$). This stands in stark contrast to other non-asymptotic analysis of optimization-based algorithms (e.g., gradient descent), which deliver characterization of iteration complexity for arbitrary optimization accuracy levels (e.g., \citet{keshavan2010matrix,candes2015phase,ma2020implicit}).

%
\end{itemize}




\noindent 
In order to address the aforementioned limitations of asymptotic theory, 
\citet{rush2018finite} developed a finite-sample analysis of AMP (for noisy linear models) 
that permits the number of iterations to reach $o\big( \frac{\log n }{ \log\log n}\big)$. 
However, $o\big( \frac{\log n }{ \log\log n}\big)$ iterations of AMP are, for the most part, 
unable to yield a (relative) convergence error of $O(n^{-\varepsilon})$ for even an arbitrarily small constant $\varepsilon>0$. 
Another recent work  \cite{celentano2021local} considered the use of spectrally initialized AMP for $\mathbb{Z}_2$ synchronization, and appended it with another gradient-type algorithm 
in order to allow for a growing number of iterations; this, however, did not reveal non-asymptotic behavior of AMP either.   
All this motivates the following question that we would like to study in this paper:  
\begin{center}
	\emph{Is it possible to develop non-asymptotic analysis of AMP beyond $o\big( \frac{\log n}{ \log \log n} \big)$ iterations?} 
\end{center}
On a technical level, the challenge lies in understanding the complicated dependence structures of AMP iterates across iterations. 
In prior analysis, the bounds on certain residual terms (e.g., the difference between the behavior of the AMP and what state evolution predicts) blow up dramatically fast in the iteration number, 
thus calling for new analysis ideas to enable tighter controls of such residual terms.

\subsection{AMP for spiked Wigner models}

In this paper, we attempt to answer the question posed above in the affirmative, focusing on the context of estimation in spiked matrix models as detailed below.  
To facilitate concrete discussions, let us first set the stage by introducing the model and algorithm studied herein, 
before moving on to describe our main results in the next subsection. 

\paragraph{Spiked Wigner models.} 
%
The spiked matrix model refers to a class of data matrices that can be decomposed into a rank-one signal and a random noise matrix, 
which was proposed by \cite{johnstone2001distribution} as a way to study PCA in high dimension and has inspired substantial subsequent works in both statistics and random matrix theory \citep{peche2006largest,baik2005phase,bai2008central,johnstone2009consistency,johnstone2018pca}. 
This paper assumes access to a rank-one deformation of a Wigner matrix $W=[W_{ij}]_{1\leq i,j\leq n}$ as follows: 
\begin{align}
\label{eqn:wigner}
	M = \lambda v^\star v^{\star\top} + W \in \mathbb{R}^{n\times n},
\end{align}
where the spiked vector $v^{\star}=[v^{\star}_i]_{1\leq i\leq n} \in \real^n$ obeys $\|\vstar\|_2=1$ and represents the signal to be estimated, 
$\lambda>0$ determines the signal-to-noise ratio (SNR), 
and the $W_{ij}$'s ($i\geq j$) are independently generated such that
\begin{align}
	W_{ij} = W_{ji} \overset{\mathrm{i.i.d.}}{\sim} \mathcal{N}\Big(0, \frac{1}{n} \Big) \qquad \text{and} \qquad W_{ii} \overset{\mathrm{i.i.d.}}{\sim}  \mathcal{N}\Big(0,\frac{2}{n} \Big) .
\end{align}
As has been shown in prior literature \citep{peche2006largest,feral2007largest,capitaine2009largest}, 
the leading eigenvalue of $M$ stands out from the semicircular bulk under the condition $\lambda > 1$;  
in contrast, it is information-theoretically infeasible to detect the planted signal if $\lambda < 1$, unless additional structural information about $\vstar$ is available.  
Prominent examples of such structural information include sparsity \citep{johnstone2009consistency,berthet2013optimal}, non-negativity \citep{montanari2015non}, cone constraints \citep{deshpande2014cone,lesieur2017constrained}, synchronization over finite groups \citep{perry2018message,javanmard2016phase}, among others. 
Nevertheless, finding the maximum likelihood estimates or Bayes-optimal estimates is often computationally intractable (due to nonconvexity), 
thus complicating the computational/statistical analyses of the iterative estimators in use.


\paragraph{AMP for spiked Wigner models.}

The AMP algorithm tailored to estimating the spiked Wigner model adopts the following update rule: 
\begin{align}
\label{eqn:AMP-updates}
	x_{t+1} = M\eta_t(x_{t}) - \big\langle\eta_t^{\prime}(x_{t}) \big\rangle \cdot \eta_{t-1}(x_{t-1}), 
	\qquad \text{ for } t\geq 1. 
\end{align}
where $\langle z \rangle \defn \frac{1}{n} \sum_{i=1}^n z_i$ for any vector $z=[z_i]_{1\leq i \leq n}\in \real^n$.  
Here, the key elements are described as follows: 
\begin{itemize}
	\item $x_t\in \real^n$ denotes the AMP iterate in the $t$-th iteration, where the initialization $x_0$ and $x_1$ can sometimes be selected in a problem-specific manner. 
	\item The scalar function $\eta_t: \real \to \real$ stands for the denoising function adopted in the $t$-th iteration, with  $\eta_t'(\cdot)$ denoting the derivative of $\eta_t(\cdot)$; 
		when applied to a vector $x$, it is understood that $\eta_t(\cdot)$ (resp.~$\eta'(\cdot)$) is applied entry-by-entry. 
	\item The first term $M\eta_t(x_{t})$ on the right-hand side of \eqref{eqn:AMP-updates} performs a power iteration to the denoised iterate $\eta_t(x_{t})$, 
		while the second term $\big\langle\eta_t^{\prime}(x_{t}) \big\rangle \cdot \eta_{t-1}(x_{t-1})$ --- often referred to as the ``Onsager term'' --- plays a crucial role in cancelling out certain correlation across iterations. 
\end{itemize}
%



\paragraph{State evolution.} 

As alluded to previously, the limiting behavior of the AMP sequence can be  pinned down through a small-dimensional recurrence termed the {\em state evolution (SE)}.  
More precisely, 
assuming that the empirical distribution of\footnote{Here, we adopt the factor $\sqrt{n}$ to be consistent with the scaling of this paper, given that $\ltwo{\vstar} = 1$.} $\{\sqrt{n} v_i^{\star}\}_{i=1}^{n}$ 
converges weakly to a distribution $\mu_{V}$ on $\mathbb{R}$ with unit second moment, the SE associated with \eqref{eqn:AMP-updates} 
is the following recurrence involving two scalar sequences $\{\alpha^\star_t\}$ and $\{\beta^\star_t\}$: 
\begin{subequations}
\label{eq:SE-Montanari}
\begin{align}
	\alpha^\star_{t+1} &= \lambda \myE\big[V \eta_t(\alpha^\star_t V + \beta^\star_t G)\big] \\
	\beta^{\star 2}_{t+1} &= \myE\big[\eta_t^2(\alpha^\star_t V + \beta^\star_t G)\big]
\end{align}
\end{subequations}
for any $t\geq 1$, 
where ${V \sim \mu_{V}}$ and ${G\sim \mathcal{N}(0,1)}$ are independent random variables. 
The SE \eqref{eq:SE-Montanari} has been studied by \cite{fletcher2018iterative} in the presence of an independent initialization, 
and by \cite{montanari2021estimation} under spectral initialization. 
As shown in \citet{montanari2021estimation}, for any fixed $t$ and any pseudo-Lipschitz function $\Psi: \mathbb{R} \times \mathbb{R} \to \mathbb{R}$, it holds almost surely that 
\begin{align}
	\lim_{n\to \infty} \frac{1}{n}\sum_{i=1}^n \Psi(\sqrt{n}v^\star_i, \sqrt{n} x_{t,i}) 
	=
	\myE \Big[\Psi(V, \alpha^\star_tV + \beta^\star_t G)\Big] 
	\label{eq:AMP-prediction-SE}
\end{align} 
when the AMP sequence $\{x_t\}$ is initialized by spectral methods. 
Informally, this result \eqref{eq:AMP-prediction-SE} uncovers that each coordinate of the AMP iterate behaves like $\alpha^\star_tV + \beta^\star_t G$ (after proper rescaling), 
containing an extra source of Gaussian-type randomness that is crucial in explaining the AMP dynamics under high-dimensional asymptotics. 
Moreover, property~\eqref{eq:AMP-prediction-SE} also suggests that 
the denoising functions $\{\eta_t\}$ can be optimally selected  
\citep{bayati2011dynamics,montanari2021estimation} as   
the minimum mean square error (MMSE) estimator (or Bayes-optimal estimator if given $\nu_V$), namely,  
\begin{align}
\label{eqn:eta-principle}
	\eta_t(x) \,=\, \Exs[V \mid \alpha^\star_t V + \beta^\star_t G = x].
\end{align}
%

Note, however, that the validity of this SE-based prediction has only been verified when $t$ is fixed and $n\rightarrow \infty$. 
It remains to see whether the SE can track the AMP behavior in a non-asymptotic manner in the presence of a possibly large number of iterations.

\subsection{A glimpse of main contributions}

The main contributions of this paper are the development of a non-asymptotic analysis framework that helps to understand the AMP behavior when the number of iterations is chosen polynomial in $n$. Our main findings are summarized as follows. 
\begin{itemize}
\item {\bf A key decomposition of AMP iterates with tractable residual terms.}
We develop in Theorem~\ref{thm:recursion} a general decomposition of the $t+1$-th iterate of AMP  as follows:
\begin{align}
\label{eqn:general-decomp}
	x_{t+1} = \alpha_{t+1} \vstar + \sum_{k = 1}^{t} \beta_{t}^k\phi_k + \xi_{t}, \qquad \text{for } t\geq 1.
\end{align}
Here, $\vstar$ is the underlying signal, $\{\phi_k\}_{k=1}^{t}$ stands for a collection of independent Gaussian vectors, 
$\alpha_{t+1}$ and $\beta_t=[\beta_t^1,\ldots,\beta_t^{t}]\in \real^{t}$ are a set of weights, 
and $\xi_{t} \in \real^{n}$ is a residual term that lies in a $t$-dimensional subspace determined by the previous iterates.  
This decomposition is fairly general with little assumption imposed on either the denoising function, the number of iterations, or $\vstar$.   
Our analysis reveals that the residual terms $\{\xi_{t}\}$ can often be bounded in a recursive yet tractable manner without blowing up rapidly. 

\item {\bf Finite-sample analysis beyond $o\big(\log n/\log \log n\big)$ iterations.}
Leveraging upon the decomposition in~\eqref{eqn:general-decomp},  
in Theorem~\ref{thm:main}, we develop an analysis framework to track $\alpha_{t+1}$ and $\beta_{t}$ in a non-asymptotic fashion,  
which intimately connects with the state evolution recurrence \eqref{eq:SE-Montanari}.
In fact, our analysis idea could yield non-asymptotic characterizations of AMP iterates for a certain polynomial number of iterations, 
which go far beyond the $o\big(\frac{\log n}{\log \log n}\big)$ iterations covered in prior art. 
All this is largely enabled due to our ability to control the residual size $\|\xi_{t}\|_2$ --- often to the order of $O\big(\sqrt{\frac{t\mathrm{poly}\log(n)}{n}} \big)$.  

\item {\bf Non-asymptotic theory for AMP with spectral initialization.} 
A widely used scheme to initialize AMP for spiked models is the spectral method, 
which often provides an informative initial estimate with non-vanishing correlation with the truth. 
Motivated by its widespread adoption in practice, 
we extend the above analysis framework to study non-asymptotic behavior of spectrally initialized AMP.
As it turns out, our AMP analysis recipe can be tightly integrated with the analysis of spectral initialization, with the aid of two auxiliary AMP sequences and 
a similar decomposition as of \eqref{eqn:general-decomp} is established for such sequences. 
Details can be found in Section~\ref{sec:spectral}.


\item {\bf Concrete consequences: non-asymptotic theory for $\mathbb{Z}_2$ synchronization and sparse PCA.}
In Section~\ref{sec:examples}, we apply our general recipe to two widely studied models
that are very different in nature: the problem of $\mathbb{Z}_{2}$ synchronization and that of sparse PCA (in the context of the sparse spiked Wigner model).
For  $\mathbb{Z}_{2}$ synchronization, 
we focus on the most challenging scenario where the spectral gap $\lambda - 1$ approaches 0, 
and characterize the non-asymptotic behavior of spectrally initialized AMP all the way up to $O\big(\frac{n}{\mathrm{poly}\log (n)}\big)$ iterations (in addition to other dependency on $\lambda-1$). 
This helps address a conjecture in \cite{celentano2021local} regarding the finite-sample behavior of spectrally initialized AMP. 
When it comes to the sparse spiked Wigner model, 
our general recipe leads to non-asymptotic characterizations of the AMP iterates as well. 
If an independent yet informative initialization is provided, then our theory allows the SNR to approach the order of the information-theoretic limit; 
otherwise, our AMP theory can be combined with two initialization schemes in order to accommodate the regime above the computational limit.


\end{itemize}







\subsection{Other related works}

The studies of the spiked Wigner model --- also under the names of deformed Wigner models or matrix denoising --- have received much attention from multiple domains, including but not limited to statistics, random matrix theory, and information theory 
(e.g., \citet{knowles2013isotropic,cheng2021tackling,el2020fundamental,bao2021singular,yan2021inference,fan2022asymptotic,lee2016bulk,perry2018optimality,zhou2023deflated,peng2012eigenvalues,simchowitz2018tight}). 
Subsuming multiple problems as special cases (e.g., phase synchronization, sparse estimation in Wigner models), 
the spiked Wigner model serves as a stylized model that helps uncover various phenomena in high dimensions, such as universality, computational-to-statistical gaps, phase transition, unreasonable effectiveness of nonconvex optimization, etc. We briefly highlight some of these aspects below.

While a large fraction of AMP theory, including the current paper, focuses on the case with i.i.d.~Gaussian noise and/or i.i.d.~Gaussian designs, 
certain \emph{universality} phenomena beyond i.i.d.~Gaussian noise have been empirically observed and theoretically established in the context of AMP 
\citep{bayati2015universality,chen2021universality,wang2022universality,dudeja2022universality}  
and in broader scenarios \citep{lee2016bulk,hu2020universality,oymak2018universality}. 
For instance, \citet{bayati2015universality} and \cite{chen2021universality} studied a random design matrix with i.i.d.~sub-Gaussian entries, 
and \citet{fan2022approximate} was able to accommodate the family of rotationally invariant designs, thus allowing for a spectral distribution that differs from the semicircle or Marcenko-Pastur law.

Additionally, for many structured estimation problems, empirical evidence suggests the potential existence of a gap between the fundamental statistical limit and what can be done computationally efficiently. 
This has inspired considerable theoretical interest towards 
solidifying such computational-to-statistical gaps; see \cite{bandeira2018notes} for a tutorial and also \cite{zdeborova2016statistical} for a connection to statistical physics. The spiked Wigner model forms an idealized model to study such gaps, for multiple structured problems like sparse PCA and non-negative PCA. 
It is also worth noting that AMP, in various settings, is able to achieve the optimal performance among polynomial-time estimators \citep{donoho2009message,celentano2022fundamental}. 
It has also been employed as a machinery to characterize the information-theoretic limits of several high-dimensional problems \citep{deshpande2014information,barbier2016mutual,reeves2019replica}.

Further, estimating the underlying signal from a spiked Wigner model 
is, for the most part, concerned with solving a highly nonconvex problem, particularly in the presence of additional structural constraints. In such cases, the  initialization schemes exert considerable influences on the subsequent AMP dynamics. In fact, a large body of existing AMP theory assumes availability of an informative initialization. 
For instance, in a special case where each entry of $\vstar$ has positive mean, 
it might be sufficient to initialize AMP with an all-one vector \citep{deshpande2014information,montanari2015non}; 
when the SNR is large enough such that $\lambda > 1$, an estimate returned by the spectral method is known to achieve strictly positive correlation with the ground-truth spike, which therefore serves as a common initialization scheme for AMP as well \citep{montanari2021estimation,fan2021tap}.

\subsection{Organization and notation}

\paragraph{Paper organization.} 
The remainder of this paper is organized as follows. 
Sections~\ref{sec:decomposition-thm-1}-\ref{sec:decomposition-thm-2} develop a general recipe that enables  non-asymptotic characterizations of the AMP in spiked models, 
assuming independent initialization. This framework is further extended in Section~\ref{sec:spectral} for the case when AMP is used along with spectral initialization. 
 Sections~\ref{sec:z2-main} and \ref{sec:sparse-main} instantiate our analysis framework to $\mathbb{Z}_{2}$ synchronization and sparse PCA, respectively, 
confirming the utility of our non-asymptotic theory. 
The proof ideas of two master theorems are presented in Section~\ref{sec:main-analysis}, with other technical details deferred to the appendices.
Section~\ref{sec:discussion} concludes the paper by pointing out several future directions.

%

\paragraph{Notation.} 
We often use 0 (resp.~1) to denote the all-zero (resp.~all-one) vector, and let $I_n$ (or simply $I$) denote the $n\times n$ identity matrix. 
For any $w\in \real$, we denote $w_+ \defn \max\{w, 0\}$.   
We denote by $\varphi(\cdot)$ (resp.~$\varphi_n(\cdot)$) the probability density function (p.d.f.) of a standard Gaussian random variable (resp.~a Gaussian random vector $\mathcal{N}(0,I_n)$). 
For any positive integer $k$, we say a function $f: \real^k \to \real$ is $L$-Lipschitz continuous for some quantity $L>0$ if, for every $z_{1}$ and $z_{2}$, one has
$|f(z_1) - f(z_2)| \leq L \cdot \ltwo{z_1 - z_2}$. 
When a function is applied to a vector, it should be understood as being applied in a component-wise manner; 
for instance,  for any vector $x=[x_i]_{1\leq i\leq n}$, we let $|x|\coloneqq [|x_i|]_{1\leq i\leq n}$ and $x_+ \defn [ \max\{x_i, 0\} ]_{1\leq i\leq n}$.
For any two vectors $x,y\in \real^n$, we write $x \circ y $ for their Kronecker product, namely, $x \circ y = (x_1y_1,\ldots, x_ny_n)^{\top} \in \real^{n}.$
For two functions $f(n)$ and $g(n)$, we write $f(n)\lesssim g(n)$ to indicate that $f(n)\leq c_1 g(n)$ for some constant $c_1>0$ that does not depend on $n$, and similarly, $f(n)\gtrsim g(n)$ means that $f(n)\geq c_2 g(n)$ for some constant $c_2>0$ independent of $n$. 
We also adopt the notation $f(n)\asymp g(n)$ to indicate that both $f(n)\lesssim g(n)$ and $f(n)\gtrsim g(n)$ hold simultaneously. 
In addition, we write $f(n) \ll g(n)$ or $f(n)=o(g(n))$ if $f(n)/g(n)\to 0$ as $n\to \infty$ and $f(n) \gg g(n)$ if $g(n)/f(n)\to 0$. 
For any matrix $M$, we let $\|M\|$ and $\|M\|_{\mathrm{F}}$ denote the spectral norm and the Frobenius norm of $M$, respectively. 
For any integer $n>0$, we let $[n]\coloneqq \{1,\cdots, n\}$. 
Also, for any vector $x\in [x_i]_{1\leq i\leq n} \in \real^n$, 
we denote by $|x|_{(i)}$ the $i$-th largest element within $\{|x_i|\}_{1\leq i\leq n}$.

In addition, given two probability measures $\mu$ and $\nu$ on $\real^{n}$, the Wasserstein distance of order $p$ between them is defined and denoted by 
\begin{align}
\label{eqn:wasserstein-p}
	W_p(\mu, \nu) \defn \bigg(\inf_{\gamma \in \mathcal{C}(\mu,\nu)}\int \|x - y\|_2^p \, \mathrm{d}\gamma (x,y)\bigg)^{1/p},
\end{align}
where  $\mathcal{C}(\mu,\nu)$ is the set comprising all couplings of $\mu$ and $\nu$ (i.e., all joint distributions $\gamma(x,y)$ whose marginal distributions are $\mu$ and $\nu$, respectively). 
We let $\mathcal{S}^{d-1}=\{x\in \real^d \mid \|x\|_2=1\}$ represent the unit sphere in $\real^d$, and
denote by $\mathbb{B}^d(r)=\{\theta \in \real^d \mid \|\theta\|_2\leq r\}$ the $d$-dimensional ball of radius $r$ centered at 0.




\section{A general recipe for non-asymptotic analysis of AMP}
\label{sec:main}

In this section, 
we develop a general recipe that leads to a non-asymptotic analysis framework for the AMP algorithm \eqref{eqn:AMP-updates}. 
This constitutes two master theorems (i.e., Theorems~\ref{thm:recursion} and \ref{thm:main}) 
that uncover the key decomposition for the AMP iterates and single out several key quantities to be controlled in order to bound the deviation between the true AMP behavior and the state evolution recurrence. 
Our analysis framework is further extended in Section~\ref{sec:spectral} to accommodate spectrally initialized AMP.


\subsection{A crucial decomposition of AMP iterates}
\label{sec:decomposition-thm-1}


We begin by presenting a key decomposition of the AMP iterates in the following theorem, 
which lies at the core of the non-asymptotic theory developed in this paper. The proof of this result is postponed to Section~\ref{sec:pf-thm-recursion}.
\begin{theos} 
\label{thm:recursion}
Suppose that the AMP algorithm~\eqref{eqn:AMP-updates} is initialized with some vector $x_0 $ obeying $\eta_0(x_0)=0$ and some vector $x_1$ independent of $W$.  
\red{Assume $\{\eta_t(\cdot)\}$ are differentiable except at finite number of points.}
Then for every $1\le t< n$,   
the AMP iterates admit the following decomposition:
\begin{align}
\label{eqn:xt-decomposition}
	x_{t+1} = \alpha_{t+1} \vstar + \sum_{k = 1}^{t} \beta_{t}^k\phi_k + \xi_{t},
\end{align}
where 
\begin{itemize}
	\item[(i)] the coefficient $\alpha_{t+1} \in \real$ obeys $\alpha_{t+1} = \lambda v^{\star\top} \eta_{t}(x_{t})$; 
	\item[(ii)] $\{\phi_k\}_{1\leq k\leq t}$ are independently generated obeying $\phi_k \overset{\mathrm{i.i.d.}}{\sim} \mathcal{N}(0, \frac{1}{n}I_n)$;
	\item[(iii)] the coefficient vector $\beta_t\coloneqq (\beta_t^1,\beta_t^2,\ldots,\beta_t^t) \in \real^t$ obeys $\|\beta_t\|_2  = \lt\|\eta_{t}(x_t)\rt\|_2$; 
	\item[(iv)] $\xi_t \in \mathbb{R}^{n}$  is some residual vector such that, with probability at least $1-O(n^{-11})$, 
\begin{align}
\label{eqn:xi-norm-main}
	\|\xi_{t}\|_2 = \Big\langle \sum_{k = 1}^{t-1} \mu_t^k\phi_k, \delta_{t}\Big\rangle - \langle\delta_{t}^{\prime}\rangle \sum_{k = 1}^{t-1} \mu_t^k\beta_{t-1}^k + \Delta_t + O\Big(\sqrt{\frac{t\log n}{n}}\|\beta_{t}\|_2 \Big)
\end{align}
holds for some some unit vector $\mu_t = [\mu_t^k]_{1\leq k\leq t-1} \in \mathbb{R}^{t-1}$, 			
where we define
\begin{subequations}
\label{eqn:delta-chorus}
\begin{align}
	v_t &\defn \alpha_t\vstar + \sum_{k = 1}^{t-1} \beta_{t-1}^k\phi_k , \label{defn:v-t-thm1}\\
	\label{defn:delta-t} \delta_{t} &\defn \eta_{t}(x_t) - \eta_{t}(v_t), \\
	\label{defn:delta-prime-t} \delta_{t}^{\prime} &\defn \eta_{t}^{\prime}(x_t) - \eta_{t}^{\prime}(v_t),\\
	\label{defn:Delta-t} \Delta_t &\defn 
	\sum_{k = 1}^{t-1} \mu_t^k \Big[\big\langle \phi_k, \eta_{t}(v_t)\big\rangle - \big\langle\eta_t^{\prime}(v_t)\big\rangle \beta_{t-1}^k\Big].
\end{align}
\end{subequations}
\end{itemize}
\end{theos}
\begin{remark}
	 The auxiliary vector $v_t$ defined in \eqref{defn:v-t-thm1}, 
	 which is a linear combination of $\vstar$ and  the Gaussian vectors $\{\phi_{k}\}$,  
can be viewed as $x_t$ with the residual term $\xi_{t-1}$ dropped (see \eqref{eqn:xt-decomposition}). As we shall see momentarily, 
$v_t$ often serves as a fairly tight and informative approximation of $x_t$. 
\end{remark}

In a nutshell, Theorem~\ref{thm:recursion} decomposes the $t+1$-th iterate of the AMP algorithm $x_{t+1}$ into three components: 
 {\em $(1)$ a signal component $\alpha_{t+1}\vstar$:} which is perfectly aligned with $\vstar$, whose strength is captured by $\alpha_{t+1}$; 
$(2)$ {\em a random noise component $\sum_{k = 1}^{t} \beta_{t}^k\phi_k$:} which behaves as a weighted superposition of $t$ i.i.d.~Gaussian vectors, although the weights $\beta_t$ might be statistically dependent on the $\phi_k$'s; 
$(3)$ {\em a residual term $\xi_{t}$:} which hopefully can be well controlled. 
%
%
This decomposition, which holds all the way up to the $n$-th iteration, is fairly general and 
plays a crucial role in obtaining non-asymptotic characterizations of $x_{t+1}$. 
In particular, it imposes little assumption (resp.~no assumption) on the denoising function (resp.~the underlying signal $\vstar$). 
In what follows, we single out several important remarks about the three components in \eqref{eqn:xt-decomposition}. 


\begin{itemize}
\item Let us first look at the random noise component $\sum_{k = 1}^{t} \beta_{t}^k\phi_k$. 
	Clearly, if $\beta_{t}$ were statistically independent from  the i.i.d.~Gaussian vectors $\{\phi_k\}$, then $\frac{1}{\|\beta_{t}\|_2}\sum_{k = 1}^{t} \beta_{t}^k\phi_k$ 
		would exhibit an ideal Gaussian distribution $\mathcal{N}\big(0, \frac{1}{n}I_n\big)$. 
	In general, however, $\beta_{t}$ exhibits delicate dependency on $\{\phi_k\}$, thus complicating matters. 
		Fortunately, the 1-Wasserstein distance between $\frac{1}{\|\beta_{t}\|_2}\sum_{k = 1}^{t} \beta_{t}^k\phi_k$ and the desired $\mathcal{N}\big(0, \frac{1}{n}I_n\big)$ remains small as long as $t$ is not too large; that is, 
\begin{align}
	W_{1}\Bigg(\mu\bigg(\frac{1}{\ltwo{\beta_t}}\sum_{i=1}^{t}\beta_{k}\phi_{k}\bigg), \,\mathcal{N}\bigg(0,\frac{1}{n}I_{n}\bigg)\Bigg) \lesssim 
	\sqrt{\frac{t\log n}{n}},
\end{align}
as asserted by Lemma~\ref{lem:wasserstein}, where $\mu(X)$ denotes the law of the random variable $X$. 
This reveals that this random noise component almost resembles an ideal Gaussian vector $\mathcal{N}\big(0, \frac{1}{n}I_n\big)$ for a wide range of $t$.

\item Next, as shall be made clear momentarily, quantities $\alpha_{t+1}$ and $\ltwo{\beta_{t}}$ in \eqref{eqn:xt-decomposition} are intimately related to the primary quantities in the state evolution formula~\eqref{eq:SE-Montanari}, although their evolutions are now described in a non-asymptotic fashion. This paves the path for a non-asymptotic characterization of its convergence behavior towards a stationary point.

%

\item The residual term $\xi_{t}$ is fairly complicated, depending heavily on the previous iterations of AMP as well as the specific choices of the denoising functions $\eta_{t}$. 
In truth, with different choices of $\eta_{t}$, the residual term $\xi_t$ might exhibit very different dependence on the salient parameters. 
		In Theorem~\ref{thm:recursion} and its analysis, we provide a recursive characterization of $\|\xi_t\|_2$ using several quantities in the preceding iteration,  
and unveil certain low-dimensional structure of the residual term $\xi_t$. 
\red{These important observations pave the way to a more systematic control of these residual terms and offer a key knob to control the non-asymptotic error for the final Gaussian approximation. }


\end{itemize}

\red{
Finally, it is worth noting that prior AMP theory often hinges upon an ingenious Gaussian conditioning technique (e.g., \cite{bayati2011dynamics,bolthausen2009high,rush2018finite}). 
The key construction therein is to write each $x_{t+1}$ as a linear combination of the past iterates $\{x_{i}\}_{i\le t}$ plus a new Gaussian vector and a new error term, as shown in \cite[Lemma 1]{bayati2011dynamics}. 
Such a linear combination together with the new Gaussian vector leads to the final Gaussian approximation. 
However, since $x_1, \ldots, x_t$ are not exactly Gaussian vectors (as the Gaussian property holds only asymptotically); when arguing about the error in $x_{t+1}$ via an inductive method, one inevitably has to deal with the accumulated error inherited from each error term in $\{x_{i}\}_{i\leq t}$. 
Following this argument and directly taking the union bound over the errors of every $x_{t}$ will result in a residual term that grows exponentially fast in the number of iterations $t$ as in \cite{rush2018finite}.
Addressing this issue calls for a more refined and effective manner to track error accumulation so as to avoid the exponential blow-up.  This inspires the development of Theorem~\ref{thm:recursion} and the ensuing theory. 
}


\subsection{Non-asymptotic error characterizations}
\label{sec:decomposition-thm-2}

Thus far, we have identified a general decomposition of the AMP iterates in Theorem~\ref{thm:recursion}, 
accompanied by a recursive formula \eqref{eqn:xi-norm-main} to describe how the size $\|\xi_{t}\|_2$ of the residual term evolves. 
Nevertheless, the formula \eqref{eqn:xi-norm-main} might remain elusive at first glance, 
as it is built upon multiple different objects in the previous iteration. 
In order to better understand the advantages of the recursive relations in Theorem~\ref{thm:recursion}, 
we single out several additional quantities, which --- if easily controllable --- help further simplify the recurrence. 
These taken collectively constitute our general recipe for non-asymptotic analysis of AMP, whose utility will be brought to light via two concrete applications in Section~\ref{sec:examples}.


\paragraph{Assumptions and key quantities.}

Let us first impose the following basic assumptions on the denoising function $\eta_{t}: \real \to \real$. 
Here and throughout, we let $\eta_t^{\prime}(\cdot)$, $\eta_t^{\second}(\cdot)$ and $\eta_t^{\third}(\cdot)$ denote respectively the first-order, second-order, and third-order derivatives of $\eta_{t}$; 
when we apply $\eta_t^{\prime}(\cdot)$, $\eta_t^{\second}(\cdot)$ and $\eta_t^{\third}(\cdot)$ to vectors, it is understood that they are applied entry-by-entry. 
\begin{assumption} 
\label{assump:eta}
For every $1\leq t\leq n$, 
it is assumed that:
\begin{itemize}
	\item $\eta_t(\cdot)$ is  continuous everywhere, and is differentiable up to the 3rd order everywhere except for a finite set $\mathcal{M}_{\mathsf{dc}}$ of points with $\big|\mathcal{M}_{\mathsf{dc}}\big|=O(1)$;
	\item $|\eta_{t}^{\prime}(w)|\leq \rho^{t}$ for any differentiable point $w$ of $\eta_t(\cdot)$; 
	\item $|\eta_{t}^{\second}(w)|\leq \rho^t_1$ for any differentiable point $w$ of $\eta_t^{\prime}(\cdot)$;  
	\item $|\eta_{t}^{\third}(w)|\leq \rho^t_2$ for any differentiable point $w$ of $\eta_t^{\second}(\cdot)$.
\end{itemize}
\red{We take $\rho \defn \max_{1\leq t\leq n} \rho^{t}$, $\rho_1 \defn \max_{1\leq t\leq n} \rho_1^{t}$ and $\rho_2 \defn \max_{1\leq t\leq n} \rho_2^{t}$.} 
%
\end{assumption}
\noindent 
For notational simplicity, we shall --- unless otherwise noted --- take $\eta_t^{\prime}(w)=\eta_t^{\second}(w)=\eta_t^{\third}(w)=0$ for any non-differentiable point, 
with the impact of these singular points explicitly taken into account in the quantity $E_t$ to be defined in Assumption~\ref{assump:A-H-eta}.



In the next assumption, we would like to isolate a few additional quantities that can often be bounded separately. 
We shall formally state this assumption after defining the following additional quantity: 
\begin{align}
\label{defi:kappa}
\notag \kappa_t^2 \defn 
\max\Bigg\{
	\Bigg\langle\int\Big[x\eta_{t}^{\prime}\Big(\alpha_t\vstar + \frac{\|\beta_{t-1}\|_2}{\sqrt{n}}x\Big)  - \frac{\ltwo{\beta_{t-1}}}{\sqrt{n}}&\eta_{t}^{\second}\Big(\alpha_t\vstar + \frac{\|\beta_{t-1}\|_2}{\sqrt{n}}x\Big)\Big]^2 \varphi_n(\dx)\Bigg\rangle, ~\\
&\bigg\langle
\int\Big[\eta_{t}^{\prime}\Big(\alpha_t\vstar + \frac{\|\beta_{t-1}\|_2}{\sqrt{n}}x\Big)\Big]^2\varphi_n(\dx)
\bigg\rangle\Bigg\},
\end{align}
where we recall that $\varphi_n(\cdot)$ is the pdf~of $\mathcal{N}(0,I_n)$ and $\langle x\rangle \defn \frac{1}{n} \sum_{i=1}^{n} x_{i}$.

\begin{assumption}
\label{assump:A-H-eta}
	For any $1\leq t \leq n$, consider arbitrary vectors $\mu_t \in \mathcal{S}^{t-1}$, $\xi_{t-1}\in \real^n$, and coefficients $(\alpha_t, \beta_{t-1}) \in \real \times \real^{t-1}$ that might all be statistically dependent on $\phi_{k}$, and define $v_t$ as in \eqref{defn:v-t-thm1} accordingly. 
In addition to imposing Assumption~\ref{assump:eta}, we assume the existence of (possibly random) quantities $A_{t}, \cdots, G_{t}$ \red{only depending on $n, t, \eta_t, \alpha_t$} such that with probability at least $1-O(n^{-11})$, the following inequalities hold
\begin{subequations}
\begin{align}
	\label{defi:A}
	\Big|\sum_{k = 1}^{t-1} \mu_t^k\Big[\big\langle \phi_k, \eta_{t}(v_t)\big\rangle - \big\langle\eta_t^{\prime}(v_t)\big\rangle \beta_{t-1}^k\Big]\Big| &\,\le\, A_t, \\
	\Big|v^{\star\top}\eta_{t}(v_t) - v^{\star\top}\int\eta_t\Big(\alpha_t \vstar + \frac{\|\beta_{t-1}\|_2}{\sqrt{n}}x\Big)\varphi_n(\dx)\Big| &\,\le\, B_t, \label{defi:B}\\
	\Big|\big\|\eta_{t}(v_t)\big\|_2^2 - \int\Big\|\eta_t\Big(\alpha_t \vstar + \frac{\|\beta_{t-1}\|_2}{\sqrt{n}}x\Big)\Big\|_2^2\varphi_n(\dx)\Big| &\,\le\, C_t, \label{defi:C}\\
	\Big\|\sum_{k = 1}^{t-1} \mu_t^k\phi_k \circ \eta_{t}^{\prime}(v_t) - \frac{1}{n}\sum_{k = 1}^{t-1} \mu_t^k\beta_{t-1}^k\eta_{t}^{\second}(v_t)\Big\|_2^2 - \kappa_t^2 &\,\le\, D_t, \label{defi:D}\\
%
%
	\big\|\eta_{t}(v_t) \circ \eta_{t}^{\prime}(v_t)\big\|_2 &\,\le\, F_t, \label{defi:F}\\
	\big\|\eta_{t}(v_t)\big\|_{\infty} &\,\le\, G_t. \label{defi:G}
%
%
\end{align}
In addition, for any non-differentiable point $m\in \mathcal{M}_{\mathsf{dc}}$, define $\theta(m)\in \real$ as 
\begin{align}
\label{defi:theta}
\theta(m)\coloneqq\sup\left\{ \theta: \,\sum_{j=1}^{n}\bigg|m-\alpha_{t}\vstar_{j}-\sum_{k=1}^{t-1}\beta_{t-1}^{k}\phi_{k,j}\bigg|^{2}\ind\bigg(\bigg|m-\alpha_{t}\vstar_{j}-\sum_{k=1}^{t-1}\beta_{t-1}^{k}\phi_{k,j}\bigg|\le\theta\bigg) \le \|\xi_{t-1}\|_{2}^{2}\right\} , 
\end{align}
and we assume the existence of some quantity $E_t$ such that, with probability at least $1-O(n^{-11})$, 
%
\begin{align}
\label{defi:E}
\sum_{m\in\mathcal{M}_{\mathsf{dc}}}\sum_{j=1}^{n}\ind\Big(\Big|m-\alpha_{t}\vstar_{j}-\sum_{k=1}^{t-1}\beta_{t-1}^{k}\phi_{k,j}\Big|\le\theta(m)\Big)\,\le\,E_{t} .
\end{align}
%
\end{subequations}
\end{assumption}



\paragraph{Error control and state evolution.} 

Armed with the above two assumptions, we are positioned to control the magnitude of the residual term $\xi_{t}$ as well as quantities $\alpha_{t+1}$ and $\beta_{t}$. 
Our result is summarized in the following theorem, with the proof deferred to Section~\ref{sec:pf-thm-main}.
\begin{theos} 
\label{thm:main}
Consider the settings of Theorem~\ref{thm:recursion}, and impose Assumptions~\ref{assump:eta}-\ref{assump:A-H-eta}.  
Then with probability at least $1-O(n^{-11})$, 
the AMP iterates \eqref{eqn:AMP-updates} satisfy the decomposition~\eqref{eqn:xt-decomposition} with 
\begin{subequations}
\label{eqn:state-evolution-finite}
\begin{align}
\label{eqn:alpha-t-genearl}
	\alpha_{t+1} &= \lambda v^{\star \top} \int{\eta}_{t}\lt(\alpha_t \vstar + \frac{\ltwo{\beta_{t-1}}}{\sqrt{n}}x\rt)\varphi_n(\dx) + \lambda\Delta_{\alpha,t} \\
	\|\beta_t\|_2^2 &= n\bigg\langle\int{\eta}_{t}^2\lt(\alpha_t \vstar + \frac{\ltwo{\beta_{t-1}}}{\sqrt{n}}x\rt)\varphi_n(\dx)\bigg\rangle + \Delta_{\beta,t}
	\label{eqn:beta-t-genearl}
\end{align}
\end{subequations}
for any $t\leq n$, where the residual terms obey
\begin{subequations}
\label{eqn:para-general}
\begin{align}
\label{eqn:delta-alpha-general}
|\Delta_{\alpha,t}| &\,\lesssim\, B_t +\rho \|\xi_{t-1}\|_2, \\
\label{eqn:delta-beta-general}
|\Delta_{\beta,t}| &\lesssim C_t + \lt(F_t + \rho_1G_t\|\xi_{t-1}\|_2 + \rho\sqrt{E_t}G_t + \rho^2\|\xi_{t-1}\|_2 \rt)\|\xi_{t-1}\|_2, \\
\label{eqn:xi-t-general}
	\|\xi_{t}\|_{2}&\le\sqrt{\kappa_{t}^{2}+D_{t}}\,\|\xi_{t-1}\|_{2}+O\Bigg(\sqrt{\frac{t\log n}{n}}\|\beta_{t}\|_{2}+A_{t}+\lt[\sqrt{\frac{t+\log n}{n}}\rho_{1}+\frac{\rho_{2}\|\beta_{t-1}\|_{2}}{n}\rt]\|\xi_{t-1}\|_{2}^{2}\nonumber\\
 & \qquad\qquad\qquad\qquad\qquad\qquad+\rho\sqrt{\frac{{E_{t}+t\log n}}{n}}\|\xi_{t-1}\|_{2}+\frac{(\rho+\rho_{1}\big\|\xi_{t-1}\big\|_{\infty})E_{t}\|\beta_{t-1}\|_{2}}{n}\Bigg) .
\end{align}
%
\end{subequations}
\end{theos}

 
Theorem~\ref{thm:main} offers an explicit and recursive way to control the quantities $\alpha_{t+1}, \ltwo{\beta_t}, \xi_{t}$, 
assuming that the quantities $A_t,\ldots,G_t$ isolated in Assumption~\ref{assump:A-H-eta} can be bounded effectively. 
Crucially, the results in \eqref{eqn:state-evolution-finite} can be viewed as the non-asymptotic analog of the asymptotic state evolution recurrence~\eqref{eq:SE-Montanari}.  
To be more precise, note that if  we assume the empirical distribution of $\{\sqrt{n}\vstar_{i}\}_{1\leq i\leq n}$ converges to some distribution $\mu_{V}$ and generate $V\sim \mu_V, G\sim \mathcal{N}(0,1)$ independently, 
then \eqref{eqn:state-evolution-finite} can be alternatively interpreted as 
\begin{subequations}
\begin{align}
	\frac{\alpha_{t+1}}{\sqrt{n}} &\approx \lambda \Exs \bigg[V \eta_t\Big( \frac{\alpha_t}{\sqrt{n}} V + \frac{\ltwo{\beta_{t-1}}}{\sqrt{n}} G \Big)\bigg] ,  \\
	\frac{\ltwo{\beta_{t}}^2}{n} &\approx \Exs \bigg[\eta_t^2 \Big(\frac{\alpha_t}{\sqrt{n}} V + \frac{\ltwo{\beta_{t-1}}}{\sqrt{n}} G \Big)\bigg] ,
\end{align}
\end{subequations}
which  --- upon proper rescaling ---  is consistent with \eqref{eq:SE-Montanari} as long as $\Delta_{\alpha}$ and $\Delta_{\beta}$ are negligible.




The basic idea of Theorem~\ref{thm:main} is to divide the ultimate goal into multiple sub-tasks, 
motivating us to bound the derivatives stated in Assumption~\ref{assump:eta} and each of the quantities $A_t,\ldots,G_t$ separately.   
This framework is fully non-asymptotic, provided that $A_t,\ldots,G_t$ admits some non-asymptotic bounds as well. 
\red{Given the generality of this result, a natural question arises as to whether it is feasible to control the parameters $\kappa_t, A_t,\ldots,G_t$ to the desired order, and to ensure that the residual term $\xi_{t}$ is sufficiently small for a broad range of iterations. 
In order to demonstrate the efficacy of this framework,   
we carry out the analysis details for two examples in the sequel: $\mathbb{Z}_2$ synchronization and sparse PCA. 
In both of these examples, we are able to demonstrate that such residual terms are exceedingly small.
More details can be seen in Section~\ref{sec:examples} when we embark on the discussion of these two concrete applications. 
}

\begin{itemize}
	\item 
	\red{Let us take a closer inspection on the left-hand side of \eqref{defi:A} concerning $A_t$. 
	Heuristically, consider the idealistic case where  $\mu_{t}$, $\alpha_t$ and $\beta_{t-1}$ are independent of $\{\phi_k\}_{1\leq k\leq t-1}$. 
	By virtue of the celebrated Stein lemma, 
	we can easily show that the quantity $\sum_{k = 1}^{t-1} \mu_t^k\big[\big\langle \phi_k, \eta_{t}(v_t)\big\rangle - \big\langle\eta_t^{\prime}(v_t)\big\rangle \beta_{t-1}^k\big]$ has zero mean.  In addition, this quantity can be viewed as a Lipschitz function of an i.i.d.~Gaussian vector,
	which is expected to concentrate sharply around its mean \citep{massart2007concentration}. 
	Such concentration results can then be extended to accommodate statistically dependent $\mu_{t}$ and $\beta_{t-1}$ via standard covering arguments (see, e.g., the uniform concentration results in Section~\ref{sec:Gaussian-concentration}). 
	Similar ideas can be applied to control $D_t$ (cf.~\eqref{defi:D}), 
	although the expression of $D_t$ is more complicated with non-zero mean. 
	From a technical point of view, the properties of the denoising function, such as being differentiable and Lipschitz continuous, are only mildly exploited in bounding these quantities. 
	As a result, having a specific form of $\eta_{t}$ may not necessarily be useful in obtaining more succinct expressions for Assumption~\ref{assump:A-H-eta}. }

	\item Similarly, the target quantities (excluding the absolute value symbols) that define $B_t$ (cf.~\eqref{defi:B}) and $C_t$ (cf.~\eqref{defi:C}) 
		are also zero-mean Lipschitz functions of i.i.d.~Gaussian vectors, 
		if we take $\alpha_t$ and $\beta_{t-1}$ to be independent of $\{\phi_k\}_{1\leq k\leq t-1}$. 
		As a result, we expect $B_t$ and $C_t$ to be controllable again using uniform Gaussian concentration results.

	\item In terms of quantity $E_t$, which captures the influence of non-differentiable points of the denoising function $\eta_{t}(\cdot)$.
		In those problems with smooth $\eta_t(\cdot)$ (e.g., $\mathbb{Z}_2$ synchronization to be explored in Section~\ref{sec:z2-main}), 
		we have $E_{t} = 0$, which allows for significant simplification of \eqref{eqn:para-general}. 
		Nonetheless, it plays a crucial role in problems with non-differentiable denoising functions, as shall be seen in the example of sparse PCA (in Section~\ref{sec:sparse-main}). 

\end{itemize}

Finally, the signal-to-noise ratio in decomposition~\eqref{eqn:xt-decomposition} is captured by $\frac{\alpha_{t+1}}{\ltwo{\beta_t}}.$ 
Clearly, if throughout the execution of AMP, each $\eta_t$ is properly normalized such that $\ltwo{\eta_t(x_t)} = \|\beta_{t}\|_2 = 1$, then $\Delta_{\beta,t} = 0$ for every $1 \leq t < n$. 
Therefore it is sufficient to focus on the dynamics of $\{\alpha_{t}\}$. 
In such case, the application of Theorem~\ref{thm:main} is further simplified by controlling quantities $A_t, B_t$, $D_t$ and $E_t$.

\subsection{Non-asymptotic analysis of spectrally initialized AMP}
\label{sec:spectral}

Theorems~\ref{thm:recursion}-\ref{thm:main} are concerned with AMP iterates when initialized at a point independent of $W.$  
Caution needs to be exercised, however, when these results are used to accommodate random initialization; 
in fact, when AMP is initialized randomly, while the decomposition still holds true, the error terms $\Delta_{\alpha,t}$ and  $\|\xi_t\|_2$  might not be negligible compared to the signal component, 
thereby calling into question the validity of the asymptotic state evolution formula. 
Alternatively, one might consider AMP with a warm start --- that is, initializing AMP at some informative point. 
Along this line, a common approach to initialize a nonconvex iterative algorithm is the spectral method \citep{chen2021spectral,chi2019nonconvex,montanari2021estimation,keshavan2010matrix}, 
which attempts estimation by computing the leading eigenvector of the data matrix and has proved effective for various low-rank estimation problems. 
Spectrally initialized AMP has previously been analyzed when $t$ is fixed and $n$ approaches infinity \citep{montanari2021estimation,celentano2021local}.


Motivated by the wide use of spectral initialization in practice, 
we pursue an extension of our non-asymptotic analysis framework to accommodate AMP with spectral initialization.  
Recognizing that the leading eigenvector of a large matrix is often computed by means of an iterative power method, 
we consider the following spectral estimate: 
\begin{itemize}
	\item[1)] generate an initial vector $\vseed \in \mathbb{R}^{n}$ uniformly at random on the $n$-dimensional sphere $\mathcal{S}^{n-1}$; 
	\item[2)] run power iteration for $s$ steps (with $s$ to be specified shortly), and yield an estimate 
\begin{align}
\label{eqn:power-initialization}
	x_1 \defn a_s M^s \vseed 
	\qquad \text{with }a_s \defn \frac{1}{\ltwo{M^s \vseed}} 
\end{align}
with $a_s$ the normalization factor. 
\end{itemize}
The reason we study this concrete power method is two fold: (i) it corresponds to the method widely implemented in practice to compute the leading eigenvector in an exceedingly accurate manner; (ii) it is iterative in nature, thus facilitating integration into the AMP analysis framework.

When we employ $x_{1}$ (cf.~\eqref{eqn:power-initialization}) to initialize the AMP algorithm \eqref{eqn:AMP-updates}, 
Theorems~\ref{thm:recursion}-\ref{thm:main} remain valid after slight modification, with an initial signal strength $\alpha_1$ 
that can be characterized accurately using the property of spectral methods. 
Our result is formally stated below;  its proof can be found in Section~\ref{sec:pf-thm-spectral}.


\begin{theos} 
\label{thm:recursion-spectral}
Suppose that the AMP algorithm \eqref{eqn:AMP-updates} is initialized with $x_0, x_1\in \real^n$, 
where $x_1$ is obtained via \eqref{eqn:power-initialization} with $s = \frac{C_v\log n}{(\lambda - 1)^2}$ for some large enough constant $C_v>0$, 
and $x_0$ obeys $\eta_0(x_0) = \frac{1}{\lambda}x_1$.  
Suppose that $1 + C_{\lambda}\big(\frac{\log n }{n}\big)^{1/9} \leq \lambda= O(1)$ for some large enough constant $C_{\lambda }>0$.  
Then for every $0\le t< n - 2s -1 $, the AMP iterates \eqref{eqn:AMP-updates} admit the following decomposition: 
\begin{align}
\label{eqn:xt-decomposition-spectral}
	x_{t+1} = \alpha_{t+1} \vstar + \sum_{k = -2\power}^{t} \beta_{t}^k\phi_k + \xi_{t},
\end{align}
where the $\phi_k$'s are independent obeying 
$\phi_k \overset{\mathrm{i.i.d.}}{\sim} \mathcal{N}(0, \frac{1}{n}I_n)$,  the $\xi_{k}$'s denote some residual vectors, 
and  
\begin{subequations}
\label{eq:alpha-beta-recursion-spect}
\begin{align}
	\alpha_{1} &=\sqrt{1-\frac{1}{\lambda^{2}}}, \qquad
	\alpha_{t+1} = \lambda v^{\star\top} \eta_{t}(x_t), \\
	\|\beta_t\|_2 &\defn \big\| \big(\beta_t^{-2\power},\ldots,\beta_t^{0},\beta_t^1,\beta_t^2,\ldots,\beta_t^t \big) \big\|_2 
	= \lt\|\eta_{t}(x_t)\rt\|_2. 
\end{align}
\end{subequations}
In particular,   
there exist some unit vectors $\{\mu_t\}$ with $\mu_t = [\mu_t^{-2s},\ldots, \mu_t^{t}]\in \mathbb{R}^{t+2\power+1}$ obeying
\begin{align}
\label{eqn:xi-norm-main-spectral}
	\|\xi_{0}\|_2 &\lesssim \frac{\log^{3.5} n}{\sqrt{(\lambda - 1)^9n}} \\
	\|\xi_{t}\|_2 &= \Big\langle \sum_{k = -2\power}^{t-1} \mu_t^k\phi_k, \delta_{t}\Big\rangle - \langle\delta_{t}^{\prime}\rangle \sum_{k = -2\power}^{t-1} \mu_t^k\beta_{t-1}^k + \Delta_t 
	+ O\Big(\sqrt{\frac{t\log n}{n}}\|\beta_{t}\|_2\Big),
	\qquad  1 \leq t < n-2s-1 
\end{align}
with probability at least $1-O(n^{-11})$, 
where we define
\begin{subequations}
\label{eqn:delta-chorus-spectral}
\begin{align}
	v_t &\defn \alpha_t\vstar + \sum_{k = -2\power}^{t-1} \beta_{t-1}^k\phi_k \\
	\label{defn:Delta-t-spectral} \Delta_t &\defn 
	\sum_{k = -2\power}^{t-1} \mu_t^k \Big[\big\langle \phi_k, \eta_{t}(v_t)\big\rangle - \big\langle\eta_t^{\prime}(v_t)\big\rangle \beta_{t-1}^k\Big], \\
	\label{defn:delta-t-spectral} 
	\delta_{t} &\defn \eta_{t}(x_t) - \eta_{t}(v_t), \\
	\label{defn:delta-prime-t-spectral} 
	\delta_{t}^{\prime} &\defn \eta_{t}^{\prime}(x_t) - \eta_{t}^{\prime}(v_t).
\end{align}
\end{subequations}
\end{theos}

%
%

Akin to Theorem~\ref{thm:recursion}, 
Theorem~\ref{thm:recursion-spectral} provides a non-asymptotic characterization for each AMP iterate in the presence of spectral initialization. 
Even though the power method does not resemble the AMP update rule, 
spectrally initialized AMP shares the same decomposition structure as in Theorem~\ref{thm:recursion}, 
except that many summations therein include $2s+1$ more vectors/coefficients in order to incorporate the influence of spectral methods.  
As can be anticipated, one can immediately derive a counterpart of Theorem~\ref{thm:main} in the presence of spectral initialization by properly modifying Assumption~\ref{assump:A-H-eta}.


\begin{cors}
\label{cor:recursion-spectral}
Consider the setting of Theorem~\ref{thm:recursion-spectral}. 
Suppose Assumptions~\ref{assump:eta}-\ref{assump:A-H-eta} are valid, 
except that the summations $\sum_{k=1}^{t-1}$ are replaced by $\sum_{k=-2s}^{t-1}$.  
With probability at least $1-O(n^{-11})$, 
the AMP iterates admit the decomposition~\eqref{eqn:xt-decomposition-spectral} with $\alpha_{t+1}$ and $\ltwo{\beta_t}$ obeying \eqref{eqn:state-evolution-finite}  and the error terms satisfying \eqref{eqn:para-general}.
\end{cors}


On a technical level, the main step towards proving Theorem~\ref{thm:recursion-spectral} consists of showing that the spectral initialization admits a similar decomposition
\begin{align}
\label{eqn:spectral-x1-gist}
	x_1 = \alpha_1 v^{\star} + \sum_{i = -2s}^{0} \beta_0^k\phi_k + O\Big(\frac{\log^{3.5} n}{\sqrt{(\lambda - 1)^7n}}\Big), 
\end{align}
for a set of $2s+1$ i.i.d.~Gaussian vectors $\{\phi_k\}_{-2s\leq k\leq 0}$. 
The primary challenge in establishing this result stems from the fact that $x_{1}$ relies heavily on $W$, 
which cannot be easily decoupled from $W$ as in Theorem~\ref{thm:recursion} with an independent initialization.  
Informally, a key observation that helps overcome this challenge (as shall be detailed in Section~\ref{sec:rep-spec}) 
is the following decomposition owing to power iterations: 
\begin{align}
	x_1 = \sum_{i = 0}^{s-1} a_iW^{i}v^{\star} + \frac{1}{\ltwo{M^s \vseed}}W^s\vseed, 
\end{align}
for certain coefficients $a_{0},\ldots,a_{s-1}\in \real$,
where  $\vseed$ is the initial vector for power iterations.  
Inspired by this decomposition, we attempt to construct an orthonormal basis of $2s+1$ dimension 
that covers $x_1$ perfectly, which can be accomplished by some AMP-style algorithms. 
These auxiliary AMP sequences can then be merged with the subsequent AMP updates, 
providing a sensible way to invoke Theorem~\ref{thm:recursion}. 

\begin{remark}
	It is worth pointing out that in our non-asymptotic analysis, it is critical to ensure that $x_{1}$ lies perfectly within the constructed $(2s+1)$-dimensional subspace; otherwise, the leakage term --- albeit of tiny magnitude --- might ruin the key independence structures that underlie our theory (to be made precise in Lemma~\ref{lem:distribution}). This issue, however, does not manifest itself if one only aims for an asymptotic characterization, making our incorporation of spectral initialization more intricate  compared to the asymptotic counterpart in \cite{montanari2021estimation}. 
\end{remark}

\section{Consequences for specific models}
\label{sec:examples}

Focusing on two important models (i.e., $\mathbb{Z}_{2}$ synchronization and sparse spiked Wigner models), 
this section develops concrete consequences of our general recipe presented in Section~\ref{sec:main},  
aimed at illustrating the effectiveness of our non-asymptotic theory.



\subsection{$\mathbb{Z}_{2}$ synchronization}
\label{sec:z2-main}

The first concrete model considered here is $\mathbb{Z}_{2}$ synchronization, 
which augments \eqref{eqn:wigner} with some binary-valued signal structure as follows:
\begin{align}
\label{eqn:wigner-Z2}
	M = \lambda v^\star v^{\star\top} + W \in \mathbb{R}^{n\times n},
	\qquad \text{where }v_i^{\star}\in \Big\{\frac{1}{\sqrt{n}}, -\frac{1}{\sqrt{n}} \Big\}, 
	~ 1\leq i\leq n.
\end{align}
It can be viewed as a special example of synchronization over compact groups \citep{singer2011angular,perry2018message,zhong2018near,gao2022sdp}. 
Given this observation matrix and a signal prior (e.g., $ \vstar_i \stackrel{\textrm{i.i.d.}}{\sim} \textsf{Unif}\big\{\frac{1}{\sqrt{n}}, -\frac{1}{\sqrt{n}}\big\}$), the Bayes-optimal estimate for the rank-one matrix $\vstar v^{\star\top}$ takes the following form:
\begin{align}
	\Xbayes \defn \myE[v^\star v^{\star\top} \mid M].
\end{align}
Computing the Bayes-optimal solution is, however, computationally infeasible due to the combinatorial nature of the underlying optimization problem. 
A recent line of research searched for nearly tight yet tractable approximation to the Bayes-optimal estimator 
\citep{peche2006largest,baik2005phase,javanmard2016phase,fan2021tap,montanari2016semidefinite}, 
with AMP being one natural choice  \citep{deshpande2017asymptotic,celentano2021local,lelarge2019fundamental}.

Recall that the majority of AMP analysis for $\mathbb{Z}_{2}$ synchronization 
operates under the assumption that $n\rightarrow \infty$ and $t$ stays fixed.  
In order to obtain an optimal estimator with finite-sample guarantees in the most challenging regime $\lambda >1$,  
the recent work \cite{celentano2021local} proposed a three-stage hybrid algorithm: 
(i) starting with a spectral initialization, (ii) running AMP updates for constant number of steps, (iii) refining by running, say, natural gradient descent method, until convergence.   
This procedure yields a polynomial-time algorithm that converges to a local minimizer $m_{\star}$ of the so-called TAP free energy (which obeys $\|m_{\star}m_{\star}^\top - \Xbayes\|_{\mathrm{F}}\to 0$ in probability)\footnote{Note here, to be consistent with other parts of the paper, we adopt a different scaling by taking $\ltwo{m_{\star}} = 1$}. 
\cite{celentano2021local} further conjectured based on numerical experiments that a spectrally initialized AMP might be actually sufficient (in the absence of a third refinement stage). This raises a natural theoretical question:  
\begin{center}
	\emph{How does spectrally initialized AMP perform when $t$ far exceeds a constant or even $o\big(\frac{\log n}{\log \log n}\big)$? }
 \end{center} 
As discussed in \cite{celentano2021local}, existing state-evolution-based arguments fell short in answering this question due to their asymptotic nature. 
In the following, we aim to answer the question positively, with the aid of our non-asymptotic framework developed in this paper.



\paragraph{Spectrally initialized AMP for $\mathbb{Z}_{2}$ synchronization.}

Let us begin by formalizing the AMP procedure to be studied herein.  
Specifically, the AMP updates take the following form for each $t\geq 1$: 
\begin{subequations}
\label{eq:AMP-z2}
\begin{align}
\label{eqn:AMP-updates-Z2}
	x_{t+1} = M\eta_t(x_{t}) - \big\langle\eta_t^{\prime}(x_{t}) \big\rangle \cdot \eta_{t-1}(x_{t-1}) 
	\qquad \text{with } \eta_t(x) = \gamma_t\tanh\lt(\pi_tx\rt),  
\end{align}
where 
\begin{align}
\label{eqn:eta-z2-new}
	\pi_t \defn \sqrt{n(\|x_t\|_2^2-1)}\qquad\text{and}
	\qquad\gamma_t \defn \lt\|\tanh\lt(\pi_tx_t\rt)\rt\|_2^{-1}. 
\end{align} 
Here, the pre-factor $\gamma_{t}$ is chosen to ensure $\ltwo{\eta_{t}(x_t)} = 1$ for normalization purpose
(note that this differs from the pre-factor adopted in \cite{celentano2021local}). 
	As already recognized in prior work, a properly rescaled $\tanh(\cdot)$ function is capable of approaching the Bayes-optimal estimator. 
The first iterate $x_1$ is obtain via the spectral method, or more precisely, the power method, that is,  
\begin{align}
\label{eqn:Z2-initialization}
	x_1 \defn \lambda  a_s M^s \vseed\qquad\text{with }~
	s \asymp \frac{\lambda^2 \log n}{(\lambda - 1)^2} ~~\text{and}~~ a_s = \frac{1}{\ltwo{M^s \vseed}}. 
\end{align} 
\end{subequations}
We shall also choose $x_0$ such that $\eta_0(x_0) = x_1/\lambda$ to be consistent with Theorem~\ref{thm:recursion-spectral}. 
Given that it is infeasible to distinguish $\vstar$ and $-\vstar$ given only the observation $M$, 
we shall assume --- without loss of generality --- $x_1^\top \vstar \geq 0$ throughout the rest of the paper.

\paragraph{Non-asymptotic theoretical guarantees.} 
We now invoke our general recipe to analyze the non-asymptotic performance of \eqref{eq:AMP-z2}. 
In order to do so, we find it helpful to first introduce the (limiting version of) state evolution (SE) tailored to the denoising function $\eta_t(\cdot)\propto \tanh(\cdot)$. 
Specifically, let us produce a scalar sequence $\{\tau_t\}$ recursively as follows:
\begin{align}
\label{eqn:tau-t-z2}
	\tau_{1} \defn \lambda^{2}-1 \qquad \text{and} \qquad
	\tau_{t+1} \defn 
	\lambda^2 \int \tanh(\tau_t + \sqrt{\tau_t}x)\varphi(\dx) ,
	\quad t\geq 1,
\end{align}
where $\varphi(\cdot)$ represents the pdf of  $\mathcal{N}(0,1)$. 
Note that this SE recurrence is consistent with what has been derived in the prior work \cite{celentano2021local}. 
With this in mind, we state in Theorem~\ref{thm:Z2} our non-asymptotic characterization for the AMP algorithm,  
whose proof can be found in Section~\ref{sec:pf-thm-Z2}.

\begin{theos} 
\label{thm:Z2}
Consider the model \eqref{eqn:wigner-Z2} with $1+ \frac{\log n}{n^{1/16}} < \lambda \leq 1.2$, and recall the scalar sequence $\{\tau_t\}$ in \eqref{eqn:tau-t-z2}.
With probability at least $1 - O(n^{-11})$, the spectrally initialized AMP \eqref{eq:AMP-z2} admits the following decomposition: 
\begin{align}
\label{eqn:z2-decomposition}
	x_{t+1} = \alpha_{t+1} \vstar + \sum_{k = -2s}^{t} \beta_{t}^k\phi_k + \xi_{t} 
	\qquad \text{for all } 0\leq t = o\lt(  \frac{n(\lambda - 1)^{10}}{\log^{7} n} \rt),
\end{align}
where the $\phi_k$'s are i.i.d.~random vectors drawn from $\mathcal{N}\big(0, \frac{1}{n}I_n\big)$, and  the parameters satisfy 
\begin{subequations}
\label{eqn:z2-final}
\begin{align}
	%
	\alpha_1^2 &= \lambda^2 - 1, \\
	\alpha_{t+1}^2 &= \lambda^2 \big( v^{\star\top} \eta_{t}(x_{t}) \big)^2= \left(1+O\bigg(\sqrt{\frac{t\log n}{(\lambda-1)^{8}n}} + 
	\frac{\log^{3.5}n}{\sqrt{(\lambda-1)^{14}n}}\bigg)\right)\tau_{t+1}, 
	\qquad t \geq 1, \label{eqn:z2-delta-alpha-final}\\
	\|\beta_{t}\|_2 &= \big\| \big[ \beta_{t}^{-2s},\cdots, \beta_{t}^{t} \big] \big\|_2  = 1,\\
	\|\xi_{t}\|_2 &\lesssim \sqrt{\frac{t\log n}{(\lambda-1)^{3}n}}+\frac{\log^{3.5}n}{\sqrt{(\lambda-1)^{9}n}} . 
\end{align}
%
%
\end{subequations}
\end{theos}
\begin{remark}
	Note that the assumption $\lambda \leq 1.2$ is not necessary and can be safely eliminated. 
	We assume $\lambda \leq 1.2$ for two reasons: (i) it represents the most challenging regime for $\mathbb{Z}_2$ synchronization; 
	(ii) assuming $\lambda \leq 1.2$ allows us to streamline some (non-critical) part of the proof. 
\end{remark}

In words, Theorem~\ref{thm:Z2} captures the finite-sample dynamics of the AMP \eqref{eq:AMP-z2} up to 
$o\big( \frac{n(\lambda - 1)^{10}}{\log^7 n} \big)$ iterations.  
Each iterate is very well approximated by a superposition of a signal component and a Gaussian component, up to a small error at most on the order of $\sqrt{\frac{t\log n}{(\lambda-1)^{3}n}}+\frac{\log^{3.5}n}{\sqrt{(\lambda-1)^{9}n}}$. Recognizing that $\|\beta_{t-1}\|_2=1$, one arrives at the following heuristic approximation: 
\begin{equation}
	x_t \approx \alpha_t \vstar + \mathcal{N}\Big( 0, \frac{1}{n}I_n\Big), 
\end{equation}
which can be rigorized under 1-Wasserstein using standard Gaussian concentration results (see, e.g., Lemma~\ref{lem:wasserstein}); 
this is consistent with the prediction of prior works (e.g., \citet{deshpande2017asymptotic}) under high-dimensional asymptotics (up to proper rescaling). 
To the best of our knowledge, Theorem~\ref{thm:Z2} delivers the first finite-sample characterization of AMP in the $\mathbb{Z}_{2}$ synchronization setting beyond $O_n(1)$ iterations. 
As asserted by the result \eqref{eqn:z2-delta-alpha-final} in Theorem~\ref{thm:Z2},  
the strength of the signal component in $x_t$ remains fairly close to the prediction of state evolution \eqref{eqn:tau-t-z2}, that is, 
\begin{equation}
	\frac{\big( \big\langle \vstar, \, \eta_t(x_t) \big\rangle \big)^2}{ \|\vstar\|_2^2 \| \eta_t(x_t) \|_2^2   } = \frac{\tau_{t+1}}{\lambda^2} 
	\left(1+O\bigg(\sqrt{\frac{t\log n}{(\lambda-1)^{8}n}} + 
	\frac{\log^{3.5}n}{\sqrt{(\lambda-1)^{14}n}}\bigg)\right) 
\end{equation}
up to $o\big( \frac{n(\lambda - 1)^{10}}{\log^7 n} \big)$ iterations, where we recall that $\| \eta_t(x_t) \|_2=\|\vstar\|_2=1$.

We also make note of a phase transition phenomenon that has been established in \citet{deshpande2017asymptotic}. Namely,  when $\lambda < 1$, the Bayes-optimal estimate converges to the zero estimator, 
meaning that no estimator whatsoever is able to obtain non-trivial estimation; in contrast, when $\lambda > 1$, it is possible to achieve non-trivial correlation with the underlying signal. 
Therefore, it suffices to focus on the scenario where $\lambda>1$. 
It is worth emphasizing that our result is fully non-asymptotic in terms of the spectral gap $\lambda - 1$ as well. 
In fact, our theory allows $\lambda$ to be exceedingly close to 1 (i.e., $\lambda - 1 = o_n(1)$), 
which is in sharp contrast to prior works that all required $\lambda\geq 1+\epsilon$ for some strictly positive constant $\epsilon$. 
Note that we have made no efforts to obtain the sharpest constant  
in the assumption $\lambda > 1+ \frac{\log n}{n^{1/16}}$, the $1/16$ herein is likely to be improved with more careful book-keeping.

\begin{remark}
As pointed out by \citet[Lemma A.7]{celentano2021local}, in the large $n$ limit, the AMP algorithm yields matching asymptotic performances as that of the Bayes-optimal estimator, in the sense that 
\begin{align*}
	\lim_{t\to \infty}\lim_{n \to \infty} \big\|\vstar v^{\star\top} - c_t \eta_t(x_t) \eta_t(x_t)^\top \big\|^2_{\mathrm{F}}
	=\lim_{t\to \infty}\lim_{n \to \infty} \big\|\vstar v^{\star\top} - \widehat{X}^{\mathrm{bayes}} \big\|^2_{\mathrm{F}}
\end{align*}
for some constant $c_t >0.$ This further implies that the minimum mean square estimation error is dictated by the (unique) fixed point of the state evolution recursion \eqref{eqn:tau-t-z2}. 
In addition, our proof of Theorem~\ref{thm:Z2} also makes explicit the convergence rate of $\tau_t$ to $\tau^{\star}$. 
To be more precise, as we shall demonstrate in Section~\ref{sec:main-recursion-z2} (see, e.g., discussions around \eqref{eq:SE-induction} and \eqref{eqn:middle}), we have
\begin{align*}
	\ltwo{\tau_{t+1} - \taustar} \leq \big(1 - (\lambda-1) \big)\ltwo{\tau_t - \taustar}. 
\end{align*}
%
This taken collectively with Theorem~\ref{thm:Z2} leads to
\begin{align}
	\alpha_t^2 - \tau^{\star} = (\lambda^2-1)\big(1 - (\lambda-1) \big)^t + 
	 O\bigg(\sqrt{\frac{t\log n}{(\lambda-1)^{8}n}} + \frac{\log^{3.5}n}{\sqrt{(\lambda-1)^{14}n}}\bigg),
\end{align}
which captures how far $\alpha_t^2$ deviates from the asymptotic limit as the iteration number $t$ increases. 
This helps answer a natural question regarding the finite-sample convergence property of spectrally initialized AMP. 
%

\end{remark}

\subsection{Sparse PCA (sparse spiked Wigner matrix)}
\label{sec:sparse-main}

Another specific model of interest is concerned with sparse PCA. 
In the statistics literature,  spiked models with sparsity constraints have been a main-stay  
for studying sparse PCA \citep{johnstone2009consistency}, 
inspiring various algorithms including regression-type methods \citep{zou2006sparse}, 
convex relaxation \citep{amini2008high,d2004direct,vu2013fantope}, iterative thresholding \citep{ma2013sparse,krauthgamer2015semidefinite,deshpande2014sparse}, 
sum of squares hierarchy \citep{hopkins2017power}, among many others. 
This paper contributes to this growing literature by studying the effectiveness of AMP for sparse PCA (see also,~\cite{deshpande2014information,montanari2021estimation}).

More specifically, this subsection considers sparse estimation in the spiked Wigner model\footnote{Note that another popular model for sparse PCA is the sparse spike Wishart model \cite{johnstone2009consistency}. We choose the spiked Wigner model as it is closer to the context studied in this paper.}, where  
we seek to estimate a $k$-sparse eigenvector $\vstar \in \mathcal{S}^{n-1}$ from the following data matrix: 
\begin{align}
\label{eqn:wigner-sparse}
	M = \lambda v^\star v^{\star\top} + W \in \mathbb{R}^{n\times n},
	\qquad \text{where } \|\vstar\|_0=k.
\end{align}
We would like to leverage upon our analysis framework to track the non-asymptotic performance of AMP in the face of the sparsity structure.


\paragraph{AMP for sparse spiked Wigner models.} 

For each $t\geq 1$, the AMP update rule takes the following form: 
\begin{subequations}
\label{eq:AMP-sparse}
\begin{align}
\label{eqn:AMP-updates-sparse}
	x_{t+1} = M\eta_t(x_{t}) - \big\langle\eta_t^{\prime}(x_{t}) \big\rangle \cdot \eta_{t-1}(x_{t-1}) 
	\qquad \text{with } \eta_t(x) = \gamma_t\mathrm{sign}(x) \circ (|x| - \tau_t 1 )_{+},  
\end{align}
where the denoising function  $\eta_t(\cdot)$ is taken to be the soft thresholding function (applied entry-by-entry) with a threshold $\tau_t$ 
	and a rescaling pre-factor to ensure $\|\eta_t(x_t)\|_2 = 1$: 
\begin{align}
\label{eqn:eta-sparse}
	\gamma_t \defn \big\|\mathrm{sign}(x_t) \circ (|x_t| - \tau_t 1)_{+} \big\|_2^{-1}.
\end{align}
\end{subequations}
It is worth noting that $\eta_{t}$ is differentiable almost everywhere except for two points (i.e., $\pm \tau_t$), 
with $\eta_{t}^{\prime}(x) = \gamma_t\mathds{1}(|x| > \tau_t).$ In addition, the threshold $\tau_t$ shall be selected to be $\tau_t\asymp\sqrt{\frac{\log n}{n}}$, 
to be specified shortly.

\subsubsection{Non-asymptotic AMP theory with an independent initialization}

To begin with, we characterize the performances of AMP when an informative yet independent initialization is available.  
For notational simplicity, we define the following function:
\begin{align}
\label{eqn:franck}
f(\alpha) \defn \frac{\lambda v^{\star \top} \displaystyle \int \mathsf{ST}_{\tau_t}\left(\alpha \vstar + \frac{x}{\sqrt{n}} \right)\varphi_n(\dx)}{\sqrt{\displaystyle \int\Big\|\mathsf{ST}_{\tau_t}\left(\alpha \vstar + \frac{x}{\sqrt{n}} \right)\Big\|_{2}^{2}\varphi_{n}(\dx)}},
\end{align}
where $\mathsf{ST}_{\tau_t}(x) \defn \mathsf{sign}(x)(|x| - \tau_t)_{+}$ for any $x\in \real$ and $\varphi_n(\cdot)$ is the pdf of $\mathcal{N}\big(0, I_n\big)$.
Let us introduce the state evolution recursion as follows (which depends only on $\lambda$ and $\vstar$):
\begin{align}
\label{eqn:alpha-star-sparse}
	\alpha_{t+1}^{\star} = f(\alpha_t^\star), 
\end{align}
with the initial condition obeying $\alpha^{\star}_2 \asymp \lambda.$ 
%
%
Our non-asymptotic theory for sparse PCA is stated below, 
 with its proof deferred to Section~\ref{sec:pf-sparse}.
\begin{theos}
\label{thm:sparse}
Consider the model \eqref{eqn:wigner-sparse} where $0<\lambda \lesssim 1$. 
Given an independent initial point $x_{1}$ obeying $\inprod{\vstar}{\eta_1(x_1)} \asymp 1$ and $\eta_0(x_0) = 0$, 
the AMP algorithm~\eqref{eq:AMP-sparse} satisfies the following decomposition:
\begin{align}
\label{eqn:sparse-decomp-dvorak}
	&x_{t+1} = \alpha_{t+1} \vstar + \sum_{k = 1}^{t} \beta_{t}^k\phi_k + \xi_{t},
	\qquad\text{for }t\geq 1, \\
\label{eqn:sparse-SE}
	\text{with }~
	\alpha_{t+1} = & \lambda v^{\star \top} \int{\eta}_{t}\left(\alpha_{t} \vstar + \frac{x}{\sqrt{n}} \right)\varphi(\dx) + \lambda\Delta_{\alpha,t},
\qquad
\|\beta_{t-1}\|_2 = 1, 
\end{align}
where it holds with probability at least $1-O(n^{-11})$ that
\begin{subequations}
\label{eqn:soccer}
\begin{align}
\lambda |\Delta_{\alpha, t}| &\lesssim \sqrt{\frac{k+t\log^3 n}{n}}, 
\qquad \|\xi_{t}\|_2 \lesssim \sqrt{\frac{k+t\log^3 n}{n}}, 
\label{eq:residual-sparse} \\
\label{eqn:se-alpha-sparse}
	&\big|\alpha_{t+1} - \alpha_{t+1}^{\star}\big| \lesssim \sqrt{\frac{k \log n + t\log^3 n}{n}},   
\end{align}
\end{subequations}
provided that 
\begin{align}
	t \lesssim \frac{n\lambda^2}{\log^3 n}\qquad \text{and} \qquad \frac{k\log n}{n\lambda^2} \lesssim 1. \label{cond:t-k}
\end{align}
\end{theos}

In a nutshell, each AMP iterate behaves almost like a signal component superimposed by a Gaussian-type component (see the decomposition \eqref{eqn:sparse-decomp-dvorak}), 
with an residual term that is well controlled  up until the number of iterations reaches $$O\left(\frac{n\lambda^2}{\log^3 n}\right).$$
If $\lambda \asymp 1$, 
then the validity of the above non-asymptotic theory is guaranteed for $O(n/\log^3 n)$ iterations, 
which is far beyond what existing theory can cover. 
It is also worth pointing out that the non-asymptotic state-evolution~\eqref{eqn:se-alpha-sparse} matches the one derived in existing literature (cf.~\eqref{eq:SE-Montanari}) when $\frac{k\log n + t\log^3 n}{n} \to 0$. 
In the sequel, we single out a few additional remarks of this result in order.

\begin{itemize}

\item 
 In comparison to several prior works (e.g., \citet{amini2008high,montanari2021estimation,ding2019subexponential}),  our results impose no assumption on either the empirical distribution of $\vstar$, or the values of the non-zero entries of $\vstar$.  
For instance, we do allow some non-zero entries of $\vstar$ to be either extremely large or exceedingly small.





\item Different from \cite{montanari2021estimation}, we permit $\lambda$ to enter the regime where $\lambda < 1$. 
	Note that in this regime, the leading eigenvector of the observed matrix $M$ becomes uninformative \citep{deshpande2017asymptotic}, 
	and therefore, vanilla spectral initialization fails to provide a warm start as required in \cite{montanari2021estimation}. Fortunately, it is still possible to obtain an informative estimate for sparse PCA in the regime where $\lambda < 1$, as long as the sparsity structure is properly exploited.

\item In fact, assuming access to an informative initialization independent of $W$, Theorem~\ref{thm:sparse} only requires $\lambda \gtrsim \sqrt{\frac{k\log n}{n}}$.
This threshold matches the known information-theoretical lower bound in order to enable consistent estimation; 
see also \citet{vu2012minimax,berthet2013optimal} for relevant messages derived for the spiked covariance model. 
\red{
In sharp contrast with $\mathbb{Z}_{2}$ synchronization, where the information-theoretical limit appears at $\lambda = 1$, here it is still possible to recover the signal for $\lambda \ll 1$ by cleverly making use of the sparsity structure, although the leading eigenvector is uninformative. 
}



\end{itemize}

Finally, an informative starting point is not always available, particularly when it is close to the information-theoretic threshold.  
Noteworthily, 
a growing body of sparse PCA literature provided evidence concerning 
the existence of  computational barriers that prevent one from finding polynomial-time algorithms to approach the information-theoretic limits \citep{berthet2013computational,lesieur2015phase,krzakala2016mutual,hopkins2017power,macris2020all}.
In light of this, we study AMP with two data-driven initialization schemes that achieve non-trivial correlation with the true spike, focusing on the scenario where the SNR rises above the computational limit. Specifically, we develop in Section~\ref{sec:init-sparse-pca}, two initialization procedures to tackle the strong and weak SNR regimes, detailed in Section~\ref{sec:result-sparse-init} and \ref{sec:result-sparse-init-weak} respectively.
\section{Main analysis}
\label{sec:main-analysis}

We present the proofs of Theorem~\ref{thm:recursion} and \ref{thm:main} in this section and defer other technical details and lemmas to the appendices.

\subsection{Proof of Theorem~\ref{thm:recursion}}
\label{sec:pf-thm-recursion}

We carry out the main analysis for Theorem~\ref{thm:recursion} in the following three steps. 

\paragraph{Step 1: constructing a key set of auxiliary sequences.}

Let us first introduce a sequence of auxiliary vectors/matrices $\{z_k, W_k, \zeta_k\}_{1 \le k \le n}$ in a recursive manner as follows. 
\begin{itemize}
	\item[(i)] With the Wigner matrix $W$ and the initialization $x_{1}$ (pre-selected independent of $W$) in place, we define  
\begin{subequations}
\label{eqn:z-w-recursion}
\begin{align}
\label{eqn:z-w-init}
	z_1 \defn \frac{\eta_1(x_1)}{\lt\|\eta_1(x_1)\rt\|_2} \in \real^n
	\qquad\text{and}\qquad 
	W_1 \defn W \in \real^{n\times n},
\end{align}
which are statistically independent from each other. 
	\item[(ii)] For any $2 \leq t \leq n$, 
concatenate the $z_{k}$'s into a matrix $U_{t-1} \defn [z_k]_{1 \le k \leq t-1} \in \real^{n\times (t-1)}$
and set 
\begin{align}
	z_t &\defn \frac{\lt(I_n - U_{t-1}U_{t-1}^{\top}\rt)\eta_{t}(x_{t})}{\lt\|\lt(I_n - U_{t-1}U_{t-1}^{\top}\rt)\eta_{t}(x_{t})\rt\|_2}, 
	\label{eqn:zt}\\
	W_t &\defn \lt(I_n - z_{t-1}z_{t-1}^{\top}\rt)W_{t-1}\lt(I_n - z_{t-1}z_{t-1}^{\top}\rt),
	\label{eqn:Wt}
\end{align}
where $\{x_t\}$ is the sequence generated by the AMP updates~\eqref{eqn:AMP-updates}. 
\end{subequations}
\end{itemize}
In view of these definitions, we immediately single out the following basic fact. 
\begin{lems} \label{lemma:zk-orthonormal}
	The set of vectors $\{z_k\}_{1\leq k\leq n}$ forms an orthonormal basis. 
\end{lems}
\begin{proof}
	First, it is clear that $U_1=z_1$ consists of orthonormal columns. 
	Next, suppose that $U_{t-1}$ contains orthonormal columns for some $t$, then   $I_n-U_{t-1}U_{t-1}^{\top}$ forms a projection matrix onto the subspace perpendicular to $U_{t-1}=[z_1,\cdots,z_{t-1}]$, and hence $\langle z_t, z_k\rangle = 0$ for all $1\leq k\leq t-1$ (cf.~\eqref{eqn:zt}). This implies that $U_t$ also consists of orthonormal columns. An induction argument thus concludes the proof. 
\end{proof}
%
%
As it turns out, $\{z_i\}_{1\leq i\leq t}$ assists in obtaining a useful decomposition of $\eta_t(x_t)$. 
By construction, for each $t$ we have 
$\eta_t(x_t) \in \mathsf{span}\big\{ z_t, U_{t-1} \big\} = \mathsf{span}\big\{ z_t, \cdots, z_{1} \big\}$.   
This together with Lemma~\ref{lemma:zk-orthonormal} allows us to decompose 
\begin{align}
\label{eqn:eta-decomposition}
	\eta_{t}(x_{t}) = \sum_{k = 1}^{t} \beta_{t}^kz_k, \qquad\text{with }\beta_{t}^k \defn \big\langle\eta_{t}(x_{t}), z_k \big\rangle ~~~(1\leq k\leq t),
\end{align}
which satisfies
\begin{align}
	\lt\|\eta_{t}(x_{t})\rt\|_2 = \lt\|\beta_{t}\rt\|_2 \qquad \text{with }~ \beta_{t} \defn \big(\beta_t^1,\beta_t^2,\ldots,\beta_t^t \big)^{\top} \in \real^{t}.
\end{align}

\paragraph{Step 2: deriving distributional properties of $W_kz_k$.}

Next, we look at some useful distributional properties of $W_kz_k$. 
Towards this end, let us generate another set of auxiliary vectors 
\begin{align}
\label{eqn:zeta-k}
	\zeta_k \defn \Big(\frac{\sqrt{2}}{2} - 1\Big) z_kz_k^{\top}W_kz_k + \sum_{i = 1}^{k - 1} g_i^kz_i, 
	\qquad 1\leq k\leq n,
\end{align}
where the $g_i^k$'s are independently drawn from $\mathcal{N}(0, \frac{1}{n})$. 
As it turns out, we can characterize the distribution of the superposition of $W_kz_k$ and $\zeta_k$, 
as stated in the following lemma.   
\begin{lems} 
\label{lem:distribution}
With $\{z_k, W_k, \zeta_k\}_{1 \le k \le n}$ defined as above, one has 
\begin{align}
\label{def:phi_k}
	\phi_k \defn W_kz_k + \zeta_k \sim \mathcal{N}\lt(0, \frac{1}{n}I_n\rt),\qquad\text{for all }1 \le k \le n.
\end{align}
Further, $\{\phi_k\}_{1\leq k\leq n}$ are statistically independent. 
{}
\end{lems}
In words, when properly augmented by i.i.d.~Gaussians in the directions $\{z_i\}_{1\leq i < t}$ and adjusting the size of $W_kz_k$ along the direction $z_k$, we arrive at an i.i.d.~Gaussian vector. 
The proof is postponed to Section~\ref{sec:pf-distribution}.

\red{	Let us take a moment to explain the intuition behind the introduction of $\zeta_{k}$. This idea can be elucidated by examining the first two iterations. 
\begin{itemize}
\item Given an initial point $x_1$ independent of $W$, it can be easily verified that $z_1^\top W_1 z_1 \sim \mathcal{N}(0, 2/n)$ and 
\begin{align}
W z_1 \sim \mathcal{N}\bigg(0, \frac{1}{n}I_n + \frac{1}{n}z_1z_1^{\top}\bigg),
\end{align}
where we recall that $z_1 \defn \eta_1(x_1)/\big\|\eta_1(x_1)\big\|_2.$
In other words, $W z_1$ exhibits an inflated variance along the direction $z_{1}$, due to the fact that the diagonal entries of $W$ have a higher variance (namely, $2/n$) than that of the off-diagonal entries (namely, $1/n$). 
Thus, if we introduce $\zeta_{1} \defn (\sqrt{2}/{2} - 1) z^\top_{1} W z_1 \cdot z_1 $ (where $z^\top_{1} W z_1 \cdot z_1$ corresponds to the projection of $W z_{1}$ to the direction $z_{1}$),
we can rewrite $W z_1$ by adding and subtracting $\zeta_{1}$ as follows:
\begin{align}
	Wz_1 = \underbrace{W z_1 + \zeta_1}_{\phi_1} + \underbrace{(- \zeta_1)}_{\xi_1}.
\end{align} 
Here, $\zeta_{1}$ helps reduce the variance of $W z_1$ along the direction $z_{1}$, given that now the projection of $\phi_{1}$ to $z_{1}$ equals 
\begin{align*}
	z_1^\top  \phi_1 \cdot z_1 = (z_1^\top W z_1 + z_1^\top \zeta_1) \cdot z_1 = \frac{\sqrt{2}}{2}z^\top_{1} W z_1 \cdot z_1.
\end{align*}
Recognizing that $z_1^\top W_1 z_1 \sim \mathcal{N}(0, 2/n)$, we can see that $\frac{\sqrt{2}}{2}z^\top_{1} W z_1 \sim \mathcal{N}(0,1)$, and as a result, $\phi_{1}\sim \mathcal{N}\big(0, \frac{1}{n} I_n\big)$.  
\item Similarly, let us take one step further to look at the case with $t=2$. By virtue of property~\eqref{defn:zWz} of Lemma~\ref{lem:distribution},  we have
\begin{align}
W_2z_2 \sim \mathcal{N}\bigg(0, \frac{1}{n}I_n - \frac{1}{n}z_1z_1^{\top} + \frac{1}{n}z_2z_2^{\top}\bigg),\quad\text{conditioned on }x_2.
\end{align}
Again, by defining $\zeta_2 \defn \Big(\frac{\sqrt{2}}{2} - 1\Big) z_2^{\top}W_2z_2 \cdot z_2 + g_1^k z_1$ with $g_1^k \sim \mathcal{N}(0,1/n)$ being an independent Gaussian random variable,  
we can decompose $W_2z_2$ as follows
\begin{align*}
	W_2z_2 = \underbrace{W_2z_2 + \zeta_2}_{\phi_2} + (-\zeta_2).
\end{align*}
As it turns out, $\phi_2$ is also a Gaussian vector $\mathcal{N}(0, \frac{1}{n}I_{n})$. 
This occurs because the term $({\sqrt{2}}/{2} - 1) z_2^{\top}W_2z_2 \cdot z_2$ helps reduce the variance of $W_2 z_2$ along the direction $z_{2}$, and $g_1^k z_1$ adds back the extra variance along the direction $z_{1}$  given that $z_1^\top W_2 z_2 = z_1^\top (I - z_1z_1^\top)W(I - z_1z_1^\top) z_2 = 0$. . 
\end{itemize}
In fact, it is generally the case that: when dealing with $W_{k} z_{k}$, an extra term $\zeta_k$ --- and hence new independent random variables $g_i^k$'s --- is added to $W_{k} z_{k}$ in order to produce a homogeneous Gaussian vector $\phi_{k}$; this extra term will then be subtracted out from $\xi_t$. Clearly, the terms $\zeta_k$ and $-\zeta_k$ cancel out each other and hence their sum remains measurable with respect to $W$.
While introducing extra randomnesses of this kind might sound counterintuitive at first glance,  
this strategy gives rise to homogeneous Gaussian vectors in conjunction with a well-controlled residual term, which turns out to be remarkably useful when tackling the two specific examples here and beyond. 
}

\paragraph{Step 3: establishing a key decomposition of $\{x_t\}$.}
Equipped with the definitions above, we claim that the AMP updates satisfy the following decomposition:
\begin{align}
	x_t \defn \alpha_t \vstar + \sum_{k = 1}^{t-1} \beta_{t-1}^k\phi_k + \xi_{t-1}, \qquad\text{for }t \ge 2,  
	\label{def:dynamics}
\end{align}
where $\alpha_{t} = \lambda v^{\star\top} \eta_{t-1}(x_{t-1})$ and $\xi_{t-1}$ denotes some residual term obeying
\[
	\xi_{t-1} \in U_{t-1}. 
\]
Here and below, we abuse the notation $U_{t-1}$ to denote the subspace spanned by the columns of $[z_1,\cdots,z_{t-1}]$.

\begin{proof}[Proof of decomposition \eqref{def:dynamics}]
The proof proceeds in an inductive manner. 
First, recalling the update rule of AMP, the definition \eqref{eqn:z-w-init}, and the assumption $\eta_0(x_0)=0$ yields 
\begin{align*}
x_{2} & =(\lambda\vstar v^{\star\top}+W)\eta_{1}(x_{1})\\
 & =\lambda v^{\star\top}\eta_{1}(x_{1})\cdot\vstar+W\eta_{1}(x_{1})=\lambda v^{\star\top}\eta_{1}(x_{1})\cdot\vstar+\ltwo{\eta_{1}(x_{1})}\cdot W_{1}z_{1}\\
 & =\alpha_{2}\vstar+\beta_{1}^{1}W_{1}z_{1}=\alpha_{2}\vstar+\beta_{1}^{1}\phi_{1}+\underset{\eqqcolon\,\xi_{1}}{\underbrace{\left(-\beta_{1}^{1}\zeta_{1}\right)}},
\end{align*}
where the penultimate identity comes from the definition of $\alpha_t$ and $\beta_{t}^k$, and the last relation arises from \eqref{def:phi_k}. 
Clearly, $\xi_{1}\in U_1$ according to \eqref{eqn:zeta-k}. 
This establishes the claim \eqref{def:dynamics} for the base case with $t=2$.

Next, suppose that the decomposition \eqref{def:dynamics} is valid for step $t$, and we aim to justify it for step $t+1$ as well. 
Towards this, let us begin by expressing $W_1$ as 
\begin{align}
W_1 = W_t + \sum_{k = 1}^{t-1}(W_k - W_{k+1}) = W_t + \sum_{k = 1}^{t-1} \lt[W_kz_kz_k^{\top} + z_kz_k^{\top}W_k - z_kz_k^{\top}W_kz_kz_k^{\top}\rt], 
	\label{eq:W1-recursive-expand}
\end{align}
which comes from the definition~\eqref{eqn:Wt}. 
Based on this decomposition and the relation \eqref{eqn:eta-decomposition}, we can express the AMP iteration as: 
\begin{align}
\notag x_{t+1} & =\alpha_{t+1}\vstar+W_{1}\eta_{t}(x_{t})-\langle\eta_{t}^{\prime}(x_{t})\rangle\eta_{t-1}(x_{t-1})=\alpha_{t+1}\vstar+W_{1}\eta_{t}(x_{t})-\langle\eta_{t}^{\prime}(x_{t})\rangle\sum_{k=1}^{t-1}\beta_{t-1}^{k}z_{k}\notag\\
 & =\alpha_{t+1}\vstar+W_{t}\eta_{t}(x_{t})+\sum_{k=1}^{t-1}\lt[W_{k}z_{k}z_{k}^{\top}+z_{k}z_{k}^{\top}W_{k}-z_{k}z_{k}^{\top}W_{k}z_{k}z_{k}^{\top}\rt]\eta_{t}(x_{t})-\langle\eta_{t}^{\prime}(x_{t})\rangle\sum_{k=1}^{t-1}\beta_{t-1}^{k}z_{k}\notag\\
 & =\alpha_{t+1}\vstar+W_{t}\eta_{t}(x_{t})+\sum_{k=1}^{t-1}\beta_{t}^{k}W_{k}z_{k}+\sum_{k=1}^{t-1}z_{k}\big\langle W_{k}z_{k},\eta_{t}(x_{t})\big\rangle-\sum_{k=1}^{t-1}z_{k}\big(\beta_{t}^{k}z_{k}^{\top}W_{k}z_{k}\big)-\langle\eta_{t}^{\prime}(x_{t})\rangle\sum_{k=1}^{t-1}\beta_{t-1}^{k}z_{k}\notag\\
 & =\alpha_{t+1}\vstar+\sum_{k=1}^{t}\beta_{t}^{k}W_{k}z_{k}+\sum_{k=1}^{t-1}z_{k}\lt[\langle W_{k}z_{k},\eta_{t}(x_{t})\rangle-\langle\eta_{t}^{\prime}(x_{t})\rangle\beta_{t-1}^{k}-\beta_{t}^{k}z_{k}^{\top}W_{k}z_{k}\rt]\label{eqn:xt-by-Wkzk}\\
 & =\alpha_{t+1}\vstar+\sum_{k=1}^{t}\beta_{t}^{k}\phi_{k}+
	\underset{\eqqcolon\, \xi_t}{\underbrace{ \sum_{k=1}^{t-1}z_{k}\lt[\langle W_{k}z_{k},\eta_{t}(x_{t})\rangle-\langle\eta_{t}^{\prime}(x_{t})\rangle\beta_{t-1}^{k}-\beta_{t}^{k}z_{k}^{\top}W_{k}z_{k}\rt]-\sum_{k=1}^{t}\beta_{t}^{k}\zeta_{k}}}.
	\label{eqn:xt-by-phik}
\end{align}
where the second line invokes \eqref{eq:W1-recursive-expand}, the fourth line makes use of the fact that 
\[
	W_t\eta_{t}(x_{t})= W_t \big( I-U_{t-1}U_{t-1}^{\top} \big) \eta_{t}(x_{t}) = W_t  (\beta_t^t z_t), 
\]
and the last line in \eqref{eqn:xt-by-phik} follows from \eqref{def:phi_k}. 
By construction, $\zeta_k\in U_k$, and hence the expression of $\xi_t$ in \eqref{eqn:xt-by-phik} immediately reveals that $\xi_t\in U_t$.
\end{proof}

Before moving on, we further take a moment to derive an alternative expression of $\xi_t$.  
Let us first make the following  observation
arising from the definition \eqref{eqn:zeta-k} and the decomposition~\eqref{eqn:eta-decomposition}:  
\begin{align*}
\sum_{k = 1}^{t} \beta_{t}^k\zeta_k = \sum_{k = 1}^{t} \beta_{t}^k\lt[\bigg(\frac{\sqrt{2}}{2} - 1\bigg)z_kz_k^{\top}W_kz_k + \sum_{i = 1}^{k - 1} g_i^kz_i\rt] = \sum_{k = 1}^{t} z_k\lt[\beta_{t}^k\bigg(\frac{\sqrt{2}}{2} - 1\bigg)z_k^{\top}W_kz_k + \sum_{i = k+1}^{t} \beta_{t}^ig_k^i\rt],
\end{align*}
where the last line holds since
\begin{align*}
\sum_{k=1}^{t}\beta_{t}^{k}\sum_{i=1}^{k-1}g_{i}^{k}z_{i} & =\sum_{i=1}^{t-1}z_{i}\sum_{k=i+1}^{t}\beta_{t}^{k}g_{i}^{k}=\sum_{k=1}^{t-1}z_{k}\sum_{i=k+1}^{t}\beta_{t}^{i}g_{k}^{i}=\sum_{k=1}^{t}z_{k}\sum_{i=k+1}^{t}\beta_{t}^{i}g_{k}^{i} .
\end{align*}
Additionally, apply the decomposition~\eqref{eqn:eta-decomposition} and the \eqref{eqn:zeta-k} once again to reach
\begin{align*}
\big\langle \zeta_k, \eta_{t}(x_t)\big\rangle 
= \lt\langle\Big(\frac{\sqrt{2}}{2} - 1\Big) z_kz_k^{\top}W_kz_k + \sum_{i = 1}^{k - 1} g_i^kz_i, \sum_{k = 1}^{t} \beta_{t}^kz_k\rt\rangle
= \bigg(\frac{\sqrt{2}}{2} - 1\bigg)\beta_t^kz_k^{\top}W_kz_k + \sum_{i = 1}^{k - 1} \beta_t^ig_i^k 
\end{align*}
for any $k\leq t$. 
Substituting the above two equalities into \eqref{eqn:xt-by-phik}, we arrive at 
\begin{align}
\xi_{t} & =\sum_{k=1}^{t-1}z_{k}\lt[\langle W_{k}z_{k},\eta_{t}(x_{t})\rangle-\langle\eta_{t}^{\prime}(x_{t})\rangle\beta_{t-1}^{k}-\beta_{t}^{k}z_{k}^{\top}W_{k}z_{k}\rt]-\sum_{k=1}^{t}\beta_{t}^{k}\zeta_{k}\nonumber\\
 & =\sum_{k=1}^{t-1}z_{k}\lt[\langle\phi_{k},\eta_{t}(x_{t})\rangle-\langle\zeta_{k},\eta_{t}(x_{t})\rangle-\langle\eta_{t}^{\prime}(x_{t})\rangle\beta_{t-1}^{k}-\beta_{t}^{k}z_{k}^{\top}W_{k}z_{k}\rt]-\sum_{k=1}^{t}\beta_{t}^{k}\zeta_{k}\nonumber\\
 & =\sum_{k=1}^{t-1}z_{k}\Bigg[\Big\langle\phi_{k},\eta_{t}\Big(\alpha_{t}\vstar+\sum_{k=1}^{t-1}\beta_{t-1}^{k}\phi_{k}+\xi_{t-1}\Big)\Big\rangle-\langle\eta_{t}^{\prime}(x_{t})\rangle\beta_{t-1}^{k}
	- \sum_{i=1}^{k-1}\beta_{t}^{i}g_{i}^{k}-\sum_{i=k+1}^{t}\beta_{t}^{i}g_{k}^{i} \notag\\
 & \qquad\qquad-\big(\sqrt{2}-1\big)\beta_{t}^{k}z_{k}^{\top}W_{k}z_{k}\bigg] - \bigg(\frac{\sqrt{2}}{2} - 1\bigg)\beta_t^tz_tz_t^{\top}W_tz_t,\label{eq:xi-expression}
\end{align}
where the last line invokes the decomposition \eqref{def:dynamics}. 

\paragraph{Step 4: bounding the residual term $\ltwo{\xi_t}$.}
\red{Everything then boils down to controlling $\ltwo{\xi_t}$.  
Let us define a vector $\mu_t =[\mu_t^k]_{1\leq k\leq t} \in \mathbb{R}^{t}$ with coordinates 
\[
	\mu_t^k \coloneqq \frac{\xi_t^{\top}z_k}{\|\xi_t\|_2}, \qquad 1\leq k\leq t.
\]
Given that  $\{z_k\}_{k\leq t}$ forms an orthonormal basis and that $\xi_t\in U_t$, one can easily see that $$\ltwo{\mu_t} = 1 \qquad \text{and} \qquad \xi_{t} = \ltwo{\xi_t} \sum_{k=1}^t \mu_t^k z_k.$$
Hence, we can deduce that 
\begin{align}
\notag \|\xi_{t}\|_{2} & = \frac{\inprod{\xi_t}{\xi_t}}{\ltwo{\xi_t}}= \frac{\inprod{\ltwo{\xi_t}\sum_{k=1}^t \mu_t^k z_k}{\xi_t}}{\ltwo{\xi_t}} = \sum_{k=1}^t \mu_t^k \inprod{z_k}{\xi_t} \\
& \stackrel{(\text{i})}{=} \sum_{k=1}^{t-1}\mu_t^{k}\Bigg[\Big\langle\phi_{k},\eta_{t}\Big(\alpha_{t}\vstar+\sum_{k=1}^{t-1}\beta_{t-1}^{k}\phi_{k}+\xi_{t-1}\Big)\Big\rangle-\langle\eta_{t}^{\prime}(x_{t})\rangle\beta_{t-1}^{k}
	- \sum_{i=1}^{k-1}\beta_{t}^{i}g_{i}^{k}-\sum_{i=k+1}^{t}\beta_{t}^{i}g_{k}^{i} \notag\\
 & \qquad\qquad-\big(\sqrt{2}-1\big)\beta_{t}^{k}z_{k}^{\top}W_{k}z_{k}\bigg] - \bigg(\frac{\sqrt{2}}{2} - 1\bigg)\beta_t^t\mu_t^tz_t^{\top}W_tz_t \nonumber\\
\notag & = \bigg\langle\sum_{k=1}^{t-1}\mu_{t}^{k}\phi_{k},\delta_{t}\bigg\rangle-\langle\delta_{t}^{\prime}\rangle\sum_{k=1}^{t-1}\mu_{t}^{k}\beta_{t-1}^{k} - \bigg(\frac{\sqrt{2}}{2} - 1\bigg)\beta_t^t\mu_t^tz_t^{\top}W_tz_t\\
\notag & \qquad\qquad-\sum_{k=1}^{t-1}\mu_{t}^{k}\lt[-\big\langle\phi_{k},\eta_{t}(v_{t})\big\rangle+\big\langle\eta_{t}^{\prime}(v_{t})\big\rangle\beta_{t-1}^{k}+(\sqrt{2}-1)\beta_{t}^{k}z_{k}^{\top}W_{k}z_{k}+\sum_{i=1}^{k-1}\beta_{t}^{i}g_{i}^{k}+\sum_{i=k+1}^{t}\beta_{t}^{i}g_{k}^{i}\rt]\\
 & =\Big\langle\sum_{k=1}^{t-1}\mu_{t}^{k}\phi_{k},\delta_{t}\Big\rangle-\langle\delta_{t}^{\prime}\rangle\sum_{k=1}^{t-1}\mu_{t}^{k}\beta_{t-1}^{k} - \bigg(\frac{\sqrt{2}}{2} - 1\bigg)\beta_t^t\mu_t^tz_t^{\top}W_tz_t \notag \\
	& \qquad\qquad +\Delta_{t} - \sum_{k=1}^{t-1}\mu_{t}^{k}\lt[(\sqrt{2}-1)\beta_{t}^{k}z_{k}^{\top}W_{k}z_{k}+\sum_{i=1}^{k-1}\beta_{t}^{i}g_{i}^{k}+\sum_{i=k+1}^{t}\beta_{t}^{i}g_{k}^{i}\rt],\label{eq:xi_bound}
\end{align}
where $(\text{i})$  invokes expression~\eqref{eq:xi-expression} and the fact that $\{z_k\}_{k\leq t}$ are orthogonal to each other, 
and the last two lines rely on the definitions in~\eqref{eqn:delta-chorus} as follows:
\begin{align*}
v_{t} & \defn\alpha_{t}\vstar+\sum_{k=1}^{t-1}\beta_{t-1}^{k}\phi_{k},\\
\Delta_{t} & \defn\sum_{k=1}^{t-1}\mu_{t}^{k}\Big[\big\langle\phi_{k},\eta_{t}(v_{t})\big\rangle-\big\langle\eta_{t}^{\prime}(v_{t})\big\rangle\beta_{t-1}^{k}\Big],\\
\delta_{t} & \defn\eta_{t}\Big(\alpha_{t}\vstar+\sum_{k=1}^{t-1}\beta_{t-1}^{k}\phi_{k}+\xi_{t-1}\Big)-\eta_{t}(v_{t}),\\
\delta_{t}^{\prime} & \defn\eta_{t}^{\prime}\Big(\alpha_{t}\vstar+\sum_{k=1}^{t-1}\beta_{t-1}^{k}\phi_{k}+\xi_{t-1}\Big)-\eta_{t}^{\prime}(v_{t}).
\end{align*}
}

To establish Theorem~\ref{thm:recursion}, it then suffices to control the last term on the right-hand side of \eqref{eq:xi_bound}. 
This is accomplished in the following lemma, whose proof is deferred to Section~\ref{sec:pf-concentration}. 
\begin{lems} \label{lem:concentration}
With probability at least $1-O(n^{-11})$, for any $t\leq n$ we have
\begin{align*}
	\bigg|\sum_{k = 1}^{t-1} \mu_t^k\Big((\sqrt{2}-1)\beta_{t}^kz_k^{\top}W_kz_k + \sum_{i = 1}^{k - 1} \beta_t^ig_i^k + \sum_{i = k+1}^{t} \beta_{t}^ig_k^i\Big)\bigg| &\lesssim \sqrt{\frac{t\log n}{n}} \|\beta_{t}\|_2. 
\end{align*}
\end{lems}

Taking this lemma collectively with equality~\eqref{eq:xi_bound} and the trivial bound $|z_t^{\top}W_tz_t| \lesssim \sqrt{\frac{\log n}{n}}$ leads to 
\begin{align}
\label{eqn:sonata}
\|\xi_t\|_2 = \Big\langle \sum_{k = 1}^{t-1} \mu_t^k\phi_k, \delta_{t}\Big\rangle - \langle\delta_{t}^{\prime}\rangle \sum_{k = 1}^{t-1} \mu_t^k\beta_{t-1}^k + \Delta_t + O\bigg(\sqrt{\frac{t\log n}{n}}\|\beta_{t}\|_2 \bigg),
\end{align}
thus completing the proof of Theorem~\ref{thm:recursion}.


\subsection{Proof of Theorem~\ref{thm:main}}
\label{sec:pf-thm-main}

Before embarking on the proof, we remind the reader of several results  that have already proven for $\alpha_t$ and $\beta_t$. 
Recall that in the proof of Theorem~\ref{thm:recursion}, we decompose the AMP iterate $x_{t+1}$ as follows 
%
\begin{align*}
	x_{t+1} = \alpha_{t+1} \vstar + \sum_{k = 1}^{t} \beta_{t}^k\phi_k + \xi_{t},
	\qquad 1\leq t\leq n,
\end{align*}
where $\xi_{t}\in U_{t}$ (some linear subspace of dimension $t$) represents some residual term, and
\begin{subequations}
\label{eq:alpha-beta-t-expansion-thm2}
\begin{align}
	\alpha_{t+1} &= \lambda v^{\star\top}\eta_t(x_t) = \lambda v^{\star\top}\eta_t\Big(\alpha_t \vstar + \sum_{k = 1}^{t-1} \beta_{t-1}^k\phi_k + \xi_{t-1}\Big), 
	\label{eq:alpha-t-expansion-thm2}\\
	\|\beta_t\|_2 &= \|\eta_t(x_t)\|_2 
	= \Big\|\eta_t\Big(\alpha_t \vstar + \sum_{k = 1}^{t-1} \beta_{t-1}^k\phi_k + \xi_{t-1} \Big)\Big\|_2.
	\label{eq:beta-t-expansion-thm2}
\end{align}
\end{subequations}
We have also shown in Theorem~\ref{thm:recursion} that with probability at least $1-O(n^{-11})$, the residual term satisfies 
\begin{align}
\notag	\|\xi_{t}\|_2 = \Big\langle \sum_{k = 1}^{t-1} \mu^k_t \phi_k, \delta_{t}\Big\rangle - \langle\delta_{t}^{\prime}\rangle \sum_{k = 1}^{t-1} \mu^k_t \beta_{t-1}^k + \Delta_t + O\Big(\sqrt{\frac{t\log n}{n}}\|\beta_{t}\|_2 \Big)\\
\label{eqn:xi-t-tmp} \leq \Big\langle \sum_{k = 1}^{t-1} \mu^k_t \phi_k, \delta_{t}\Big\rangle - \langle\delta_{t}^{\prime}\rangle \sum_{k = 1}^{t-1} \mu^k_t \beta_{t-1}^k + A_t + O\Big(\sqrt{\frac{t\log n}{n}}\|\beta_{t}\|_2 \Big),
\end{align} 
where the last step invokes property~\eqref{defi:A} in Assumption~\ref{assump:A-H-eta} as well as the definition \eqref{defn:Delta-t} of $\Delta_t$.



\paragraph{Step 1: bounding $\Delta_{\alpha,t} $ and $\Delta_{\beta,t} $ in terms of $\delta_t$ and $\delta_t^{\prime}$.}
We begin by controlling the size of the term $\Delta_{\alpha,t} $. In view of its definition in \eqref{eqn:alpha-t-genearl}, 
we have 
\begin{align*}
	\Delta_{\alpha,t} 
	&\defn \frac{\alpha_{t+1}}{\lambda} - v^{\star\top}\int\eta_t\Big(\alpha_t \vstar + \frac{\|\beta_{t-1}\|_2}{\sqrt{n}}x\Big)\varphi_n(\dx), \\
	& = v^{\star\top}\delta_t + v^{\star\top}\eta_{t}\Big(\alpha_t \vstar + \sum_{k = 1}^{t-1} \beta_{t-1}^k\phi_k\Big) - v^{\star\top}\int\eta_t\Big(\alpha_t\vstar + \frac{\|\beta_{t-1}\|_2}{\sqrt{n}}x\Big)\varphi_n(\dx),
\end{align*}
where the second line follows from \eqref{eq:alpha-t-expansion-thm2} and the definition \eqref{defn:delta-t} of $\delta_t$. 
As a direct consequence of the assumption~\eqref{defi:B}, we obtain
\begin{align}
\label{eqn:tmp-Delta-alpha}
	|\Delta_{\alpha,t}| &\leq \lt|\inprod{\vstar}{\delta_t}\rt| + B_t. 
\end{align}
We then move on to the term $\Delta_{\beta,t}$. 
Recognizing that 
\[
	\| \beta_t \|_2^2 = \|\eta_t(x_t)\|_2^2  = \Big\| \eta_{t}\Big(\alpha_t\vstar + \sum_{k = 1}^{t-1} \beta_{t-1}^k\phi_k \Big) + \delta_t \Big\|_2^2,
\]
we can combine it with the definition \eqref{eqn:beta-t-genearl} to obtain
\begin{align*}
	\Delta_{\beta,t} &\defn \|\beta_t\|_2^2 - \int\Big\|\eta_t\Big(\alpha_t\vstar + \frac{\|\beta_{t-1}\|_2}{\sqrt{n}}x\Big)\Big\|_2^2\varphi_n(\dx)\\
	&= \Big\langle2\eta_{t}\Big(\alpha_t\vstar + \sum_{k = 1}^{t-1} \beta_{t-1}^k\phi_k\Big), \delta_t\Big\rangle + \|\delta_t\|_2^2 
	+ \Big\|\eta_{t}\Big(\alpha_t\vstar + \sum_{k = 1}^{t-1} \beta_{t-1}^k\phi_k\Big)\Big\|_2^2 - \int\Big\|\eta_t\Big(\alpha_t\vstar + \frac{\|\beta_{t-1}\|_2}{\sqrt{n}}x\Big)\Big\|_2^2\varphi_n(\dx).
\end{align*}
By virtue of the assumption~\eqref{defi:C}, we obtain 
\begin{align}
\label{eqn:tmp-Delta-beta}
	|\Delta_{\beta,t}| &\leq \Big|\Big\langle2\eta_{t}\Big(\alpha_t\vstar + \sum_{k = 1}^{t-1} \beta_{t-1}^k\phi_k\Big), \delta_t\Big\rangle\Big| + \|\delta_t\|_2^2 + C_t.
\end{align}

\paragraph{Step 2: bounding $\delta_t$ and $\delta_t^{\prime}$.}

To further control the right-hand side of \eqref{eqn:tmp-Delta-alpha} and \eqref{eqn:tmp-Delta-beta}, 
we proceed by bounding terms associated with $\delta_{t}$. 
Given that $\eta_t(\cdot)$ is assumed to be continuous, one can derive 
\begin{align}
\notag \delta_{t} & =\eta_{t}\Big(\alpha_{t}\vstar+\sum_{k=1}^{t-1}\beta_{t-1}^{k}\phi_{k}+\xi_{t-1}\Big)-\eta_{t}\Big(\alpha_{t}\vstar+\sum_{k=1}^{t-1}\beta_{t-1}^{k}\phi_{k}\Big)\\
\notag & ={\displaystyle \int}_{0}^{1}\bigg\{\eta_{t}^{\prime}\Big(\alpha_{t}\vstar+\sum_{k=1}^{t-1}\beta_{t-1}^{k}\phi_{k}+\tau\xi_{t-1}\Big)\circ\xi_{t-1}\bigg\}\mathrm{d}\tau\\
 & =\eta_{t}^{\prime}\Big(\alpha_{t}\vstar+\sum_{k=1}^{t-1}\beta_{t-1}^{k}\phi_{k}\Big)\circ\xi_{t-1}+{\displaystyle \int}_{0}^{1}\bigg\{\bigg[\eta_{t}^{\prime}\Big(\alpha_{t}\vstar+\sum_{k=1}^{t-1}\beta_{t-1}^{k}\phi_{k}+\tau\xi_{t-1}\Big)-\eta_{t}^{\prime}\Big(\alpha_{t}\vstar+\sum_{k=1}^{t-1}\beta_{t-1}^{k}\phi_{k}\Big)\bigg]\circ\xi_{t-1}\bigg\}\mathrm{d}\tau,
\label{eqn:shostakovich-delta-t-123}
\end{align}
where the second line invokes the fundamental theorem of calculus. 
Note that $\eta_t^{\prime}(\cdot)$ has a finite number of discontinuous points.  
Recalling that $|\eta_t^{\prime}(w)|\leq \rho$ and $|\eta_t^{\second}(w)|\leq \rho_1$ for any continuous point $w\in \real$ (see Assumption~\ref{assump:eta}), 
we have
\begin{align}
	&\left|\eta_{t}^{\prime}\Big(\alpha_{t}\vstar+\sum_{k=1}^{t-1}\beta_{t-1}^{k}\phi_{k}+\tau\xi_{t-1}\Big)-\eta_{t}^{\prime}\Big(\alpha_{t}\vstar+\sum_{k=1}^{t-1}\beta_{t-1}^{k}\phi_{k}\Big)\right| \notag\\
	&\qquad \leq\left|{\displaystyle \int}_{0}^{1}\bigg\{\eta_{t}^{\second}\Big(\alpha_{t}\vstar+\sum_{k=1}^{t-1}\beta_{t-1}^{k}\phi_{k}+\tau_{1}\tau\xi_{t-1}\Big)\circ\big(\tau\xi_{t-1}\big)\bigg\}\mathrm{d}\tau_{1}\right|+ 2\rho\Gamma \notag\\
 &\qquad \leq\rho_{1}\big|\xi_{t-1}\big|+ 2\rho\Gamma,
	\label{eq:diff-eta-t-prime-UB135}
\end{align}
where $\Gamma= [\Gamma_j]_{1\leq j\leq n} \in \real^n$ is a term reflecting the influence of discontinuous points. 
More precisely, $\Gamma_j$ denotes the number of discontinuities of $\eta_t^{\prime}(\cdot)$ encountered between 
$\big[\alpha_{t}\vstar_j+\sum_{k=1}^{t-1}\beta_{t-1}^{k}\phi_{k,j}, \alpha_{t}\vstar_j+\sum_{k=1}^{t-1}\beta_{t-1}^{k}\phi_{k,j}+\xi_{t-1,j}\big]$.   
Note that if a point $m$ is contained in an interval $[a,b]$,
then one must have $a+\tau(b-a)=m$ for some $\tau\in[0,1]$,
which requires that $|b-a|\geq|\tau(b-a)|=|a-m|$. This basic fact allows us to take
\begin{align}
	\Gamma_j =  \sum_{m\in\mathcal{M}_{\mathsf{dc}}}\ind\bigg\{ \big|\xi_{t-1,j}\big|\geq\Big|\alpha_{t}v_{j}^{\star}+\sum_{k=1}^{t-1}\beta_{t-1}^{k}\phi_{k,j}
	-m\Big|\bigg\} \eqqcolon \sum_{m\in\mathcal{M}_{\mathsf{dc}}} \Gamma_{j}(m).
	\label{eq:Gamma-ub-discontinuous}
\end{align}
Substitution into \eqref{eqn:shostakovich-delta-t-123} yields
\begin{align}
\Big|\delta_{t} & -\eta_{t}^{\prime}\Big(\alpha_{t}\vstar+\sum_{k=1}^{t-1}\beta_{t-1}^{k}\phi_{k}\Big)\circ\xi_{t-1}\Big|
	\leq \rho_{1}\big|\xi_{t-1}\big|^{2}+ 2\rho\Gamma \circ \big|\xi_{t-1}\big|
.
\label{eqn:shostakovich-delta-t}
\end{align}
Similarly, we can repeat the same argument (particularly \eqref{eq:diff-eta-t-prime-UB135} and \eqref{eqn:shostakovich-delta-t}) to bound $\delta_{t}^{\prime}$ as follows:  
\begin{align}
 & \bigg|\delta_{t}^{\prime}-\eta_{t}^{\second}\Big(\alpha_{t}\vstar+\sum_{k=1}^{t-1}\beta_{t-1}^{k}\phi_{k}\Big)\circ\xi_{t-1}\bigg| \notag\\
 & =\bigg|\eta_{t}^{\prime}\Big(\alpha_{t}\vstar+\sum_{k=1}^{t-1}\beta_{t-1}^{k}\phi_{k}+\xi_{t-1}\Big)-\eta_{t}^{\prime}\Big(\alpha_{t}\vstar+\sum_{k=1}^{t-1}\beta_{t-1}^{k}\phi_{k}\Big)-\eta_{t}^{\second}\Big(\alpha_{t}\vstar+\sum_{k=1}^{t-1}\beta_{t-1}^{k}\phi_{k}\Big)\circ\xi_{t-1}\bigg| \notag\\
 & \leq\bigg|{\displaystyle \int}_{0}^{1}\bigg\{\eta_{t}^{\second}\Big(\alpha_{t}\vstar+\sum_{k=1}^{t-1}\beta_{t-1}^{k}\phi_{k}+\tau\xi_{t-1}\Big)\circ\xi_{t-1}\bigg\}\mathrm{d}\tau-\eta_{t}^{\second}\Big(\alpha_{t}\vstar+\sum_{k=1}^{t-1}\beta_{t-1}^{k}\phi_{k}\Big)\circ\xi_{t-1}\bigg|+2\rho\Gamma  \notag\\
	& \leq\rho_{2}\big|\xi_{t-1}\big|^{2}+ 2\rho\Gamma +2\rho_{1}\Gamma \circ \big|\xi_{t-1}\big|. 
	\label{eqn:shostakovich-delta-t-prime}
\end{align}
With the above bounds on $\delta_t$ and $\delta_t^{\prime}$ in place, 
we are ready to establish the advertised results \eqref{eqn:delta-alpha-general}, \eqref{eqn:delta-beta-general} and \eqref{eqn:xi-t-general}, 
which we will look at one by one in the sequel.

\paragraph{Step 3: establishing inequality~\eqref{eqn:xi-t-general}.} 
With these relations in place, let us start with controlling quantity $\|\xi_{t}\|_2$. 
In view of expression~\eqref{eqn:xi-t-tmp}, it requires us to bound $\langle \mu_t^k\phi_k, \delta_{t}\rangle - \langle\delta_{t}^{\prime}\rangle \sum_{k = 1}^{t-1} \mu_t^k\beta_{t-1}^k$. 
Taking the bounds~\eqref{eqn:shostakovich-delta-t} and \eqref{eqn:shostakovich-delta-t-prime} collectively with \eqref{eqn:xi-t-tmp}, 
and recalling the definition \eqref{defn:v-t-thm1} of $v_t$, we arrive at 
\begin{align}
\|\xi_{t}\|_{2} & \leq\Big\langle\sum_{k=1}^{t-1}\mu_{t}^{k}\phi_{k},\delta_{t}\Big\rangle-\langle\delta_{t}^{\prime}\rangle\sum_{k=1}^{t-1}\mu_{t}^{k}\beta_{t-1}^{k}+A_{t}+O\Big(\sqrt{\frac{t\log n}{n}}\|\beta_{t}\|_{2}\Big)\notag\\
 & =\bigg\langle\sum_{k=1}^{t-1}\mu_{t}^{k}\phi_{k},\,\eta_{t}^{\prime}(v_{t})\circ\xi_{t-1}\bigg\rangle-\bigg\langle\eta_{t}^{\second}(v_{t})\circ\xi_{t-1}\bigg\rangle\sum_{k=1}^{t-1}\mu_{t}^{k}\beta_{t-1}^{k}\notag\\
 & \qquad+\rho_{1}\bigg\langle\bigg|\sum_{k=1}^{t-1}\mu_{t}^{k}\phi_{k}\bigg|,\big|\xi_{t-1}\big|^{2}\bigg\rangle+\rho_{2}\Big\langle\big|\xi_{t-1}\big|^{2}\Big\rangle\bigg|\sum_{k=1}^{t-1}\mu_{t}^{k}\beta_{t-1}^{k}\bigg|\notag\\
 & \qquad+2\rho\bigg\langle\bigg|\sum_{k=1}^{t-1}\mu_{t}^{k}\phi_{k}\bigg|,\,\Gamma\circ \big|\xi_{t-1}\big|\bigg\rangle+\Big\{2\rho\langle \Gamma \rangle + 2\rho_{1}\big\langle\Gamma\circ \big|\xi_{t-1}\big|\big\rangle\Big\}\bigg|\sum_{k=1}^{t-1}\mu_{t}^{k}\beta_{t-1}^{k}\bigg| \notag\\
	& \qquad +A_{t}+O\Big(\sqrt{\frac{t\log n}{n}}\|\beta_{t}\|_{2}\Big). \label{eq:xi-t-UB-13579}
\end{align} 
This leaves us with several terms to control, which is the content of the lemma below; the proof is deferred to Section~\ref{sec:pf-lem-recursion}.
\begin{lems} \label{lem:recursion}
Consider any $t\leq n$. Given $\kappa_t$ defined in~\eqref{defi:kappa}, it holds that 
\begin{subequations}
\begin{align}
	\bigg\langle\sum_{k=1}^{t-1}\mu_{t}^{k}\phi_{k},\eta_{t}^{\prime}(v_{t})\circ\xi_{t-1}\bigg\rangle-\Big\langle\eta_{t}^{\second}(v_{t})\circ\xi_{t-1}\Big\rangle\sum_{k=1}^{t-1}\mu_{t}^{k}\beta_{t-1}^{k} &\le\sqrt{\kappa_{t}^{2}+D_{t}}\,\|\xi_{t-1}\|_{2}	 
	\label{eq:lem-recursion-smooth-part-1} \\
	\rho_{1}\bigg\langle\bigg|\sum_{k=1}^{t-1}\mu_{t}^{k}\phi_{k}\bigg|,\big|\xi_{t-1}\big|^{2}\Big\rangle+\rho_{2}\Big\langle\big|\xi_{t-1}\big|^{2}\Big\rangle\bigg|\sum_{k=1}^{t-1}\mu_{t}^{k}\beta_{t-1}^{k}\bigg| & \lesssim\lt( \rho_{1}  \frac{\sqrt{t}+\sqrt{\log n}}{\sqrt{n}} +\frac{\rho_{2}\|\beta_{t-1}\|_{2}}{n}\rt)\|\xi_{t-1}\|_{2}^{2}
\label{eq:lem-recursion-smooth-part-2}
\end{align}
hold with probability at least $1-O(n^{-11})$. 
In addition, one has 
\begin{align}
 & 2\rho\bigg\langle\bigg|\sum_{k=1}^{t-1}\mu_{t}^{k}\phi_{k}\bigg|,\,\Gamma\circ\big|\xi_{t-1}\big|\bigg\rangle+\Big\{2\rho\langle\Gamma\rangle+2\rho_{1}\big\langle\Gamma\circ\big|\xi_{t-1}\big|\big\rangle\Big\}\bigg|\sum_{k=1}^{t-1}\mu_{t}^{k}\beta_{t-1}^{k}\bigg| \notag\\
 & \qquad\lesssim\rho\sqrt{\frac{(E_{t}+t)\log n}{n}}\big\|\xi_{t-1}\big\|_{2}+\frac{(\rho+\rho_{1}\big\|\xi_{t-1}\big\|_{\infty})E_{t}\big\|\beta_{t-1}\big\|_{2}}{n}. 	\label{eqn:new-version}
\end{align}
\end{subequations}
\end{lems}
Combining Lemma~\ref{lem:recursion} with \eqref{eq:xi-t-UB-13579} immediately completes the proof of inequality~\eqref{eqn:xi-t-general}.

\paragraph{Step 4: establishing inequalities \eqref{eqn:delta-alpha-general} and \eqref{eqn:delta-beta-general}.}
Finally, we return to establish the advertised bounds on $|\Delta_{\alpha,t}|$ and $|\Delta_{\beta,t}|$. 
Towards this, we are in need of the following lemma, whose proof is provided in Section~\ref{sec:pf-lem-recursion2}.
\begin{lems} 
\label{lem:recursion2}
The following inequalities hold true:
\begin{subequations}
\begin{align}
\label{eqn:cello}
	|\langle \vstar, \delta_t\rangle| &\lesssim \rho\|\xi_{t-1}\|_2, \\
\label{eqn:viola}
	\Big|\Big\langle\eta_{t}\big(\alpha_t\vstar + \sum_{k = 1}^{t-1} \beta_{t-1}^k\phi_k\big), \delta_t \Big\rangle\Big| &\lesssim F_t\|\xi_{t-1}\|_2 + \rho_1G_t\|\xi_{t-1}\|_2^2 + \rho \sqrt{E_t}G_t \|\xi_{t-1}\|_2, \\
%
\label{eqn:violin}
\|\delta_t\|_2^2 &\lesssim \rho^2\|\xi_{t-1}\|_2^2.
\end{align}
\end{subequations}
\end{lems}
Substituting the results in Lemma~\ref{lem:recursion2} into inequalities~\eqref{eqn:tmp-Delta-alpha} and \eqref{eqn:tmp-Delta-beta} immediately   
establishes \eqref{eqn:delta-alpha-general} and \eqref{eqn:delta-beta-general}.  
We have thus completed the proof of Theorem~\ref{thm:main}.

\section{Discussion}
\label{sec:discussion}

In this paper, 
we have proposed a general recipe towards analyzing the finite-sample performance of the AMP algorithm when applied to spiked Wigner models. 
Our analysis framework makes explicit a crucial decomposition of each AMP iterate (as a superposition of a signal term and a Gaussian-type stochastic component), 
with a residual term that can be tracked recursively without exploding rapidly. 
Further, this analysis framework can be seamlessly integrated with spectral initialization. 
The power of our analysis strategy has been demonstrated via two concrete applications:  $\mathbb{Z}_{2}$ synchronization and sparse PCA. 
In both cases, explicit non-asymptotic behaviors of AMP have been derived up to a polynomial number of iterations, thereby revealing new insights about the finite-sample convergence properties of  AMP. 
\red{After finishing this paper, 
we have posted a companion paper \cite{li2023approximate} that --- built upon and extending the analysis framework herein ---  characterizes the finite-sample dynamics of random initialized AMP for the problem of $\mathbb{Z}_2$ synchronization. It is proved there that AMP is capable of escaping random initialization and entering a local refinement phase within at most $O(\log n)$ iterations. In other words, an informative initialization is not crucial at all for the effectiveness of AMP. }


Our work leaves open a variety of questions; we conclude the paper by highlighting a few of them. 

\begin{itemize}
\item 
Firstly, while we have illustrated the effectiveness of our master theorems with two examples of different flavor, 
there is no shortage of other signal structures that are of practical interest. 
For instance,  one might wonder how AMP behaves non-asymptotically when the signal $\vstar$ is known to satisfy certain shape constraints (e.g., having non-negative entries, residing in a monotone or convex cone \citep{bandeira2019computational,wei2019geometry}). In some of these cases, the natural denoising functions might not be separable, therefore while the decomposition in Theorem~\ref{thm:recursion} still holds true, controlling those residual terms is significantly more complicated. 

\item 
Secondly, our analysis is tailored to the spiked Wigner model where the noise takes the form of an independent Gaussian matrix. 
It remains unclear whether our non-asymptotic characterizations can be generalized to accommodate non-Gaussian noise matrices \citep{bayati2015universality,chen2021universality,dudeja2022universality}. 
Developing universality results in a non-asymptotic manner is an important yet highly challenging task worthy of future investigation. 
 
\item
\red{Additionally, we have observed in our empirical simulations that: in many examples, AMP continues to work well even when $t$ further increases (to a point that goes far beyond $n$).
This suggests that the dependence on $t$ in our statistical bounds might fall short of tightness in general. 
		How to tighten the statistical performance guarantees for large $t$ remains an interesting open question, which is left for future studies. }

\item
\red{Finally, moving beyond spiked models, 
we expect that our non-asymptotic framework can be generalized to accommodate other important settings such as sparse linear regression and generalized linear models (GLMs). 
In fact, the update rule of AMP for regression and GLMs can often be viewed as AMP applied to  \emph{asymmetric} matrix models; 
more specifically, given an asymmetric design matrix $X$, AMP for GLMs maintains two sequences of updates as follows: 
\begin{align*}
	s_{t} &= XF_t(\beta_{t}) - \lt\langle F_t^{\prime}\rt\rangle G_{t-1}(s_{t-1}), \\
	\beta_{t+1} &= X^{\top}G_t(s_t) - \lt\langle G_t^{\prime}\rt\rangle F_{t}(\beta_{t}),
\end{align*}
thus resembling the updating rule analyzed in the current paper. 
To control these two sequences of updates, one can employ similar analysis ideas as the ones developed for Theorem~\ref{thm:recursion}, 
		while in the meantime keeping track of two sets of orthogonal basis and two sequences of Gaussian random vectors. 
Given that these two sequences rely heavily on each together, caution needs to be exercised when dealing with their accumulated errors.  
Carrying out such analysis fully is fairly involved, and hence we will leave it for future investigation.  
}

 \end{itemize}


\vspace{1cm}

\begin{center}
    {\large APPENDIX}
\end{center}

\appendix


\section{Preliminaries: useful concentration results}
\label{sec:Gaussian-concentration}

This section gathers a few useful concentration results concerning functions of random vectors that will be applied multiple times throughout this paper.

\subsection{List of concentration lemmas}

The first result is concerned with Gaussian concentration for Lipschitz-continuous functions, 
whose proof can be found in Section~\ref{sec:pf-Gaussian}. 
Here and below, we remind the reader that $\mathbb{B}^d(r)$ indicates the $d$-dimensional Euclidean ball with radius $r$ centered at 0.  
\begin{lems}
\label{lem:Gauss}
Consider an $n$-dimensional Gaussian vector $X \sim \mathcal{N}(0, I_n)$, 
and a set of functions $f_{\theta} : \mathbb{R}^n \to \mathbb{R}$ as parameterized by $\theta \in \Theta \subseteq \mathbb{B}^d(r)$.  
Let $\mathcal{E}$ be some convex set  obeying $\mathbb{P}(X \in \mathcal{E}) \geq 1 - O(n^{-11})$. 
Assume that for any fixed $\theta,\widetilde{\theta} \in \Theta$ and any given $Z_1, Z_2 \in \mathcal{E}$, we have
\begin{align}
\label{eqn:gauss-lipschitz}
	|f_{\theta}(Z_1) - f_{\theta}(Z_2)| \le \sigma \ltwo{Z_1 - Z_2}
	\qquad\text{and}\qquad
	\lt\|f_{\theta}(Z) - f_{\widetilde{\theta}}(Z)\rt\|_2 \le L \ltwo{\theta - \widetilde{\theta}}.
\end{align}
In addition, suppose that for any fixed $\theta \in \Theta$, we have
\begin{align}
\label{eqn:gauss-lipschitz-B-proj}	
	\big|\myE\lt[f_{\theta}(\mathcal{P}_{\mathcal{E}}(X)) - f_{\theta}(X)\rt]\big| \le B,
\end{align}
where $\mathcal{P}_{\mathcal{E}}(\cdot)$ denotes the Euclidean projection onto the set $\mathcal{E}$.
Then for any $\epsilon < r$,
\begin{align}
\label{eqn:Gauss-target}
	\sup_{\theta\in \Theta} \big|f_{\theta}(X) - \myE\lt[f_{\theta}(X)\rt] \big| \lesssim \sigma\sqrt{d\log\lt(\frac{nr}{\epsilon}\rt)} + L\epsilon + B
\end{align}
holds with probability at least $1-O(n^{-11})$. 
\end{lems}
As an immediate consequence of Lemma~\ref{lem:Gauss}, we can take $\epsilon\asymp n^{-200}$ to yield the following result:
\begin{cors}
	\label{cor:Gauss}
Under the assumptions of Lemma~\ref{lem:Gauss}, suppose the convex set $\mathcal{E}$ obeys
\begin{subequations}
\label{eq:f-grad-bound-poly}
\begin{align}
	\lt\|f_{\theta}(Z) - f_{\widetilde{\theta}}(Z)\rt\|_2 &\lesssim n^{100} \ltwo{\theta - \widetilde{\theta}}
	\qquad\text{for all }Z\in\mathcal{E}\text{ and all }\theta,\widetilde{\theta}\in \Theta;\\
\big|\myE\lt[f_{\theta}(\mathcal{P}_{\mathcal{E}}(X))-f_{\theta}(X)\rt]\big| & \lesssim n^{-100}.
\end{align}
\end{subequations}
Then with probability at least $1-O(n^{-11})$ one has 
\[
\sup_{\theta\in\Theta}\big|f_{\theta}(X)-\myE\lt[f_{\theta}(X)\rt]\big|\lesssim\sigma\sqrt{d\log\lt(nr\rt)}+n^{-100}.
\]
\end{cors}



Next, we develop concentration results for a family of functions that include indicator functions.  
Consider a set of independent random vectors $X_1,\ldots, X_{m} \in \real^{n}$ with $m\leq n$, 
and for each $1\leq i\leq m$, 
consider a collection of functions $f_{i, \theta}, h_{i, \theta} : \mathbb{R}^n \to \mathbb{R}$ indexed by $\theta \in \mathbb{B}^d(r)$. 
The following concentration bound --- whose proof is deferred to Section~\ref{sec:scriabin} --- proves useful when establishing our main results. 

\begin{lems}\label{lem:Gauss-jump}
Suppose that for any given $\theta \in \Theta \subseteq \mathbb{B}^d(r)$, 
the random variable $f_{i, \theta}(X_i) \ge 0$ is $\sigma_i$-subexponential. 
Assume that there exist a set of events $\mathcal{E}_i$ ($1\leq i\leq m\leq n)$ obeying $\mathbb{P}(\bigcap_{i}\mathcal{E}_i) > 1 - O(n^{-11})$ such that:
	for any $i \in [m]$ and any $\theta,\widetilde{\theta}\in \Theta$, 
\begin{subequations}
\label{eqn:lips-jump}
\begin{align}
	\big\| f_{i, \theta}(Z_i) - f_{i, \widetilde{\theta}}(Z_i) \big\|_2 
	+ \big\| h_{i, \theta}(Z_i) - h_{i, \widetilde{\theta}}(Z_i) \big\|_2 
	&\leq L\big\| \theta - \widetilde{\theta} \big\|_2
	\qquad \text{for all }Z_i \text{ with }\ind_{\mathcal{E}_i}(Z_i)=1 , \\ 
	%
	\myE\big[ f_{i, \theta}(X_i) \ind\lt(\mathcal{E}_i^{\mathrm{c}}\rt)\big] &\le B.
\end{align}
\end{subequations}
Also, for any $i\in [m]$ and any $\theta\in \Theta$, define
\begin{equation}
	\varrho_{i,\theta} \defn \myE\big[f_{i, \theta}(X_i)\big] + L\epsilon+\sigma_i\log n , \qquad 1\leq i\leq m. 
\end{equation}
Then for any $0<\epsilon < r$, with probability at least $1-O(n^{-11})$ one has
\begin{align}
 & \lt|\sum_{i=1}^{m}\Big(f_{i,\theta}(X_{i})\ind\big(h_{i,\theta}(X_{i})>\tau\big)-\myE\Big[f_{i,\theta}(X_{i})\ind\big(h_{i,\theta}(X_{i})>\tau\big)\Big]\Big)\rt|\nonumber\\
	& \notag\qquad\lesssim\sqrt{\sum_{i=1}^{m}\varrho_{i,\theta}^{2}\Big(\mathbb{P}\lt(h_{i,\theta}(X_{i})>\tau\rt)+\frac{1}{n}\Big)d\log\frac{rn}{\epsilon}}+\Big(\max_{1\leq i\leq m}\varrho_{i,\theta}\Big) d\log\frac{rn}{\epsilon} \\
 & \qquad\qquad\qquad+mL\epsilon + mB+\sum_{i=1}^{m}\varrho_{i,\theta}\mathbb{P}\Big(\tau-(3L+1)\epsilon\le h_{i,\theta}(X_{i})\le\tau+(3L+1)\epsilon\Big) 
	\label{eqn:scriabin}
\end{align}
simultaneously for all $\theta \in \Theta$ and all $\tau \in [-r,r]$.
\end{lems}
Similar to Corollary~\ref{cor:Gauss}, 
we can take $\epsilon = n^{-200}$ to derive the following immediate consequence. 
\begin{cors}
	\label{cor:Gauss-jump}
Under the assumptions of Lemma~\ref{lem:Gauss-jump}, suppose that  
\begin{subequations}
	\label{eq:f-grad-bound-poly-jump}
\begin{align}
	\big\| f_{i, \theta}(Z_i) - f_{i, \widetilde{\theta}}(Z_i) \big\|_2 
	+ \big\| h_{i, \theta}(Z_i) - h_{i, \widetilde{\theta}}(Z_i) \big\|_2 
	&\leq n^{100} \big\| \theta - \widetilde{\theta} \big\|_2
	\qquad \text{for all }Z_i \text{ with }\ind_{\mathcal{E}_i}(Z_i)=1  \\ 
	%
	\myE\big[ f_{i, \theta}(X_i) \ind\lt(\mathcal{E}_i^{\mathrm{c}}\rt)\big] &\le n^{-100}
\end{align}
\end{subequations}
for any $i\in [m]$ and any $\theta,\widetilde{\theta}\in \Theta$. Also, suppose that 
\[
	\mathbb{P}\Big(\tau-400n^{-100}\le h_{i,\theta}(X_{i})\le\tau+400n^{-100}\Big) \lesssim m^{-1} ,
	\qquad 1\leq i\leq m
\]
for any $\theta\in \Theta$ and any $\tau \in [-r,r]$. 
If we redefine
\begin{equation}
	\varrho_{i,\theta} \defn \myE\big[f_{i, \theta}(X_i)\big]  +\sigma_i\log n , \qquad 1\leq i\leq m,  
\end{equation}
then with probability at least $1-O(n^{-11})$ one has 
\begin{align*}
 & \lt|\sum_{i=1}^{m}\Big(f_{i,\theta}(X_{i})\ind\big(h_{i,\theta}(X_{i})>\tau\big)-\myE\Big[f_{i,\theta}(X_{i})\ind\big(h_{i,\theta}(X_{i})>\tau\big)\Big]\Big)\rt|\nonumber\\
	& \notag\qquad\lesssim\sqrt{\sum_{i=1}^{m}\varrho_{i,\theta}^{2}\Big(\mathbb{P}\lt(h_{i,\theta}(X_{i})>\tau\rt)+\frac{1}{n}\Big)d\log(rn)}+\Big(\max_{1\leq i\leq m}\varrho_{i,\theta}\Big)d\log(rn) + \frac{n+d\log(rn)}{n^{100}}
%
\end{align*}
simultaneously for all $\theta \in \Theta$ and all $\tau \in [-r,r]$. 
\end{cors}


The third result is concerned with norms of (linear combinations of) independent Gaussian vectors; 
the proof can be found in Section~\ref{sec:pf-brahms}. 
Here and throughout, for every vector $x\in \real^{n}$, we adopt the convention and let $|x|_{(i)}$ denote its $i$-th largest entry in magnitude. 
\begin{lems}
\label{lem:brahms-lemma}
	Consider a collection of independent Gaussian vectors $\{\phi_k\}_{1\leq k\leq n}$ with $\phi_k \overset{\mathrm{i.i.d.}}{\sim} \mathcal{N}(0,\frac{1}{n}I_n)$. 
With probability at least $1-\delta$, it holds that 
\begin{subequations}
\label{eqn:vive-brahms}
\begin{align}
	\label{eqn:brahms}
	\Big| \max_{1\leq k\leq t-1} \|\phi_k\|_2 -1 \Big| & \lesssim \sqrt{\frac{\log \frac{n}{\delta}}{n}}, \\
	\label{eqn:long}
	 \sup_{a = [a_k]_{1\leq k< t} \in \mathcal{S}^{t-2}} \bigg| \Big\|\sum_{k = 1}^{t-1} a_k\phi_k\Big\|_2 - 1 \bigg| & \lesssim \sqrt{\frac{t\log \frac{n}{\delta}}{n}}, \\
	\label{eqn:vive}
\sup_{a=[a_{k}]_{1\leq k<t}\in\mathcal{S}^{t-2}}\sum_{i=1}^{s}\Big|\sum_{k=1}^{t-1}a_{k}\phi_{k}\Big|_{(i)}^{2}
	&\lesssim\frac{(t+s)\log \frac{n}{\delta}}{n},\qquad \forall 1 \leq s\leq n.	
\end{align}
\end{subequations}
 \end{lems}


Finally, we state a lemma that quantifies the 1-Wasserstein distance between a weighted combination of independent Gaussian vectors (with the weights being possibly dependent on the Gaussian vectors) and an i.i.d.~Gaussian vector. The proof of this lemma can be found in Section~\ref{sec:proof-wasserstein}.
\begin{lems} 
\label{lem:wasserstein}
Consider a set of i.i.d.~random vectors $\phi_k \overset{\mathrm{i.i.d.}}{\sim} \mathcal{N}(0, \frac{1}{n}I_n)$, 
as well as any unit vector $\beta =[\beta_i]_{1\leq i\leq t} \in \mathcal{S}^{t-1}$ that might be statistically dependent on $\{\phi_k\}$. 
Then the 1-Wasserstein distance (cf.~\eqref{eqn:wasserstein-p}) between the distribution of $\sum_{i=1}^{t}\beta_{k}\phi_{k}$ --- denoted by $\mu\big(\sum_{i=1}^{t}\beta_{k}\phi_{k}\big)$  --- and $\mathcal{N}\big(0,\frac{1}{n}I_{n}\big)$ obeys 
\begin{align}
\label{eqn:gaussian-approx}
W_{1}\bigg(\mu\Big(\sum_{i=1}^{t}\beta_{k}\phi_{k}\Big),\mathcal{N}\Big(0,\frac{1}{n}I_{n}\Big)\bigg) \lesssim \sqrt{\frac{t\log n}{n}}. 
\end{align}
%
\end{lems}
%


\subsection{Proof of Lemma~\ref{lem:Gauss}}
\label{sec:pf-Gaussian}

Let us define
\begin{align*}
	g_{\theta}(X) \defn f_{\theta}(\mathcal{P}_{\mathcal{E}}(X)).
\end{align*}
By Lipschitz property of $f_{\theta}$ (cf.~\eqref{eqn:gauss-lipschitz}), we can obtain the Lipschitz property for $g_{\theta}$ as follows:
\begin{align*}
	\lt|g_{\theta}(X) - g_{\theta}(Y)\rt| = \lt|f_{\theta}(\mathcal{P}_{\mathcal{E}}(X)) - f_{\theta}(\mathcal{P}_{\mathcal{E}}(Y))\rt| \le \sigma\|\mathcal{P}_{\mathcal{E}}(X) - \mathcal{P}_{\mathcal{E}}(Y)\|_2 \le \sigma\|X - Y\|_2,
\end{align*}
where the last step uses the non-expansiveness of Euclidean projection onto convex sets. 
Gaussian isoperimetric inequalities (e.g., \citet[Theorem 3.8]{massart2007concentration}) then tells us that,  for any fixed $\theta \in \Theta$, 
%
\begin{align}
\label{eqn:schubert}
	\big|g_{\theta}(X) - \myE\lt[g_{\theta}(X)\rt]\big| \leq \sigma\sqrt{2\log\frac{1}{\delta}}
\end{align}
holds with probability at least $1 - \delta$.

Next, we need to establish uniform concentration over all $\theta\in \Theta$. 
Towards this, let us construct an $\epsilon$-net $\mathcal{N}_{\epsilon}$ for $\Theta$ with smallest size such that:
 for any $\theta \in \Theta$, there exists some $\thetahat \in \mathcal{N}_{\epsilon}$ obeying $\ltwo{\theta - \widehat{\theta}} \le \epsilon$.
Given that $\Theta \subseteq \real^d$, 
it is easily seen that the cardinality of the $\epsilon$-net can be chosen such that $|\mathcal{N}_{\epsilon}| \le (\frac{2r}{\epsilon})^{d}$ \citep[Chapter 4.2]{vershynin2018high}. 
Taking \eqref{eqn:schubert} with the union bound over the set $\mathcal{N}_{\epsilon}$ reveals that: with probability at least $1 - \delta$, 
\begin{align}
\label{eqn:schubert-impromptu}
	\sup_{\theta \in \mathcal{N}_{\epsilon}} \big|g_{\theta}(X) - \myE\lt[g_{\theta}(X)\rt] \big| \leq \sigma \sqrt{2d \log\Big(\frac{2 r}{\delta\epsilon}\Big)}. 
\end{align}
%

With the above concentration result in place, we are ready prove the advertised inequality~\eqref{eqn:Gauss-target}.  
First, recalling that $\lt|\myE\lt[g_{\theta}(X) - f_{\theta}(X)\rt]\rt| \le B$ and $f_{\theta}(X) = g_{\theta}(X)$ with probability at least $1-O(n^{-11})$, one has 
\begin{align*}
	\big|f_{\theta}(X) - \myE\lt[f_{\theta}(X)\rt]\big| 
	& = \big|g_{\theta}(X) - \myE\lt[f_{\theta}(X)\rt]\big|
	 \le \big|g_{\theta}(X) - \myE\lt[g_{\theta}(X)\rt]\big| + \big|\myE [f_{\theta}(X)] - \myE[g_{\theta}(X)]\big|\\
	&\le \big|g_{\theta}(X) - \myE\lt[g_{\theta}(X)\rt]\big| + B 
\end{align*}
with probability at least $1-O(n^{-11})$.
In addition, our assumption \eqref{eqn:gauss-lipschitz} also indicates that 
\begin{align*}
	&\big|g_{\theta_1}(X) - g_{\theta_2}(X)\big| = \big|f_{\theta_1}(\mathcal{P}_{\mathcal{E}}(X)) - f_{\theta_2}(\mathcal{P}_{\mathcal{E}}(X))\big| 
	\le L\|\theta_1 - \theta_2\|_2, \qquad \text{and} \\
	\big|\myE [g_{\theta_1}(X)] - \myE[g_{\theta_2}(X)]\big| &= \big|\myE[f_{\theta_1}(\mathcal{P}_{\mathcal{E}}(X)) - f_{\theta_2}(\mathcal{P}_{\mathcal{E}}(X))]\big| \le \myE\big[|f_{\theta_1}(\mathcal{P}_{\mathcal{E}}(X)) - f_{\theta_2}(\mathcal{P}_{\mathcal{E}}(X))|\big] \leq L\|\theta_1 - \theta_2\|_2. 
\end{align*}
%
Consequently, for every $\theta\in \Theta$, we obtain 
\begin{align*}
	\lt|f_{\theta}(X) - \myE\lt[f_{\theta}(X)\rt]\rt|
	&\le \lt|g_{\widehat{\theta}}(X) - \myE\lt[g_{\widehat{\theta}}(X)\rt]\rt| + 2L\epsilon + B \\
	&\lesssim \sigma\sqrt{d\log\lt(\frac{nr}{\epsilon}\rt)} + L\epsilon + B,
\end{align*}
with probability exceeding $1-O(n^{-11})$, 
where we invoke the Lipschitz property for $g_{\theta}$ and $\myE g_{\theta}$ with respect to $\theta$,
 and the last inequality follows from relation~\eqref{eqn:schubert-impromptu} with $\delta = n^{-11}$.
We have thus established Lemma~\ref{lem:Gauss}. 


\subsection{Proof of Lemma~\ref{lem:Gauss-jump}}
\label{sec:scriabin}

\paragraph{Step 1: establishing concentration for any fixed $\theta$ and $\tau$.}

For notational simplicity, let us introduce
\begin{align*}
	\mu_i &\defn \mathbb{E}\big[f_{i, \theta}(X_i)\ind(\mathcal{E}_i)\big] \\
	Z_i &\defn f_{i, \theta}(X_i)\ind\lt(h_{i, \theta}(X_i) > \tau\rt)\ind(\mathcal{E}_i) 
\end{align*}
for any $i\in [m]$, and any fixed $\theta\in \Theta$ and $\tau \in [-r,r]$. 
The goal of this step is to show the following Bernstein-type inequality: for any given $\theta\in \Theta$ and $\tau \in [-r,r]$,  
\begin{align}
\label{eqn:sum-zi}
	\Big|\sum_{i = 1}^m Z_i - \mathbb{E}[Z_i]\Big| \lesssim  
	\sqrt{\sum_{i = 1}^m \lt(\mu_i+\sigma_i\log n\rt)^2\lt(\mathbb{P}\big(h_{i, \theta}(X_i) > \tau\big) + \frac{1}{n}\rt)\log\frac{1}{\delta}}
	 + \max_{i} \lt(\mu_i+\sigma_i\log n\rt)\log\frac{1}{\delta} 
\end{align}
holds with probability at least $1 - \delta$, where $\delta$ can be any value in $(0,1)$. 
The remainder of this step is devoted to establishing \eqref{eqn:sum-zi}. 


We find it useful to first single out several preliminary facts. 
Recognizing that $f_{i, \theta}(X_i)$ is assumed to be non-negative, one has  $\mu_{i} \le \mathbb{E}\big[f_{i, \theta}(X_i)\big]$.  
Given that $f_{i, \theta}(X_i)$ is assumed to be $\sigma_i$-subexponential, 
we see that $f_{i, \theta}(X_i)\ind(\mathcal{E}_i)$ is also $\sigma_i$-subexponential, 
which further implies that the centered version $f_{i, \theta}(X_i)\ind(\mathcal{E}_i)-\mu_i$ is $O(\sigma_i)$-subexponential 
(see \citet[Exercise 2.7.10]{vershynin2018high}); 
this means that 
there exists some universal constant $c_5\geq 1$ such that
\begin{align}
	\mathbb{P}\Big( f_{i, \theta}(X_i)\ind(\mathcal{E}_i) \geq \mu_i + \tau \Big)
	\leq 
	\mathbb{P}\lt( \big|f_{i, \theta}(X_i)\ind(\mathcal{E}_i) - \mu_i \big| \geq \tau \rt)
	 \leq 2\exp\Big(- \frac{\tau}{c_5 \sigma_i} \Big) 
	 \label{eq:sub-exponential-simgai-135}
\end{align}
for any $\tau \geq 0$ and any $i \in [m]$.  
In what follows, we shall follow similar ideas for proving Bernstein's inequality (e.g., \cite{wainwright2019high}). 

For every integer $k\geq 1$, 
let us first look at the $k$-moment of $Z_{i} - \mathbb{E}[Z_i]$. 
Note that for any two non-negative numbers $a,b\geq 0$, one has $|(a-b)^k|\leq a^k + b^k$. 
This fact taken together with Jensen's inequality gives  
\begin{align}
\notag\mathbb{E}\lt[\big|Z_{i}-\mathbb{E}[Z_{i}]\big|^{k}\rt] & \leq\mathbb{E}\lt[Z_{i}^{k}\rt]+\big(\mathbb{E}[Z_{i}]\big)^{k}\leq2\mathbb{E}\lt[Z_{i}^{k}\rt]\\
\notag & =2\mathbb{E}\lt[\big(f_{i,\theta}(X_{i})\big)^{k}\ind\lt(h_{i,\theta}(X_{i})>\tau\rt)\ind(\mathcal{E}_{i})\rt]\\
	\notag & \leq2\mathbb{E}\lt[\big(f_{i,\theta}(X_{i})\big)^{k}\ind(\mathcal{E}_{i})\ind\lt(h_{i,\theta}(X_{i})>\tau\rt)\ind\big(f_{i,\theta}(X_{i})\leq \mu_{i}+c_5\sigma_{i}k\log n\big)\rt]\\
\notag & \qquad+2\mathbb{E}\lt[\big(f_{i,\theta}(X_{i})\big)^{k}\ind(\mathcal{E}_{i})\ind\big(f_{i,\theta}> \mu_{i}+ c_5\sigma_{i}k\log n\big)\rt]\\
 & \stackrel{(\text{i})}{\leq}
	2\lt(\mu_{i}+c_5\sigma_{i}k\log n\rt)^{k}\mathbb{E}\big[\ind\lt(h_{i,\theta}(X_{i})>\tau\rt)\big]+16(\mu_{i}+c_{5}\sigma_{i}k\log n)^{k}\exp(-k\log n) \notag\\
 & \leq 16 \lt(\mu_{i}+ c_5\sigma_{i}k\log n\rt)^{k}\Big(\mathbb{P}\big(h_{i,\theta}(X_{i})>\tau\big)+\frac{1}{n}\Big). 
\label{eqn:bernstein-condition}	
\end{align}
%
To justify why (i) is valid, we note that
\begin{align}
\mathbb{E}\left[\big(f_{i,\theta}(X_{i})\big)^{k}\ind(\mathcal{E}_{i})\ind\big(f_{i,\theta}(X_{i})>\mu_{i}+c_5\sigma_{i}k\log n\big)\right] 
	& \leq\int_{(\mu_{i}+c_5\sigma_{i}k\log n)^{k}}^{\infty}\mathbb{P}\left\{ \big(f_{i,\theta}(X_{i})\ind(\mathcal{E}_{i})\big)^{k}\geq\tau\right\} \mathrm{d}\tau\notag\\
 & =\int_{(\mu_{i}+c_5\sigma_{i}k\log n)^{k}}^{\infty}\mathbb{P}\left\{ f_{i,\theta}(X_{i})\ind(\mathcal{E}_{i})\geq\tau^{1/k}\right\} \mathrm{d}\tau\notag\\
 & =\int_{\mu_{i}+c_5\sigma_{i}k\log n}^{\infty}kx^{k-1}\mathbb{P}\Big\{ f_{i,\theta}(X_{i})\ind(\mathcal{E}_{i})\geq x\Big\}\mathrm{d}x\notag\\
 & =\int_{c_5\sigma_{i}k\log n}^{\infty}k(x+\mu_{i})^{k-1}\mathbb{P}\Big\{ f_{i,\theta}(X_{i})\ind(\mathcal{E}_{i})\geq x+\mu_{i}\Big\}\mathrm{d}x\notag\\
	& \overset{(\mathrm{ii})}{\leq}2k\int_{c_5\sigma_{i}k\log n}^{\infty}(x+\mu_{i})^{k-1}\exp\Big(-\frac{x}{c_{5}\sigma_{i}}\Big)\mathrm{d}x, \label{eq:RHS-135}
\end{align}
where (ii) follows from inequality~\eqref{eq:sub-exponential-simgai-135}. Now the right-hand side of the above inequality can be further controlled as 
\begin{align}
\eqref{eq:RHS-135} & =2c_{5}\sigma_{i}k\int_{k\log n}^{\infty}(c_{5}\sigma_{i}x+\mu_{i})^{k-1}\exp(-x)\mathrm{d}x\notag\\
 & \overset{(\mathrm{iii})}{\leq}2c_{5}\sigma_{i}k\sum_{l=k\log n}^{\infty}(c_{5}\sigma_{i}l+\mu_{i})^{k-1}\exp(-l)\notag\\
 & \overset{(\mathrm{iv})}{\leq}2c_{5}\sigma_{i}k\cdot(c_{5}\sigma_{i}k\log n+\mu_{i})^{k-1}\exp(-k\log n)\sum_{i=0}^{\infty}\Big(\frac{2}{e}\Big)^{i} \notag\\
 & \leq8(c_{5}\sigma_{i}k\log n+\mu_{i})^{k}\exp(-k\log n),
	\label{eq:exp-tail}
\end{align}
where (iii) is valid since the function $(ax+b)^{k}e^{-x}$ with $a,b>0$ is decreasing
in $x$ for any $x>k$, and (iv) holds since for any $l\geq k\log n$,
one has
\[
	\frac{\big((l+1)c_{5}\sigma_{i}+\mu_{i} \big)^{k}\exp(-(l+1))}{\big(lc_{5}\sigma_{i}+\mu_{i}\big)^{k}\exp(-l)}\le e^{-1}\Big(1+\frac{1}{l}\Big)^{k}\le \frac{2}{e},
\]
namely, $\big(lc_{5}\sigma_{i}+\mu_{i}\big)^{k}\exp(-l)$ decreases geometrically in $l$ with a contraction factor $2/e$.  
This validates Step (i) in \eqref{eqn:bernstein-condition}.

In view of \eqref{eqn:bernstein-condition}, letting $\widetilde{Z}_i \coloneqq Z_{i}-\mathbb{E}[Z_{i}]$ (so that $\mathbb{E}[\widetilde{Z}_i]=0$) and using the power series expansion, we obtain 
\begin{align*}
\mathbb{E}\lt[\sum_{k=0}^{\infty}\frac{\lambda^{k}\widetilde{Z}_{i}^{k}}{k!}\rt] & =1+\mathbb{E}\lt[\sum_{k=2}^{\infty}\frac{\lambda^{k}\widetilde{Z}_{i}^{k}}{k!}\rt]\le\exp\lt(\sum_{k=2}^{\infty}\frac{\mathbb{E}\lt[\lambda^{k}\widetilde{Z}_{i}^{k}\rt]}{k!}\rt)\\
 & \le\exp\lt(\sum_{k=2}^{\infty}\frac{8\lambda^{k}\lt(\mu_{i}+c_{5}\sigma_{i}k\log n\rt)^{k}}{k!}\lt(\mathbb{P}\big(h_{i,\theta}(X_{i})>\tau\big)+\frac{1}{n}\rt)\rt)\\
 & \le\exp\lt(16e^{2}c_5^2\lambda^{2} \lt(\mu_{i}+\sigma_{i}\log n\rt)^{2}\lt(\mathbb{P}\big(h_{i,\theta}(X_{i})>\tau\big)+\frac{1}{n}\rt)\rt) 
\end{align*}
for any $\lambda>0$ obeying $c_5\lambda(\mu_i+\sigma_i\log n) \le (2e)^{-1}$, 
where the first line applies the elementary inequality $1+x\leq\exp(x)$ for any $x\in \real$,
and the last line holds since, by taking  $z = c_5\lambda(\mu_i+\sigma_i\log n) \le (2e)^{-1}$, one has
\[
\sum_{k=2}^{\infty}\frac{\big[\lambda(\mu_{i}+c_{5}\sigma_{i}k\log n)\big]^{k}}{k!}\stackrel{(\text{v})}{\le}\sum_{k=2}^{\infty}\frac{\big[\lambda c_{5}(\mu_{i}+\sigma_{i}\log n)k\big]^{k}}{\sqrt{2\pi k}\,k^{k}e^{-k}}\leq\sum_{k=2}^{\infty}(ez)^{k}\leq e^{2}z^{2}\sum_{i=0}^{\infty}\frac{1}{2^{i}}=2e^{2}z^{2},
\]
where (v) follows from the fact $c_5\geq 1$ and the well-known Stirling inequality $\sqrt{2\pi}k^{k+\frac{1}{2}}e^{-k}\leq k!$. 
Given the above convergence of the power series, we conclude that for any $\lambda>0$ obeying $c_5\lambda(\mu_i+\sigma_i\log n) \le (2e)^{-1}$, 
\begin{align}
	\mathbb{E}\big[\exp\big(\lambda \widetilde{Z}_i\big)\big] = \mathbb{E}\lt[\sum_{k=0}^\infty \frac{\lambda^k \widetilde{Z}_i^k}{k!}\rt] 
\le\exp\lt(16e^{2}c_5^2\lambda^{2} \lt(\mu_{i}+\sigma_{i}\log n\rt)^{2}\lt(\mathbb{P}\big(h_{i,\theta}(X_{i})>\tau\big)+\frac{1}{n}\rt)\rt)
	\eqqcolon \exp(\lambda^2 \nu_0^2 ). 
\end{align}

%
%
%
%
%
%
%

To finish up, letting $L_{0}=\max_{1\leq i\leq n}2ec_{5}(\mu_{i}+\sigma_{i}\log n)$, we can apply Markov's inequality to obtain
\begin{align*}
\mathbb{P}\bigg(\sum_{i=1}^{m}\widetilde{Z}_{i}>\tau\bigg) & \le\min_{0<\lambda<\frac{1}{L_{0}}}\Bigg\{\exp(-\lambda\tau)\cdot\mathbb{E}\Big[\exp\Big(\lambda\sum_{i=1}^{m}\widetilde{Z}_{i}\Big)\Big]\Bigg\}\le\min_{0<\lambda<\frac{1}{L_{0}}}\Big\{\exp(-\lambda\tau)\exp\big(\lambda^{2}\nu_{0}^{2}\big)\Big\}.
\end{align*}
Repeating standard arguments for establishing Bernstein's inequality (see, e.g., \citet[Step 4 in Pages 118-119]{vershynin2018high}), 
one can immediately conclude that
\begin{align*}
	\sum_{i = 1}^m \widetilde{Z}_i 
	&\lesssim \max \bigg\{ \sqrt{\nu_0 \log \frac{1}{\delta}} + L_0 \log \frac{1}{\delta} \bigg\}  \\
	&\asymp 
	\sqrt{\sum_{i = 1}^m \lt(\mu_i+\sigma_i\log n\rt)^2\lt(\mathbb{P}\lt(h_{i, \theta}(X_i) > \tau\rt) + \frac{1}{n}\rt)\log\frac{1}{\delta}}
	+ 
	 \max_{i}\big\{ \mu_i+\sigma_i\log n \big\} \log\frac{1}{\delta} 
\end{align*}
with probability exceeding $1-\delta$. 
Repeating the same argument reveals that the above inequality continues to hold if $\sum_{i = 1}^m \widetilde{Z}_i$ is replaced with 
$-\sum_{i = 1}^m \widetilde{Z}_i$. 
This in turn establishes \eqref{eqn:sum-zi} any fixed $\theta\in \Theta$ and $\tau \in [-r,r]$.

\paragraph{Step 2: controlling the difference between the $\epsilon$-net and the remaining parameters.} 
To show uniform concentration over all $\theta$ and $\tau$, we intend to invoke an $\epsilon$-net-based argument. 
Towards this, let us 
construct an $\epsilon$-net $\mathcal{N}_{\epsilon}^{\Theta} \subseteq \Theta \subseteq \mathbb{B}^d(r)$ for the $d$-dimensional ball $\mathbb{B}^d(r)$ of radius $r$
--- which can be chosen to have cardinality $|\mathcal{N}_{\epsilon}^{\Theta}| \le (\frac{3r}{\epsilon})^{d}$ \citep[Chapter 4.2]{vershynin2018high} --- 
such that for any $\theta \in \Theta$, there exists some $\thetahat \in \mathcal{N}_{\epsilon}^{\Theta}$ satisfying $\ltwo{\thetahat - \theta} \le \epsilon<r$.
In addition, we construct another $\epsilon$-net $\mathcal{N}_{\epsilon}^{[-r,r]} \subseteq [-r, r]$ 
obeying $|\mathcal{N}_{\epsilon}^{[-r,r]}| \le \frac{2r}{\epsilon}$
for the interval $[-r,r]$, 
such that for any $\tau\in [-r,r]$, there exists $\widetilde{\tau}\in \mathcal{N}_{\epsilon}^{[-r,r]}$
obeying $|\tau - \widetilde{\tau}|\leq \epsilon$.

Let us now look at an arbitrary 
$\theta \in \Theta$ and its nearest neighbor $\thetahat$ in $\mathcal{N}_{\epsilon}^{\Theta}$ (so that $\|\thetahat - \theta\|_2 \le \epsilon$). 
In view of the Lipschitz property~\eqref{eqn:lips-jump}, we can deduce that
\begin{align}
\notag\sum_{i=1}^{m}f_{i,\theta}(X_{i})\ind\big(h_{i,\theta}(X_{i})>\tau\big)\ind(\mathcal{E}_{i}) & =\sum_{i=1}^{m}f_{i,\theta}(X_{i})\ind\lt(h_{i,\thetahat}(X_{i})>\tau+h_{i,\thetahat}(X_{i})-h_{i,\theta}(X_{i})\rt)\ind\lt(\mathcal{E}_{i}\rt)\\
 & \leq\sum_{i=1}^{m}f_{i,\thetahat}(X_{i})\ind\lt(h_{i,\thetahat}(X_{i})>\tau+h_{i,\thetahat}(X_{i})-h_{i,\theta}(X_{i})\rt)\ind\lt(\mathcal{E}_{i}\rt)+O(mL\epsilon) \notag\\
 & \le\sum_{i=1}^{m}f_{i,\thetahat}(X_{i})\ind\lt(h_{i,\thetahat}(X_{i})>\widehat{\tau}_{-}\rt)\ind(\mathcal{E}_{i}) +O(mL\epsilon)
\label{eqn:jump-lips-up}
\end{align}
for some point $\widehat{\tau}_- \in \mathcal{N}_{\epsilon}^{[-r,r]}$ satisfying
\begin{align*}
	\tau - (L+1)\epsilon \le \widehat{\tau}_- \le \tau + h_{i, \thetahat}(X_i) - h_{i, \theta}(X_i).
\end{align*}
Here, the second line in \eqref{eqn:jump-lips-up} applies the Lipschitz continuity of $f_{i,\theta}$ w.r.t.~$\theta$, 
while the last line 
relies on  the Lipschitz condition that $|h_{i, \thetahat}(X_i) - h_{i, \theta}(X_i)|  \leq L\epsilon$. 
Similarly, we have the following lower bound:
\begin{align}
\label{eqn:jump-lips-down}
\sum_{i = 1}^m f_{i, \theta}(X_i)\ind\big(h_{i, \theta}(X_i) > \tau\big)\ind\lt(\mathcal{E}_i\rt)  
\ge \sum_{i = 1}^m f_{i, \thetahat}(X_i)\ind\lt(h_{i, \thetahat}(X_i) > \widehat{\tau}_+\rt) \ind(\mathcal{E}_{i}) - O(mL\epsilon)
\end{align}
for some point $\widehat{\tau}_+ \in \mathcal{N}_{\epsilon}^{[-r,r]}$ obeying
\begin{align*}
\tau + h_{i, \thetahat}(X_i) - h_{i, \theta}(X_i) \le \widehat{\tau}_+ \le \tau + (L+1)\epsilon.
\end{align*}

Next, we turn attention to the mean term $\mathbb{E}\big[f_{i, \theta}(X_i)\ind\lt(h_{i, \theta}(X_i) > \tau\rt)\big].$
Consider $\widehat{\tau} \in \mathcal{N}_{\epsilon}^{[-r,r]}$ such that
\begin{align*}
 	|\widehat{\tau} - \tau| \le (L+1)\epsilon.
 \end{align*} 
Clearly, there might be more than one points in the $\epsilon$-net that are within distance  
$(L+1)\epsilon$ to $\tau$, and we shall specify the choice of $\widehat{\tau}$ momentarily. 
Recall the assumption $\myE\big[ f_{i, \theta}(X_i) \ind\lt(\mathcal{E}_i^{\mathrm{c}}\rt)\big] \le B$ (cf.~\eqref{eqn:lips-jump}) 
and the non-negativity of $f_{i,\theta}$ to arrive at 
\begin{align}
&\mathbb{E}\big[f_{i, \theta}(X_i)\ind\lt(h_{i, \theta}(X_i) > \tau\rt)\big] 
\leq \mathbb{E}\big[f_{i, \theta}(X_i)\ind\lt(h_{i, \theta}(X_i) > \tau\rt)\ind\lt(\mathcal{E}_i\rt)\big] + O(B) \notag\\
&\qquad \leq \mathbb{E}\lt[f_{i, \theta}(X_i)\ind\lt(h_{i, \theta}(X_i) > h_{i, \thetahat}(X_i) - h_{i, \theta}(X_i) + \tau\rt)\ind\lt(\mathcal{E}_i\rt)\rt] + O\lt( B\rt) \notag\\
&\qquad\leq \mathbb{E}\lt[f_{i, \thetahat}(X_i)\ind\lt(h_{i, \thetahat}(X_i) > h_{i, \thetahat}(X_i) - h_{i, \theta}(X_i) + \tau\rt)\ind\lt(\mathcal{E}_i\rt)\rt] + O\lt(L\epsilon + B\rt) \notag\\
&\qquad\leq \mathbb{E}\lt[f_{i, \thetahat}(X_i)\ind\lt(h_{i, \thetahat}(X_i) > \widehat{\tau}\rt)\ind\lt(\mathcal{E}_i\rt)\rt] \notag\\
&\qquad \quad + O\lt(\mathbb{E}\lt[f_{i, \thetahat}(X_i)\ind\lt(\tau-(L+1)\epsilon \le h_{i, \thetahat}(X_i) \le \tau+(L+1)\epsilon)\rt)\ind\lt(\mathcal{E}_i\rt)\rt] + L\epsilon + B\rt). 
	\label{eq:E-f-indicator-1729}
\end{align}
Here, the second inequality follows from the Lipschitz continuity of $f_{i,\theta}$ w.r.t.~$\theta$ (cf.~\eqref{eqn:lips-jump}), and
the last inequality holds since 
\begin{align}
\notag &\lt\{Z_i : \ind\lt(h_{i, \thetahat}(Z_i) > h_{i, \thetahat}(Z_i) - h_{i, \theta}(Z_i) + \tau\rt) \ne \ind\lt(h_{i, \thetahat}(Z_i) > \widehat{\tau}\rt)\rt\} \\
\notag &\subseteq \Big\{Z_i : \widehat{\tau} \le h_{i, \thetahat}(Z_i) \le h_{i, \thetahat}(Z_i) - h_{i, \theta}(Z_i) + \tau\Big\}
\cup \Big\{Z_i : h_{i, \thetahat}(Z_i) - h_{i, \theta}(Z_i) + \tau \le h_{i, \thetahat}(Z_i) \le \widehat{\tau}\Big\} \\
&\subseteq \Big\{Z_i : \tau-(L+1)\epsilon \le h_{i, \thetahat}(Z_i) \le \tau+(L+1)\epsilon\Big\}. \label{eqn:set-relation}
\end{align}
where we invoke again $|h_{i, \thetahat}(Z_i) - h_{i, \theta}(Z_i)| \leq L\epsilon$ and 
$|\widehat{\tau} - \tau| \le (L+1)\epsilon$.  
Let us augment the notation $\mu_i$ to make explicit the dependency on $\theta_i$ as follows
\begin{equation}
	\mu_{i,\theta} \defn \mathbb{E}\big[f_{i, \theta}(X_i)\ind(\mathcal{E}_i)\big] .
	\label{eq:defn-mu-i-theta}
\end{equation}
Then an application of the bound~\eqref{eq:exp-tail} with $k=1$ leads directly to
\begin{align*}
 & \mathbb{E}\lt[f_{i,\thetahat}(X_{i})\ind\lt(\tau-(2L+1)\epsilon\le h_{i,\widehat{\theta}}(X_{i})\le\tau+(2L+1)\epsilon)\rt)\ind\lt(\mathcal{E}_{i}\rt)\rt]\\
 & \qquad=\mathbb{E}\lt[f_{i,\thetahat}(X_{i})\ind\lt(\tau-(2L+1)\epsilon\le h_{i,\widehat{\theta}}(X_{i})\le\tau+(2L+1)\epsilon)\rt)\ind\lt(\mathcal{E}_{i}\rt)\ind\lt(f_{i,\thetahat}(X_{i})\leq\mu_{i,\thetahat}+c_5\sigma_{i}\log n\rt)\rt]\\
 & \qquad\qquad+\mathbb{E}\lt[f_{i,\thetahat}(X_{i})\ind\lt(\tau-(2L+1)\epsilon\le h_{i,\widehat{\theta}}(X_{i})\le\tau+(2L+1)\epsilon)\rt)\ind\lt(\mathcal{E}_{i}\rt)\ind\lt(f_{i,\thetahat}(X_{i})>\mu_{i,\thetahat}+c_5\sigma_{i}\log n\rt)\rt]\\
 & \qquad\lesssim(\mu_{i,\thetahat}+\sigma_{i}\log n)\mathbb{P}\Big(\tau-(2L+1)\epsilon\le h_{i,\widehat{\theta}}(X_{i})\le\tau+(2L+1)\epsilon)\Big)+\frac{\mu_{i,\thetahat}+\sigma_{i}\log n}{n}.
\end{align*}
Substituting it into \eqref{eq:E-f-indicator-1729} yields 
\begin{align}
	\notag &\mathbb{E}\Big[f_{i, \theta}(X_i)\ind\big(h_{i, \theta}(X_i) > \tau\big)\Big] 
	- 
	\mathbb{E}\lt[f_{i, \thetahat}(X_i)\ind\big(h_{i, \thetahat}(X_i) > \widehat{\tau}\big)\ind\lt(\mathcal{E}_i\rt)\rt]\\ 
	&\qquad \lesssim 
	(\mu_{i,\thetahat} + \sigma_i\log n)\mathbb{P}\Big(\tau-(2L+1)\epsilon \le h_{i, \thetahat}(X_i) \le \tau+(2L+1)\epsilon\Big) + \frac{\mu_{i,\thetahat} + \sigma_i\log n}{n} + L\epsilon + B. 
	\label{eqn:expectation-jump-123}
\end{align}
Clearly, repeating the above argument shows that \eqref{eqn:expectation-jump-123} continues to hold if the left-hand side of \eqref{eqn:expectation-jump-123} is replaced by
$\mathbb{E}\big[f_{i, \thetahat}(X_i)\ind\big(h_{i, \thetahat}(X_i) > \widehat{\tau}\big)\ind\lt(\mathcal{E}_i\rt)\big] - \mathbb{E}\big[f_{i, \theta}(X_i)\ind\big(h_{i, \theta}(X_i) > \tau\big)\big]$. As a result, we conclude that
\begin{align}
	\notag &\Big| \mathbb{E}\Big[f_{i, \theta}(X_i)\ind\big(h_{i, \theta}(X_i) > \tau\big)\Big] 
	- 
	\mathbb{E}\lt[f_{i, \thetahat}(X_i)\ind\big(h_{i, \thetahat}(X_i) > \widehat{\tau}\big)\ind\lt(\mathcal{E}_i\rt)\rt] \Big|\\ 
	&\qquad \lesssim 
	(\mu_{i,\thetahat} + \sigma_i\log n)\mathbb{P}\Big(\tau-(2L+1)\epsilon \le h_{i, \thetahat}(X_i) \le \tau+(2L+1)\epsilon\Big) + \frac{\mu_{i,\thetahat} + \sigma_i\log n}{n} + L\epsilon + B. 
	\label{eqn:expectation-jump}
\end{align}

\paragraph{Step 3: establishing uniform convergence.}
We are now ready to establish the advertised concentration result~\eqref{eqn:scriabin}. 
Recall that the concentration result~\eqref{eqn:sum-zi} holds for every fixed $(\thetahat,\widehat{\tau})$ pair. 
By taking the union bound over all points in $\mathcal{N}_{\epsilon}^{\Theta}\times \mathcal{N}_{\epsilon}^{[-r,r]}$
and setting $\delta = n^{-11}(\frac{\epsilon}{3r})^{d+1}$, 
we can see that with probability at least $1 - \delta  (\frac{3r}{\epsilon})^{d+1} = 1 - O(n^{-11})$,  
\begin{align}
\label{eqn:sum-zi-eps-net}
	&\bigg|\sum_{i = 1}^m f_{i, \thetahat}(X_i)\ind\lt(h_{i, \thetahat}(X_i) > \widehat{\tau}\rt) \ind(\mathcal{E}_i)  
	- \mathbb{E}\Big[f_{i, \thetahat}(X_i) \ind\lt(h_{i, \thetahat}(X_i) > \widehat{\tau}\rt) \ind(\mathcal{E}_i) \Big]\bigg|   \notag\\
	&\qquad \lesssim\sqrt{\sum_{i = 1}^m \big(\mu_{i,\thetahat}+\sigma_i\log n\big)^2\lt(\mathbb{P}\big(h_{i, \thetahat}(X_i) > \widehat{\tau} \big) 
	+ \frac{1}{n}\rt) d\log\frac{nr}{\epsilon}}
	 + \max_{i} \big(\mu_{i,\thetahat}+\sigma_i\log n\big) d\log\frac{nr}{\epsilon} 
\end{align}
holds simultaneously for all $(\thetahat,\widehat{\tau})\in \mathcal{N}_{\epsilon}^{\Theta}\times \mathcal{N}_{\epsilon}^{[-r,r]}$.

Consider an arbitrary point $\theta \in \Theta$ and $\tau\in [-r,r]$;   
let  $\thetahat$ be its closest point in $\mathcal{N}_{\epsilon}^{\theta}$, 
and take $\widehat{\tau}$ to be either $\widehat{\tau}_-$ or $\widehat{\tau}_+$.
Combining \eqref{eqn:jump-lips-up}, \eqref{eqn:jump-lips-down} and \eqref{eqn:expectation-jump} 
and using the assumption $\mathbb{P}(\cap_{i}\mathcal{E}_i)\geq 1-O(n^{-11})$ lead to: with probability at least $1-O(n^{-11})$, 
\begin{align*}
&\lt|\sum_{i = 1}^m \Big(f_{i, \theta}(X_i)\ind\lt(h_{i, \theta}(X_i) > \tau\rt) - \mathbb{E}\big[f_{i, \theta}(X_i)\ind\big(h_{i, \theta}(X_i) > \tau\big)\big]\Big)\rt| \\
&\qquad
=\lt|\sum_{i = 1}^m \Big(f_{i, \theta}(X_i)\ind\lt(h_{i, \theta}(X_i) > \tau\rt)\ind(\mathcal{E}_i) - \mathbb{E}\big[f_{i, \theta}(X_i)\ind\big(h_{i, \theta}(X_i) > \tau\big)\big]\Big)\rt| \\
&\qquad \le 
\lt|\sum_{i = 1}^m \lt(f_{i, \thetahat}(X_i)\ind\big(h_{i, \thetahat}(X_i) > \widehat{\tau}\big)\ind(\mathcal{E}_i) - \mathbb{E}\lt[f_{i, \thetahat}(X_i)\ind\big(h_{i, \thetahat}(X_i) > \widehat{\tau}\big)\ind\lt(\mathcal{E}_i\rt)\rt]\rt)\rt| \\
&\qquad\qquad+ O\bigg(mL\epsilon + mB + \sum_i \big(\mu_{i,\thetahat} + \sigma_i\log n\big)\Big[\mathbb{P}\Big(\tau-(2L+1)\epsilon \le h_{i, \thetahat}(X_i) \le \tau+(2L+1)\epsilon \Big) + \frac{1}{n}\Big]\bigg)\\
&\qquad\lesssim \sqrt{\sum_{i = 1}^m \big(\mu_{i,\thetahat}+\sigma_i\log n\big)^2\Big(\mathbb{P}\lt(h_{i, \widehat{\theta}}(X_i) > \widehat{\tau}\rt)+\frac{1}{n} \Big)d\log\frac{rn}{\epsilon}} + \max_i \big(\mu_{i,\thetahat}+\sigma_i\log n\big)d\log\frac{rn}{\epsilon} \\
&\qquad\qquad+ mL\epsilon + mB + \sum_i \big(\mu_{i,\thetahat} + \sigma_i\log n\big)
	\underbrace{\mathbb{P}\Big(\tau-(2L+1)\epsilon \le h_{i, \thetahat}(X_i) \le \tau+(2L+1)\epsilon \Big)}_{\eqqcolon\, \widehat{p}_i}, 
\end{align*}
where the last inequality follows from \eqref{eqn:sum-zi-eps-net}. 
Additionally, recognizing that $|h_{i, \thetahat}(X_i) - h_{i, \theta}(X_i)| \leq  L\epsilon$ (see \eqref{eqn:lips-jump}), 
we see that
\[
\widehat{p}_{i}\leq\mathbb{P}\Big(\tau-(3L+1)\epsilon\le h_{i,\theta}(X_{i})\le\tau+(3L+1)\epsilon\Big)\eqqcolon p_{i}.
\]

Finally, we remind the readers of the set relation~\eqref{eqn:set-relation}. Repeating the argument in \eqref{eq:E-f-indicator-1729}, we obtain
\begin{align*}
	\mathbb{P}\lt(h_{i, \widehat{\theta}}(X_i) > \widehat{\tau}\rt) &\le 
\mathbb{P}\big(h_{i, \theta}(X_i) > \tau\big) + 
\mathbb{P}\Big(\tau-(2L+1)\epsilon \le h_{i, \theta}(X_i) \le \tau+(2L+1)\epsilon \Big) \\ 
	&\leq \mathbb{P}\big(h_{i, \theta}(X_i) > \tau\big) + p_i.
\end{align*}
We also make note of the following elementary relation 
\begin{align*}
	\sqrt{\sum_{i = 1}^m \big(\mu_{i,\thetahat}+\sigma_i\log n\big)^2 p_i d\log\frac{rn}{\epsilon}} 
	&\leq 
	\sqrt{\lt\{ \max_i \big( \mu_{i,\thetahat}+\sigma_i\log n\big)d\log\frac{rn}{\epsilon} \rt\} \sum_{i = 1}^m \big(\mu_{i,\thetahat}+\sigma_i\log n\big) p_i } \\
	&\leq 
	\max_i \big(\mu_{i,\thetahat}+\sigma_i\log n\big)d\log\frac{rn}{\epsilon} 
	+
	\sum_i \big(\mu_{i,\thetahat} + \sigma_i\log n \big)p_i.
\end{align*}
Putting the above pieces together then yields 
\begin{align}
 & \lt|\sum_{i=1}^{m}\Big(f_{i,\theta}(X_{i})\ind\lt(h_{i,\theta}(X_{i})>\tau\rt)-\mathbb{E}\big[f_{i,\theta}(X_{i})\ind\big(h_{i,\theta}(X_{i})>\tau\big)\big]\Big)\rt| \notag\\
 & \qquad\lesssim\sqrt{\sum_{i=1}^{m}\lt(\mu_{i,\thetahat}+\sigma_{i}\log n\rt)^{2}\Big(\mathbb{P}\lt(h_{i,\theta}(X_{i})>\tau\rt)+p_{i}+\frac{1}{n}\Big)d\log\frac{rn}{\epsilon}}+\max_i\big(\mu_{i,\thetahat}+\sigma_{i}\log n\big)d\log\frac{rn}{\epsilon} \notag\\
 & \qquad\qquad\qquad+mL\epsilon+mB+\sum_{i}(\mu_{i,\thetahat}+\sigma_{i}\log n)p_{i} \notag\\
 & \qquad\lesssim\sqrt{\sum_{i=1}^{m}\lt(\mu_{i,\thetahat}+\sigma_{i}\log n\rt)^{2}\Big(\mathbb{P}\lt(h_{i,\theta}(X_{i})>\tau\rt)+\frac{1}{n}\Big)d\log\frac{rn}{\epsilon}}+\max\big(\mu_{i,\thetahat}+\sigma_{i}\log n\big)d\log\frac{rn}{\epsilon} \notag\\
 & \qquad\qquad\qquad+mL\epsilon+mB+\sum_{i}(\mu_{i,\thetahat}+\sigma_{i}\log n)p_{i}.
	\label{eq:sum-f-dc-abded}
\end{align}
Finally, recall that the Lipschitz continuity of $f_{i,\theta}$ w.r.t.~$\theta$ gives 
\[
	\mu_{i,\thetahat} = \mathbb{E}\big[f_{i, \thetahat}(X_i)\ind(\mathcal{E}_i)\big] 
	\leq \mathbb{E}\big[f_{i, \theta}(X_i)\ind(\mathcal{E}_i)\big] + L\epsilon 
	\le \mathbb{E}\big[f_{i, \theta}(X_i)\big] + L\epsilon .
\]
Substitution into \eqref{eq:sum-f-dc-abded} thus completes the proof.

%
%
%
%


\subsection{Proof of Lemma~\ref{lem:brahms-lemma}} 
\label{sec:pf-brahms}

For a set of random vectors $\{\phi_k\}_{k=1}^{t-1}$ independently drawn from $\mathcal{N}(0,\frac{1}{n}I_n)$, 
standard concentration results for Wishart matrices (e.g., \citet[Example 6.2]{wainwright2019high}) together with the union bound tell us that  
\begin{align}
\label{eqn:simple-rm}
	\lt\|(\phi_1, \ldots, \phi_{t-1})^{\top}(\phi_1, \ldots, \phi_{t-1}) - I_{t-1}\rt\| \lesssim \sqrt{\frac{t\log \frac{n}{\delta}}{n}},\qquad\text{for any } 1 < t \leq n 
\end{align}
with probability at least $1 - \delta$. 
Two immediate consequences of \eqref{eqn:simple-rm} are in order.
\begin{itemize}
	\item First, taking $t=2$ in \eqref{eqn:simple-rm} reveals that with probability at least $1 - \delta$, 
\[
\Big|\|\phi_{1}\|_{2}^{2}-1\Big|\lesssim\sqrt{\frac{\log \frac{n}{\delta}}{n}}\text{\ensuremath{\qquad}}\Longrightarrow\qquad\Big|\|\phi_{1}\|_{2}-1\Big|=\frac{\big|\|\phi_{1}\|_{2}^{2}-1\big|}{\big|\|\phi_{1}\|_{2}+1\big|} \leq \Big|\|\phi_{1}\|_{2}^{2}-1\Big| \lesssim\sqrt{\frac{\log \frac{n}{\delta}}{n}}.
\]	
	Clearly, this inequality holds if $\phi_1$ is replaced by any other $\phi_k$. Taking the union bound over all $1\leq k \leq n$ establishes inequality \eqref{eqn:brahms}.

	
	\item As another direct consequence of \eqref{eqn:simple-rm}, we have, with probability at least $1 - \delta$,  
\begin{align*}
	\sup_{a = [a_k]_{1\leq k< t} \in \mathcal{S}^{t-2}} 
	 \bigg| \Big\|\sum_{k = 1}^{t-1} a_k\phi_k \Big\|_2^2 - 1 \bigg| \lesssim \sqrt{\frac{t\log \frac{n}{\delta}}{n}} ,
\end{align*}
		which allows us to establish the claim \eqref{eqn:long} as follows: 
\begin{align}
\label{eqn:spect-brahms}
	 \sup_{a = [a_k]_{1\leq k< t} \in \mathcal{S}^{t-2}} 
	\bigg| \Big\|\sum_{k = 1}^{t-1} a_k\phi_k \Big\|_2 - 1 \bigg| 
	= \sup_{a = [a_k]_{1\leq k< t} \in \mathcal{S}^{t-2}} 
	 \frac{\Big| \big\|\sum_{k = 1}^{t-1} a_k\phi_k \big\|_2^2 - 1 \Big|}{ \big\|\sum_{k = 1}^{t-1} a_k\phi_k \big\|_2 + 1  }
	\lesssim \sqrt{\frac{t\log \frac{n}{\delta}}{n}}.
\end{align}

\end{itemize}

%

Next, we turn attention to the claim~\eqref{eqn:vive}. 
Here and throughout, for any vector $x\in \real^n$ and any index set $S\subseteq [n]$, 
we let $x_S$ denote the subvector of $x$ formed by the entries of $x$ at indices from $S$. 
Following the discretization argument \citep[Chapter 5]{wainwright2019high}, we can construct an $\epsilon$-net $\mathcal{N}_{\epsilon}$ on $\mathcal{S}^{t-2}$ --- which can be chosen such that its cardinality does not exceed $(3/\epsilon)^t$ \citep[Eq.~(4.10)]{vershynin2018high} --- such that for any $a\in \mathcal{S}^{t-2}$, one can find a point $\widetilde{a}\in \mathcal{N}_{\epsilon}$ such that $\|a-\widetilde{a}\|_2\leq \epsilon<1$. 

\begin{itemize}
	\item We first bound the supermum over $\mathcal{N}_{\epsilon}$. 
Note that for any fixed $a \in \mathcal{N}_\epsilon \subseteq \mathcal{S}^{t-2}$ and any subset $S \subseteq [n]$ with $|S| = s$, 
the vector $(\sum_{k = 1}^{t-1} a_k\phi_k)_S$ is a Gaussian vector drawn from $\mathcal{N}(0, \frac{1}{n}I_s)$. 
Applying \citet[Proposition 1]{hsu2012tail} then implies that
%
%
%
\[
\mathbb{P}\left\{ \bigg\|\Big(\sum_{k=1}^{t-1}a_{k}\phi_{k}\Big)_{S}\bigg\|_{2}>\frac{2}{\sqrt{n}}(\sqrt{s}+\sqrt{\tau})\right\} \leq\mathbb{P}\left\{ \bigg\|\Big(\sum_{k=1}^{t-1}a_{k}\phi_{k}\Big)_{S}\bigg\|_{2}^{2}>\frac{s}{n}+\frac{2\sqrt{s\tau}}{n}+\frac{2\tau}{n}\right\} \leq e^{-\tau}
\]
for any $\tau>0$. 
		Setting $\tau=  \log \big( \frac{1}{\delta} (\frac{3}{\epsilon})^t {n \choose s} \big)$ and combining this inequality with the union bound over all $a\in \mathcal{N}_{\epsilon}$ and all $S\subseteq [n]$ with $|S|=s$  lead to
\begin{align*}
 & \mathbb{P}\Bigg\{\sup_{\substack{a\in\mathcal{N}_{\epsilon}\\
S\subset[n],|S|=s
}
}\bigg\|\Big(\sum_{k=1}^{t-1}a_{k}\phi_{k}\Big)_{S}\bigg\|_{2}>\frac{2\sqrt{s}}{\sqrt{n}}+\frac{2}{\sqrt{n}}\sqrt{\log\left(\frac{1}{\delta}\Big(\frac{3}{\epsilon}\Big)^{t}{n \choose s}\right)}\Bigg\}\\
 & \qquad\leq\sum_{\substack{a\in\mathcal{N}_{\epsilon}\\
S\subset[n],|S|=s
}
}\mathbb{P}\Bigg\{\bigg\|\Big(\sum_{k=1}^{t-1}a_{k}\phi_{k}\Big)_{S}\bigg\|_{2}>\frac{2\sqrt{s}}{\sqrt{n}}+\frac{2}{\sqrt{n}}\sqrt{\log\bigg(\frac{1}{\delta}\Big(\frac{3}{\epsilon}\Big)^{t}{n \choose s}\bigg)}\Bigg\}\\
 & \qquad\leq\Big(\frac{3}{\epsilon}\Big)^{t}{n \choose s}\exp\left(-\log\bigg(\frac{1}{\delta}\Big(\frac{3}{\epsilon}\Big)^{t}{n \choose s}\bigg)\right)
	\leq \delta. 
\end{align*}
		Taking $\epsilon = (\delta/n)^{10}$ and using ${n \choose s} \leq n^s$ imply that: with probability exceeding $1-\delta$, 
\begin{equation}
\sup_{\substack{a\in\mathcal{N}_{\epsilon}\\
S\subset[n],|S|=s
}
	}\bigg\|\Big(\sum_{k=1}^{t-1}a_{k}\phi_{k}\Big)_{S}\bigg\|_{2}\lesssim\frac{\sqrt{s \log\frac{n}{\delta}}}{\sqrt{n}}+\frac{\sqrt{t\log \frac{n}{\delta}}}{\sqrt{n}}\label{eq:sup-Neps-a-S-sum}
\end{equation}

\item Next, consider an arbitrary vector $a\in \mathcal{S}^{t-2}$ and let $\widetilde{a}\in \mathcal{N}_{\epsilon}$ obey $\|a-\widetilde{a}\|_2\leq \epsilon= (\delta/n)^{10}$. 
	Then \eqref{eq:sup-Neps-a-S-sum} together with the triangle inequality tells us that with probability exceeding $1-\delta$, 
	\begin{align*}
\bigg\|\Big(\sum_{k=1}^{t-1}a_{k}\phi_{k}\Big)_{S}\bigg\|_{2} & \leq\bigg\|\Big(\sum_{k=1}^{t-1}\widetilde{a}_{k}\phi_{k}\Big)_{S}\bigg\|_{2}+\bigg\|\Big(\sum_{k=1}^{t-1}(a_{k}-\widetilde{a}_{k})\phi_{k}\Big)_{S}\bigg\|_{2}\\
		& \lesssim\frac{\sqrt{s\log \frac{n}{\delta}}}{\sqrt{n}}+\frac{\sqrt{t\log \frac{n}{\delta}}}{\sqrt{n}}+\|a-\widetilde{a}\|_{2}\Big\|\Big[\phi_{1},\cdots,\phi_{t-1}\Big]\Big\|\\
		& \asymp\frac{\sqrt{s\log \frac{n}{\delta}}}{\sqrt{n}}+\frac{\sqrt{t\log \frac{n}{\delta}}}{\sqrt{n}},
\end{align*}
		where the last line holds since $\|a-\widetilde{a}\|_2\leq (\delta/n)^{10}$ and, with probability exceeding $1-\delta$, $\big\|\big[\phi_{1},\cdots,\phi_{t-1}\big]\big\|\leq \sqrt{t/\delta}$ \citep[Chapter 4.4]{vershynin2018high}. Given that $a$ can be an arbitrary vector lying within $\mathcal{S}^{t-2}$, we have concluded the proof of the claim~\eqref{eqn:vive}.

\end{itemize}

\subsection{Proof of Lemma~\ref{lem:wasserstein}}
\label{sec:proof-wasserstein}



Recall the definition \eqref{eqn:wasserstein-p} of the Wasserstein metric between to probability measures.
In view of the celebrated Kantorovich-Rubinstein duality, the 1-Wasserstein distance admits the following dual representation:
\begin{align}
	W_1(\mu, \nu) = \sup \Big\{\mathbb{E}_{\mu}[f] - \mathbb{E}_{\nu}[f] : f \text{ is } 1\text{-Lipschitz}\Big\}, 
	\label{eq:K-R-duality}
\end{align}
which is the key to establishing this lemma. 


Let us start by considering any given $1$-Lipschitz function $f$. 
It is assumed without loss of generality that $f(0) = 0$ (as the expression \eqref{eq:K-R-duality} only involves the difference of $f$), which together with the 1-Lipschitz property gives 
\begin{align}
	|f(x)| = |f(x) - f(0)| \le \|x\|_2 .
	\label{eq:fx-size-norm2}
\end{align}
For any fixed unit vector $\widetilde{\beta}=[\widetilde{\beta}_k]_{1\leq k\leq t}\in \mathcal{S}^{t-1}$, the vector
$\sum_{i = 1}^t \widetilde{\beta}_k \phi_k$ clearly follows a Gaussian distribution $\mathcal{N}(0, \frac{1}{n}I_n)$.  
Applying Gaussian isoperimetric inequalities (e.g., \citet[Theorem 3.8]{massart2007concentration}) yields
\begin{align}
\label{eqn:blues}
	f\Big(\sum_{i = 1}^t \widetilde{\beta}_k \phi_k\Big) - \mathop{\mathbb{E}}\limits_{g \sim \mathcal{N}(0, \frac{1}{n}I_n)}\big[f(g)\big] 
	\le \sqrt{\frac{2\log\frac{1}{\delta}}{n}} 
\end{align}
with probability at least $1 - \delta$. 
Next, let us construct an $\epsilon$-net $\mathcal{N}_{\epsilon}$ of $\mathcal{S}^{t-1}$ with cardinality not exceeding $(2/\epsilon )^t$, 
such that for any $\widehat{\beta} \in  \mathcal{S}^{t-1}$, one can find a point $\widetilde{\beta}\in \mathcal{N}_{\epsilon}$ obeying $\|\widehat{\beta} - \widetilde{\beta}\|_2\leq \epsilon$. 
Taking the above inequality with the union bound over $\mathcal{N}_{\epsilon}$ then leads to
\begin{align*}
\sup_{\widetilde{\beta}\in\mathcal{N}_{\epsilon}}\bigg\{ f\Big(\sum_{i=1}^{t}\widetilde{\beta}_{k}\phi_{k}\Big)-\mathop{\mathbb{E}}\limits _{g\sim\mathcal{N}(0,\frac{1}{n}I_{n})}\big[f(g)\big]\bigg\} & \lesssim\sqrt{\frac{t\log n}{n}}
\end{align*}
with probability at least $1-O(n^{-11})$. 
Armed with this result, for an arbitrary $\widehat{\beta} \in  \mathcal{S}^{t-1}$ one can show that
\begin{align*}
f\Big(\sum_{i=1}^{t}\widehat{\beta}_{k}\phi_{k}\Big)-\mathop{\mathbb{E}}\limits _{g\sim\mathcal{N}(0,\frac{1}{n}I_{n})}\big[f(g)\big] & \leq\Bigg\{ f\Big(\sum_{i=1}^{t}\widetilde{\beta}_{k}\phi_{k}\Big)-\mathop{\mathbb{E}}\limits _{g\sim\mathcal{N}(0,\frac{1}{n}I_{n})}\big[f(g)\big]\Bigg\}+f\Big(\sum_{i=1}^{t}\widehat{\beta}_{k}\phi_{k}\Big)-f\Big(\sum_{i=1}^{t}\widetilde{\beta}_{k}\phi_{k}\Big)\\
 & \lesssim\sqrt{\frac{t\log n}{n}}+\bigg\|\sum_{i=1}^{t}\widehat{\beta}_{k}\phi_{k}-\sum_{i=1}^{t}\widetilde{\beta}_{k}\phi_{k}\bigg\|_{2}\asymp\sqrt{\frac{t\log n}{n}}+\bigg\|\sum_{i=1}^{t}\big(\widetilde{\beta}_{k}-\widehat{\beta}_{k}\big)\phi_{k}\bigg\|_{2}\\
 & \lesssim\sqrt{\frac{t\log n}{n}}+\big\|\widetilde{\beta}-\widehat{\beta}\big\|_{2}\big\|\left[\phi_{1},\cdots,\phi_{t}\right]\big\|\lesssim\sqrt{\frac{t\log n}{n}}+\epsilon\,\big\|\left[\phi_{1},\cdots,\phi_{t}\right]\big\|\\
 & \lesssim\sqrt{\frac{t\log n}{n}}+\frac{\sqrt{t}}{n^{5}}\asymp\sqrt{\frac{t\log n}{n}}
\end{align*}
with probability at least $1-O(n^{-11})$, 
where the second line results from the 1-Lipschitz property of $f$, 
and the last line takes $\epsilon = 1 / n^5$ and invokes standard random matrix theory \citep[Chapter 4.4]{vershynin2018high} 
that asserts 
\begin{equation}
	\mathbb{P} \Big\{ \big\|\left[\phi_{1},\cdots,\phi_{t}\right]\big\| \leq C_8 \sqrt{t} \Big\} \geq 1-O(n^{-11})
\end{equation}
for some constant $C_8>0$. 
Given that the above inequality holds simultaneously for all $\widehat{\beta} \in  \mathcal{S}^{t-1}$, 
we have
\begin{align}
\sup_{\widehat{\beta}=[\widehat{\beta}_{k}]_{1\leq k\leq t}\in\mathcal{S}^{t-1}}\bigg\{ f\Big(\sum_{i=1}^{t}\widehat{\beta}_{k}\phi_{k}\Big)-\mathop{\mathbb{E}}\limits _{g\sim\mathcal{N}(0,\frac{1}{n}I_{n})}\big[f(g)\big]\bigg\} & \leq C_7\sqrt{\frac{t\log n}{n}}\label{eq:sup-f-beta-phi-Ef}
\end{align}
with probability exceeding $1-O(n^{-11})$, where $C_7>0$ is some universal constant.

Next, we would like to use \eqref{eq:sup-f-beta-phi-Ef} to bound $\mathbb{E}\big[f\big(\sum_{i=1}^{t}\beta_{k}\phi_{k}\big)\big]$. 
Let us define the following event:
\begin{align*}
\mathcal{E}_{1} & \coloneqq\left\{ f\Big(\sum_{i=1}^{t}\beta_{k}\phi_{k}\Big)\leq\mathop{\mathbb{E}}\limits _{g\sim\mathcal{N}(0,\frac{1}{n}I_{n})}\big[f(g)\big]+C_{7}\sqrt{\frac{t\log n}{n}}\right\} ,
\end{align*}
which clearly obeys 
$
	\mathbb{P}(\mathcal{E}_1)\geq 1-O(n^{-11}) .
$
One can then decompose
\begin{align}
\mathbb{E}\bigg[f\Big(\sum_{i=1}^{t}\beta_{k}\phi_{k}\Big)\bigg] & =\mathbb{E}\bigg[f\Big(\sum_{i=1}^{t}\beta_{k}\phi_{k}\Big)\ind(\mathcal{E}_{1})\bigg]+\mathbb{E}\bigg[f\Big(\sum_{i=1}^{t}\beta_{k}\phi_{k}\Big)\ind(\mathcal{E}_{1}^{\mathrm{c}})\bigg]
	\label{eq:Exp-f-beta-phi-decompose}
\end{align}
The first term on the right-hand side of \eqref{eq:Exp-f-beta-phi-decompose} can be controlled as follows: 
\begin{align*}
\mathbb{E}\bigg[f\Big(\sum_{i=1}^{t}\beta_{k}\phi_{k}\Big)\ind(\mathcal{E}_{1})\bigg] & \leq\mathbb{E}\bigg[\bigg\{\mathop{\mathbb{E}}\limits _{g\sim\mathcal{N}(0,\frac{1}{n}I_{n})}\big[f(g)\big]+C_{7}\sqrt{\frac{t\log n}{n}}\bigg\}\ind(\mathcal{E}_{1})\bigg]\\
 & \leq\mathop{\mathbb{E}}\limits _{g\sim\mathcal{N}(0,\frac{1}{n}I_{n})}\big[f(g)\big]+C_{7}\sqrt{\frac{t\log n}{n}}+\Bigg|\mathop{\mathbb{E}}\limits _{g\sim\mathcal{N}(0,\frac{1}{n}I_{n})}\big[f(g)\big]+C_{7}\sqrt{\frac{t\log n}{n}}\,\Bigg|\,\mathbb{P}(\mathcal{E}_{1}^{\mathrm{c}})\\
 & \leq\mathop{\mathbb{E}}\limits _{g\sim\mathcal{N}(0,\frac{1}{n}I_{n})}\big[f(g)\big]+O\left(\sqrt{\frac{t\log n}{n}}\right).
\end{align*}
Here, the last line holds since $\mathbb{P}(\mathcal{E}_{1}^{\mathrm{c}})\leq O(n^{-11})$ and
\[
\bigg|\mathop{\mathbb{E}}\limits _{g\sim\mathcal{N}(0,\frac{1}{n}I_{n})}\big[f(g)\big]\bigg|\leq\mathop{\mathbb{E}}\limits _{g\sim\mathcal{N}(0,\frac{1}{n}I_{n})}\left[\|g\|_{2}\right]\leq1+\mathop{\mathbb{E}}\limits _{g\sim\mathcal{N}(0,\frac{1}{n}I_{n})}\left[\|g\|_{2}^{2}\right]=2,
\]
where the first inequality arises from \eqref{eq:fx-size-norm2}. 
When it comes to the second term on the right-hand side of \eqref{eq:Exp-f-beta-phi-decompose}, 
we make the observation that
\begin{align*}
\mathbb{E}\bigg[f\Big(\sum_{i=1}^{t}\beta_{k}\phi_{k}\Big)\ind(\mathcal{E}_{1}^{\mathrm{c}})\bigg] & \leq\mathbb{E}\bigg[\bigg\|\sum_{i=1}^{t}\beta_{k}\phi_{k}\bigg\|_{2}\ind(\mathcal{E}_{1}^{\mathrm{c}})\bigg]\leq\mathbb{E}\bigg[\|\beta\|_{2} \cdot \big\|\big[\phi_{1},\cdots,\phi_{t}\big]\big\|\ind(\mathcal{E}_{1}^{\mathrm{c}})\bigg]\\
 & =\mathbb{E}\bigg[\big\|\big[\phi_{1},\cdots,\phi_{t}\big]\big\|\ind(\mathcal{E}_{1}^{\mathrm{c}})\bigg]\leq\mathbb{E}\Big[\big\|\big[\phi_{1},\cdots,\phi_{t}\big]\big\|_{\mathrm{F}}\ind(\mathcal{E}_{1}^{\mathrm{c}})\Big]\\
 & \leq\sqrt{\mathbb{E}\bigg[\big\|\big[\phi_{1},\cdots,\phi_{t}\big]\big\|_{\mathrm{F}}^{2}\bigg]}\sqrt{\mathbb{E}\big[\ind(\mathcal{E}_{1}^{\mathrm{c}})\big]}\\
 & \leq\sqrt{t}\cdot O(n^{-11})\leq O(n^{-10}), 
\end{align*}
where the first inequality comes from \eqref{eq:fx-size-norm2}, the second line is valid since $\|\beta\|_2=1$, and the third line invokes the Cauchy-Schwarz inequality.  
Substituting the above two inequalities into \eqref{eq:Exp-f-beta-phi-decompose}, we obtain
\begin{align}
\mathbb{E}\bigg[f\Big(\sum_{i=1}^{t}\beta_{k}\phi_{k}\Big)\bigg] & 
	\leq\mathop{\mathbb{E}}\limits _{g\sim\mathcal{N}(0,\frac{1}{n}I_{n})}\big[f(g)\big]+O\left(\sqrt{\frac{t\log n}{n}}\right).
	\label{eqn:w1-arbitrary-f}
\end{align}

To finish up, combine \eqref{eqn:w1-arbitrary-f} with \eqref{eq:K-R-duality} to arrive at
\[
W_{1}\bigg(\mu\Big(\sum_{i=1}^{t}\beta_{k}\phi_{k}\Big),\mathcal{N}\Big(0,\frac{1}{n}I_{n}\Big)\bigg)=\sup\bigg\{\mathbb{E}\bigg[f\Big(\sum_{i=1}^{t}\beta_{k}\phi_{k}\Big)\bigg]-\mathbb{E}_{g\sim\mathcal{N}(0,\frac{1}{n}I_{n})}\big[f(g)\big]:f\text{ is }1\text{-Lipschitz}\bigg\}\lesssim \sqrt{\frac{t\log n}{n}}.
\]

\section{Proof of auxiliary lemmas for master theorems (Theorems~\ref{thm:recursion}-\ref{thm:main})}


\subsection{Proof of Lemma~\ref{lem:distribution}}
\label{sec:pf-distribution}

Before embarking on the proof, let us introduce some notation and basic properties. 
Recall that $\{z_k\}_{k\leq t}$ are orthonormal (see Lemma~\ref{lemma:zk-orthonormal}) and $U_{t-1}=[z_1,\cdots,z_{t-1}] \in \real^{n\times (t-1)}$. 
For any $1\leq k< n$, 
we let $U_k^{\perp}\in \real^{n\times (n-k)}$ represent the orthogonal complement of $U_k$ (such that  $U_k^{\top} U_k^{\perp} = 0$ and $U_k^{\perp\top}U_k^{\perp}=I_{n-k}$). 
We also define the projection of $W_{k+1}$ onto $U_k^{\perp}$ as follows
\begin{align}
	\widetilde{W}_{k+1} &\defn U_k^{\perp \top} W_{k+1} U_k^{\perp} 
\end{align}
which together with the construction \eqref{eqn:Wt} clearly satisfies
\begin{align}
	\widetilde{W}_{k+1} = U_k^{\perp \top} (I_n - z_{k}z_{k}^{\top}) W_{k}(I_n - z_{k}z_{k}^{\top}) U_k^{\perp} 
	= U_k^{\perp \top}  W_{k} U_k^{\perp} = \cdots  
	= U_k^{\perp \top} W U_k^{\perp} \in \real^{(n-k)\times (n-k)}.
	\label{eq:tilde-W-k-W-relation}
\end{align}
In view of the construction, we also have
\begin{align}
	W_{k+1} &= \lt(I_n - z_{k}z_{k}^{\top}\rt)W_{k}\lt(I_n - z_{k}z_{k}^{\top}\rt) = \cdots = \lt( I_n - U_k U_k^{\top} \rt) W \lt( I_n - U_k U_k^{\top} \rt) \notag\\
	 &= U_k^{\perp} U_k^{\perp \top} W U_k^{\perp} U_k^{\perp \top}   = U_k^{\perp} \widetilde{W}_{k+1} U_k^{\perp\top}. 
	 \label{eq:W-k-tilde-W-relation}
\end{align}

To establish Lemma~\ref{lem:distribution}, the first step lies in proving the following claim. 
In the sequel, let us prove this crucial claim first before moving on to the next step. 
\begin{claim} \label{claim:independence}
 	Consider any $2\leq k\leq n$. Conditional on $\{z_i\}_{i < k}$ and $x_1$, the following hold:
	\begin{itemize}
		\item $\widetilde{W}_k$ is a (rescaled) Wigner matrix in the sense that its entries $\big\{(\widetilde{W}_k)_{ij} \mid i\geq j \big\}$ are independent obeying
			\begin{align}
				(\widetilde{W}_k)_{ii} \sim \mathcal{N}\Big(0,\frac{2}{n}\Big)
				\qquad \text{and} \qquad 
				(\widetilde{W}_k)_{ij} = (\widetilde{W}_k)_{ji} \sim \mathcal{N}\Big(0,\frac{1}{n}\Big) \quad \text{for any }i>j ;
				\label{eq:proj-Wk-Wigner}
			\end{align}
		\item $W_k$ is conditionally independent of $\{W_iz_i\}_{i < k}$;
		\item the randomness of $x_k$ and $z_{k}$ comes purely from that of $\{W_iz_i\}_{i < k}$ and $x_1$, and hence $x_k$ and $z_{k}$ are conditionally independent of $W_k$.
	\end{itemize}
 \end{claim} 
\begin{proof}[Proof of Claim~\ref{claim:independence}] 
The proof of this claim proceeds via an inductive argument. 

\paragraph{The base case with $k=2$.}
Consider first the  case when $k = 2$. In view of the definition~\eqref{eqn:Wt}, we have 
\begin{align*}
	W_2 &= \big(I - z_{1}z_{1}^{\top}\big)W \big(I - z_{1}z_{1}^{\top}\big)  
\end{align*}
	where $z_1$ is independent from $W$. Let $z_1^{\perp}=\real^{n\times (n-1)}$ denote the orthogonal complement of $z_1$ (so that $z_1^{\top}z_1^{\perp}=0$ and $z_1^{\perp\top}z_1^{\perp}=I_{n-1}$), and define the projection of $W_2$ onto $z_1^{\perp}$ (see \eqref{eq:tilde-W-k-W-relation}) obeys: 
\begin{equation}
	\widetilde{W}_2 = z_1^{\perp \top} W_2 z_1^{\perp} = z_1^{\perp \top} W z_1^{\perp} 
	\overset{\mathrm{d}}{=}  e_1^{\perp \top} W e_1^{\perp}, 
\end{equation}
%
where the last relation arises from the rotational invariance of the Wigner matrix (with $e_1$ denoting the first standard basis vector). 
Therefore, it is readily seen that: conditioned on $z_{1}$, 
\begin{itemize}
	\item $\widetilde{W}_{2}$ is a (rescaled) Wigner matrix in $\real^{(n-1)\times (n-1)}$ obeying \eqref{eq:proj-Wk-Wigner};
	\item  $\widetilde{W}_{2}$ --- and hence $W_{2}$ --- is statistically independent from $Wz_1$. 

\end{itemize}
In addition,  recalling the update rule \eqref{eqn:AMP-updates}, the definition \eqref{eqn:z-w-init} of $z_1$ and the assumption $\eta_0(x_0)=0$, we have
\begin{align*}
	x_2 
	&= \lambda\vstar v^{\star \top}\eta_1(x_1) + W\eta_1(x_1)
	= \big( \lambda v^{\star \top}z_1 \ltwo{\eta_1(x_1)} \big) \cdot \vstar+  \ltwo{\eta_1(x_1)} \cdot Wz_1,
\end{align*}
where the last step relies on the definition \eqref{eqn:z-w-init} of $z_{1}$. 
Given that $z_1$ is fully determined by $x_1$, 
we see that the randomness of $x_{2}$ --- and hence that of $z_2$ --- comes entirely from $Wz_{1}$ and $x_1$. 
We have thus established the advertised claim for the case with $k=2.$

\paragraph{The induction step.} 
Next, assuming that the claim holds for all step $i$ with $i\leq k$, let us extend it to the $(k+1)$-th step. 
To begin with, the inductive assumption tells us that: conditional on $\{z_i\}_{i < k}$ and $x_1$,   
\begin{itemize}
	\item[(i)] $W_k$ is independent of $\{W_iz_i\}_{i < k}$;
	\item[(ii)] the randomness of $z_k$ purely comes from $\{W_iz_i\}_{i < k}$, and hence $W_k$ is also independent of $z_k$. 
\end{itemize}
Taking these two conditions together reveals that:  
if we condition on  $\{z_i\}_{i \leq k}$ and $x_1$ (namely, we condition on an additional variable $z_k$ compared to the  above induction hypothesis), 
then clearly $W_k$ is still independent of $\{W_iz_i\}_{i < k}$. 
Recalling that 
\begin{equation}
	W_{k+1} = \lt(I_n - z_{k}z_{k}^{\top}\rt)W_{k}\lt(I_n - z_{k}z_{k}^{\top}\rt),
	\label{eq:expression-Wk+1-Wk}
\end{equation}
we can readily conclude that: conditioned on $\{z_i\}_{i \leq k}$ and $x_1$, 
\begin{itemize}
	\item $W_{k+1}$ is also independent of $\{W_iz_i\}_{i < k}$, 
given the conditional independence between $W_k$ and $\{W_iz_i\}_{i < k}$ and the fact that $z_k$ is being conditioned now; 
	\item 
		$W_{k+1}$ is independent of $\{W_iz_i\}_{i < k}$. 
\end{itemize}
As a result, in order to show that $W_{k+1}$ is conditionally independent from $\{W_iz_i\}_{i \leq k}$, 
it suffices to justify that it is conditionally independent from $W_kz_k$, which we shall accomplish next.

 Recall that $\widetilde{W}_k$ is a rescaled Wigner matrix independent of $z_k$ (see Property (ii) above) when conditioned on  $\{z_i\}_{i < k}$ and $x_1$. 
 Akin to the argument for the base case, the rotational invariance of the Wigner matrix together with  expression \eqref{eq:expression-Wk+1-Wk} tells us that: 
conditional on $\{z_i\}_{i \leq k}$ and $x_1$, 
\begin{itemize}
	\item $\widetilde{W}_{k+1}$ is a (rescaled) Wigner matrix in $\real^{(n-k)\times (n-k)}$ obeying \eqref{eq:proj-Wk-Wigner};
	\item  $\widetilde{W}_{k+1}$ --- and hence $W_{k+1}$ --- is statistically independent from $W_kz_k$.
\end{itemize}
We can thus conclude that: conditional on $\{z_i\}_{i \leq k}$ and $x_1$, 
both $W_{k+1}$ and $\widetilde{W}_{k+1}$ are independent from $\big\{ W_iz_i\big\}_{1\leq i\leq k}$. 

%
%
%
%
%
In addition, given the AMP update rule \eqref{eqn:AMP-updates}, it is legitimate to write
\begin{align*}
	x_{k+1} &= (\lambda\vstar v^{\star \top} + W)\eta_k(x_{k}) - \big\langle\eta_k^{\prime}(x_{k})\big\rangle \cdot \eta_{k-1}(x_{k-1})\\
	&= \lambda\vstar v^{\star \top}\eta_k(x_{k}) + \sum_{i = 1}^{k} \beta_{k}^i W_iz_i + \sum_{i = 1}^{k-1} z_i\Big[\langle W_iz_i, \eta_{k}(x_{k})\rangle - \langle\eta_k^{\prime}(x_{k})\rangle \beta_{k-1}^i - \beta_{k}^iz_i^{\top}W_iz_i\Big],
\end{align*}
where the last equality follows from expression~\eqref{eqn:xt-by-Wkzk}. 
Clearly,  $x_{k+1}$ is determined by $x_k$, $\big\{W_iz_i\big\}_{i\leq k}$, and $\{z_i\}_{i\leq k}$ (given that $\beta_k^i$ is also determined by $z_i$ and $x_k$), 
in addition to other deterministic objects. 
Moreover,  our induction hypothesis asserts that the randomness of $x_k$ and $z_k$ all comes from $\big\{W_iz_i\big\}_{i< k}$ and $x_1$.  
Consequently, these taken collectively imply that all randomness of $x_{k+1}$ (and hence $z_{k+1}$) comes from $\{W_iz_i\}_{i \leq k}$  and $x_1$.
We have thus established the claim for step $k+1.$
To finish up, applying the inductive argument concludes the proof of Claim~\ref{claim:independence}. 
\end{proof}

Armed with  the results in Claim~\ref{claim:independence}, we can characterize the conditional distribution of $W_kz_k$. 
Given that the $z_i$'s are orthonormal (cf.~Lemma~\ref{lemma:zk-orthonormal}), we can apply Claim~\ref{claim:independence} to show that: 
conditional on $\{z_i\}_{1\leq i< k}$ and $x_1$,  
\begin{subequations}
\label{defn:zWz}	
\begin{align}
z_{i}^{\top}W_{k}z_{k} & =z_{i}^{\top}U_{k-1}^{\perp}U_{k-1}^{\perp\top}WU_{k-1}^{\perp}U_{k-1}^{\perp\top}z_{k}=0\qquad\qquad\text{for }i< k,
	\label{defn:zWz-i-less-k}\\
z_{k}^{\top}W_{k}z_{k} & =z_{k}^{\top}U_{k-1}^{\perp}U_{k-1}^{\perp\top}WU_{k-1}^{\perp}U_{k-1}^{\perp\top}z_{k}=\left(U_{k-1}^{\perp\top}z_{k}\right)^{\top}\widetilde{W}_{k}\left(U_{k-1}^{\perp\top}z_{k}\right)\overset{\mathrm{d}}{=}e_{1}^{\top}\widetilde{W}_{k}e_{1}\sim\mathcal{N}\Big(0,\frac{2}{n}\Big), \\
U_{k}^{\perp\top}W_{k}z_{k} & =U_{k}^{\perp\top}U_{k-1}^{\perp}U_{k-1}^{\perp\top}WU_{k-1}^{\perp}U_{k-1}^{\perp\top}z_{k}=\left(U_{k-1}^{\perp\top}U_{k}^{\perp}\right)^{\top}\widetilde{W}_{k}\left(U_{k-1}^{\perp\top}z_{k}\right)\sim\mathcal{N}\Big(0,\frac{1}{n}I_{n-k}\Big),
\end{align}
\end{subequations}
 where we have made use of the fact in Claim~\ref{claim:independence} that, conditional on $\{z_i\}_{1\leq i< k}$ and $x_1$, 
 $\widetilde{W}_{k}$ is a rescaled Wigner matrix independent of $z_k$.   
Therefore, if we generate i.i.d.~Gaussian random variables $g_i^k \sim \mathcal{N}(0,\frac{1}{n})$ for all $i < k$,  
then conditional on $\{z_i\}_{i \leq k}$ and $x_1$, it follows that
\begin{align}
	\phi_{k} & \defn W_{k}z_{k}+\Big(\frac{\sqrt{2}}{2}-1\Big)z_{k}^{\top}W_{k}z_{k}\cdot z_{k}+\sum_{i=1}^{k-1}g_{i}^{k}z_{i} \label{eq:defn-phi-k-proof}\\
 & =\bigg(\sum_{i=1}^{k}z_{i}z_{i}^{\top}\bigg)W_{k}z_{k}+\big(U_{k}^{\perp}U_{k}^{\perp\top}\big)W_{k}z_{k}+\Big(\frac{\sqrt{2}}{2}-1\Big)z_{k}^{\top}W_{k}z_{k}\cdot z_{k}+\sum_{i=1}^{k-1}g_{i}^{k}z_{i} \notag\\
 & =\frac{\sqrt{2}}{2}\left(z_{k}^{\top}W_{k}z_{k}\right)z_{k}+\sum_{i=1}^{k-1}g_{i}^{k}z_{i}+U_{k}^{\perp}\big(U_{k}^{\perp\top}W_{k}z_{k}\big) \notag\\
	& \sim\mathcal{N}\Big(0,\frac{1}{n}I_{n}\Big). \label{eq:phi-k-distribution}
\end{align}
 Here, the penultimate line makes use of \eqref{defn:zWz-i-less-k} and a little algebra, 
 whereas the last line is valid since, along each basis direction (i.e., $z_1,\cdots,z_k$ and each column of $U_k^{\perp}$), the projection of $\phi_k$ is independent $\mathcal{N}(0,1/n)$. 
 In fact, \eqref{eq:phi-k-distribution} tells us that the conditional distribution of $\phi_{k}$ is always $\mathcal{N}(0, \frac{1}{n}I_n)$ no matter what value the sequence $\{z_i\}_{i \leq k}$ takes, thus indicating the (unconditional) distribution of $\phi_{k}$ as follows: 
 \begin{align}
\phi_{k}  \sim\mathcal{N}\Big(0,\frac{1}{n}I_{n}\Big). \label{eq:phi-k-distribution-marginal}
\end{align}
%


Finally, we demonstrate that $\{\phi_i\}_{1 \le i \leq k}$ are independent. 
To this end, we first observe that $\phi_{k}$ is independent of $\{z_i\}_{i < k}$ and $x_1$, 
which is an immediate consequence of the conditional distribution derivation \eqref{eq:phi-k-distribution}. 
Further, combining Claim~\ref{claim:independence} with the definition \eqref{eq:defn-phi-k-proof} of $\phi_k$ (which depends only on $W_kz_k$  and $\{g_i^k\}$ conditional on $\{z_i\}_{1\leq i\leq k}$) reveals that: 
conditional on $\{z_i\}_{i < k}$ and $x_1$, $\phi_k$ is statistically independent from $\phi_{k-1},\cdots,\phi_1$. 
Letting us abuse the notation and use $f$ to represent the pdf of the random vectors of interest, we obtain
\begin{align*}
 & f(\phi_{k},\phi_{k-1},\cdots,\phi_{1})={\displaystyle \int}f\big(\phi_{k},\phi_{k-1},\cdots,\phi_{1}\mid z_{k-1},\cdots,z_{1},x_{1}\big)\mu\left(\mathrm{d}z_{k-1},\cdots,\mathrm{d}z_{1},\mathrm{d}x_{1}\right)\\
 & ={\displaystyle \int}f\big(\phi_{k}\mid z_{k-1},\cdots,z_{1},x_{1}\big)f\big(\phi_{k-1},\cdots,\phi_{1}\mid z_{k-1},\cdots,z_{1},x_{1}\big)\mu\left(\mathrm{d}z_{k-1},\cdots,\mathrm{d}z_{1},\mathrm{d}x_{1}\right)\\
 & =f(\phi_{k}){\displaystyle \int}f\big(\phi_{k-1},\cdots,\phi_{1}\mid z_{k-1},\cdots,z_{1},x_{1}\big)\mu\left(\mathrm{d}z_{k-1},\cdots,\mathrm{d}z_{1},\mathrm{d}x_{1}\right)\\
 & =f\big(\phi_{k}\big)f\big(\phi_{k-1},\cdots,\phi_{1}\big),
\end{align*}
where the second line holds since, as shown above, $\phi_k$ is independent of $\phi_{k-1},\cdots,\phi_1$ when conditioned on $z_1,\cdots,z_{k-1}$ and $x_1$, 
and the third line makes use of the statistical independence between $\phi_k$ and $z_{k-1},\cdots,z_1,x_1$.  
Repeating the above derivation gives
\begin{align*}
f(\phi_{k},\phi_{k-1},\cdots,\phi_{1}) & =f\big(\phi_{k}\big)f\big(\phi_{k-1},\cdots,\phi_{1}\big)=\cdots=f\big(\phi_{k}\big)f\big(\phi_{k-1}\big)\cdots f\big(\phi_{1}\big),  \end{align*}
thereby justifying that $\{\phi_i\}_{1 \le i \leq k}$ are statistically independent.

%
%



\subsection{Proof of Lemma~\ref{lem:concentration}}
\label{sec:pf-concentration}


To begin with, it is seen from property~\eqref{defn:zWz} that $z_k^{\top}W_kz_k$ follows a normal distribution with variance $2/n$ (given that this distribution is independent of $\{z_i\}_{1\leq i\leq k}$ and $x_1$). 
Standard Gaussian concentration inequalities \citep[Chapter 2.6]{vershynin2018high} together with the union bound tell us that
\begin{equation}
	\max_{1\leq k\leq n} \big| z_k^{\top}W_kz_k \big| \lesssim \sqrt{\frac{\log n}{n}}
\end{equation}
with probability at least $1-n^{-11}$. Consequently, we have
\begin{align}
\Big|\sum_{k=1}^{t-1}\mu_{t}^{k}\beta_{t}^{k}z_{k}^{\top}W_{k}z_{k}\Big| & \leq\Big|\max_{1\leq k\leq n}z_{k}^{\top}W_{k}z_{k}\Big|\cdot\sum_{k=1}^{t-1}\big|\mu_{t}^{k}\beta_{t}^{k}\big|\leq\Big|\max_{1\leq k\leq n}z_{k}^{\top}W_{k}z_{k}\Big|\cdot\|\mu_{t}\|_{2}\|\beta_{t}\|_{2}\nonumber \\
 & \lesssim\sqrt{\frac{\log n}{n}}\|\beta_{t}\|_{2},\label{eq:sum-mu-beta-zWz}
\end{align}
with probability at least $1-n^{-11}$, where we remind the reader of the notation $\mu_t=[\mu_t^k]_{1\leq k\leq t}$ and $\beta_t=[\beta_t^k]_{1\leq k\leq t}$ and the fact that $\|\mu_t\|_2=1$.

Next, we turn to the following term 
\begin{align*}
\mathcal{I}_{1} & \coloneqq\sum_{k=1}^{t-1}\mu_{t}^{k}\Big(\sum_{i=1}^{k-1}\beta_{t}^{i}g_{i}^{k}+\sum_{i=k+1}^{t}\beta_{t}^{i}g_{k}^{i}\Big) 
	\eqqcolon \sum_{k=1}^{t-1}\mu_{t}^{k} \varrho_k,
\end{align*}
where we recall that each random variable $g_i^{k}$ with $i\neq k$ is independently generated from $\mathcal{N}(0,1/n)$,
which is also independent from $\beta_{t}$ (but not $\mu_t$). 
Conditional on $\beta_t$, one has
\begin{align*}
\mathsf{Var}\left(\varrho_{k}\mid\beta_{t}\right) & \coloneqq\frac{1}{n}\sum_{i=1}^{k-1}\big(\beta_{t}^{i}\big)^{2}+\frac{1}{n}\sum_{i=k+1}^{t}\big(\beta_{t}^{i}\big)^{2}\leq\frac{\|\beta_{t}\|_{2}^{2}}{n},
\end{align*}
which combined with Gaussian concentration inequalities \citep[Chapter 2.6]{vershynin2018high} and the union bound yields 
\begin{align}
	\max_{1\leq k\leq n}|\varrho_k| \lesssim \frac{\|\beta_{t}\|_{2}\sqrt{\log n}}{\sqrt{n}}
	\label{eq:var-rho-k-UB}
\end{align}
with probability at least $1-O(n^{-11})$. As a result, the Cauchy-Schwarz inequality gives
\begin{align}
|\mathcal{I}_{1}| & =\Big|\sum_{k=1}^{t-1}\mu_{t}^{k}\varrho_{k}\Big|\leq\|\mu_{t}\|_{2}\sqrt{\sum_{k=1}^{t-1}\varrho_{k}^{2}}
	\lesssim \sqrt{\frac{t\log n}{n}} \|\beta_{t}\|_{2} \label{eq:UB-I1-135}
\end{align}
with probability at least $1-O(n^{-11})$, where we have used \eqref{eq:var-rho-k-UB} and the fact $\|\mu_t\|_2=1$. 

Combining \eqref{eq:sum-mu-beta-zWz} and \eqref{eq:UB-I1-135} immediately finishes the proof.

\subsection{Proof of Lemma~\ref{lem:recursion}}
\label{sec:pf-lem-recursion}

Lemma~\ref{lem:recursion} involves bounds concerning the continuous part of the function and that of the discontinuous part, 
which we shall prove separately.  

\paragraph{The continuous part: proof of inequalities \eqref{eq:lem-recursion-smooth-part-1} and \eqref{eq:lem-recursion-smooth-part-2}.} First, some basic algebra leads to 
\begin{align}
\notag\Big\langle\sum_{k=1}^{t-1}\mu_{t}^{k}\phi_{k},\eta_{t}^{\prime}(v_{t})\circ\xi_{t-1}\Big\rangle-\langle\eta_{t}^{\second}(v_{t})\circ\xi_{t-1}\rangle\sum_{k=1}^{t-1}\mu_{t}^{k}\beta_{t-1}^{k} & =\bigg\langle\sum_{k=1}^{t-1}\mu_{t}^{k}\phi_{k}\circ\eta_{t}^{\prime}(v_{t}),\xi_{t-1}\bigg\rangle-\bigg\langle\frac{1}{n}\sum_{k=1}^{t-1}\mu_{t}^{k}\beta_{t-1}^{k}\eta_{t}^{\second}(v_{t}),\xi_{t-1}\bigg\rangle\\
 & \le\bigg\|\sum_{k=1}^{t-1}\mu_{t}^{k}\phi_{k}\circ\eta_{t}^{\prime}(v_{t})-\frac{1}{n}\sum_{k=1}^{t-1}\mu_{t}^{k}\beta_{t-1}^{k}\eta_{t}^{\second}(v_{t})\bigg\|_{2}\|\xi_{t-1}\|_{2}.\label{eqn:shostakovich-simple}
\end{align}
The condition \eqref{defi:D} imposed in Assumption~\ref{assump:A-H-eta} tells us that
\begin{align*}
	\Big\|\sum_{k = 1}^{t-1} \mu_t^k\phi_k \circ \eta_{t}^{\prime} - \frac{1}{n}\sum_{k = 1}^{t-1} \mu_t^k\beta_{t-1}^k\eta_{t}^{\second}\Big\|_2^2 
	& \le \kappa_t^2 + D_t,
\end{align*}
which taken collectively with inequality~\eqref{eqn:shostakovich-simple} concludes the proof of inequality \eqref{eq:lem-recursion-smooth-part-1}. 

When it comes to the second claim \eqref{eq:lem-recursion-smooth-part-2}, we observe that for any $t\leq n$, 
\begin{align*}
\rho_{1}\bigg\langle\bigg|\sum_{k=1}^{t-1}\mu_{t}^{k}\phi_{k}\bigg|,\big|\xi_{t-1}\big|^{2}\Big\rangle+\rho_{2}\Big\langle\big|\xi_{t-1}\big|^{2}\Big\rangle\bigg|\sum_{k=1}^{t-1}\mu_{t}^{k}\beta_{t-1}^{k}\bigg| & \leq\rho_{1}\bigg\{\max_{1\leq i\leq n}\Big|\sum_{k=1}^{t-1}\mu_{t}^{k}\phi_{k}\Big|_{i}\bigg\}\cdot\|\xi_{t-1}\|_{2}^{2}+\frac{\rho_{2}}{n}\|\xi_{t-1}\|_{2}^{2}\bigg|\sum_{k=1}^{t-1}\mu_{t}^{k}\beta_{t-1}^{k}\bigg|\\
 & \lesssim\lt(\rho_{1}\sqrt{\frac{t\log n}{n}}+\frac{\rho_{2}\|\beta_{t-1}\|_{2}}{n}\rt)\|\xi_{t-1}\|_{2}^{2}. 
\end{align*}
Here, the last line makes use of two properties:
(i) $\big| \sum_{k = 1}^{t-1} \mu_t^k\beta_{t-1}^k \big| \le \ltwo{\mu_t^k} \|\beta_{t-1}\|_2 = \ltwo{\beta_{t-1}}$ (given that $\mu_t$ is constructed as a unit vector); 
(ii) the standard Gaussian concentration inequalities \citep[Chapter 4.4]{vershynin2018high} indicating that, with probability at least $1 - O(n^{-11})$, 
\[
	\max_{1\leq i\leq n}\Big|\sum_{k=1}^{t-1}\mu_{t}^{k}\phi_{k}\Big|_{i}\leq\max_{1\leq i\leq n}\|\mu_{t}\|_{2}\big\|[\phi_{1,i},\cdots,\phi_{t-1,i}]\big\|_{2}=\max_{1\leq i\leq n}\big\|[\phi_{1,i},\cdots,\phi_{t-1,i}]\big\|_{2}\lesssim \frac{\sqrt{t}+\sqrt{\log n}}{\sqrt{n}} .
\]
This establishes inequality \eqref{eq:lem-recursion-smooth-part-2}.


\paragraph{The discontinuous part: proof of inequality \eqref{eqn:new-version}.} 

We first make the observation that: the quantity $\theta(m)$ defined in expression~\eqref{defi:theta} obeys  
\begin{align}
\label{eqn:zero-norm-comparison}
\sum_{j=1}^{n}\ind\Big(\Big|\alpha_{t}\vstar_{j}+\sum_{k=1}^{t-1}\beta_{t-1}^{k}\phi_{k,j}-m\Big|\le|\xi_{t-1,j}|\Big)
	\le\sum_{j=1}^{n}\ind\Big(\Big|\alpha_{t}\vstar_{j}+\sum_{k=1}^{t-1}\beta_{t-1}^{k}\phi_{k,j}-m\Big|\le\theta(m)\Big), 
\end{align}
which can be proved using the definition \eqref{defi:theta} as follows.
\begin{proof}[Proof of \eqref{eqn:zero-norm-comparison}]
By defining the set
\begin{align*}
	\mathcal{J} \defn \bigg\{j\in [n] : \Big|\alpha_tv_j^{\star} + \sum_{k = 1}^{t-1} \beta_{t-1}^k\phi_{k, j} - m\Big| \le |\xi_{t-1, j}|\bigg\}, 
\end{align*}
we can easily see that
\begin{align}
	\sum_{j \in \mathcal{J}}\Big|\alpha_tv_j^{\star} + \sum_{k = 1}^{t-1} \beta_{t-1}^k\phi_{k, j} - m_i\Big|^2 \le \sum_{j \in \mathcal{J}} |\xi_{t-1, j}|^2 \le \|\xi_{t-1}\|_2^2 .
	\label{eq:Jtrue-condition}
\end{align}
Additionally,  if we define another set $\mathcal{J}^{\prime}$ as follows
\begin{align}
	\mathcal{J}^{\prime} \defn \bigg\{j\in [n] : \Big|\alpha_tv_j^{\star} + \sum_{k = 1}^{t-1} \beta_{t-1}^k\phi_{k, j} - m\Big| \le \theta(m) \bigg\},
	\label{eq:Jprime-condition}
\end{align}
then in view of definition of $\theta(m)$, $\mathcal{J}^{\prime}$ is clearly the index set with the largest cardinality obeying 
$$\sum_{j \in \mathcal{J}^{\prime}}\Big|\alpha_tv_j^{\star} + \sum_{k = 1}^{t-1} \beta_{t-1}^k\phi_{k, j} - m\Big|^2 \le \|\xi_{t-1}\|_2^2.$$
	Since $\mathcal{J}$ also satisfies this relation (cf.~\eqref{eq:Jtrue-condition}),  we arrive at $\lt|\mathcal{J}^{\prime}\rt| \ge \lt|\mathcal{J}\rt|$, thus validating inequality~\eqref{eqn:zero-norm-comparison}. 
\end{proof}

Next, for any $m\in \mathcal{M}_{\mathsf{dc}}$, define $\Gamma(m) \coloneqq \big[\Gamma_j(m) \big]_{1\leq j\leq n}\in \real^n$ (see \eqref{eq:Gamma-ub-discontinuous}). 
Equipped with the above relation \eqref{eqn:zero-norm-comparison},  one can show that
\begin{align}
2\rho\bigg\langle\bigg|\sum_{k=1}^{t-1}\mu_{t}^{k}\phi_{k}\bigg|,\,\Gamma\circ\big|\xi_{t-1}\big|\bigg\rangle & =2\rho\bigg\langle\bigg|\sum_{k=1}^{t-1}\mu_{t}^{k}\phi_{k}\bigg|\circ\Gamma,\,\big|\xi_{t-1}\big|\bigg\rangle=\sum_{m\in\mathcal{M}_{\mathsf{dc}}}2\rho\bigg\langle\bigg|\sum_{k=1}^{t-1}\mu_{t}^{k}\phi_{k}\bigg|\circ\Gamma(m),\,\big|\xi_{t-1}\big|\bigg\rangle \notag\\
 & \leq\sum_{m\in\mathcal{M}_{\mathsf{dc}}}2\rho\Big\|\sum_{k=1}^{t-1}\mu_{t}^{k}\phi_{k}\circ\Gamma(m)\Big\|_{2}\cdot\big\|\xi_{t-1}\big\|_{2}.
	\label{eqn:sonnet}
\end{align}
To control the right-hand side of \eqref{eqn:sonnet}, 
we first apply inequality~\eqref{eqn:vive} in Lemma~\ref{lem:brahms-lemma} with $s = t$ to obtain 
\begin{align*}
	\sum_{i = 1}^t \Big|\sum_{k = 1}^{t-1} \mu_t^k\phi_k\Big|_{(i)}^2 \lesssim \frac{t\log{n}}{n}
\end{align*}
with probability at least $1-O(n^{-11})$. 
This relation in turn implies that for every $j \geq t$, 
\begin{align*}
	\Big|\sum_{k = 1}^{t-1} \mu_t^k\phi_k\Big|_{(j)}^2 \leq \frac{1}{t} \sum_{i = 1}^t \Big|\sum_{k = 1}^{t-1} \mu_t^k\phi_k\Big|_{(i)}^2 \lesssim \frac{\log{n}}{n}.
\end{align*}
With these two inequalities in mind, we can deduce that 
\begin{align}
\Big\|\sum_{k=1}^{t-1}\mu_{t}^{k}\phi_{k}\circ\Gamma(m)\Big\|_{2}^{2} & \leq
\sum_{i=1}^{t}\Big|\sum_{k=1}^{t-1}\mu_{t}^{k}\phi_{k}\Big|_{(i)}^{2}+\Big|\sum_{k=1}^{t-1}\mu_{t}^{k}\phi_{k}\Big|_{(t+1)}^{2}\cdot\big\|\Gamma(m)\big\|_{1} \notag\\
 & \lesssim\frac{t\log n}{n}+\frac{\log n}{n}\sum_{j=1}^{n}\ind\Big(\Big|\alpha_{t}\vstar_{j}+\sum_{k=1}^{t-1}\beta_{t-1}^{k}\phi_{k,j}-m\Big|\le|\xi_{t-1,j}|\Big). 
	\label{eqn:l2-decomposition}
\end{align}
This taken collectively with inequality~\eqref{eqn:sonnet} leads to 
\begin{align*}
2\rho\bigg\langle\bigg|\sum_{k=1}^{t-1}\mu_{t}^{k}\phi_{k}\bigg|,\,\Gamma\circ\big|\xi_{t-1}\big|\bigg\rangle & \lesssim\rho\sum_{m\in\mathcal{M}_{\mathsf{dc}}}\left(\sqrt{\frac{t\log n}{n}}+\sqrt{\frac{\log n}{n}}\sqrt{\sum_{j=1}^{n}\ind\Big(\Big|\alpha_{t}\vstar_{j}+\sum_{k=1}^{t-1}\beta_{t-1}^{k}\phi_{k,j}-m\Big|\le|\xi_{t-1,j}|\Big)}\right)\big\|\xi_{t-1}\big\|_{2}\\
 & \lesssim\rho\sum_{m\in\mathcal{M}_{\mathsf{dc}}}\left(\sqrt{\frac{t\log n}{n}}+\sqrt{\frac{\log n}{n}}\sqrt{\sum_{j=1}^{n}\ind\Big(\Big|\alpha_{t}\vstar_{j}+\sum_{k=1}^{t-1}\beta_{t-1}^{k}\phi_{k,j}-m\Big|\le \theta(m) \Big)}\right) \big\|\xi_{t-1}\big\|_{2}\\
 & \lesssim\rho\sqrt{\frac{(E_{t}+t)\log n}{n}}\big\|\xi_{t-1}\big\|_{2},
\end{align*}
where the second inequality comes from \eqref{eqn:zero-norm-comparison}, and the last inequality makes use of the definition \eqref{defi:E} of $E_t$. 

%
%
%
%
%
%

Similar calculations lead to 
\begin{align*}
 & \Big\{2\rho\langle\Gamma\rangle+2\rho_{1}\big\langle\Gamma\circ\big|\xi_{t-1}\big|\big\rangle\Big\}\bigg|\sum_{k=1}^{t-1}\mu_{t}^{k}\beta_{t-1}^{k}\bigg|\leq2\Big(\rho+\rho_{1}\big\|\xi_{t-1}\big\|_{\infty}\Big)\langle\Gamma\rangle\big\|\mu_{t}\big\|_{2}\big\|\beta_{t-1}\big\|_{2}\\
 & \qquad=\frac{2(\rho+\rho_{1}\big\|\xi_{t-1}\big\|_{\infty})}{n}\Big\{\sum_{m\in\mathcal{M}_{\mathsf{dc}}}\sum_{j=1}^{n}\Gamma_{j}(m)\Big\}\big\|\beta_{t-1}\big\|_{2}\\
 & \qquad\leq\frac{2(\rho+\rho_{1}\big\|\xi_{t-1}\big\|_{\infty})\big\|\beta_{t-1}\big\|_{2}}{n}\sum_{m\in\mathcal{M}_{\mathsf{dc}}}\sum_{j=1}^{n}\ind\Big(\Big|\alpha_{t}\vstar_{j}+\sum_{k=1}^{t-1}\beta_{t-1}^{k}\phi_{k,j}-m\Big|\le|\xi_{t-1,j}|\Big)\\
 & \qquad\leq\frac{2(\rho+\rho_{1}\big\|\xi_{t-1}\big\|_{\infty})\big\|\beta_{t-1}\big\|_{2}}{n}\sum_{m\in\mathcal{M}_{\mathsf{dc}}}\sum_{j=1}^{n}\ind\Big(\Big|\alpha_{t}\vstar_{j}+\sum_{k=1}^{t-1}\beta_{t-1}^{k}\phi_{k,j}-m\Big|\le\theta(m)\Big)\\
 & \qquad\leq\frac{2(\rho+\rho_{1}\big\|\xi_{t-1}\big\|_{\infty})E_{t}\big\|\beta_{t-1}\big\|_{2}}{n}
\end{align*}
Taking the above pieces collectively, we demonstrate that
\begin{align*}
 & 2\rho\bigg\langle\bigg|\sum_{k=1}^{t-1}\mu_{t}^{k}\phi_{k}\bigg|,\,\Gamma\circ\big|\xi_{t-1}\big|\bigg\rangle+\Big\{2\rho\langle\Gamma\rangle+2\rho_{1}\big\langle\Gamma\circ\big|\xi_{t-1}\big|\big\rangle\Big\}\bigg|\sum_{k=1}^{t-1}\mu_{t}^{k}\beta_{t-1}^{k}\bigg|\\
 & \qquad\lesssim\rho\sqrt{\frac{(E_{t}+t)\log n}{n}}\big\|\xi_{t-1}\big\|_{2}+\frac{(\rho+\rho_{1}\big\|\xi_{t-1}\big\|_{\infty})E_{t}\big\|\beta_{t-1}\big\|_{2}}{n}
\end{align*}
as claimed.


\subsection{Proof of Lemma~\ref{lem:recursion2}}
\label{sec:pf-lem-recursion2}

Inequality~\eqref{eqn:cello} and \eqref{eqn:violin} directly results from the Lipschitz property of $\eta_t$ and the fact that $\ltwo{\vstar} = 1.$
We then move on to proving inequality~\eqref{eqn:viola}. 
Recall from \eqref{eqn:shostakovich-delta-t} that $\delta_{t}$ obeys
\begin{align}
\Big|\delta_{t} & -\eta_{t}^{\prime}\Big(\alpha_{t}\vstar+\sum_{k=1}^{t-1}\beta_{t-1}^{k}\phi_{k}\Big)\circ\xi_{t-1}\Big|
	\leq \rho_{1}\big|\xi_{t-1}\big|^{2}+ 2\rho\Gamma \circ \big|\xi_{t-1}\big|.
\label{eqn:shostakovich-delta-t-135}
\end{align}
In view of  \eqref{eqn:shostakovich-delta-t-135}, we have
\begin{align*}
\Big|\big\langle\eta_{t}(v_{t}),\delta_{t}\big\rangle\Big| & \leq\Big|\big\langle\eta_{t}(v_{t}),\eta_{t}^{\prime}(v_{t})\circ\xi_{t-1}\big\rangle\Big|+\rho_{1}\Big|\big\langle\eta_{t}(v_{t}),\big|\xi_{t-1}\big|^{2}\big\rangle\Big|+2\rho\Big|\big\langle\eta_{t}(v_{t}),\Gamma\circ\big|\xi_{t-1}\big|\big\rangle\Big|\\
 & =\Big|\big\langle\eta_{t}(v_{t})\circ\eta_{t}^{\prime}(v_{t}),\xi_{t-1}\big\rangle\Big|+\rho_{1}\big\|\eta_{t}(v_{t})\big\|_{\infty}\Big\|\big|\xi_{t-1}\big|^{2}\Big\|_{1}+2\rho\Big|\big\langle\eta_{t}(v_{t})\circ\Gamma,\big|\xi_{t-1}\big|\big\rangle\Big|\\
 & \leq\big\|\eta_{t}(v_{t})\circ\eta_{t}^{\prime}(v_{t})\big\|_{2}\big\|\xi_{t-1}\big\|_{2}+\rho_{1}\big\|\eta_{t}(v_{t})\big\|_{\infty}\big\|\xi_{t-1}\big\|_{2}^{2}+2\rho\big\|\eta_{t}(v_{t})\big\|_{\infty}\big\|\Gamma\big\|_{2}\big\|\xi_{t-1}\big\|_{2}\\
 & \lesssim F_{t}\big\|\xi_{t-1}\big\|_{2}+\rho_{1}G_{t}\big\|\xi_{t-1}\big\|_{2}^{2}+ \rho G_{t}\sqrt{E_{t}}\big\|\xi_{t-1}\big\|_{2}. 
\end{align*}
Here, the last line follows from Assumptions~\eqref{defi:F} and \eqref{defi:G}, as well as the fact that
\[
	\big\|\Gamma\big\|_{2} \leq \big\|\widetilde{\Gamma}\big\|_{2} 
	\leq  \sum_{m\in\mathcal{M}_{\mathsf{dc}}} \big\|\widetilde{\Gamma}(m)\big\|_{2} 
	= \sum_{m\in\mathcal{M}_{\mathsf{dc}}} \bigg( \sum_{j=1}^n \widetilde{\Gamma}_j(m) \bigg)^{1/2},
\]
where $\widetilde{\Gamma}(m)\coloneqq\big[\widetilde{\Gamma}_{j}(m)\big]_{1\leq j\leq n}\in\real^{n}$
with $$\widetilde{\Gamma}_{j}(m)=\ind\Big(\Big|\alpha_{t}\vstar_{j}+\sum\nolimits_{k=1}^{t-1}\beta_{t-1}^{k}\phi_{k,j}-m\Big|\le\theta(m)\Big).$$   
Further invoking $\sum_{m\in\mathcal{M}_{\mathsf{dc}}}\sum_{j=1}^{n}\widetilde{\Gamma}_{j}(m) \leq E_{t}$ (see \eqref{defi:E}) and $|\mathcal{M}_{\mathsf{dc}}|=O(1)$ gives 
\begin{align*}
	\big\|\Gamma\big\|_{2} \leq \sum_{m\in\mathcal{M}_{\mathsf{dc}}} \sqrt{E_t} \asymp \sqrt{E_t}. 
\end{align*}
This concludes the proof of the claim  \eqref{eqn:viola}. 

\section{Analysis for spectral initialization: Proof of Theorem~\ref{thm:recursion-spectral}}
\label{sec:pf-thm-spectral}

To establish Theorem~\ref{thm:recursion-spectral}, 
our strategy is to construct some auxiliary AMP sequences that are intimately connected to spectral initialization (obtained via a sequence of power iterations), 
thus allowing us to analyze spectrally initialized AMP by means of the theory developed in Theorems~\ref{thm:recursion} and \ref{thm:main}. 
Note that the auxiliary AMP sequence to be introduced below is designed only for analysis purposes, and is not implemented during the execution of the real algorithm.

Throughout this section, we denote by $\lambda_{\max}$ (resp.~$\widehat{v}^{\star}$) the leading eigenvalue (resp.~eigenvector) of $M$, 
and let $\lambda_i(M)$ represent the $i$-th largest eigenvalue (in magnitude) of $M$. 



\subsection{Preliminaries: non-asymptotic eigenvalue and eigenvector analysis}
\label{sec:preliminary-evalue-evector}

Understanding the performance of spectral methods requires careful control of the eigenvalues and eigenvectors of the random matrices of interest. 
Before embarking on the proof, we gather several useful non-asymptotic eigenvalue/eigenvector perturbation bounds.   
\begin{itemize}
	\item \citet[Corollary~3.9]{bandeira2016sharp} asserts that (by taking $\varepsilon$ therein to be $(\log n/n)^{1/3}$)  
\begin{subequations}
\label{eq:banderia-W-ub}
\begin{align}
	\| W \|  \leq 2 + O\bigg(  \Big(\frac{\log n}{n}\Big)^{1/3} \bigg),
\end{align}
	holds with probability at least $1-O(n^{-15})$. This combined with Weyl's inequality further leads to
\begin{align}
	\big|\lambda_2(M)\big| 
	\leq \| W \|  \leq 2 + O\bigg(  \Big(\frac{\log n}{n}\Big)^{1/3} \bigg). 
	\label{eq:banderia-lambda-i-ub}
\end{align}
\end{subequations}
		
 	\item \citet[Theorem 3.1]{peng2012eigenvalues} establishes that, with probability at least $1-O(n^{-11})$ one has 
 \begin{align}
 	\lambda+\frac{1}{\lambda}- C_9\sqrt{\frac{ \log n}{n(\lambda-1)^{5}}}
 	\leq \lambda_{\max} 
 	\leq \lambda+\frac{1}{\lambda}+C_9 \sqrt{\frac{ \log n}{n}}
 	\label{eq:minyu-lambda-LB}
 \end{align}
		for some large enough constant $C_9>0$, provided that $ 1 + \big(\frac{\log n}{n}\big)^{1/5} < \lambda = O(1)$. 	



\item Applying Weyl's inequality (i.e., $|\lambda_{2}(M)| \leq \|W\|$)  
	and \citet[Theorem 6.1]{simchowitz2018tight} (with $\epsilon=\frac{1}{8}\frac{\lambda + \frac{1}{\lambda}-2}{\lambda + \frac{1}{\lambda}}\min\big\{\frac{1}{2},\frac{1}{\lambda^2-1}\big\}$ and $\kappa=1/2$ taken therein) yield
\begin{subequations}
\begin{align}
	\lambda_{\max}-|\lambda_{2}(M)| &\geq \lambda_{\max} - \| W \| 
	\geq \frac{\lambda + \frac{1}{\lambda}-2}{4} = \frac{(\lambda-1)^2}{4\lambda} 
	\label{eq:lambda-1-gap-non-asymptotic-UCB} \\
	\bigg|\lambda_{\max}-\lambda-\frac{1}{\lambda}\bigg| 
	&\leq\min\left\{ \frac{(\lambda-1)^{2}}{16\lambda},\frac{1}{8\lambda}\cdot\frac{\lambda-1}{\lambda+1}\right\} 	 
	\label{eq:lambda-1-non-asymptotic-UCB}
\end{align}
\end{subequations}
 with probability at least $1-O(n^{-11})$, provided that 
\begin{align}
	\lambda-1\geq C_{\lambda}\Big(\frac{\log (n\lambda) }{n}\Big)^{1/6}
	\label{eq:lambda-1-gap-condition-UCB}
\end{align}
for some sufficiently large constant $C_{\lambda}>0$. 
A direct consequence of \eqref{eq:lambda-1-non-asymptotic-UCB} is that
\begin{subequations}
\label{eq:lambda-max-two}
\begin{align}
	\lambda_{\max}-2 &\geq\lambda+\frac{1}{\lambda}-2-\frac{(\lambda-1)^{2}}{16\lambda}=\frac{15(\lambda-1)^{2}}{16\lambda}>0, \\
	\lambda_{\max} &\leq \lambda+\frac{1}{\lambda} + \frac{1}{8\lambda} \leq 3\lambda. 
\end{align}
\end{subequations}

\item

In addition, we make note of an immediate consequence of \eqref{eq:lambda-1-non-asymptotic-UCB} as follows:
\begin{equation}
	|\widetilde{\lambda} -\lambda|\leq\min\left\{ C_9\sqrt{\frac{ \log n}{n(\lambda-1)^{7}}}, \frac{\lambda-1}{4},\frac{1}{2(\lambda+1)}\right\}, 
	\label{eq:lambda-max-correct-bound}
\end{equation}
where we recall $\widetilde{\lambda}\defn\frac{\lambda_{\max}+\sqrt{\lambda_{\max}^{2}-4}}{2}.$
\begin{proof}[Proof of inequality \eqref{eq:lambda-max-correct-bound}]
In view of \eqref{eq:lambda-1-non-asymptotic-UCB}, one can write $\lambda_{\max}=\lambda+\frac{1}{\lambda}+\Delta$
for some $\Delta$ with $|\Delta|\leq\min\big\{ C_9\sqrt{\frac{ \log n}{n(\lambda-1)^{5}}}, \frac{(\lambda-1)^{2}}{16\lambda},\frac{1}{8\lambda}\frac{\lambda-1}{\lambda+1}\big\}$.
It is readily seen that
\begin{align*}
\Bigg|\frac{\lambda_{\max}+\sqrt{\lambda_{\max}^{2}-4}}{2}-\lambda\Bigg| & =\Bigg|\frac{\lambda+\frac{1}{\lambda}+\Delta+\sqrt{\left(\lambda+\frac{1}{\lambda}+\Delta\right)^{2}-4}}{2}-\lambda\Bigg|\\
 & =\Bigg|\frac{\sqrt{\left(\lambda+\frac{1}{\lambda}+\Delta\right)^{2}-4}-\sqrt{\left(\lambda+\frac{1}{\lambda}\right)^{2}-4}}{2}+\frac{\Delta}{2}\Bigg|\\
 & 
 \leq\frac{1}{2}\frac{|2\Delta\left(\lambda+\frac{1}{\lambda}\right)+\Delta^{2}|}{\sqrt{\left(\lambda+\frac{1}{\lambda}+\Delta\right)^{2}-4}+\sqrt{\left(\lambda+\frac{1}{\lambda}\right)^{2}-4}}+\frac{|\Delta|}{2}\\
 & \leq 
 \frac{3|\Delta|\lambda}{\lambda-\frac{1}{\lambda}}+\frac{|\Delta|}{2}
 \leq\frac{4|\Delta|\lambda}{\lambda-1},
\end{align*}
where the last line follows since $\lambda+\frac{1}{\lambda}\leq2\lambda$
and $|\Delta|\leq\lambda+\frac{1}{\lambda}$. This directly concludes the proof.  
\end{proof}

\end{itemize}

Furthermore, the following lemma develops a non-asymptotic bound on the correlation between the leading eigenvector $\widehat{v}^{\star}$ and the ground truth $\vstar$; the proof can be found in Section~\ref{sec:lem:cor-evector}. 
\begin{lems}
\label{lem:cor-evector}
Suppose that $1 + C_{\lambda}\big(\frac{\log n }{n}\big)^{1/9} \leq \lambda= O(1)$ for some large enough constant $C_{\lambda }>0$.  
The correlation between $\vstar$ and the leading eigenvector $\widehat{v}^{\star}$ of $M$ satisfies 
\begin{equation}
\label{eq:cor-evector}
\big|\big\langle\widehat{v}^{\star},v^{\star}\big\rangle\big|=\sqrt{1-\frac{1}{\lambda^{2}}}+O\Big(\sqrt{\frac{\log n}{(\lambda-1)^{9}n}}\Big)
\end{equation}
with probability at least $1-O(n^{-11})$. 
\end{lems}

\subsection{Constructing an AMP-style basis that covers $\widehat{v}^{\star}$ approximately}
\label{sec:spec-vhatstar}

In this subsection, we design an auxiliary AMP sequence that allows us to construct a set of orthonormal vectors $\{y_s\}_{1\leq k\leq s}$,  
whose span {\em approximately} covers the leading eigenvector $\widehat{v}^{\star}$ of $M$. 

\paragraph{Construction of auxiliary AMP iterates.}
Let us produce the following iterative procedure initialized at the truth $\vstar$: 
\begin{align}
\label{eqn:amp-for-spectral}
	\omega_{t+1} = W \omega_t - \omega_{t-1}, \qquad (t\geq 1) \qquad\text{with }~ w_0 = 0 ~\text{ and }~ w_1 = \vstar.
\end{align}
This iterative procedure involves a power iteration  $W \omega_t$ in each iteration,
while at the same time it takes the form of AMP updates (by subtracting $\omega_{t-1}$ and choosing the denoiser to be the identity function). 
Note, however, 
that the power iteration  $W \omega_t$ in \eqref{eqn:amp-for-spectral} is concerned with only the noise matrix $W$, 
which stands in stark contrast to \eqref{eqn:AMP-updates} that consists of computing $(\lambda \vstar v^{\star\top} + W)\eta_t(x_t)$ and involves the signal component $\lambda \vstar v^{\star\top}$. 
In fact, the signal component comes into play in \eqref{eqn:amp-for-spectral} only through the initial vector  $ v_1 = \vstar.$


\paragraph{Other auxiliary sequences derived from $\{\omega_k\}$.}
Akin to our proof of Theorem~\ref{thm:recursion} (see Section~\ref{sec:pf-thm-recursion}), 
we find it useful to look at several auxiliary sequences $\{y_k\}_{k \ge 1}$, $\{\zeta_k^{\prime}\}$ and $\{\psi_k\}$ derived based on $\{\omega_k\}$, 
which will assist in analyzing $\{\omega_t\}$. 
\begin{itemize}
	\item[(i)] Given that $\omega_1= v^{\star}$, we define
\begin{align}
	y_1 \defn \omega_1 = \vstar \in \mathcal{S}^{n-1},\qquad\text{and}\qquad W_1^{\prime} \defn W. 
	\label{eqn:y-auxiliary-init}
\end{align}
%

	\item[(ii)] For each $2\leq t<n$, concatenate $\{y_{k}\}$ into a matrix $V_{t-1} \defn [y_k]_{1 \le k \leq t-1} \in \real^{n\times (t-1)}$
and define
\begin{equation}
\label{eqn:spect-seq-1}
\begin{aligned}
	y_t &\defn \frac{\lt(I - V_{t-1}V_{t-1}^{\top}\rt)\omega_{t}}{\lt\|\lt(I - V_{t-1}V_{t-1}^{\top}\rt)\omega_{t}\rt\|_2}, 	\\
	W_t^\prime &\defn \lt(I - y_{t-1}y_{t-1}^{\top}\rt)W_{t-1}^{\prime}\lt(I - y_{t-1}y_{t-1}^{\top}\rt).
\end{aligned}
\end{equation}
According to Lemma~\ref{lemma:zk-orthonormal}, the $y_k$'s constructed above are orthonormal, and $\omega_t \in \mathsf{span}\{y_1,\ldots,y_t\}$. 

	\item[(iii)] Additionally, if we generate $\{g_i^k\}_{1\leq i,k\leq n}$ as i.i.d.~$\mathcal{N}(0, \frac{1}{n})$ and define
\begin{align}
	\zeta_k^{\prime} \defn \Big(\frac{\sqrt{2}}{2} - 1\Big) y_ky_k^{\top}W_k^{\prime}y_k + \sum_{i = 1}^{k - 1} g_i^ky_i \in \real^n,
	\qquad 1\leq k\leq n ,
\end{align}
then Lemma~\ref{lem:distribution} reveals that the $\psi_k$'s constructed below are i.i.d.~obeying
\begin{align}
\label{def:psi_k}
	\psi_k \defn W_k^{\prime}y_k + \zeta_k^{\prime} \overset{\text{i.i.d.}}{\sim} \mathcal{N}\lt(0, \frac{1}{n}I_n\rt),\qquad 1 \le k \le n.
\end{align}
\end{itemize}
Clearly, $\{y_k, W_k^\prime, \zeta_k^{\prime}, \psi_k\}$ plays the same role as $\{z_k, W_k, \zeta_k, \phi_k\}$ in expression \eqref{eqn:z-w-recursion} 
in the proof of Theorem~\ref{thm:recursion}.

\paragraph{Connections between $\{\omega_k\}$ and spectral initialization.}

We now discuss some important connections between $\{\omega_k\}$ and the leading eigenvector $\widehat{v}^{\star}$ of $M$. 
One basic fact to connect \eqref{eqn:amp-for-spectral} with the power method 
is that: $W^t v^{\star}$ can be linearly represented by the iterates $\{\omega_i\}_{1 \le i \le t+1}$, 
as stated in the following lemma. Intuitively, this fact makes sense as the update rule \eqref{eqn:amp-for-spectral} resembles that of the power method.
\begin{lems}
\label{lem:linear-combination-power-method-AMP}
For every $t\geq 0$, 
$W^t v^{\star}$ is a linear combination of $\{\omega_i\}_{1 \le i \le t+1}$.
\end{lems}
\begin{proof}[Proof of Lemma~\ref{lem:linear-combination-power-method-AMP}]
We shall establish this result by induction.
First, the claim is trivially true for $t= 0$ since $\omega_1=v^{\star}$. 
Now, suppose the statement further holds true for $t-1$, i.e.,
\begin{align}
\label{eqn:magic-flute}
	W^{t-1}v^{\star} = \sum_{i = 1}^{t} c_{t}^i \omega_i \qquad \text{for some coefficients } c_{t}^i,   
\end{align}
and we would like to extend it to $t$. Towards this, observe that
\begin{align*}
	W^tv^{\star} = W\big(W^{t-1}v^{\star}\big) = W\sum_{i = 1}^t c_t^i\omega_i = \sum_{i = 1}^t c_t^i(\omega_{i+1} + \omega_{i-1}),
\end{align*}
where the last step follows since $W\omega_i = \omega_{i+1}+\omega_{i-1}$ (cf.~\eqref{eqn:amp-for-spectral}). 
For notational simplicity, we shall also set 
\begin{align}
	c_t^i = 0 \qquad \text{for any }i > t \text{ or } i = 0. \label{eq:ct-i-zero-coefficient}
\end{align}
As a result, we can write 
\begin{subequations}
\label{eqn:linear-induc}
\begin{equation}
	W^t v^{\star} = \sum_{i = 1}^{t+1} c_{t+1}^i \omega_i 
\end{equation}
with the coefficients 
\begin{align}
	c_{t+1}^{i}=c_{t}^{i+1}+c_{t}^{i-1} ;
	\label{eq:c-t-i-recursion}
\end{align}
here, we have invoked \eqref{eq:ct-i-zero-coefficient}. 
%
%
\end{subequations}
This validates the claimed result. 
\end{proof}
Moreover, it turns out that  $\widehat{v}^{\star}$ can be approximately  
represented as (i) a linear combination of $\{y_k\}_{1 \le k \le s}$, 
and also (ii) a linear combination of the set of independent Gaussian vectors $\{\psi_k\}_{1 \le k \le s}$ (cf.~\eqref{def:psi_k}).    
This is asserted by the following lemma, whose proof can be found in Section~\ref{sec:pf-lem-spec}. 
\begin{lems} 
\label{lem:spec}
	Let $s = \frac{C_v\lambda^2 \log n}{(\lambda - 1)^2}$ for some sufficiently large constant $C_v>0$. 
	Assume that $1+C_{\lambda} (\frac{\log^{7} n}{n})^{1/6} < \lambda =O(1)$ for some large enough constant $C_{\lambda}>0$. 
	With probability at least $1-O(n^{-11})$, there exist coefficients $c_i $ $(1\leq i\leq s)$ such that  
\begin{align}
\label{eqn:spectral-expansion}
	\bigg\| \widehat{v}^{\star} - \sum_{i = 1}^{\power} c_iy_i \bigg\|_2 \lesssim \frac{\log^{3.5} n}{\sqrt{(\lambda - 1)^6n}} 
	\qquad \text{and} \qquad
\bigg\| \widehat{v}^{\star} - c_1v^{\star} + \frac{1}{\widetilde{\lambda}}\sum_{i = 1}^{\power} c_i\psi_i \bigg\|_2 \lesssim \frac{\log^{3.5} n}{\sqrt{(\lambda - 1)^6n}}, \end{align}
where $\widetilde{\lambda} \defn \frac{2}{\lambda_{\max} - \sqrt{\lambda_{\max}^2 - 4}}$, and $c_{i+1} = \widetilde{\lambda}^{-1} c_i$ for all $i\geq 1$. 
\end{lems}
%
This approximate linear representation of $\widehat{v}^{\star}$ 
plays a crucial role in explaining why spectrally initialized AMP yields a similar decomposition as 
another AMP with independent initialization.

\subsection{Constructing another AMP-style basis that covers $x_1$ exactly}
\label{sec:rep-spec}
Next, we turn to our spectral estimate obtained through the power method:
\[
	x_{1} = a_s M^s \widetilde{v} \qquad \text {with }a_s = \frac{1}{\|M^s \widetilde{v}\|_2},
\]
where $\widetilde{v} \sim \mathcal{N}(0,\frac{1}{n}I_n)$ is the initial vector of the power method chosen randomly. 
Based upon our results in Section~\ref{sec:spec-vhatstar}, we intend to further augment $\{y_k\}_{1\leq k\leq s}$ 
into a set of $2s+1$ orthonormal vectors $\{\widehat{y}_t\}_{1\leq t \leq 2s+1}$ 
--- again via a certain auxiliary AMP sequence ---
such that $x_{1}$ falls {\em perfectly} within $\mathsf{span} \{\widehat{y}_1,\ldots, \widehat{y}_{2s+1}\}$.

\paragraph{Preliminaries about the power method.}

Standard convergence analysis for the power method tells us that: 
if we take $s \geq \widetilde{C}_v \frac{ \lambda_{\max} \log n}{\lambda_{\max}-|\lambda_2(M)|}$ for some constant $\widetilde{C}_v>0$ large enough
and if $\widetilde{v}\sim \mathcal{N}(0,\frac{1}{n}I_n)$, then with probability exceeding $1-O(n^{-11})$ we can guarantee that  
\begin{subequations}
\label{eqn:power-error}
\begin{align}
	\|x_1 - \widehat{v}^{\star}\|_2 &\lesssim \frac{1}{n^{12}} 
	\qquad \text{ and } \qquad \\
	a_s = \frac{1}{\ltwo{M^s \vseed}} 
	&\leq \frac{1}{\lambda_{\max}^s |\langle \widehat{v}^{\star}, \widetilde{v}\rangle | }
	\lesssim \frac{n^{11.5}}{\lambda_{\max}^s},
\end{align}
\end{subequations}
where the last inequality is valid since 
$\langle \widehat{v}^{\star}, \widetilde{v}\rangle \stackrel{\textrm{d}}{=} \frac{\inprod{g}{v}}{\ltwo{g}},~g\sim \mathcal{N}(0,\frac{1}{n}I_n)$ and hence $|\langle \widehat{v}^{\star}, \widetilde{v}\rangle| \gtrsim n^{-11.5}$ with probability at least $1-O(n^{-11})$. 
In addition,  \eqref{eq:lambda-1-gap-non-asymptotic-UCB} and \eqref{eq:lambda-1-non-asymptotic-UCB} 
allow us to control $\frac{\lambda_{\max}}{\lambda_{\max} - |\lambda_2(M)|}$, 
thus indicating that \eqref{eqn:power-error} is guaranteed to hold as long as 
\[
	s \geq C_v \frac{ \lambda \log n}{(\lambda - 1)^2}
\]
for some constant $C_v$ large enough.


%

In addition, we remark that there exist a set of coefficients $a_{0},\ldots, a_{s-1} \in \real$ that allow us to express
\begin{align}
\label{eqn:x1-decmp-violin}
	x_1 = \sum_{i = 0}^{s-1} a_iW^{i}v^{\star} + a_sW^s\vseed, 
	\qquad \text{with }a_s = \frac{1}{\ltwo{M^s \vseed}}. 
\end{align}
\begin{proof}[Proof of \eqref{eqn:x1-decmp-violin}]
Recall that $x_1$ is proportional to $	(\lambda v^{\star}v^{\star\top} + W)^s \vseed$. 
By expanding $(\lambda v^{\star}v^{\star\top} + W)^s$, we know that each term in the expansion takes one of the following forms:  
\[
	\text{(i) }\lambda v^{\star}v^{\star\top}A_{1}\widetilde{v}\text{ for some matrix }A_{1};\quad
	\text{(ii) }W^{i}(\lambda v^{\star}v^{\star\top})A_{2}\widetilde{v}\text{ for some matrix }A_{2}\text{ and some }i; \quad
	\text{(iii) }W^{s}\widetilde{\nu}.
\]
Clearly, in each case the term falls within the span of $\{v^{\star}, W^i v^{\star}, W^s \vseed \}$, thus concluding the proof.
\end{proof}

\paragraph{Construction of a set of basis vectors using another auxiliary AMP.} 

Based on the decomposition \eqref{eqn:x1-decmp-violin}, 
we intend to show that $x_{1}$ can be linearly represented (in an exact manner) using a set of $2s+1$ orthonormal basis vectors, 
in a way similar to Lemma~\ref{lem:spec}. 
Towards this end, we design another AMP-type algorithm (with the denoising functions taken as the identity function):
\begin{align}
\label{eqn:amp-for-power}
	u_{t+1} = W u_t - u_{t-1} \quad (t>s), \qquad\text{with }~ u_{s} = 0 ~\text{ and }~ u_{s+1} = \vseed, 
\end{align}
Despite the use of the same update rule, 
a key difference between \eqref{eqn:amp-for-power} and \eqref{eqn:amp-for-spectral} lies in that $u_t$ starts from $\vseed$ (i.e., the vector used to initialize the power method), while $\omega_t$ starts from the ground-truth vector $\vstar$.

Akin to our analysis for Theorem~\ref{thm:recursion}, we generate a sequence of orthonormal vectors $\{\widehat{y}_k\}$ and auxiliary random matrices $\widehat{W}_k$ as follows: recalling the sequence $\{y_k, W_k\}$ defined in \eqref{eqn:spect-seq-1}, we take
\begin{equation}
\label{eqn:spect-seq-2}
\begin{aligned}
	\widehat{y}_t &\defn y_t, \qquad &&1\leq t\leq s, \\
	\widehat{W}_t &\defn W_t^{\prime}, \qquad &&1\leq t\leq s, \\
	\widehat{y}_t &\defn \frac{\big(I - \widehat{V}_{t-1}\widehat{V}_{t-1}^{\top}\big)u_{t}}{\big\|\big(I - \widehat{V}_{t-1}\widehat{V}_{t-1}^{\top}\big)u_{t}\big\|_2}, 	\qquad && s< t \leq 2s+1,\\
	\widehat{W}_t &\defn \big(I - \widehat{y}_{t-1}\widehat{y}_{t-1}^{\top}\big)\widehat{W}_{t-1}\big(I - \widehat{y}_{t-1}\widehat{y}_{t-1}^{\top}\big),
	\qquad && s< t \leq 2s+1, 
\end{aligned}
\end{equation}
where $\widehat{V}_{t-1} \defn [\widehat{y}_k]_{1 \le k \leq t-1} \in \real^{n\times (t-1)}$. 
The orthonormality of the sequence $\{\widehat{y}_k\}_{1\leq k\leq  2s+1}$ can be seen by repeating the proof of Lemma~\ref{lemma:zk-orthonormal}. 
In addition,  let us further generate the following vectors
\begin{align}
	\widehat{\psi}_k \defn \widehat{W}_k\widehat{y}_k + \widehat{\zeta}_k \overset{\text{i.i.d.}}{\sim} \mathcal{N}\lt(0, \frac{1}{n}I_n\rt),\qquad\text{for all }1 \le k \le n,
\end{align}
where
\begin{align}
	\widehat{\zeta}_k \defn \Big(\frac{\sqrt{2}}{2} - 1\Big) \widehat{y}_k\widehat{y}_k^{\top}\widehat{W}_k\widehat{y}_k + \sum_{i = 1}^{k - 1} g_i^k\widehat{y}_i,
\end{align}
with the $g_i^k$'s independently drawn from $\mathcal{N}(0, \frac{1}{n})$. 
Then Lemma~\ref{lem:distribution} and its analysis immediately tell us that the $\widehat{\psi}_k$'s are statistically independent obeying
\begin{align}
	\widehat{\psi}_k  \overset{\text{i.i.d.}}{\sim} \mathcal{N}\lt(0, \frac{1}{n}I_n\rt),\qquad\text{for all }1 \le k \le n,
\end{align}
%


\paragraph{Linear representation of $x_1.$}

We are positioned to represent $x_1$ over the set of basis vectors  It turns out that $x_1$ can be represented approximately as the linear combination of $\{\widehat{y}_k\}$, or the set of independent Gaussian vectors $\{\widehat{\psi}_k\}$. 
Our result is formally stated as follows. 
\begin{lems}
\label{lem:alt-x1}
With probability exceeding $1-O(n^{-11})$, we have
\begin{align}
\label{eqn:alt-x1}
	x_1 = \sum_{i = 1}^{2s+1} b_i\widehat{y}_i \qquad \text{and} \qquad
	\Bigg\| x_{1}-c_{1}v^{\star}-\frac{1}{\lambda}\sum_{i=1}^{2s+1}b_{i}\widehat{\psi}_{i}\Bigg\|_{2} & \lesssim\frac{\log^{3.5}n}{\sqrt{(\lambda-1)^{7}n}}
\end{align}
with 
\begin{equation}
	\label{eq:defn-bi-x1}
	b_i \defn \inprod{\widehat{y}_i}{x_1}\qquad \text{for any }1\leq i\leq 2s+1.
\end{equation}
\end{lems}
\noindent The proof of this lemma is deferred to Section~\ref{sec:proof-eqn:alt-x1}

\subsection{Analysis for spectrally initialized AMP}

We are now positioned to develop non-asymptotic analysis for the spectrally initialized AMP, namely, the AMP sequence $\{x_t\}$ (cf.~\eqref{eqn:AMP-updates}) when initialized to $x_1$ (i.e., the output of the power method).

\paragraph{Auxiliary sequences derived from $\{x_t\}$.}

Akin to the proof of Theorem~\ref{thm:recursion}, we introduce a sequence of auxiliary vectors/matrices $\{z_k, W_k, \zeta_k\}_{-2s\le k \le n}$ in a recursive manner in order to help understand the dynamics of $x_{t}$:  
\begin{itemize}
\item 
For any $k$ with $-2s \leq k \leq 0$, set 
\begin{align*}
 	 z_{k} \defn \widehat{y}_{k+2s+1}, \quad
 	 W_{k} \defn \widehat{W}_{k+2s+1}^{\prime}, \quad
 	 \zeta_{k} \defn \zeta_{k+2s+1}^{\prime}, \quad
 	 \phi_{k} \defn \widehat{\psi}_{k+2s+1}, 
 \end{align*} 
where $\{\widehat{y}_{k}, \widehat{W}_k , \widehat{\psi}_k\}$ have been introduced in Section~\ref{sec:rep-spec}.

\item For any $1\leq k \leq n$, define 
\begin{subequations}
\begin{align}
	z_1 \defn \frac{\big(I - \widehat{V}_{2s+1}\widehat{V}_{2s+1}^{\top}\big)\eta_1(x_1)}{\big\|\big(I - \widehat{V}_{2s+1}\widehat{V}_{2s+1}^{\top}\big)\eta_1(x_1)\big\|_2} \in \real^n,
\end{align}
where we remind the readers that $\widehat{V}_{2s+1}=[\widehat{y}_1,\ldots,\widehat{y}_{2s+1}]$. 
Further, we take 
\begin{equation}
\begin{aligned}
	U_{k-1} &\defn \big[\widehat{V}_{2s+1},~[z_i]_{1 \le i \leq k-1}\big] \in \real^{n\times (k+2\power)},\\
	z_{k} &\defn \frac{\lt(I - U_{k-1}U_{k-1}^{\top}\rt)\eta_{k}(x_{k})}{\lt\|\lt(I - U_{k-1}U_{k-1}^{\top}\rt)\eta_{k}(x_{k})\rt\|_2}, 	\\
	W_k &\defn \lt(I - z_{k-1}z_{k-1}^{\top}\rt)W_{k-1}\lt(I - z_{k-1}z_{k-1}^{\top}\rt).
\end{aligned}
\end{equation}
\end{subequations}
\end{itemize}

\noindent With these definitions in place,  we see that for each $t\geq 1$, the vectors $\{z_k\}^{t}_{-2s}$ are orthonormal whose span contains $\eta_t(x_t)$ (see Lemma~\ref{lemma:zk-orthonormal} and the text right after),  
This allows us to decompose
\begin{align}
\label{eqn:eta-decomposition-augment}
	\eta_{t}(x_{t}) = 
	\sum_{k = -2s}^{t} \beta_{t}^kz_k, \qquad\text{with }\beta_{t}^k \defn \langle\eta_{t}(x_{t}), z_k\rangle
\end{align}
and ensure that $\lt\|\eta_{t}(x_{t})\rt\|_2 = \lt\|\beta_{t}\rt\|_2$ with $\beta_{t} \defn (\beta_t^{-2\power},\ldots,\beta_t^{0},\beta_t^1,\beta_t^2,\ldots,\beta_t^t)^{\top} \in \real^{t+2s+1}.$ 
Additionally, we introduce the following vectors as in Lemma~\ref{lem:distribution}:  
\begin{align}
\label{eqn:zeta-k-augment}
	\zeta_k \defn \Big(\frac{\sqrt{2}}{2} - 1\Big) z_kz_k^{\top}W_kz_k + \sum_{i = -2\power}^{k - 1} g_i^kz_i
	\qquad\text{for } k\geq 1,
\end{align}
with each $g_i^k$ independently generated from $\mathcal{N}(0, \frac{1}{n})$, and we set 
\begin{align}
	\phi_k \defn W_kz_k + \zeta_k ,\qquad\text{for all } -2s \le k < n-2s.
\end{align}
Consequently, repeating exactly the same argument as in the proof of Lemma~\ref{lem:distribution} reveals that
\begin{align}
	\phi_k  \overset{\text{i.i.d.}}{\sim} \mathcal{N}\lt(0, \frac{1}{n}I_n\rt),\qquad\text{for all } -2s \le k < n-2s.
\end{align}


\paragraph{Analyzing spectrally initialized AMP via our general recipe.}

Recall that Theorem~\ref{thm:recursion} offers a general recipe in deriving the decomposition for $x_{t}$. 
As it turns out, the same induction-based proof idea developed for Theorem~\ref{thm:recursion} continues to work for analyzing spectrally initialized AMP.  
In fact, assuming validity for the initialization (which we shall justify momentarily), such proof arguments lead to:
\begin{align}
	x_t \defn \alpha_t \vstar + \sum_{k = -2\power}^{t-1} \beta_{t-1}^k\phi_k + \xi_{t-1}, \qquad\text{for }1 \leq t < n-2s;  
	\label{def:dynamics-augment}
\end{align}
here, $\alpha_{t+1} = \lambda v^{\star\top} \eta_{t}(x_t)$, and $\xi_{t-1}$ is the residual term obeying 
\begin{align*}
	\xi_t &= \sum_{k = -2\power}^{t-1} z_k\bigg[\Big\langle \phi_k, \eta_{t}\big(\alpha_t\vstar + \sum_{k = -2\power}^{t-1} \beta_{t-1}^k\phi_k + \xi_{t-1}\big)\Big\rangle - \langle\eta_t^{\prime}(x_{t})\rangle \beta_{t-1}^k - (\sqrt{2}-1) \beta_{t}^kz_k^{\top}W_kz_k - \sum_{i = -2\power}^{k - 1} \beta_t^ig_i^k - \sum_{i = k+1}^{t} \beta_{t}^ig_k^i\bigg].
\end{align*}
Following Steps 2 and 3 verbatim in Section~\ref{sec:pf-thm-recursion}, we see that \eqref{def:dynamics-augment}
holds true for $t+1$ and satisfies 
\begin{align*}
\|\xi_t\|_2 
&= \Big\langle \sum_{k = -2\power}^{t-1} \mu_t^k\phi_k, \delta_{t}\Big\rangle - \langle\delta_{t}^{\prime}\rangle \sum_{k = -2\power}^{t-1} \mu_t^k\beta_{t-1}^k + \Delta_t 
- 
	\sum_{k = -2\power}^{t-1} \mu_t^k\lt[ (\sqrt{2}-1) \beta_{t}^kz_k^{\top}W_kz_k + \sum_{i = -2\power}^{k - 1} \beta_t^ig_i^k + \sum_{i = k+1}^{t} \beta_{t}^ig_k^i\rt], 
\end{align*}
where $\Delta_t$, $\delta_{t}$ and $\delta_{t}^{\prime}$ are defined in \eqref{eqn:delta-chorus-spectral}.
Further, taking this collectively with Lemma~\ref{lem:concentration} establishes \eqref{eqn:xi-norm-main-spectral}.

We still need to verify that the spectral estimate $x_1$ also satisfies the desired decomposition \eqref{def:dynamics-augment}. 
Towards this, note that if we choose $x_1 = \lambda \eta_{0}(x_0)$ 
(e.g., taking $x_1= \lambda x_0$ and choosing $\eta_0$ to be  identity), 
then 
\begin{align}
	\beta_0^k \coloneqq \frac{1}{\lambda} b_{k+2s+1} 
	= \frac{1}{\lambda} \inprod{x_1}{\widehat{y}_{k+2s+1}} 
	= \frac{1}{\lambda} \inprod{x_1}{z_{k}} = \big\langle \eta_{0}(x_0), z_k \big\rangle  ,
	\label{eq:defn-beta0-spect}
\end{align}
where we have used $\eqref{eq:defn-bi-x1}$. This combined with inequality~\eqref{eqn:alt-x1} gives
\begin{align*}
	\bigg\| x_1 - c_1v^{\star} -  \sum_{i = -2s}^{0} \beta_0^k\phi_k \bigg\|_2 = 
	\bigg\| x_1 - c_1v^{\star} - \frac{1}{\lambda} \sum_{i = -2s}^{0} b_{k+2s+1}\widehat{\psi}_{k+2s+1} \bigg\|_2 
	\lesssim \frac{\log^{3.5} n}{\sqrt{(\lambda - 1)^7n}}, 
\end{align*}
In addition, Lemma~\ref{lem:spec} tells us that 
\[
\big|c_{1}-\big\langle y_{1},\widehat{v}^{\star}\big\rangle\big|=\bigg|\Big\langle y_{1},\sum_{i=1}^{s}c_{i}y_{i}\Big\rangle-\big\langle y_{1},\widehat{v}^{\star}\big\rangle\bigg|\leq\bigg\|\sum_{i=1}^{s}c_{i}y_{i}-\widehat{v}^{\star}\bigg\|_{2}\lesssim\frac{\log^{3.5}n}{\sqrt{(\lambda-1)^{6}n}},
\]
where we have used the fact that $y_{1}=v^{\star}$ and the orthonormality of $\{y_i\}$. 
Combining this with Lemma~\ref{lem:cor-evector} implies that
\[
\Big|c_{1}- \sqrt{1-\frac{1}{\lambda^{2}}} \Big| \leq 
\big|c_{1}-\big\langle y_{1},\widehat{v}^{\star}\big\rangle\big| + \Big|\big\langle y_{1},\widehat{v}^{\star}\big\rangle - \sqrt{1-\frac{1}{\lambda^{2}}} \Big|
\lesssim\frac{\log^{3.5}n}{\sqrt{(\lambda-1)^{9}n}}. 
\]
Putting the preceding results together, we can express
\begin{align}
\label{eqn:ronaldinho}
	x_1 =  \sqrt{1-\frac{1}{\lambda^{2}}} \,v^{\star} +  \sum_{i = -2s}^{0} \beta_0^k\phi_k + \xi_0
\end{align}
where $\|\xi_0\|_2\lesssim \frac{\log^{3.5}n}{\sqrt{(\lambda-1)^{9}n}}$. 
This justifies the validity of \eqref{def:dynamics-augment} for the base case with $t=1$, thus concluding the   proof of Theorem~\ref{thm:recursion-spectral}.

\subsection{Proof of Lemma~\ref{lem:cor-evector}}
\label{sec:lem:cor-evector}

Given the rotational invariance of the Wigner matrix $W$, we shall
assume without loss of generality that $v^{\star}=e_{1}$ throughout
this proof. We also introduce the convenient notation $W=\left[\begin{array}{cc}
W_{1,1} & w_{n-1}^{\top}\\
w_{n-1} & W_{n-1}
\end{array}\right]$, where $W_{n-1}\in\mathbb{R}^{(n-1)\times(n-1)}$ and $w_{n-1}\in\mathbb{R}^{n-1}$
are statistically independent. 

To begin with, it is readily seen from (\ref{eq:lambda-1-gap-non-asymptotic-UCB})
that $\lambda_{\max}I_{n-1}-W_{n-1}$ is invertible with probability
at least $1-O(n^{-11})$. Apply \citet[Theorem~5]{li2021minimax}
to show that
\[
\big|\big\langle\widehat{v}^{\star},v^{\star}\big\rangle\big|^{2}=\frac{1}{1+\big\|\big(\lambda_{\max}I_{n-1}-W_{n-1}\big)^{-1}w_{n-1}\big\|_{2}^{2}}\eqqcolon\frac{1}{1+\big\|\sqrt{n}R(\lambda_{\max})w_{n-1}\big\|_{2}^{2}},
\]
where for notational simplicity we define, for any $\lambda_{0}>\|W\|$,
\[
R(\lambda_{0})\coloneqq\frac{1}{\sqrt{n}}\big(\lambda_{0}I_{n-1}-W_{n-1}\big)^{-1}.
\]
In what follows, let us first analyze the target quantity for any
fixed $\lambda_{0}\in\big[\frac{7}{8}\big(\lambda+\frac{1}{\lambda}\big)+\frac{1}{4},3\lambda]$. 
\begin{itemize}
\item First of all, \citet[Lemmas~5.6 and 5.7]{peng2012eigenvalues} combined
with a little algebra implies that
\[
\Bigg|\big\| R(\lambda_{0})\big\|_{\mathrm{F}}^{2}-\frac{1}{\big(\frac{\lambda_{0}+\sqrt{\lambda_{0}^{2}-4}}{2}\big)^{2}-1}\Bigg|=\Bigg|\mathsf{Tr}\Big[\big(\lambda_{0}I_{n-1}-W_{n-1}\big)^{-2}\Big]-\frac{1}{\big(\frac{\lambda_{0}+\sqrt{\lambda_{0}^{2}-4}}{2}\big)^{2}-1}\Bigg|\lesssim\sqrt{\frac{\log n}{(\lambda-1)^{8}n}}
\]
holds with probability at least $1-O(n^{-15})$, provided that $1<\lambda=O(1)$.
As a result, 
\begin{equation}
\big\| R(\lambda_{0})\big\|_{\mathrm{F}}^{2}=\frac{1}{\big(\frac{\lambda_{0}+\sqrt{\lambda_{0}^{2}-4}}{2}\big)^{2}-1}+O\bigg(\sqrt{\frac{\log n}{(\lambda-1)^{8}n}}\bigg)\lesssim \frac{1}{\lambda-1},\label{eq:Rlambda-F}
\end{equation}
where the last relation holds since $\lambda_{0} \geq  \lambda+\frac{1}{\lambda}-\frac{1}{4}\big(\lambda+\frac{1}{\lambda}-2\big)$ and hence
$\frac{\lambda_{0}+\sqrt{\lambda_{0}^{2}-4}}{2}-1 \gtrsim \lambda - 1$ (by repeating the proof of inequality \eqref{eq:lambda-max-correct-bound} with $\Delta = \frac{1}{4}\big(\lambda+\frac{1}{\lambda}-2$). 
In addition, it follows from (\ref{eq:banderia-W-ub}) that
\begin{equation}
\|R(\lambda_{0})\|\leq\frac{1}{\sqrt{n}}\frac{1}{\lambda_{0}-\|W_{n-1}\|}\leq\frac{1}{\sqrt{n}}\cdot\frac{1}{\frac{7}{8}\big(\lambda+\frac{1}{\lambda}\big)+\frac{1}{4}-2-O\big((\frac{\log n}{n})^{1/3}\big)}\lesssim\frac{1}{(\lambda-1)^{2}\sqrt{n}}.\label{eq:Rlambda-op}
\end{equation}
\item Further, invoking \citet[Eq.~(2.1)]{rudelson2013hanson} (with $\varepsilon$
therein taken to be $C_{9}\frac{\|R(\lambda_{0})\|\sqrt{\log n}}{\|R(\lambda_{0})\|_{\mathrm{F}}}$
for some large enough constant $C_{9}>0$) reveals that, conditional
on $W_{n-1}$, 
\[
\Big|\big\|\sqrt{n}R(\lambda_{0})w_{n-1}\big\|^{2}-\|R(\lambda_{0})\|_{\mathrm{F}}^{2}\Big|\leq C_{9}\|R(\lambda_{0})\|_{\mathrm{F}}\|R(\lambda_{0})\|\sqrt{\log n}\lesssim\sqrt{\frac{\log n}{(\lambda-1)^{5}n}}
\]
holds with probability at least $1-O(n^{-15})$, provided that $C_{9}\frac{\|R(\lambda_{0})\|\sqrt{\log n}}{\|R(\lambda_{0})\|_{\mathrm{F}}}<1$
(which is guaranteed to hold due to (\ref{eq:Rlambda-F}) and (\ref{eq:Rlambda-op})). 
\item Combine the above results to yield, with probability at least $1-O(n^{-15})$,
\begin{align*}
\Bigg|\frac{1}{1+\big\|\sqrt{n}R(\lambda_{0})w_{n-1}\big\|_{2}^{2}}-\frac{1}{1+\frac{1}{\big(\frac{\lambda_{0}+\sqrt{\lambda_{0}^{2}-4}}{2}\big)^{2}-1}}\Bigg| & \leq\Bigg|\big\|\sqrt{n}R(\lambda_{0})w_{n-1}\big\|_{2}^{2}-\frac{1}{\big(\frac{\lambda_{0}+\sqrt{\lambda_{0}^{2}-4}}{2}\big)^{2}-1}\Bigg|\lesssim\sqrt{\frac{\log n}{(\lambda-1)^{8}n}}.
\end{align*}
\end{itemize}
Next,  invoke standard epsilon-net argument \citep[Chapter 4.2]{vershynin2018high}
to show that
\begin{align*}
\Bigg|\frac{1}{1+\big\|\sqrt{n}R(\lambda_{0})w_{n-1}\big\|_{2}^{2}}-\frac{1}{1+\frac{1}{\big(\frac{\lambda_{0}+\sqrt{\lambda_{0}^{2}-4}}{2}\big)^{2}-1}}\Bigg| & \lesssim\sqrt{\frac{\log n}{(\lambda-1)^{8}n}},\qquad\forall\lambda_{0}\in\Big[\frac{7}{8}\big(\lambda+\frac{1}{\lambda}\big)+\frac{1}{4},3\lambda\Big]
\end{align*}
with probability at least $1-O(n^{-11})$; we omit this standard argument
here for the sake of brevity. Recognizing that $\lambda_{\max}\in\big[\frac{7}{8}\big(\lambda+\frac{1}{\lambda}\big)+\frac{1}{4},3\lambda]$
(see \eqref{eq:lambda-1-non-asymptotic-UCB} and \eqref{eq:lambda-max-two}) and defining $\widetilde{\lambda}\coloneqq\frac{\lambda_{\max}+\sqrt{\lambda_{\max}^{2}-4}}{2}$,
we immediately obtain
\begin{align*}
\big|\big\langle\widehat{v}^{\star},v^{\star}\big\rangle\big|^{2} & =\frac{1}{1+\big\|\sqrt{n}R(\lambda_{\max})w_{n-1}\big\|_{2}^{2}}=\frac{1}{1+\frac{1}{\widetilde{\lambda}^{2}-1}}+O\Big(\sqrt{\frac{\log n}{(\lambda-1)^{8}n}}\Big)=1-\frac{1}{\widetilde{\lambda}^{2}}+O\Big(\sqrt{\frac{\log n}{(\lambda-1)^{8}n}}\Big)\\
	& =1-\frac{1}{\lambda^{2}}+ O\Big( \frac{|\lambda - \widetilde{\lambda} |}{\lambda \widetilde{\lambda}} \Big) + O\Big(\sqrt{\frac{\log n}{(\lambda-1)^{8}n}}\Big)
  =1-\frac{1}{\lambda^{2}}+O\Big(\sqrt{\frac{\log n}{(\lambda-1)^{8}n}}\Big),
\end{align*}
where the last inequality comes from (\ref{eq:lambda-max-correct-bound})
and $\lambda\asymp1$. Consequently, we arrive at
\[
\Bigg|\big|\big\langle\widehat{v}^{\star},v^{\star}\big\rangle\big|-\sqrt{1-\frac{1}{\lambda^{2}}}\Bigg|=\frac{\Big|\big|\big\langle\widehat{v}^{\star},v^{\star}\big\rangle\big|^{2}-\big(1-\frac{1}{\lambda^{2}}\big)\Big|}{\big|\big\langle\widehat{v}^{\star},v^{\star}\big\rangle\big|+\sqrt{1-\frac{1}{\lambda^{2}}}}\lesssim\frac{\sqrt{\frac{\log n}{(\lambda-1)^{8}n}}}{\sqrt{1-\frac{1}{\lambda^{2}}}}\asymp\sqrt{\frac{\log n}{(\lambda-1)^{9}n}}.
\]


\subsection{Proof of Lemma~\ref{lem:spec}}
\label{sec:pf-lem-spec}

Recall that the leading eigenvector $\widehat{v}^{\star}$ of $M$ satisfies $M \widehat{v}^{\star} = \lambda_{\max} \widehat{v}^{\star}$. 
In view of the Neumann expansion for eigenvectors  (see, e.g.~\cite[Theorem 2]{chen2021asymmetry}), $\widehat{v}^{\star}$ admits the following expansion: 
\begin{align}
\label{eqn:x0-neumann}
	\widehat{v}^{\star} = \widetilde{c}_0\sum_{t = 0}^{\infty} \frac{1}{\lambda_{\max}^{t}} W^tv^{\star}, 
	\qquad \text{with }\widetilde{c}_0 = \frac{\lambda}{\lambda_{\max}} \langle v^{\star},\widehat{v}^{\star}\rangle ,
\end{align}
with the proviso that $\|W\|<\lambda_{\max}$ --- a condition that has been guaranteed in \eqref{eq:lambda-1-gap-non-asymptotic-UCB}.  
Clearly, one has
\begin{equation}
|\widetilde{c}_0|=\frac{\lambda}{\lambda_{\max}}\big|\langle v^{\star},\widehat{v}^{\star}\rangle\big|\leq\big|\langle v^{\star},\widehat{v}^{\star}\rangle\big|\leq
	1. \label{eq:coefficient-c0-small-than-1}
\end{equation}

Next, it follows from Lemma~\ref{lem:linear-combination-power-method-AMP} that $W^t v^{\star}$ can be written as a linear combination of $\{\omega_i\}_{1 \le i \le t+1}$ with $\omega_{i}$ defined in \eqref{eqn:amp-for-spectral}. 
Substituting \eqref{eqn:linear-induc} into expression~\eqref{eqn:x0-neumann} yields 
\begin{align}
	\widehat{v}^{\star} &= \widetilde{c}_0\sum_{t = 0}^{\infty} \frac{1}{\lambda_{\max}^{t}} W^tv^{\star}
	= \widetilde{c}_0\sum_{t = 0}^{\infty} \frac{1}{\lambda_{\max}^{t}} \Big(\sum_{i = 1}^{t+1} c_{t+1}^i \omega_i\Big)
	= \sum_{i=1}^{\infty} \Big(\widetilde{c}_0\sum_{t = i-1}^{\infty} \lambda_{\max}^{-t}c_{t+1}^i \Big) \omega_i 
	 \notag\\
	&= \sum_{i=1}^{\infty} \Big(\widetilde{c}_0\sum_{t = 0}^{\infty} \lambda_{\max}^{-t}c_{t+1}^i \Big) \omega_i \eqqcolon \sum_{i=1}^{\infty} c_i \omega_i,	
	\qquad \text{with } c_i \defn \widetilde{c}_0\sum_{t = 0}^{\infty} \lambda_{\max}^{-t}c_{t+1}^i \text{ for }i\geq 1, 
	\label{eq:vhat-ci-defn}
\end{align}
where the second line has made use of \eqref{eq:ct-i-zero-coefficient}. 
To proceed, let us claim for the moment that the following relations hold true for all $t$ obeying $\frac{t^{5}\log n}{n}=o(1)$:  
\begin{subequations}
\label{eqn:spec-finale}
\begin{align}
& c_i = \widetilde{\lambda}^{-1}\cdot c_{i-1}, \qquad \text{where  }\widetilde{\lambda}^{-1} \defn \frac{\lambda_{\max} - \sqrt{\lambda_{\max}^2 - 4}}{2},
\label{eqn:spec-finale-1}\\
& \|\omega_t\|_2  = 1+O\Big(\sqrt{\frac{t^5\log n}{n}}\Big),\label{eqn:spec-finale-3} \\
&\big\| \omega_t - \psi_{t-1} \big\|_2  \lesssim \sqrt{\frac{t^5\log n}{n}} 
	\quad \text{and} \quad  \big\| \omega_t -  y_t \big\|_2 \lesssim \sqrt{\frac{t^5\log n}{n}},
\label{eqn:spec-finale-2} \\
&\sum_{i=1}^{\frac{C_{v}\log n}{\widetilde{\lambda}-1}}(c_{i})^{2} \leq4.\label{eq:sum-of-square-ci} 
\end{align}
\end{subequations}
In particular, when $\lambda = O(1)$, it follows from \eqref{eq:lambda-max-correct-bound} that, with probability at least $1-O(n^{-11})$, 
\begin{subequations}
\begin{equation}
	\big| \widetilde{\lambda}  - \lambda \big| 
	= \bigg| \frac{\lambda_{\max}+\sqrt{\lambda_{\max}^2-4}}{2} - \lambda \bigg|
	\leq \frac{\lambda - 1 }{4} ,
	\label{eq:minyu-lambda-tilde-bound}
\end{equation}
and as a result, 
\begin{equation}
	\widetilde{\lambda} - 1 \geq \lambda - \frac{\lambda - 1 }{4} - 1
	= \frac{3(\lambda - 1) }{4} .
	\label{eq:minyu-lambda-tilde-minus-1-bound}
\end{equation}
\end{subequations}


The preceding claims in \eqref{eqn:spec-finale} allow us to complete the proof of Lemma~\ref{lem:spec}. 
To see this, note that by virtue of expression~\eqref{eqn:spec-finale-1} and \eqref{eqn:spec-finale-3}, 
we can truncate the infinite sum by keeping the first $\frac{C_v\log n}{\widetilde{\lambda}-1}$ terms for some large enough constant $C_v>0$, namely, 
\begin{align}
\bigg\|\widehat{v}^{\star}-\sum_{i=1}^{\frac{C_{v} \log n}{\widetilde{\lambda}-1}}c_{i}\omega_{i}\bigg\|_{2} 
	& =\bigg\| \widetilde{c}_0 \sum_{i=\frac{C_{v}\log n}{\widetilde{\lambda}-1}+1}^{\infty}c_{i}\omega_{i}\bigg\|_{2}
	\leq O\bigg(|\widetilde{c}_0|\sum_{i=\frac{C_{v}\log n}{\widetilde{\lambda}-1}}^{\infty}\widetilde{\lambda}^{-i}\bigg)
	\leq O\Big(\frac{|\widetilde{c}_0|\widetilde{\lambda}^{-\frac{C_{v}\log n}{\widetilde{\lambda}-1}}}{\widetilde{\lambda}-1}\Big) \notag\\
	&=O\Big(\frac{|\widetilde{c}_0| n^{-\frac{C_{v}\log\widetilde{\lambda}}{\widetilde{\lambda}-1}}}{\widetilde{\lambda}-1}\Big)=O\Big(\frac{1}{(\widetilde{\lambda}-1)n}\Big) ,
	\label{eq:v-truncated-sum-bound135}
\end{align}
where the last line holds when $C_v$ is large enough and uses the fact that $|\widetilde{c}_0|\leq 1$ (cf.~\eqref{eq:coefficient-c0-small-than-1}).
Notice that here we truncate at the first $\frac{C_{v} \log n}{\widetilde{\lambda}-1}$ terms. If one decides to keep the first $s$ terms for $s \geq \frac{C_{v} \log n}{\widetilde{\lambda}-1}$, it only results in a smaller truncation error. 

Taking the relation~\eqref{eqn:spec-finale-2} and \eqref{eq:v-truncated-sum-bound135} together allows us to demonstrate that 
\begin{align*}
\bigg\|\widehat{v}^{\star}-\sum_{i=1}^{\frac{C_{v}\log n}{\widetilde{\lambda}-1}}c_{i}y_{i}\bigg\|_{2} & \leq\bigg\|\widehat{v}^{\star}-\sum_{i=1}^{\frac{C_{v}\log n}{\widetilde{\lambda}-1}}c_{i}\omega_{i}\bigg\|_{2}+O\bigg(\sum_{i=1}^{\frac{C_{v}\log n}{\widetilde{\lambda}-1}}|c_{i}|\sqrt{\frac{i^{5}\log n}{n}}\bigg)\\
 & \lesssim\frac{1}{(\widetilde{\lambda}-1)n}+\bigg(\sum_{i=1}^{\frac{C_{v}\log n}{\widetilde{\lambda}-1}}|c_{i}|^{2}\bigg)^{\frac{1}{2}}\bigg(\sum_{i=1}^{\frac{C_{v}\log n}{\widetilde{\lambda}-1}}\frac{i^{5}\log n}{n}\bigg)^{\frac{1}{2}}\\
 & \lesssim\bigg(\frac{\log^{7}n}{(\widetilde{\lambda}-1)^{6}n}\bigg)^{\frac{1}{2}}\lesssim \frac{\log^{3.5}n}{\sqrt{(\lambda-1)^{6}n}}, 
\end{align*}
where the penultimate relation comes from \eqref{eq:sum-of-square-ci}, 
and the last relation results from \eqref{eq:minyu-lambda-tilde-minus-1-bound}. 
Also, similar to the arguments in \eqref{eq:v-truncated-sum-bound135} we can obtain
\[
\bigg\|\widetilde{c}_{0}\sum_{i=\frac{C_{v}\log n}{\widetilde{\lambda}-1}+1}^{s}c_{i}\omega_{i}\bigg\|_{2}\lesssim\frac{1}{(\widetilde{\lambda}-1)n}\qquad\text{and}\qquad\bigg\|\widetilde{c}_{0}\sum_{i=\frac{C_{v}\log n}{\widetilde{\lambda}-1}+1}^{s}c_{i}y_{i}\bigg\|_{2}\lesssim\frac{1}{(\widetilde{\lambda}-1)n}.
\]
As a result, we arrive at
\[
\bigg\|\widehat{v}^{\star}-\sum_{i=1}^{s}c_{i}y_{i}\bigg\|_{2}\leq\bigg\|\widehat{v}^{\star}-\sum_{i=1}^{\frac{C_{v}\log n}{\widetilde{\lambda}-1}}c_{i}\omega_{i}\bigg\|_{2}+\bigg\|\widetilde{c}_{0}\sum_{i=\frac{C_{v}\log n}{\widetilde{\lambda}-1}+1}^{\infty}c_{i}\omega_{i}\bigg\|_{2}\lesssim\frac{\log^{3.5}n}{\sqrt{(\lambda-1)^{6}n}}.
\]

Repeating the same argument and recognizing that  $y_{1} = \vstar$ lead to
\begin{equation}
\bigg\|\widehat{v}^{\star}-c_{1}v^{\star}-\frac{1}{\widetilde{\lambda}}\sum_{i=1}^{s}c_{i}\psi_{i}\bigg\|_{2}
	=\bigg\|\widehat{v}^{\star}-c_{1}y_{1}-\sum_{i=2}^{s+1}c_{i}\psi_{i-1}\bigg\|\lesssim\frac{\log^{3.5}n}{\sqrt{(\lambda-1)^{6}n}}.
	\label{eq:vstar-truncated-13579}
\end{equation}
%
This concludes the proof of Lemma~\ref{lem:spec}, as long as the claims in \eqref{eqn:spec-finale} can be justified. 
As a consequence, the remainder of this section is dedicated to proving \eqref{eqn:spec-finale}.


\subsubsection{Proof of claim~\eqref{eqn:spec-finale}}

\paragraph{Proof of recurrence relation~\eqref{eqn:spec-finale-1}.} 
For any $i \ge 2$, it follows from the definition \eqref{eq:vhat-ci-defn} of $c_i$ and the relation \eqref{eq:c-t-i-recursion} that
\begin{align}
	c_{i+1} + c_{i-1} = \widetilde{c}_0\sum_{t = 0}^{\infty} \lambda_{\max}^{-t}(c_{t+1}^{i+1} + c_{t+1}^{i-1}) 
	= \widetilde{c}_0\sum_{t = 0}^{\infty} \lambda_{\max}^{-t}c_{t+2}^i 
	= \lambda_{\max} \widetilde{c}_0\sum_{t = 1}^{\infty} \lambda_{\max}^{-t}c_{t+1}^i
	= \lambda_{\max}c_i,
	\label{eq:recursion-ci-ci}
\end{align}
where the last step is valid since $c_{t}^{i} = 0$ for $i > t.$
To analyze this recurrence relation \eqref{eq:recursion-ci-ci}, 
let us look at the two roots of the characteristic equation $r^2 - \lambda_{\max} r + 1=0$, namely, 
$r_1 = \frac{\lambda_{\max} - \sqrt{\lambda_{\max}^2 - 4}}{2} $, $r_2 = \frac{\lambda_{\max} + \sqrt{\lambda_{\max}^2 - 4}}{2} $.  
It is well known that the solution to \eqref{eq:recursion-ci-ci} can be expressed via these two roots as follows:
\begin{align}
	c_i = a_1(r_1)^i + a_2(r_2)^i, \qquad i\geq 0
	\label{eqn:ci-roots}
\end{align}
for some coefficients $a_1,a_2\in \real$.

In view of \eqref{eq:lambda-max-correct-bound}, one has 
\[
	r_2  \geq 1 +  (\lambda - 1) - \bigg|\frac{\lambda_{\max}+\sqrt{\lambda_{\max}^{2}-4}}{2}-\lambda\bigg|
		\geq  1+ \frac{3(\lambda-1)}{4} > 1,
\]
which also indicates that $r_1 = 1/r_2 < 1$. 
In addition, we claim that 
\begin{align}
	0\leq c_t^i \le 2^t \qquad t\geq 0,~ i\geq 0.
	\label{eq:ct-i-exponential}
\end{align}
This relation can be easily shown by induction: (i) we first learn 
		from Lemma~\ref{lem:linear-combination-power-method-AMP} that $v^{\star}=c_1^1\omega_1 =c_1^1v^{\star}$ and hence $c_1^1=1$, which together with $c_1^i=0$ ($i=0$ or $i>1$) 
justifies \eqref{eq:ct-i-exponential} when $t=1$; 
(ii) if \eqref{eq:ct-i-exponential} is valid for $t$, then it follows from 
\eqref{eq:c-t-i-recursion} that $0\leq c_{t+1}^i = c_t^{i+1} + c_t^{i-1} \le 2^t+2^t\leq 2^{t+1}$, thus establishing \eqref{eq:ct-i-exponential} for $t+1$ --- and hence its validity for all $t\geq 0$. 
Combine \eqref{eq:ct-i-exponential} with~\eqref{eq:vhat-ci-defn} to show the boundedness of $c_i$ in the sense that:
\begin{align}
	|c_{i}| & =|\widetilde{c}_{0}|\sum_{t=0}^{\infty}\lambda_{\max}^{-t}c_{t+1}^{i}\le2|\widetilde{c}_{0}|\sum_{t=0}^{\infty}(\lambda_{\max}/2)^{-t}=\frac{4|\widetilde{c}_{0}|\lambda_{\max}}{\lambda_{\max}-2}\leq\frac{64\lambda^2}{5(\lambda-1)^2},
	\label{eq:ci-upper}
\end{align}
which relies on \eqref{eq:coefficient-c0-small-than-1} and \eqref{eq:lambda-max-two}. 
The boundedness of $c_i$ for any $i\geq 0$ necessarily implies that $a_2=0$ in \eqref{eqn:ci-roots} (otherwise $c_i$ will blow up as $i$ grows given that $r_2>1$). 
We can thus conclude that $c_i = a_1r_1^i$ $(i\geq 0)$ holds for some $a_1\neq 0$, thus implying that
\begin{align*}
\frac{c_i}{c_{i-1}} = r_1 = \frac{\lambda_{\max} - \sqrt{\lambda_{\max}^2 - 4}}{2}.
\end{align*}
%

\paragraph{Proof of inequality \eqref{eqn:spec-finale-3}.}

As discussed previously, the iterates $\{\omega_{t}\}$ in \eqref{eqn:amp-for-spectral} form another sequence of AMP updates with the denoising functions taken to be the identity function. 
In view of Theorem~\ref{thm:recursion}, the iterates $\{\omega_{t}\}$ admit the decomposition
\begin{align}
	\omega_t = \sum_{k = 1}^{t-1} \beta_{t-1}^k\psi_k + \xi_{t-1}; 
	\label{eq:omegat-decomp-spec}
\end{align}
here, we abuse the notation by taking $\beta_{t}^k \defn \langle \omega_{t}, y_k\rangle$ (cf.~\eqref{eqn:eta-decomposition})
(which satisfies $\ltwo{\beta_t} = \ltwo{\omega_t}$) and letting  $\xi_{t-1}$ denote the residual term.

In order to control $\omega_{t}$, we need to bound the size of $\xi_{t-1}$.
 Specializing the expression \eqref{eq:xi_bound} to the special choice of $\eta_t$ (i.e., the identity function), 
we obtain
\begin{align}
\label{eqn:spect-tenor}
	\ltwo{\xi_t} &= 
	\Big\langle \sum_{k = 1}^{t-1} \mu_t^k\psi_k, \xi_{t-1} \Big\rangle
	+
	\sum_{k = 1}^{t-1} \mu_{t}^k\bigg[\Big\langle \psi_k, \sum_{j = 1}^{t-1} \beta_{t-1}^j\psi_j\Big\rangle - \beta_{t-1}^k 
	- (\sqrt{2}-1)\beta_{t}^ky_k^{\top}W_k^{\prime}y_k - \sum_{i = 1}^{k - 1} \beta_t^ig_i^k - \sum_{i = k+1}^{t} \beta_{t}^ig_k^i\bigg].
\end{align}
Here, $\{y_k, W_k'\}$ have been defined in expression~\eqref{eqn:spect-seq-1}, whereas $\mu_t$ is a unit vector in $\mathcal{S}^{t-2}.$ 
We then control each term in \eqref{eqn:spect-tenor} separately. 
First, observe that with probability at least $1 - O(n^{-11})$, 
\begin{align}
\label{eqn:trio}
	\Big|\Big\langle \sum_{k = 1}^{t-1} \mu_t^k\psi_k, \xi_{t-1}\Big\rangle\Big|
	\leq 
	\Big\|\sum_{k = 1}^{t-1} \mu_t^k\psi_k\Big\|_2 \cdot \ltwo{\xi_{t-1}}
	\leq 
	\Big(1+O\Big(\sqrt{\frac{t\log n}{n}}\Big)\Big)\ltwo{\xi_{t-1}}
\end{align}
holds for every $t \in [n]$, where the last inequality follows from \eqref{eqn:spect-brahms}. 
In view of Lemma~\ref{lem:concentration},  with probability at least $1 - O(n^{-11})$ one has
\begin{align}
\label{eqn:spect-baritone}
	\bigg|\sum_{k = 1}^{t-1} (\sqrt{2}-1) \mu_{t}^k\Big[\beta_{t}^ky_k^{\top}W_k^{\prime}y_k + \sum_{i = 1}^{k - 1} \beta_t^ig_i^k + \sum_{i = k+1}^{t} \beta_{t}^ig_k^i\Big]\bigg|
	\lesssim
	\sqrt{\frac{t\log n}{n}}\ltwo{\beta_t}.
\end{align}
In addition, if we write matrix $\Psi \defn [\psi_1,\ldots,\psi_{t-1}] \in \real^{n\times (t-1)}$, then property~\eqref{eqn:simple-rm} and $\|u_t\|_2=1$ give  
\begin{align}
	\notag \bigg|\Big\langle \sum_{k = 1}^{t-1} \mu_t^k\psi_k, \sum_{j = 1}^{t-1} \beta_t^j\psi_j\Big\rangle - \sum_{k = 1}^{t-1} \mu_t^k\beta_t^k \bigg|
	&= 
	\Big|\mu_t^\top\Psi^\top \Psi \beta_t - \mu_t^\top I_{t-1} \beta_t \Big| \\
	&\leq 
	\ltwo{\mu_t}\ltwo{\beta_t} \big\|\Psi^\top \Psi  - I_{t-1} \big\|
	\lesssim 
	\sqrt{\frac{t\log n}{n}}\ltwo{\beta_t} \label{eqn:spect-soprano}
\end{align}
holds with probability at least $1-O(n^{-11})$. 
Taking the decomposition~\eqref{eqn:spect-tenor} collectively with \eqref{eqn:trio}, \eqref{eqn:spect-baritone} and \eqref{eqn:spect-soprano} 
and using $\|\omega_{t-1}\|_2 = \|\beta_{t-1}\|_2$, 
we arrive at 
\begin{align}
\label{eqn:spect-xi-to-beta}
	\ltwo{\xi_t} \leq \Big(1+ C_3 \sqrt{\frac{t\log n}{n}}\Big)\ltwo{\xi_{t-1}}
	+
	C_3 \sqrt{\frac{t\log n}{n}}  \ltwo{\omega_t} 
\end{align}
for some large enough constant $C_3>0$.

Additionally, invoke \eqref{eq:omegat-decomp-spec} to obtain 
\begin{align}
\label{eqn:spect-vt-ni}
\|\omega_t\|_2 
\notag \leq \Big\|\sum_{k = 1}^{t-1} \beta_{t-1}^k\psi_k\Big\|_2 + \|\xi_{t-1}\|_2 
&\stackrel{(\textrm{i})}{\leq} 
\Big(1 + O\Big(\sqrt{\frac{t\log n}{n}}\Big)\Big) \|\beta_{t-1}\|_2 + \|\xi_{t-1}\|_2 \\
&\leq \Big(1 + C_3 \sqrt{\frac{t\log n}{n}}\Big) \|\omega_{t-1}\|_2 + \|\xi_{t-1}\|_2 ,
\end{align}
provided that the constant $C_3>0$ is large enough. 
Here, (i) comes from \eqref{eqn:spect-brahms}, 
and we remind the readers that $\|\omega_{t-1}\|_2 = \|\beta_{t-1}\|_2$ and $\ltwo{\omega_1} = 1$.

Clearly, 
the inequalities \eqref{eqn:spect-xi-to-beta} and \eqref{eqn:spect-vt-ni} taken together
lead  to a recurrence relation involving $\ltwo{\xi_t}$ and $\ltwo{\omega_t}$. 
Based on this, we claim that for all $t$ obeying $\frac{t^{5}\log n}{n}=o(1)$, one has
\begin{equation}
\|\xi_{t}\|_{2}\leq C_{5}\sqrt{\frac{t^{3}\log n}{n}}\qquad\text{and}\qquad\|\omega_{t}\|_{2}\leq1+C_{5}\sqrt{\frac{t^{5}\log n}{n}}
	\label{eqn:spec-xi}
\end{equation}
for some universal constant $C_{5}=2C_{3}$. Clearly, (\ref{eqn:spec-xi}) is satisfied
when $t=1$, given that $\xi_1 = 0$ (as $\omega_{2} = W\vstar$) and $\ltwo{\omega_1} = \ltwo{\vstar} = 1$. 
Suppose now that (\ref{eqn:spec-xi}) is valid for
the $t$-th iteration, then we can deduce that
\begin{align}
\|\omega_{t+1}\|_{2} & \leq\Big(1+C_{3}\sqrt{\frac{(t+1)\log n}{n}}\Big)\|\omega_{t}\|_{2}+\|\xi_{t}\|_{2} \notag\\
 & \leq\Big(1+C_{3}\sqrt{\frac{(t+1)\log n}{n}}\Big)\left(1+C_{5}\sqrt{\frac{t^{5}\log n}{n}}\right)+C_{5}\sqrt{\frac{t^{3}\log n}{n}} \notag\\
 & =1+C_{5}\sqrt{\frac{(t+1)\log n}{n}}\left(\frac{C_{3}}{C_{5}}+t^{2}+C_{3}\sqrt{\frac{t^{5}\log n}{n}}+t\right) \notag\\
	& \leq1+C_{5}\sqrt{\frac{(t+1)\log n}{n}}\left(t+1\right)^{2}=1+C_{5}\sqrt{\frac{(t+1)^{5}\log n}{n}}; \label{eq:omega-t-recursion}\\
\|\xi_{t+1}\|_{2} & \leq\Big(1+C_{3}\sqrt{\frac{(t+1)\log n}{n}}\Big)\|\xi_{t}\|_{2}+C_{3}\sqrt{\frac{(t+1)\log n}{n}}\,\|\omega_{t+1}\|_{2} \notag\\
 & \leq C_{5}\Big(1+C_{3}\sqrt{\frac{(t+1)\log n}{n}}\Big)\sqrt{\frac{t^{3}\log n}{n}}+C_{3}\sqrt{\frac{(t+1)\log n}{n}}\left(1+C_{5}\sqrt{\frac{t^{5}\log n}{n}}\right) \notag\\
 & \leq C_{5}\sqrt{\frac{(t+1)\log n}{n}}\left(t+\frac{C_{3}}{C_{5}}+2C_{3}\sqrt{\frac{t^{5}\log n}{n}}\right) \notag\\
 & \leq C_{5}\sqrt{\frac{(t+1)\log n}{n}}(t+1)=C_{5}\sqrt{\frac{(t+1)^{3}\log n}{n}}; 
\end{align}
here, the last lines in both of the above bounds hold true since $C_{3}/C_{5}=1/2$
and $\frac{t^{5}\log n}{n}=o(1)$. This in turn justifies the validity
of the claim (\ref{eqn:spec-xi}) for the $(t+1)$-th iteration. Hence, by induction, we have established (\ref{eqn:spec-xi})
for all $t$ obeying $\frac{t^{5}\log n}{n}=o(1)$.

Armed with \eqref{eqn:spec-xi}, we can invoke \eqref{eq:omegat-decomp-spec} again to derive
\begin{align*}
\|\omega_t\|_2 
\notag \geq \Big\|\sum_{k = 1}^{t-1} \beta_{t-1}^k\psi_k\Big\|_2 - \|\xi_{t-1}\|_2 
&\stackrel{(\textrm{i})}{\geq} 
\Big(1 - O\Big(\sqrt{\frac{t\log n}{n}}\Big)\Big) \|\beta_{t-1}\|_2 - \|\xi_{t-1}\|_2 \\
&\geq \Big(1 - C_3 \sqrt{\frac{t\log n}{n}}\Big) \|\omega_{t-1}\|_2 - C_{5}\sqrt{\frac{t^{3}\log n}{n}} 
\end{align*}
for some large enough constant $C_3>0$. Repeat the argument in \eqref{eq:omega-t-recursion} to yield
\[
	\|\omega_{t}\|_{2}\geq 1 -C_{5}\sqrt{\frac{t^{5}\log n}{n}} .
\]
This taken collectively with \eqref{eqn:spec-xi} finishes the proof of inequality \eqref{eqn:spec-finale-3}.

\paragraph{Proof of inequality~\eqref{eqn:spec-finale-2}.}

To streamline the presentation of our proof, let us first make note of the following result, 
the proof of which is deferred to the end of this section: 
%
\begin{align}
\label{eqn:spect-beta-ni}
	\big\| \big[\beta_t^{1},\beta_t^{2},\cdots, \beta_t^{t-1} \big] \big\|_2 &\lesssim \sqrt{\frac{t^5\log n}{n}}.
\end{align}
%
With this result, \eqref{eqn:spec-xi} and \eqref{eqn:spec-finale-3} in mind, we are ready to prove \eqref{eqn:spec-finale-2}. 
First, it follows from \eqref{eq:omegat-decomp-spec} that 
\begin{align*}
\big\|\omega_{t}-\psi_{t-1}\big\|_{2} & \leq\big\|(\beta_{t-1}^{t-1}-1)\psi_{t-1}\big\|_{2}+\Big\|\sum_{k=1}^{t-2}\beta_{t-1}^{k}\psi_{k}\Big\|_{2}+\|\xi_{t-1}\|_{2}\\
 & \leq\bigg(1+O\Big(\sqrt{\frac{t\log n}{n}}\Big)\bigg)\Big(\big|\beta_{t-1}^{t-1}-1\big|+\big\|\big[\beta_{t-1}^{1},\cdots,\beta_{t-1}^{t-2}\big]\big\|_{2}+\|\xi_{t-1}\|_{2}\Big),\\
 & \leq\bigg(1+O\Big(\sqrt{\frac{t\log n}{n}}\Big)\bigg)\Big(\big|\|\omega_{t-1}\|_{2}-1\big|+2\big\|\big[\beta_{t-1}^{1},\cdots,\beta_{t-1}^{t-2}\big]\big\|_{2}+\|\xi_{t-1}\|_{2}\Big),
\end{align*}
where the second line results from the properties \eqref{eqn:brahms} and \eqref{eqn:simple-rm},
and the last line is valid since 
\[
\beta_{t-1}^{t-1}=\big\langle\omega_{t-1},y_{t-1}\big\rangle=\frac{\omega_{t-1}^{\top}(I-V_{t-2}V_{t-2}^{\top})\omega_{t-1}}{\|(I-V_{t-2}V_{t-2}^{\top})\omega_{t-1}\|_{2}}\geq0
\]
\[
\Longrightarrow\quad\big|\beta_{t-1}^{t-1}-1\big|=\big||\beta_{t-1}^{t-1}|-1\big|\leq\big|\|\beta_{t-1}\|_{2}-1\big|+\big\|\big[\beta_{t-1}^{1},\cdots,\beta_{t-1}^{t-2}\big]\big\|_{2}=\big|\|\omega_{t-1}\|_{2}-1\big|+\big\|\big[\beta_{t-1}^{1},\cdots,\beta_{t-1}^{t-2}\big]\big\|_{2}.
\]
Taking this collectively with \eqref{eqn:spec-xi},  \eqref{eqn:spect-beta-ni} and \eqref{eqn:spec-finale-3} yields the first part of the advertised bound \eqref{eqn:spec-finale-2}.

Regarding the second part of \eqref{eqn:spec-finale-2}, reorganizing the expression \eqref{eqn:spect-seq-1} of $y_{t}$ gives 
\begin{align*}
	y_t 
	&= \frac{\omega_t - V_{t-1}V_{t-1}^{\top}\omega_t}{\big\|\omega_t - V_{t-1}V_{t-1}^{\top}\omega_t\big\|_2}
	= \omega_t 
	+ \bigg(\frac{1-\big\|\omega_t - V_{t-1}V_{t-1}^{\top}\omega_t\big\|_2}{\big\|\omega_t - V_{t-1}V_{t-1}^{\top}\omega_t\big\|_2}\bigg)\omega_t 
	- \frac{V_{t-1}V_{t-1}^{\top}\omega_t}{\big\|\omega_t - V_{t-1}V_{t-1}^{\top}\omega_t\big\|_2}.
	%
\end{align*}
In view of relations \eqref{eqn:spect-beta-ni} and \eqref{eqn:spec-finale-3}, we can deduce that
\begin{align*}
	\big\|V_{t-1}V_{t-1}^{\top}\omega_t\big\|_2 &= \big\|V_{t-1}^{\top}\omega_t\big\|_2 = \big\|\big[\beta_{t}^{1},\cdots,\beta_{t}^{t-1}\big]\big\|_{2}
	\lesssim 
	\sqrt{\frac{t^5\log n}{n}}
	 \\
	\big\|\omega_t - V_{t-1}V_{t-1}^{\top}\omega_t\big\|_2 
	&\leq \|\omega_t\|_2 + \big\|V_{t-1}^{\top}\omega_t \big\|_2 
	= 1+ O\Big(\sqrt{\frac{t^5\log n}{n}}\Big) ,\\
	%
	\big\|\omega_t - V_{t-1}V_{t-1}^{\top}\omega_t\big\|_2 
	&\geq \|\omega_t\|_2 - \big\|V_{t-1}^{\top}\omega_t \big\|_2 
	\geq 1- O\Big(\sqrt{\frac{t^5\log n}{n}}\Big) .
\end{align*}
Taking the preceding bounds collectively and using \eqref{eqn:spec-finale-3} once again, we immediately reach
\begin{align*}
	\big\| y_t - \omega_t \big\|_2 \leq O\Big(\sqrt{\frac{t^5\log n}{n}}\Big) ,
\end{align*}
thus validating the second part of inequality~\eqref{eqn:spec-finale-2}.

\paragraph{Proof of inequality~\eqref{eq:sum-of-square-ci}.}

It follows from \eqref{eq:v-truncated-sum-bound135} and the triangle inequality that
\begin{align*}
\bigg\|\widehat{v}^{\star}-\sum_{i=1}^{\frac{C_{v}\log n}{\widetilde{\lambda}-1}}c_{i}y_{i}\bigg\|_{2} & \leq\bigg\|\widehat{v}^{\star}-\sum_{i=1}^{\frac{C_{v}\log n}{\widetilde{\lambda}-1}}c_{i}\omega_{i}\bigg\|_{2}+\bigg\|\sum_{i=1}^{\frac{C_{v}\log n}{\widetilde{\lambda}-1}}c_{i}(\omega_{i}-y_{i})\bigg\|_{2}\\
 & \lesssim\frac{1}{(\widetilde{\lambda}-1)n}+\bigg(\sum_{i=1}^{\frac{C_{v}\log n}{\widetilde{\lambda}-1}}|c_{i}|^{2}\bigg)^{\frac{1}{2}}\bigg(\sum_{i=1}^{\frac{C_{v}\log n}{\widetilde{\lambda}-1}}\big\|\omega_{i}-y_{i}\big\|_{2}^{2}\bigg)^{\frac{1}{2}}\\
 & \lesssim\frac{1}{(\widetilde{\lambda}-1)n}+\bigg(\sum_{i=1}^{\frac{C_{v}\log n}{\widetilde{\lambda}-1}}|c_{i}|^{2}\bigg)^{\frac{1}{2}}\bigg(\sum_{i=1}^{\frac{C_{v}\log n}{\widetilde{\lambda}-1}}\frac{i^{5}\log n}{n}\bigg)^{\frac{1}{2}}\\
 & \lesssim\frac{1}{(\lambda-1)n}+\bigg(\sum_{i=1}^{\frac{C_{v}\log n}{\widetilde{\lambda}-1}}|c_{i}|^{2}\bigg)^{\frac{1}{2}}
	\sqrt{\frac{\log^{7}n}{(\lambda-1)^{6}n}},
\end{align*}
where the penultimate inequality invokes \eqref{eqn:spec-finale-2}, and the last line results from \eqref{eq:minyu-lambda-tilde-minus-1-bound}. 
This combined with the orthonormality of $\{y_{i}\}$ implies that
\begin{align*}
\bigg(\sum_{i=1}^{\frac{C_{v}\log n}{\widetilde{\lambda}-1}}(c_{i})^{2}\bigg)^{1/2} & =\bigg\|\sum_{i=1}^{\frac{C_{v}\log n}{\widetilde{\lambda}-1}}c_{i}y_{i}\bigg\|_{2}\leq\big\|\widehat{v}^{\star}\big\|_{2}+\bigg\|\widehat{v}^{\star}-\sum_{i=1}^{\frac{C_{v}\log n}{\widetilde{\lambda}-1}}c_{i}y_{i}\bigg\|_{2}\leq1+O\Bigg(\bigg(\sum_{i=1}^{\frac{C_{v}\log n}{\widetilde{\lambda}-1}}|c_{i}|^{2}\bigg)^{\frac{1}{2}}\sqrt{\frac{\log^{7}n}{(\lambda-1)^{6}n}}\Bigg)\\
 & \leq1+\frac{1}{2}\bigg(\sum_{i=1}^{\frac{C_{v}\log n}{\widetilde{\lambda}-1}}|c_{i}|^{2}\bigg)^{\frac{1}{2}},
\end{align*}
where the last line is valid as long as $n(\lambda-1)^{6}/\log^{7}n$
is sufficiently large. Rearranging terms, we are left with $\sum_{i=1}^{\frac{C_{v}\log n}{\widetilde{\lambda}-1}}(c_{i})^{2}\leq4$
as claimed.

\paragraph{Proof of inequality \eqref{eqn:spect-beta-ni}.}

Finally, we finish the proof by establishing inequality~\eqref{eqn:spect-beta-ni}. 
From the definition of $\beta_t$ and \eqref{eq:omegat-decomp-spec}, one can derive a recursive relation as follows: 
\begin{align}
\notag\big\|\big[\beta_{t}^{1},\ldots,\beta_{t}^{t-1}\big]\big\|_{2} & =\|V_{t-1}^{\top}\omega_{t}\|_{2}=\bigg\| V_{t-1}^{\top}\Big(\sum_{k=1}^{t-1}\beta_{t-1}^{k}\psi_{k}+\xi_{t-1}\Big)\bigg\|_{2}\\
\notag & \le\Big\|\sum_{k=1}^{t-2}\beta_{t-1}^{k}\psi_{k}\Big\|_{2}+|\beta_{t-1}^{t-1}|\cdot\big\| V_{t-1}^{\top}\psi_{t-1}\big\|_{2}+\big\|\xi_{t-1}\big\|_{2}\\
\notag
 & \le\Big(1+C_{4}\sqrt{\frac{t\log n}{n}}\Big)\big\|\big[\beta_{t-1}^{1},\ldots,\beta_{t-1}^{t-2}\big]\big\|_{2}+\|\omega_{t-1}\|_{2} \cdot \big\| V_{t-1}^{\top}\psi_{t-1}\big\|_{2}+\big\|\xi_{t-1}\big\|_{2}\\
 & \le\Big(1+C_{4}\sqrt{\frac{t\log n}{n}}\Big)\big\|\big[\beta_{t-1}^{1},\ldots,\beta_{t-1}^{t-2}\big]\big\|_{2}+C_4\|\omega_{t-1}\|_{2}\sqrt{\frac{t\log n}{n}}+\big\|\xi_{t-1}\big\|_{2}
\label{eqn:intermission}
\end{align}
for some large enough constant $C_{4}>0$. 
Here, the penultimate inequality uses \eqref{eqn:spect-brahms} and the fact $|\beta_{t-1}^{t-1}|\leq \|\beta_{t-1}\|_2=\|\omega_{t-1}\|_2$,  
while the last inequality would be guaranteed if we could establish the following result:  
\begin{align}
\label{eqn:second-simple-rm}
	\|V_{t-1}^{\top}\psi_{t-1}\|_2 
	\lesssim \sqrt{\frac{t\log n}{n}}. 
\end{align}
We shall assume the validity of \eqref{eqn:second-simple-rm} for the moment, and return to prove it shortly. 
Taking \eqref{eqn:intermission} together with \eqref{eqn:spec-xi} yields
\begin{align}
\big\|\big[\beta_{t}^{1},\ldots,\beta_{t}^{t-1}\big]\big\|_{2} & \le\Big(1+C_{4}\sqrt{\frac{t\log n}{n}}\Big)\big\|\big[\beta_{t-1}^{1},\ldots,\beta_{t-1}^{t-2}\big]\big\|_{2}+C_{6}\sqrt{\frac{t^{3}\log n}{n}} 
	\label{eqn:intermission-2}
\end{align}
for some sufficiently large constant $C_6>0$.

We then claim that for all $t$ obeying $\frac{t^{3}\log n}{n}=o(1)$, 
\begin{equation}
	\big\|\big[\beta_{t}^{1},\ldots,\beta_{t}^{t-1}\big]\big\|_{2} \leq C_7\sqrt{\frac{t^5 \log n}{n}}
	\label{eq:claim-beta-ub}	
\end{equation}
holds for some large enough constant $C_7>0$. 
Regarding the base case,  we observe that
\begin{align}
\label{eqn:spect-inti-beta}
	|\beta_2^1| = \big| \inprod{\omega_2}{y_1} \big| =  \big| v^{\star\top} W \vstar \big| \leq C_7 \sqrt{\frac{\log n}{n}} 
\end{align}
with probability at least $1 - O(n^{-12})$, provided that $C_7>0$ is large enough.  
Assuming that \eqref{eq:claim-beta-ub} is valid for the $(t-1)$-th iteration, we further have
\begin{align*}
\notag\big\|\big[\beta_{t}^{1},\ldots,\beta_{t}^{t-1}\big]\big\|_{2} & \le C_{7}\Big(1+C_{4}\sqrt{\frac{t\log n}{n}}\Big)\sqrt{\frac{(t-1)^{5}\log n}{n}}+C_{6}\sqrt{\frac{t^{3}\log n}{n}}\\
 & =C_{7}\sqrt{\frac{t^{3}\log n}{n}}\left\{ \Big(1+C_{4}\sqrt{\frac{t\log n}{n}}\Big)\left(t-1\right)+\frac{C_{6}}{C_{7}}\right\} \\
 & \leq C_{7}\sqrt{\frac{t^{3}\log n}{n}}\cdot\left\{ t-1+C_{4}\sqrt{\frac{t^{3}\log n}{n}}+\frac{C_{6}}{C_{7}}\right\} \leq C_{7}\sqrt{\frac{t^{5}\log n}{n}},
\end{align*}
where the last inequality holds true as long as $C_{7}\geq2C_{6}$
and $\frac{t^{3}\log n}{n}=o(1)$. 
This justifies the claim \eqref{eq:claim-beta-ub} for the $t$-th iteration. The standard induction argument then establishes \eqref{eq:claim-beta-ub} for all $t$ obeying $\frac{t^{3}\log n}{n}=o(1)$.

We now come back to prove \eqref{eqn:second-simple-rm}. Towards this, we first note that: by construction, $\psi_{t-1}$ is independent of $V_{t-1}$. To justify this, recall that it has been established in the last paragraph of Section~\ref{sec:pf-distribution} that:
$\psi_{t-1}$ follows a Gaussian distribution $\mathcal{N}(0,\frac{1}{n}I_n)$ no matter what value the sequence $\{y_k\}_{1\leq k \leq t-1}$ takes;
therefore, in view of the definition of statistical independence, $\psi_{t-1}$ is independent of $\{y_k\}_{1\leq k \leq t-1}$ and hence $V_{t-1}$ (as $V_{t-1}$ is obtained by simply concatenating $y_1,\ldots,y_{t-1}$).
Therefore, $V_{t-1}^{\top}\psi_{t-1}$ is essentially $\mathcal{N}(0,\frac{1}{n}I_{t-1})$, 
and hence standard Gaussian concentration results \citep[Chapter 4.4]{vershynin2018high} imply that
\[
	\mathbb{P}\Big( \big\| V_{t-1}^{\top}\psi_{t-1} \big\|_2 \geq 5\sqrt{\frac{t\log n}{n}} \Big) \leq O(n^{-11})
\]
as claimed. This concludes the proof.

\subsection{Proof of Lemma~\ref{lem:alt-x1}}
\label{sec:proof-eqn:alt-x1}
Repeating the proof of Lemma~\ref{lem:linear-combination-power-method-AMP}, 
we can show that each $W^s\widetilde{v}$ is a linear combination of $\{u_{s+1}, \ldots, u_{2s+1}\}.$ 
Taking this together with the decomposition~\eqref{eqn:x1-decmp-violin} and Lemma~\ref{lem:linear-combination-power-method-AMP} 
reveals that  $x_{1}$ can be expressed as 
\begin{align}
\label{eqn:basic-exp-x1}
	x_1 = \sum_{i = 1}^{2s+1} b_i\widehat{y}_i, \qquad \text{with } b_i = \inprod{x_1}{\widehat{y}_i},
\end{align}
given that $\{\widehat{y}_i\}_{i=1}^{2s+1}$ are orthonormal and span the subspace containing $\{\omega_i\}_{i=1}^{s}$ and $\{u_i\}_{i=s+1}^{2s+1}.$
%

Next, we move on to show that $\ltwo{[b_{s+1},\ldots, b_{2s+1}]}$ is small. 
More specifically, recall from  Lemma~\ref{lem:linear-combination-power-method-AMP} that 
\[
\sum_{i=0}^{s-1}a_{i}W^{i}v^{\star}\in\mathsf{span}\{\omega_{1},\cdots,\omega_{s}\}=\mathsf{span}\big\{\widehat{y}_{1},\cdots,\widehat{y}_{s}\big\},
\]
and hence by virtue of \eqref{eqn:x1-decmp-violin}, 
\begin{align}
\notag \big\|[b_{s+1},\cdots,b_{2s+1}]\big\|_{2} & =\Big\| x_{1}-\sum_{i=1}^{s}b_{i}\widehat{y}_{i}\Big\|_{2}\overset{\mathrm{(i)}}{\leq}\Big\| x_{1}-\sum_{i=0}^{s-1}a_{i}W^{i}v^{\star}\Big\|_{2}=\big\| a_{s}W^{s}\widetilde{v}\big\|_{2}\overset{\mathrm{(ii)}}{\lesssim}\frac{\|W\|^{s}n^{11.5}}{\lambda_{\max}^{s}}\\
 & \asymp\bigg(1-\frac{\lambda_{\max}-\|W\|}{\lambda_{\max}}\bigg)^{s}n^{11.5}\lesssim\frac{1}{n},
\end{align}
where (i) follows since $\sum_{i=1}^{s}b_{i}\widehat{y}_{i}$ is the
Euclidean projection of $x_{1}$ onto $\mathsf{span}\big\{\widehat{y}_{1},\cdots,\widehat{y}_{s}\big\}$ while $\sum_{i=0}^{s-1}a_{i}W^{i}v^{\star} \in \mathsf{span}\big\{\widehat{y}_{1},\cdots,\widehat{y}_{s}\big\}$; 
(ii) makes use of (\ref{eqn:power-error}), and the last inequality
invokes (\ref{eq:lambda-1-gap-non-asymptotic-UCB}) and (\ref{eq:lambda-1-non-asymptotic-UCB}) 
and is valid if $s\geq\frac{C_{v}\lambda^2 \log n}{(\lambda-1)^{2}}$ for some
sufficiently large constant $C_{v}>0$. 
Moreover, putting expressions~\eqref{eqn:basic-exp-x1} and \eqref{eqn:spectral-expansion} together yields
\begin{align}
\notag \big\|[b_{1},\cdots b_{s}]-[c_{1},\cdots,c_{s}]\big\|_{2} & =\Big\|\sum_{i=1}^{s}(b_{i}-c_{i})\widehat{y}_{i}\Big\|_{2}\leq\Big\|\sum_{i=1}^{s}(b_{i}-c_{i})\widehat{y}_{i}+\sum_{i=s+1}^{2s+1}b_{i}\widehat{y}_{i}\Big\|_{2}\\
\notag & \leq\big\| x_{1}-\widehat{v}^{\star}\big\|_{2}+\Big\|\widehat{v}^{\star}-\sum_{i=1}^{s}c_{i}\widehat{y}_{i}\Big\|_{2}\\
 & \lesssim\frac{1}{n^{12}}+\frac{\log^{3.5}n}{\sqrt{(\lambda-1)^{6}n}}\asymp\frac{\log^{3.5}n}{\sqrt{(\lambda-1)^{6}n}}
\end{align}
with probability at least $1-O(n^{-11})$. 
In light of the above two relations, we can further derive 
\begin{align*}
	\Big\|\sum_{i = 1}^{2\power+1} b_i\widehat{\psi}_i - \sum_{i = 1}^{\power} c_i\psi_i\Big\|_2 
	& \leq \Big\|\sum_{i = 1}^{s} b_i\widehat{\psi}_i - \sum_{i = 1}^{\power} c_i\psi_i\Big\|_2 
	+ \Big\|\sum_{i = s+1}^{2s+1} b_i\widehat{\psi}_i \Big\|_2 \\
	& = \Big\|\sum_{i = 1}^{s} b_i{\psi}_i - \sum_{i = 1}^{\power} c_i\psi_i\Big\|_2 
	+ \Big\|\sum_{i = s+1}^{2s+1} b_i\widehat{\psi}_i \Big\|_2 \\	
	&\leq \Big(1 + O\Big(\sqrt{\frac{t\log n}{n}}\Big)\Big) \Big(\big\|[b_{1},\cdots b_{s}]-[c_{1},\cdots,c_{s}]\big\|_{2} + \big\|[b_{s+1},\ldots,b_{2s+1}]\big\|_{2}\Big)  	\\
	&\lesssim \frac{\log^{3.5} n}{\sqrt{(\lambda - 1)^6n}} 
\end{align*}
with probability at least $1 - O(n^{-11})$, 
where the penultimate line applies the concentration result~\eqref{eqn:spect-brahms}. 
Hence, taking this collectively with \eqref{eqn:power-error} and Lemma~\ref{lem:spec}, we can demonstrate that
\begin{align}
\Bigg\| x_{1}-c_{1}v^{\star}-\frac{1}{\widetilde{\lambda}}\sum_{i=1}^{2s+1}b_{i}\widehat{\psi}_{i}\Bigg\|_{2} & \leq\Bigg\| x_{1}-c_{1}v^{\star}-\frac{1}{\widetilde{\lambda}}\sum_{i=1}^{\power}c_{i}\psi_{i}\Bigg\|_{2}+\frac{1}{\widetilde{\lambda}}\Big\|\sum_{i=1}^{2\power+1}b_{i}\widehat{\psi}_{i}-\sum_{i=1}^{\power}c_{i}\psi_{i}\Big\|_{2} \notag\\
 & \leq\Bigg\|\widehat{v}^{\star}-c_{1}v^{\star}-\frac{1}{\widetilde{\lambda}}\sum_{i=1}^{\power}c_{i}\psi_{i}\Bigg\|_{2}+\big\| x_{1}-\widehat{v}^{\star}\big\|_{2}+O\bigg(\frac{\log^{3.5}n}{\sqrt{(\lambda-1)^{6}n}}\bigg) \notag\\
 & \lesssim\frac{\log^{3.5}n}{\sqrt{(\lambda-1)^{6}n}}+\frac{1}{n^{12}}+\frac{\log^{3.5}n}{\sqrt{(\lambda-1)^{6}n}}\asymp\frac{\log^{3.5}n}{\sqrt{(\lambda-1)^{6}n}} .
	\label{eq:x1-approx-UB}
\end{align}

Finally, it results from \eqref{eq:lambda-max-correct-bound} and
\eqref{eqn:spect-brahms} that 
\begin{align*}
\bigg\|\frac{1}{\lambda}\sum_{i=1}^{2s+1}b_{i}\widehat{\psi}_{i}-\frac{1}{\widetilde{\lambda}}\sum_{i=1}^{2s+1}b_{i}\widehat{\psi}_{i}\bigg\|_{2} & \lesssim\frac{|\lambda-\widetilde{\lambda}|}{\lambda^{2}}\bigg\|\sum_{i=1}^{2s+1}b_{i}\widehat{\psi}_{i}\bigg\|\lesssim\bigg(1+O\Big(\sqrt{\frac{s}{n}}\Big)\bigg)|\lambda-\widetilde{\lambda}|\cdot\big\|[b_{1},\cdots,b_{2s+1}]\big\|_{2}\\
 & \lesssim\sqrt{\frac{\log n}{n(\lambda-1)^{7}}},
\end{align*}
where we also use the fact that $\big\|[b_{1},\cdots,b_{2s+1}]\big\|_{2}\asymp1$
(a direct consequence of \eqref{eqn:basic-exp-x1} and the orthonormality
of $\{\widehat{y}_{i}\}$). 
This  together with \eqref{eq:x1-approx-UB} and the triangle inequality immediately concludes the proof.

\section{$\mathbb{Z}_2$ synchronization: Proof of Theorem~\ref{thm:Z2}}
\label{sec:pf-thm-Z2}

With the denoising functions selected as in \eqref{eqn:eta-z2-new}, we first point out that 
\begin{align}
	\|\beta_{t}\|_2 = \ltwo{\eta_{t}(x_{t})} = 1 ,
	\qquad t\geq 1 
\end{align}
throughout the execution of AMP. 
This basic fact helps simplify the analysis, 
as there is no need to control the related quantity $\Delta_{\beta, t}$ (see \eqref{eqn:delta-beta-general}) given that $\|\beta_{t-1}\|_2$ is fixed.
As a result, this section focuses attention on characterizing the dynamics of $\alpha_{t}.$

\paragraph{Induction hypotheses.}

The proof of Theorem~\ref{thm:Z2} is built upon Theorem~\ref{thm:recursion-spectral} as well as the analysis framework laid out in  
Theorem~\ref{thm:main} (or Corollary~\ref{cor:recursion-spectral}). 
The proof is inductive in nature; 
more specifically, we aim to show, by induction, that for every $t$ obeying~\eqref{eqn:z2-decomposition}, 
the AMP iterates $\{x_t\}$ satisfy the desired decomposition \eqref{eqn:z2-decomposition} in Theorem~\ref{thm:Z2} 
while satisfying the following properties: 
\begin{subequations}
\label{Z2-induction}
\begin{align}
\big(1+o(1)\big)\lambda &\geq \alpha_t \ge \big(1+o(1)\big)\sqrt{\lambda^2 - 1} 
	\label{eq:Z2-induction-alphat}\\
	\| \xi_{t-1}\|_2 & \leq C_1	\sqrt{\frac{(t+s)\log n}{(\lambda - 1)^3 n}}  + 
	C_1 \lt( 1 - \frac{1}{40}(\lambda - 1) \rt) ^{t-1}  \frac{\log^{3.5} n}{\sqrt{(\lambda - 1)^9 n}}  \eqqcolon S_t
\label{eqn:Z2-induction-st}
\end{align}
for some large enough constant $C_1>0$. 
\end{subequations}
\red{Given that $s\asymp \frac{\log n}{(\lambda-1)^2}$, we find it helpful to note    
\begin{align}
\label{eqn:crude-st}
\notag	S_t 
	\leq 
	\underset{\eqqcolon \,\widetilde{S}_t}{\underbrace{ C_1	\sqrt{\frac{(t+s)\log n}{(\lambda - 1)^3 n}}  + C_1   \frac{\log^{3.5} n}{\sqrt{(\lambda - 1)^9 n}} }} 
	& \leq C_1	\sqrt{\frac{t\log n}{(\lambda - 1)^3 n}} + C_1 \sqrt{\frac{s\log n}{(\lambda - 1)^3 n}}
	+ C_1 \frac{\log^{3.5} n}{\sqrt{(\lambda - 1)^9 n}} \\
	& \lesssim \sqrt{\frac{t\log n}{(\lambda-1)^3n}} + \sqrt{\frac{\log^7 n}{(\lambda-1)^9 n}}, 
\end{align}
where the first line uses the basic inequality that $\sqrt{a+b} \leq \sqrt{a}+\sqrt{b}$ for any positive numbers $a, b$; for the second line, we make use of the inequality $\sqrt{\frac{s\log n}{(\lambda-1)^3n}} \lesssim \sqrt{\frac{\log^7 n}{(\lambda-1)^9n}}$ that holds under the condition $s \asymp \frac{\log n}{(\lambda - 1)^2}.$ 
}
%
%

We first verify these hypotheses for the base case. 
In view of Theorem~\ref{thm:recursion-spectral}, the spectral initialization $x_1$ (defined in \eqref{eqn:Z2-initialization})
admits the decomposition~\eqref{eqn:z2-decomposition} and satisfies 
\begin{align}
	\label{eq:init-alpha-beta-Z2}
	\alpha_1 = \sqrt{\lambda^2 - 1} , \qquad \ltwo{\beta_{0}} = 1, \qquad
	\|\xi_{0}\|_2 \lesssim \frac{\log^{3.5} n}{\sqrt{(\lambda - 1)^9n}}.
\end{align}
This validates the induction hypotheses \eqref{Z2-induction} for the base case with $t = 1$. 
In order to carry out the induction argument, 
we shall --- throughout the rest of the proof --- assume that the induction hypotheses \eqref{Z2-induction} hold true for every iteration $k\leq t$,
and attempt to show their validity for the $(t+1)$-th iteration.

\paragraph{Organization of the proof. }

The proof is organized as follows.
Section~\ref{sec:prelim-z2} collects a couple of preliminary facts (e.g., basic concentration inequalities, derivatives of the denoising function, and tight estimates of $\pi_t$ and $\gamma_t$) that will be used throughout the induction argument. 
Section~\ref{sec:z2-key} develops upper bounds on several key quantities (e.g., $A_t, B_t, D_t$) that underlie our analysis framework in Theorem~\ref{thm:main} and Corollary~\ref{cor:recursion-spectral}. 
The main recursion is established in Section~\ref{sec:xi-z2}; specifically, Section~\ref{sec:z2-xi-t} is devoted to establishing the bound for $\ltwo{\xi_t}$, Section~\ref{sec:z2-delta-alpha} studies the size of $\Delta_{\alpha, t}$, while Section~\ref{sec:main-recursion-z2} is dedicated to the analysis of $\alpha_{t}$.


\subsection{Preliminary facts}
\label{sec:prelim-z2}

Before embarking on the main proof of Theorem~\ref{thm:Z2}, let us gather some preliminary facts  that shall be used multiple times throughout the proof.

\subsubsection{Basic concentration results}
\label{sec:basic-concentration-buble}

We begin by stating some concentration results that follow directly from the results in Section~\ref{sec:Gaussian-concentration}. 
Recall that the $\phi_k$'s are i.i.d.~drawn from $\phi_{k} \stackrel{\text{i.i.d}}{\sim} \mathcal{N}(0, \frac{1}{n}I_n)$, 
and for every $x\in \real^n$ we denote by $|x|_{(i)}$ its $i$-th largest entry in magnitude.  
In the statement of Lemma~\ref{lem:Gauss}, we mention some convex set $\mathcal{E}$, which we shall select as follows. 
For any fixed $1\leq t < n-2s$ and $1\leq \tau\leq n$, let us define the following set:  
\begin{align}
	&\notag \mathcal{E}_\tau \defn 
	\left\{
	\{\phi_k\}:
	\max_{-2s\leq k\leq t-1} \|\phi_k\|_2  < 1+ C_5\sqrt{\frac{\log \frac{n}{\delta}}{n}}\right\} 
	\bigcap \left\{\{\phi_k\}: \sup_{a \in \mathcal{S}^{2s+t-1}} \Big\|\sum_{k = -2s}^{t-1} a_k\phi_k\Big\|_2 < 1 + C_5\sqrt{\frac{(t+s)\log \frac{n}{\delta} }{n}} \right\} \\
	&\hspace{2cm} \bigcap 
	\left\{\{\phi_k\}:  \sup_{a = [a_k]_{-2s\leq k< t} \in \mathcal{S}^{2s+t-1}}  \sum_{i = 1}^{\tau} \Big|\sum_{k = -2s}^{t-1} a_k\phi_k\Big|_{(i)}^2 
	<  \frac{C_5(t + s+ \tau)\log \frac{n}{\delta}}{n} \right\}  \label{eq:eps-set}
\end{align}
for some large enough constant $C_5>0$. 
It is easily seen that $\mathcal{E}_{\tau}$ is a convex set with respect to $(\phi_{-2s},\ldots,\phi_{t-1})$. 
Additionally, Lemma~\ref{lem:brahms-lemma} together with the union bound reveals that $\{\mathcal{E}_\tau\}$ is a set of high-probability events: 
\begin{align}
\label{eqn:eps-interset}
	\mprob(\{\phi_k\} \in \mathcal{E} ) \geq 1 -  \delta, \qquad \text{with }
	\mathcal{E} \defn \bigcap_{\tau = 1}^n \mathcal{E}_{\tau}  
\end{align}

In addition, Lemma~\ref{lem:Gauss} and Corollary~\ref{cor:Gauss} entail 
bounding the expected difference between a function $f$ and its projection onto $\mathcal{E}$ (see \eqref{eqn:gauss-lipschitz-B-proj}). 
Here, we state a simple result that leads to a useful bound in this regard. 
Specifically, denote $\Phi \defn \sqrt{n}(\phi_{-2s},\ldots,\phi_{t-1})$, 
and consider any given function $f: \real^{n\times (2s+t)}\to \real$ obeying 
\begin{equation}
	|f(\Phi)| \lesssim  n^{100} \Big(\max_k \|\phi_k\|_2\Big)^{100} .
	\label{eq:f-Phi-poly}
\end{equation}
Denoting by $\mathcal{P}_{\mathcal{E}}(\cdot)$ the Euclidean projection onto the set 
$\mathcal{E}$ and taking $\delta \asymp n^{-300}$, 
we assert that 
\begin{align}
\label{eqn:brahms-conc}
	 \myE\big[\big|f(\Phi) - f(\mathcal{P}_{\mathcal{E}}(\Phi)) \big|\big] 
	\lesssim n^{-100}.
\end{align}
In light of this result, we shall choose the set $\mathcal{E}$ with $\delta \asymp n^{-300}$ throughout the rest of this section.

\begin{proof}[Proof of inequality~\eqref{eqn:brahms-conc}]
We divide into two cases depending on the value of $\max_k \|\phi_k\|_2$, namely, 
\begin{align*}
	\mathbb{E}\big[ \big|f(\Phi) - f(\mathcal{P}_{\mathcal{E}}(\Phi)) \big|\big] 
&
= \mathbb{E}\Big[\big|f(\Phi) - f(\mathcal{P}_{\mathcal{E}}(\Phi)) \big| \ind\lt(\mathcal{E}^{\mathrm{c}}\rt)\Big]
\lesssim \mathbb{E}\lt[n^{100} \Big(\max_k \|\phi_k\|_2\Big)^{100} \ind\lt(\mathcal{E}^{\mathrm{c}}\rt)\rt] \\
&\lesssim \mathbb{E}\lt[n^{100} \Big(\max_k \|\phi_k\|_2\Big)^{100} \ind\lt(\mathcal{E}^{\mathrm{c}}\rt)\ind\bigg(\max_k \|\phi_k\|_2 \le 1+C_5\sqrt{\frac{\log \frac{n}{\delta}}{n}}\bigg)\rt] \\
&\qquad+ \mathbb{E}\lt[n^{100} \Big(\max_k \|\phi_k\|_2\Big)^{100} \ind\lt(\mathcal{E}^{\mathrm{c}}\rt) \ind\bigg(\max_k \|\phi_k\|_2 > 1+C_5\sqrt{\frac{\log \frac{n}{\delta}}{n}}\bigg)\rt].
\end{align*} 
First, it is easily seen from \eqref{eqn:eps-interset} that 
\begin{align*}
	\myE\lt[n^{100} \Big(\max_k \|\phi_k\|_2\Big)^{100} \ind\lt(\mathcal{E}^{\mathrm{c}}\rt)\ind\Big(\max_k \|\phi_k\|_2 \le 1+C_5\sqrt{\frac{\log \frac{n}{\delta}}{n}}\Big)\rt] \lesssim n^{200}\delta, 
\end{align*}
provided that $\log\frac{1}{\delta} \lesssim \log n$. 
In addition,  one can deduce that
\begin{align*}
	&  n^{100} \sum_k \myE\lt[\|\phi_k\|_2^{100} \ind\lt(\mathcal{E}^{\mathrm{c}}\rt) \ind\Big(\max_k \|\phi_k\|_2 > 1+C_5\sqrt{\frac{\log \frac{n}{\delta}}{n}}\Big)\rt] \\
	&\qquad \lesssim n^{100} \sum_k \myE\lt[\|\phi_k\|_2^{100} \ind\lt(\mathcal{E}^{\mathrm{c}}\rt) \ind\Big(\|\phi_k\|_2 \le 1+C_5\sqrt{\frac{\log \frac{n}{\delta}}{n}}\Big)\ind\Big(\max_k \|\phi_k\|_2 > 1+C_5\sqrt{\frac{\log \frac{n}{\delta}}{n}}\Big)\rt] \\
	&\qquad \qquad+ n^{100} \sum_k \myE\lt[\|\phi_k\|_2^{100}  \ind\Big(\|\phi_k\|_2 > 1+C_5\sqrt{\frac{\log \frac{n}{\delta}}{n}}\Big)\rt] \\
	&\qquad \lesssim n^{100} \sum_k \lt(1+ C_5\sqrt{\frac{\log \frac{n}{\delta}}{n}}\rt)^{100} \delta + n^{100} \sum_k \int_{C_5\sqrt{\frac{\log \frac{n}{\delta}}{n}}}^{\infty} (1+x)^{100} \exp(-\frac{nx^2}{2})\dx \le n^{200} \delta + n^{-100}, 
\end{align*}
provided that $\log\frac{1}{\delta} \lesssim \log n$ and that $C_5$ is large enough. 
	Putting these two cases together and choosing $\delta \asymp n^{-300}$ finish the proof. 
\end{proof}

\subsubsection{Properties about the denoising function $\eta_t$} \label{sec:property-eta-z2}
Recall that the denoising function is
\[
	\eta_t(x) = \gamma_t \tanh(\pi_t x) 
	\qquad \text{with } \pi_t = \sqrt{n(\|x_t\|_2^2-1)} \text{ and } \gamma_t = \lt\|\tanh\lt(\pi_tx_t\rt)\rt\|_2^{-1}.
\]	
In this subsection, we single out several useful properties related to $\eta_t(\cdot)$.

\paragraph{Tight estimates of $\pi_t$ and $\gamma_t$. } 
Given that $\eta_t(\cdot)$ involves two quantities $\pi_t$ and $\gamma_t$, 
we first develop tight bounds on the sizes of them in the following, 
which are legitimate under the induction hypotheses \eqref{Z2-induction}. The proof is deferred to  Section~\ref{sec:pf-z2-pi}.
\begin{lems} 
\label{lem:pit-gammat}
Under the induction hypotheses \eqref{Z2-induction}, we have
\begin{subequations}
\label{eq:parameters}
\begin{align}
\label{eqn:don}
	\pi_t &= \bigg(1 + O\bigg( \frac{S_t}{\alpha^2_t} \bigg) \bigg)\alpha_t\sqrt{n} = \big(1+o(1)\big) \alpha_t \sqrt{n} \\
\label{eqn:giovanni}
\gamma_t^{-2} &=
\left(1+O\bigg(\frac{S_{t}}{\alpha_{t}}+\frac{S_{t}}{\alpha_{t}^{3}}\bigg)\right)n\int\tanh\big(\alpha_{t}(\alpha_{t}+x)\big)\varphi(\mathrm{d}x)\asymp\alpha_{t}^{2}n
\end{align}
\end{subequations}
with probability exceeding $1-O(n^{-11})$. 
\end{lems}
%

\paragraph{Bounds on derivatives and gradients.} 
Next, we look at the derivatives and gradients of the denoising function. 
As can be straightforwardly seen,   
the function $\eta_t(x) \defn \gamma_t\tanh(\pi_tx)$ is smooth everywhere, whose first three derivatives are given by
\begin{gather}
\begin{aligned}
\label{eqn:super-basic}
	\eta_{t}^{\prime}(x) &= \gamma_t\pi_t\big(1-\tanh^2(\pi_tx)\big) \\
	 \eta_{t}^{\second}(x) &= -2\gamma_t\pi_t^2\tanh(\pi_tx)\big(1-\tanh^2(\pi_tx)\big)\\ 
	\eta_{t}^{(\third)}(x) &= -2\gamma_t\pi_t^3\big(1-\tanh^2(\pi_tx)\big)\big(1-3\tanh^2(\pi_tx)\big)
\end{aligned}
\end{gather}
for any $x\in \real$. Combining the identities with \eqref{eq:parameters} and the fact $|\tanh(x)| \leq 1$, we can easily validate that 
\begin{gather}
\label{eqn:chocolate}
\begin{aligned}
	&|\eta_t(x)| \lesssim \frac{1}{\alpha_t\sqrt{n}} , \qquad\qquad
	&&|\eta_t^{\prime} (x) | \lesssim 1 \eqqcolon \rho, \\
	& |\eta_t^{\second} (x)| \lesssim \alpha_t\sqrt{n} \lesssim \sqrt{n} \eqqcolon \rho_1, 
	\qquad  && |\eta_{t}^{(\third)} (x)| \lesssim \alpha_t^2n \lesssim n \eqqcolon \rho_2.
	\end{aligned}
\end{gather}

Next, let us consider any given  vectors  $\mu =[\mu^k]_{-2s\leq k \leq t-1} \in \mathcal{S}^{t+2s-1}$,  $\beta =[\beta^k]_{-2s\leq k \leq t-1} \in \mathcal{S}^{t+2s-1}$,
 and any given $\alpha \in \real$ obeying $\lambda \geq \alpha \geq \sqrt{\lambda^2 - 1}$ (note that, for the moment, we shall treat them as fixed parameters independent of $\{\phi_{k}\}$).  
We shall also define 
\begin{align*}
	\eta_{t}^{(i)}\big(v(\alpha,\beta)\big) \defn \eta_{t}^{(i)}\Big(\alpha \vstar + \sum_{k = -2s}^{t-1} \beta^k\phi_k\Big) \in \real^n, 
	\quad
	\text{with }v(\alpha,\beta) \defn \alpha\vstar + \sum_{k = -2s}^{t-1} \beta^k\phi_k, 
\end{align*}
where the superscript $i$ denotes the $i$-th derivative (computed in an entrywise manner). 
In what follows, we collect several elementary results that are useful for our main proof. 
\begin{subequations}
\label{eqn:-z2-basic-derivatives}
\begin{align}
	\Big\|\nabla_{\phi_j} \Big\langle \sum_{k = -2s}^{t-1} \mu^k\phi_k, a\Big\rangle\Big\|_2 &\le |\mu^j| \cdot \|a\|_2, \qquad 
	&&\text{for any given } a\in\real^n \label{mu-phi}\\
	\Big\|\nabla_{\phi_j} \big\langle \eta_t^{(s)}\big(v(\alpha,\beta)\big), a\big\rangle\Big\|_2 &\le |\beta^{j}|\cdot \big\|\eta_t^{(s+1)}\big(v(\alpha,\beta)\big) \circ a \big\|_2, \qquad 
	&&\text{for any given } a\in\real^n \label{eta-phi}\\
	\Big\|\nabla_{\mu} \Big\langle \sum_{k = -2s}^{t-1} \mu^k\phi_k, a\Big\rangle\Big\|_2 &\le \|a\|_2\sum_{k = -2s}^{t-1} \|\phi_k\|_2, \qquad 
	&&\text{for any given } a\in\real^n \label{mu-mu}\\
	\Big\|\nabla_{\mu, \beta} \Big( \sum_{k = -2s}^{t-1} \mu^k\beta^k \Big) \Big\|_2 &\le \|\mu\|_2 + \|\beta\|_2 = 2,  \label{mu-beta}\\
	\Big\|\nabla_{\beta} \Big\langle \eta_t^{(s)}\big(v(\alpha,\beta)\big), a\Big\rangle\Big\|_2 &\le \|a\|_2\cdot \big\|\eta_t^{(s+1)}\big(v(\alpha,\beta)\big) \big\|_2\cdot\sum_{k=-2s}^{t-1}\|\phi_k\|_2, \qquad 
	&&\text{for any given } a\in\real^n.\label{eta-beta}
\end{align}
\end{subequations}
The proofs of these results are fairly elementary and are hence omitted for the sake of brevity.


\subsubsection{Proof of tight estimates of $\pi_t$ and $\gamma_t$ (Lemma~\ref{lem:pit-gammat})}
\label{sec:pf-z2-pi}

\paragraph{Bounding quantity $\pi_{t}$.}

In view of Lemma~\ref{lem:brahms-lemma} and the fact that $v^{\star\top}\big[\phi_{1},\cdots,\phi_{t-1}\big]\sim \mathcal{N}(0,\frac{1}{n}I_{t-1})$, we have 
\begin{align*}
\lt\|\sum_{k = -2s}^{t-1} \beta_{t-1}^k\phi_k\rt\|_2 &= 1 + O\lt(\sqrt{\frac{(t+s)\log n}{n}}\rt), \\
\Big|\Big\langle\vstar,\sum_{k=-2s}^{t-1}\beta_{t-1}^{k}\phi_{k}\Big\rangle\Big|
	&=\Big|\Big\langle v^{\star\top}\big[\phi_{-2s},\cdots,\phi_{t-1}\big],\big[\beta_{t-1}^{-2s},\cdots\beta_{t-1}^{t-1}\big]\Big\rangle\Big|\leq\Big\|v^{\star \top}\big[\phi_{-2s},\cdots,\phi_{t-1}\big]\Big\|_{2}\|\beta_{t-1}\|_{2} \\
	&\lesssim\sqrt{\frac{(t+s)\log n}{n}}
\end{align*}
with probability exceeding $1-O(n^{-11})$, where we recall that $\|\beta_{t-1}\|_{2}=1$.  
As a result, recalling the induction hypothesis that $|\alpha_t|\leq \lambda \lesssim 1$ (see \eqref{eq:Z2-induction-alphat}), we arrive at
\[
	\Big\|\alpha_{t}\vstar+\sum_{k=-2s}^{t-1}\beta_{t-1}^{k}\phi_{k}\Big\|_{2}^{2}=\alpha_{t}^{2}+2\alpha_{t}\Big\langle\vstar,\sum_{k=-2s}^{t-1}\beta_{t-1}^{k}\phi_{k}\Big\rangle+\Big\|\sum_{k=-2s}^{t-1}\beta_{t-1}^{k}\phi_{k}\Big\|_{2}^{2}=\alpha_{t}^{2}+1+O\Big(\sqrt{\frac{(t+s)\log n}{n}}\Big)
	\lesssim 1.
\]
Invoke the other induction hypothesis \eqref{eqn:Z2-induction-st} and the condition $s\asymp \frac{\log n}{(\lambda-1)^2}$ to obtain 
\begin{align*}
\|x_t\|_2^2 
&= \Big\|\alpha_t \vstar + \sum_{k = -2s}^{t-1} \beta_{t-1}^k\phi_k + \xi_{t-1}\Big\|_2^2
= \Big\|\alpha_{t}\vstar+\sum_{k=-2s}^{t-1}\beta_{t-1}^{k}\phi_{k}\Big\|_{2}^{2} 
+ 2 \Big\langle \xi_{t-1}, \alpha_t \vstar + \sum_{k = -2s}^{t-1} \beta_{t-1}^k\phi_k\Big\rangle  + \ltwo{\xi_{t-1}}^2\\
	&= \alpha_t^2 + 1 + O\Big(\sqrt{\frac{(t+s)\log n}{n}}\Big) 
	+ O\bigg( \big\| \xi_{t-1} \big\|_2  \Big\| \alpha_t \vstar + \sum_{k = -2s}^{t-1} \beta_{t-1}^k\phi_k\Big\|_2 \bigg)
	+ \ltwo{\xi_{t-1}}^2\\
&=  \alpha_t^2 + 1 +  O\Big(\sqrt{\frac{(t+s)\log n}{n}}\Big) + O\lt(\ltwo{\xi_{t-1}}\rt) = \alpha_t^2 + 1 +  O(S_t).
\end{align*}
Therefore, we can conclude that 
\begin{align}
\pi_t = \sqrt{n(\|x_t\|_2^2-1)} 
&= \alpha_t\sqrt{n} \lt(1 + O\lt(\frac{S_t}{\alpha^2_t}\rt)\rt),
	\label{eq:gamma-part1}
\end{align}
where we use the induction hypothesis \eqref{Z2-induction}.   
This establishes the advertised relation~\eqref{eqn:don} about $\pi_{t}$.

\paragraph{Bounding quantity $\gamma_t$.} 
Before proceeding, we find it helpful to first establish a connection between 
$\|\tanh\lt(\pi_tx_t\rt)\|_2$ and $\|\tanh(\pi_tv_t)\|_2$, where we recall that $x_{t} = v_{t} + \xi_{t-1}$. 
Recognizing that $|\tanh(x)| \leq 1$ and $|\tanh'(x)| \leq 1$, we can guarantee that, for each $1\leq i\leq n$, 
\begin{align*}
	\big|\tanh\lt(\pi_tx_{t,i}\rt) - \tanh\lt(\pi_tv_{t,i}\rt)\big| 
\leq \pi_t |{\xi}_{t-1,i}|  \lesssim \alpha_t\sqrt{n} \cdot |{\xi}_{t-1,i}|, 
\end{align*}
where the last inequality follows from~\eqref{eq:gamma-part1}.
By virtue of the the induction hypothesis \eqref{eqn:Z2-induction-st},
we can obtain
\begin{align}
\notag \lt| \big\|\tanh\lt(\pi_tx_t\rt)\big\|_2^2 - \big\|\tanh\lt(\pi_tv_t\rt)\big\|_2^2\rt| 
&\leq \big\|\tanh\lt(\pi_tx_t\rt) - \tanh\lt(\pi_tv_t\rt)\big\|_2 \cdot \big\|\tanh\lt(\pi_tx_t\rt)+\tanh\lt(\pi_tv_t\rt)\big\|_2\\
&\lesssim \alpha_tn\|\xi_{t-1}\|_2 \lesssim \alpha_tnS_t .
\label{eq:gamma-part2}
\end{align}

The above relation \eqref{eq:gamma-part2}
allows us to turn attention to the quantity $\lt\|\tanh\lt(\pi_tv_t\rt)\rt\|_2^2$, 
towards which we would like to invoke Lemma~\ref{lem:Gauss} 
to control the following quantity 
\begin{align*}
\big\|\tanh\lt(\pi_tv_t\rt)\big\|_2^2 - n\int \tanh^2\lt(\frac{\pi_t}{\sqrt{n}}\lt(\alpha_t+ x\rt)\rt) \varphi(\dx). 
\end{align*} 
Given that for any coordinate $1\leq i\leq n$, one has $\sqrt{n}\vstar_{i} \in \{+1, -1\}$ and hence (due to symmetry)
\begin{align*}
	\int \tanh^2 \lt(\pi_t\lt(\alpha_t \vstar_i + \frac{x}{\sqrt{n}}\rt)\rt) \varphi(\dx)
	=
	\int \tanh^2 \lt(\frac{\pi_t}{\sqrt{n}}\lt(\alpha_t + x\rt)\rt) \varphi(\dx), 
\end{align*}
we are motivated to look at the following function 
\begin{align*}
	f_{\theta}(\Phi) \defn \big\| \tanh(\pi v) \big\|_2^2 - \int\bigg\|\tanh\lt(\pi \lt(\alpha \vstar + \frac{1}{\sqrt{n}}x\rt) \rt)\bigg\|_2^2\varphi_n(\dx)
	\qquad
	\text{where } v \defn \alpha \vstar + \sum_{k = -2s}^{t-1} \beta^k\phi_k. 
\end{align*}
where we define 
$$
	\Phi= \sqrt{n} \big[ \phi_{-2s},\ldots, \phi_{t-1} \big] ,
	\qquad  
	\theta = [\alpha, \beta, \pi] \in \real^{t+2s+2} 
	\qquad \text{and} \qquad \beta = [\beta^{-2s},\cdots, \beta^{t-1}]. 
$$
Clearly, 
in order to bound $\big\| \tanh(\pi_t v_t) \big\|_2^2 - \int\big\|\tanh\big(\pi_t \big(\alpha_t \vstar + \frac{1}{\sqrt{n}}x\big) \big)\big\|_2^2\varphi_n(\dx)$, 
it suffices to develop a bound on $f_{\theta}(\Phi)$ uniformly over all $\theta$ within the following set: 
\begin{equation}
	\Theta \defn \Big\{ \theta = (\alpha,\beta,\pi) \mid \|\beta \|_2 = 1,  \sqrt{\lambda - 1} \lesssim \alpha \lesssim 1, 
	\pi \asymp  \alpha \sqrt{n} 
	\Big\}.
\end{equation}

Towards this end, observe that
\begin{align*}
	\lt\|\nabla_{\Phi} f_\theta(\Phi)\rt\|_2 
	\le \frac{2\pi\|\beta\|_2}{\sqrt{n}}\lt\|\tanh(\pi v) \circ \tanh^{\prime}(\pi v) \rt\|_2 
	\leq 2\pi\|\beta\|_2
	&\lesssim \alpha \sqrt{n}, 
\end{align*}
where we have used the facts that $\|\beta\|_2=1$, and $\pi \asymp  \alpha \sqrt{n}$. 
Additionally, it is straightforward to check that $f_{\theta}(\Phi)$ obeys $\|\nabla_{\theta} f_{\theta}(\Phi) \|_2 \lesssim n^{100} $
for all $Z\in \mathcal{E}$ and $|f_{\theta}(\Phi)| \lesssim n^{100} \big( \max_k \|\phi_k\|_2 \big)^{100}$. 
For any fixed $\theta$, it is readily seen that $\mathbb{E}[f_{\theta}(\Phi)]=0$. 
Applying Corollary~\ref{cor:Gauss} in conjunction with \eqref{eqn:brahms-conc} yields
\[
	\sup_{\theta\in \Theta} \bigg| \frac{1}{\alpha}  f_{\theta}(\Phi) \bigg| 
	\lesssim \sqrt{n(t+s)\log n}
\]
with probability at least $1-O(n^{-11})$. This in turn leads to
\begin{align}
	\left|\big\| \tanh(\pi_t v_t) \big\|_2^2 - \int\Big\|\tanh\lt(\pi_t\lt(\alpha_t v^{\star} + \frac{1}{\sqrt{n}}x\rt)\rt)\Big\|_2^2\varphi(\dx)\right| 
	\leq \alpha_t \sup_{\theta\in \Theta} \frac{1}{\alpha} \lt|
	f_{\theta}(\Phi)
	\rt| 
	\lesssim \alpha_t\sqrt{(t+s)n\log n}. \label{eq:gamma-part3}
\end{align}

Putting~\eqref{eq:gamma-part1},~\eqref{eq:gamma-part2}, and~\eqref{eq:gamma-part3} together leads to
\begin{align}
 \|\tanh(\pi_tx_t)\|_2^2 
&= n\int \tanh^2\lt(\frac{\pi_t}{\sqrt{n}}\lt(\alpha_t+ x\rt)\rt) \varphi(\dx) + O\Big(\alpha_tnS_t + \alpha_t\sqrt{(t+s)n\log n}\Big) \notag\\
	& = n\int \tanh^2\lt(\frac{\pi_t}{\sqrt{n}}\lt(\alpha_t+ x\rt)\rt) \varphi(\dx) + O\big(\alpha_tnS_t \big). 
	\label{eq:identity-tanh-pixt}
\end{align}
In view of the mean value theorem and the fact that $|(\tanh^{2})^{\prime}(w)|=|2\tanh(w)\tanh^{\prime}(w)|\leq 2$ (and hence $\tanh^{2}$ is 2-Lipschitz continuous), we have
\begin{align}
\bigg|\tanh^{2}\Big(\frac{\pi_{t}}{\sqrt{n}}(\alpha_{t}+x)\Big)-\tanh^{2}\Big(\alpha_{t}(\alpha_{t}+x)\Big)\bigg| & \leq2\left|\Big(\frac{\pi_{t}}{\sqrt{n}}-\alpha_{t}\Big)(\alpha_{t}+x)\right|\leq2\left|\frac{\pi_{t}}{\sqrt{n}}-\alpha_{t}\right|\big(\alpha_{t}+|x|\big),
	\label{eq:tanh2-diff-123}
\end{align}
which together with \eqref{eq:gamma-part1} yields
\begin{align}
 & \bigg|{\displaystyle \int}\tanh^{2}\Big(\frac{\pi_{t}}{\sqrt{n}}(\alpha_{t}+x)\Big)\varphi(\mathrm{d}x)-{\displaystyle \int}\tanh^{2}\Big(\alpha_{t}(\alpha_{t}+x)\Big)\varphi(\mathrm{d}x)\bigg| \notag\\
&\qquad \leq2\left|\frac{\pi_{t}}{\sqrt{n}}-\alpha_{t}\right|\left({\displaystyle \int}\alpha_{t}\varphi(\mathrm{d}x)+{\displaystyle \int}|x|\varphi(\mathrm{d}x)\right)
 \lesssim\left|\frac{\pi_{t}}{\sqrt{n}}-\alpha_{t}\right|
	\lesssim \frac{S_t}{\alpha_t}. 
	\label{eq:tanh2-diff-456}
\end{align}
Substitution into \eqref{eq:identity-tanh-pixt} gives
\begin{align}
	\big\|\tanh(\pi_{t}x_{t}) \big\|_{2}^{2}  
	&=n\int\tanh^{2}\big(\alpha_{t} (\alpha_{t}+x )\big)\varphi(\mathrm{d}x)+O\bigg(\alpha_{t}nS_{t} + \frac{nS_t}{\alpha_t}
	\bigg) 
	\notag\\
	& =n\int\tanh \big(\alpha_{t} (\alpha_{t}+x)\big)\varphi(\mathrm{d}x)
	+ O\left(\alpha_{t}^{2}n\bigg(\frac{S_{t}}{\alpha_{t}}+\frac{S_{t}}{\alpha_{t}^{3}}\bigg)\right) ,
\label{eqn:giovanni-2}
\end{align}
where the last line follows from \citet[Eq.~(B.4) in Appendix B.2]{deshpande2017asymptotic}.

%
%

Finally, we justify that $\int\tanh \lt(\alpha_{t}\lt(\alpha_{t}+x\rt)\rt)\varphi(\mathrm{d}x) \asymp \alpha_t^2$. 
Towards this, we make the observation that
\begin{align*}
 & \int\tanh(\alpha_{t}^{2}+\alpha_{t}x)\varphi(\mathrm{d}x)\\
 & \qquad=\int_{0}^{2\alpha_{t}}\tanh(\alpha_{t}^{2}+\alpha_{t}x)\varphi(\mathrm{d}x)+\int_{\alpha_{t}}^{\infty}\left\{ \tanh\big(\alpha_{t}^{2}+\alpha_{t}(\alpha_{t}+z)\big)+\tanh\big(\alpha_{t}^{2}+\alpha_{t}(\alpha_{t}-z)\big)\right\} \varphi(\mathrm{d}x)\\
 & \qquad\geq\int_{0}^{2\alpha_{t}}\tanh(\alpha_{t}^{2}+\alpha_{t}x)\varphi(\mathrm{d}x)\geq\int_{0}^{2\alpha_{t}}\tanh'(3\alpha_{t}^{2})(\alpha_{t}^{2}+\alpha_{t}x)\varphi(\mathrm{d}x)\\
 & \qquad\geq\Big(1-\tanh^{2}(3\alpha_{t}^{2})\Big)\Big(2\alpha_{t}^{3}+\varphi(2\alpha_{t})2\alpha_{t}^{2}\Big)\asymp\alpha_{t}^{2}, 
\end{align*}
where the first inequality follows since $\tanh\big(\alpha_{t}^{2}+\alpha_{t}(\alpha_{t}+z)\big)+\tanh\big(\alpha_{t}^{2}+\alpha_{t}(\alpha_{t}-z)\big)\geq 0$ 
for any $z\geq 0$, 
the second inequality holds since $\tanh^{\prime}(0)=0$ and $\tanh^{\prime}(w)$ is decreasing in $w$ for $w\geq 0$, 
and the last line uses $\tanh^{\prime}(w)=1-\tanh^2(w)$ and the induction hypothesis that $\alpha_t\lesssim 1$ (cf.~\eqref{eq:Z2-induction-alphat}).  
Additionally, it results from the Taylor expansion as well as the
facts $\tanh(0)=0$ and $|\tanh^{\second}(w)|\leq1$ that
\begin{align*}
\int\tanh(\alpha_{t}^{2}+\alpha_{t}x)\varphi(\mathrm{d}x) & \leq\int(\alpha_{t}^{2}+\alpha_{t}x)\varphi(\mathrm{d}x)+\frac{1}{2}\int|\alpha_{t}^{2}+\alpha_{t}x|^{2}\varphi(\mathrm{d}x)\\
 & \leq\int(\alpha_{t}^{2}+\alpha_{t}x)\varphi(\mathrm{d}x)+\int\alpha_{t}^{4}\varphi(\mathrm{d}x)+\int\alpha_{t}^{2}x^{2}\varphi(\mathrm{d}x)
  =2\alpha_{t}^{2} + \alpha_t^4.
\end{align*}
Consequently, we have justified that 
\begin{equation}
	\int\tanh(\alpha_{t}^{2}+\alpha_{t}x)\varphi(\mathrm{d}x)\asymp \alpha_t^2
	\label{eq:int-alpha2-tanh-123}
\end{equation}
given that $\alpha_t\lesssim 1$ (cf.~\eqref{eq:Z2-induction-alphat}). 
Combining this with~\eqref{eqn:giovanni-2} and the induction hypothesis that $\sqrt{\lambda - 1} \lesssim \alpha_t \lesssim 1$ (cf.~\eqref{eq:Z2-induction-alphat}), we reach
\begin{align}
\label{eqn:tanh-basic-alpha}
\big\|\tanh(\pi_{t}x_{t}) \big\|_{2}^{2} & =
	\left(1+ O\bigg(\frac{S_{t}}{\alpha_{t}}+\frac{S_{t}}{\alpha_{t}^{3}}\bigg)\right)
	n\int\tanh\big(\alpha_{t}(\alpha_{t}+x)\big)\varphi(\mathrm{d}x)\asymp n\alpha_{t}^{2} ,
\end{align}
provided that $\frac{S_{t}}{\alpha_{t}}+\frac{S_{t}}{\alpha_{t}^{3}} \ll 1$. 
This concludes the proof of Lemma~\ref{lem:pit-gammat}.


\subsection{Controlling several key quantities $A_t,B_t,D_t$}
\label{sec:z2-key}

By virtue of Theorem~\ref{thm:main} or Corollary~\ref{cor:recursion-spectral}, 
the behavior of $\alpha_{t+1}$ is governed by a couple of key quantities as defined in Assumptions~\ref{assump:A-H-eta} (except that those sums w.r.t.~$\sum_{k=1}^{t-1}$ there should be replaced with $\sum_{k=-2s}^{t-1}$ to account for spectral initialization). 
Several immediate remarks are in order. 
\begin{itemize}
	\item As alluded to previously, there is no need to bound $\Delta_{\beta,t}$ given that $\|\beta_t\|_2$ is fixed. As a result, 
		there is no need in controlling $C_t, F_{t}$ and $G_t$. 
		
	\item Given that the denoising function is smooth everywhere, we clearly have $E_t=0$. 

%
\end{itemize}
With these remarks in mind, 
the proof of Theorem~\ref{thm:Z2} largely consists of identifying sufficiently small quantities $A_{t}, B_t,  D_{t}$ such that \eqref{defi:A}, \eqref{defi:B} and \eqref{defi:D} are satisfied with high probability, which forms the main content of this subsection. 
The analysis in this subsection operates under the induction hypotheses \eqref{Z2-induction}.

%

\subsubsection{Quantity $A_t$ in \eqref{defi:A}}
\label{sec:control-A-Z2}

Recall that this part is concerned with bounding the following quantity 
\begin{align}
	 \Big\langle \sum_{k = -2s}^{t-1} \mu^k_t \phi_k, \eta_{t}(v_t)\Big\rangle - \lt\langle\eta_t^{\prime}(v_t) \rt\rangle \sum_{k = -2s}^{t-1} \mu^k_t\beta_{t-1}^k ;
	 \label{eq:target-An-expression}
\end{align}
note that the summation starts from $k=-2s$ in order to take into account spectral initialization. 
In order to analyze this quantity, we introduce
\begin{equation}
	\Phi \defn \sqrt{n}(\phi_{-2s},\ldots,\phi_{t-1}),\qquad 
	\theta \defn (\mu,\alpha,\beta, \pi, \gamma) \in \mathcal{S}^{2s+t-1} \times \real \times \mathcal{S}^{2s+t-1} \times \real \times \real, 
	\label{eq:defn-Phi-theta-Z2-A}
\end{equation}
and define the following function
\begin{align*}
	f_\theta(\Phi) \defn \Big\langle \sum_{k = -2s}^{t-1} \mu^{k} \phi_k, \eta
	\Big(\alpha \vstar + \sum_{k=-2s}^{t-1} \beta^{k} \phi_k\Big)\Big\rangle 
	- \Big\langle\eta^{\prime}\Big(\alpha \vstar + \sum_{k=-2s}^{t-1} \beta^{k} \phi_k\Big) \Big\rangle
	 \sum_{k = -2s}^{t-1} \mu^k\beta^{k},  
\end{align*}
where 
\begin{equation}
	\eta (w) \coloneqq \gamma^{-1} \tanh(\pi w) ;
	\label{eq:defn-eta-Z2-A}
\end{equation}
here, we suppress the dependency on $\alpha$, $\pi$ and $\gamma$ in the notation $\eta(\cdot)$ for simplicity.  
Clearly, the quantity \eqref{eq:target-An-expression} can be expressed as $f_\theta(\Phi)$ with $\alpha=\alpha_t$, $\mu = \mu_t, \beta = \beta_{t-1}, \pi = \pi_t$ and $\gamma = \gamma_t$; 
these parameters, however, are statistically dependent on $\Phi$. As a result, we resort to Lemma~\ref{lem:Gauss} in order to obtain a uniform control over all parameters 
within a suitable region
\begin{equation}
	\Theta \defn \Big\{ \theta = (\mu,\alpha,\beta,\pi,\gamma) \mid \|\mu \|_2 = \|\beta \|_2 = 1,  \sqrt{\lambda - 1} \lesssim \alpha \lesssim 1, 
	\pi \asymp \gamma^{-1} \asymp \alpha \sqrt{n} \Big\}.
	\label{eq:parameter-set-Z2-A}
\end{equation}
It follows immediately from the calculation in Section~\ref{sec:property-eta-z2} that, for any $\theta \in \Theta$ and any $x\in \real$, 
\begin{gather}
\label{eqn:chocolate-At-Z2}
\begin{aligned}
	&|\eta(x)| \lesssim \frac{1}{\alpha \sqrt{n}} , \qquad
	|\eta^{\prime} (x) | \lesssim 1 , \qquad
	 |\eta^{\second} (x)| \lesssim \alpha \sqrt{n}  ,
	\qquad  && |\eta^{\third} (x)| \lesssim \alpha ^2n .
	\end{aligned}
\end{gather}

Clearly, we can see that (i) $\|\theta\|_2\lesssim \sqrt{n}$ for any $\theta \in \Theta$,
(ii) $\|\nabla_{\theta} f_{\theta} (Z) \|_2 \lesssim n^{100}$ for any $Z\in \mathcal{E}$ (see \eqref{eqn:eps-interset}),   
and (iii) $|f_{\theta}(\Phi)|\lesssim n^{100}\big(\max_{k}\|\phi_{k}\|_{2}\big)^{100}$.  
Then according to \eqref{eqn:brahms-conc}, 
it would be natural to invoke Corollary~\ref{cor:Gauss} to obtain uniform control of $f_{\theta}(\Phi)$. 
The main step then boils down to bounding $\nabla_{\Phi} f_{\theta}(\Phi)$, which we accomplish in what follows.

Letting $v \defn \alpha \vstar + \sum_{k=-2s}^{t-1} \beta^{k} \phi_k$ for notational simplicity, 
we can directly bound the derivative of $f$ w.r.t.~$\Phi$ as follows:   
\begin{align*}
\big\|\nabla_{\Phi}f_{\theta}(\Phi)\big\|_{2} & \le\frac{\|\mu\|_{2}}{\sqrt{n}}\lt\|\eta_{t}(v)\rt\|_{2}+\frac{\|\beta\|_{2}}{\sqrt{n}}\Big\|\sum_{k=-2s}^{t-1}\mu^{k}\phi_{k}\circ\eta^{\prime}(v)\Big\|_{2}+\bigg(\frac{\|\beta\|_{2}}{n\sqrt{n}}\lt\|\eta^{\second}(v)\rt\|_{2}\bigg)\big(\|\mu\|_{2}\|\beta\|_{2}\big)\\
 & = \frac{1}{\sqrt{n}}\lt\|\eta_{t}(v)\rt\|_{2}+\frac{1}{\sqrt{n}}\Big\|\sum_{k=-2s}^{t-1}\mu^{k}\phi_{k}\circ\eta^{\prime}(v)\Big\|_{2}+\frac{1}{n\sqrt{n}}\lt\|\eta^{\second}(v)\rt\|_{2}\\
 & \lesssim\frac{1}{\alpha\sqrt{n}} ,
\end{align*}
where the first inequality follows from \eqref{mu-phi} and~\eqref{eta-phi}, 
the second line relies on the condition $\|\mu\|_2 = \|\beta\|_2=1$, 
and the last inequality makes use of~\eqref{eqn:chocolate}, the condition $\sqrt{\lambda - 1}\lesssim \alpha\lesssim 1$, and the fact that $\big\| \sum_{k=-2s}^{t-1}\mu^{k}\phi_{k} \big\|_2 \lesssim 1$ (see the second pair of curly brackets in \eqref{eq:eps-set}). 
Applying Corollary~\ref{cor:Gauss} then gives 
\begin{align}
\label{eqn:moonlight}
	\sup_{\theta \in \Theta} \big| \alpha f_\theta(\Phi) - \alpha \myE [f_\theta(\Phi)] \big|
	\lesssim 
	\sqrt{\frac{(t+s)\log n}{n}}
\end{align}
with probability at least $1-O(n^{-11})$.

In addition, we observe that: for any fixed parameter $\theta$, Stein's lemma reveals that
\begin{align*}
	\myE \big[f_\theta(\Phi)\big] = \myE \left[\Big\langle \sum_{k = -2s}^{t-1} \mu^{k} \phi_k, \eta_{t}\Big(\alpha \vstar + \sum_{k=-2s}^{t-1} \beta^{k} \phi_k\Big)\Big\rangle 
	- \Big\langle\eta_t^{\prime}\Big(\alpha \vstar + \sum_{k=-2s}^{t-1} \beta^{k} \phi_k\Big) \Big\rangle
	 \sum_{k = -2s}^{t-1} \mu^k\beta^{k}\right] = 0, 
\end{align*}
which together with \eqref{eqn:moonlight} gives
\begin{align}
\label{eqn:moonlight}
	\sup_{\theta \in \Theta} \Big\{ \alpha\, \big|f_\theta(\Phi)  \big| \Big\} 
	\lesssim 
	\sqrt{\frac{(t+s)\log n}{n}}
\end{align}
with probability at least $1-O(n^{-11})$. 
Consequently, it is  sufficient to take
\begin{align}
\label{eqn:Z2-At}
	A_t \asymp \frac{1}{\alpha_t}\sqrt{\frac{(t+s)\log n}{n}}.
\end{align}


\subsubsection{Quantity $B_t$ in \eqref{defi:B}}

Regarding the quantity $B_t$, we need to examine the following function 
\begin{align*}
	f_{\theta}(\Phi) \defn v^{\star \top}\eta(v),
	\qquad
	\text{with } v \defn \alpha \vstar + \sum_{k = -2s}^{t-1} \beta^k\phi_k, 
\end{align*}
where $\Phi$ and $\theta$ are defined in \eqref{eq:defn-Phi-theta-Z2-A}, and $\eta(\cdot)$ is defined in \eqref{eq:defn-eta-Z2-A}. 
Clearly, the target quantity on the left-hand side of \eqref{defi:B} 
can be viewed as $f_{\theta}(\Phi)$ with $\alpha = \alpha_t, \beta = \beta_{t-1}, \mu = \mu_t, \pi = \pi_t$ and $\gamma = \gamma_t$. 
When it comes to the convex set $\mathcal{E}$ (cf.~\eqref{eqn:eps-interset}) and the parameter set $\Theta$ (cf.~\eqref{eq:parameter-set-Z2-A}),
 it is straightforward to verify that $\|\nabla_{\theta} f_{\theta}(Z) \|_2\lesssim n^{100}$ for any $Z\in \mathcal{E}$ 
 and $| f_{\theta}(Z) | \lesssim n^{100} \big(\|Z \|_{\mathrm{F}} \big)^{100}$.  
Therefore, 
in view of \eqref{eqn:brahms-conc}, 
we shall resort to Corollary~\ref{cor:Gauss} to obtain uniform control of $f_{\theta}(\Phi)$ over all $\theta \in \Theta$.


Invoking inequality \eqref{eta-phi} yields  
\begin{align*}
	\lt\|\nabla_{\Phi} f_{\theta}(\Phi)\rt\|_2 &\le \frac{\|\beta\|_2}{\sqrt{n}}\lt\|\vstar \circ \eta^{\prime}(v)\rt\|_2 \lesssim \frac{1}{\sqrt{n}}, 
\end{align*}
where the last inequality holds since, according to~\eqref{eqn:chocolate-At-Z2}, 
\begin{align}
	\|\eta^{\prime}(v) \circ \vstar\|_2 &\lesssim \|\vstar\|_2
	 = 1.
\end{align}
Apply Corollary~\ref{cor:Gauss} to arrive at
\begin{align*}
	\sup_{\theta \in \Theta} \lt|v^{\star \top}\eta(v) - \myE[v^{\star \top}\eta(v)]\rt| &\lesssim \sqrt{\frac{(t+s)\log n}{n}} 
\end{align*}
with probability at least $1-O(n^{-11})$.
In addition, for every given $\theta$ we have
\begin{align*}
	\myE[v^{\star \top}\eta(v)] 
	&= \myE \left[v^{\star \top} \eta\Big(\alpha \vstar + \sum_{k = -2s}^{t-1} \beta^k\phi_k\Big)\right] 
	= v^{\star \top}\int\eta \lt(\alpha \vstar + \frac{1}{\sqrt{n}}x\rt)\varphi_n(\dx),
\end{align*}
where  $\varphi_n(\cdot)$ is the CDF of $\mathcal{N}(0,I_n)$.  
Therefore, in view of the definition \eqref{defi:B}, it is sufficient to set  
\begin{align}
\label{eqn:Z2-Bt}
	B_t \asymp \sqrt{\frac{(t+s)\log n}{n}}.
\end{align}

%

\subsubsection{Quantity $D_t$ in \eqref{defi:D}}

With regards to quantity $D_t$, we aim to justify that 
\begin{align}
	\Big\|\sum_{k = -2s}^{t-1} \mu_t^k\phi_k \circ \eta_{t}^{\prime}(v_t) - \frac{1}{n}\sum_{k = -2s}^{t-1} \mu_t^k\beta_{t-1}^k\eta_{t}^{\second}(v_t)\Big\|_2^2 - \kappa_t^2
	\lesssim  D_t \asymp
	\sqrt{\frac{(t+s)\log^2 n}{n}} . \label{eqn:Z2-Dt}
\end{align}
In order to prove this, let us introduce the following function 
\begin{align*}
	&f_\theta(\Phi) \defn \Big\|\sum_{k = -2s}^{t-1} \mu^k\phi_k \circ \eta^{\prime}(v) - \frac{1}{n}\sum_{k = -2s}^{t-1} \mu^k\beta^k\eta^{\second}(v)\Big\|_2^2 - \kappa^2,\\
	&\qquad\qquad \text{with } v \defn \alpha \vstar + \sum_{k = -2s}^{t-1} \beta^k\phi_k;  
\end{align*}
here,  $\Phi$ and $\theta$ are defined in \eqref{eq:defn-Phi-theta-Z2-A},  $\eta(\cdot)$ is defined in \eqref{eq:defn-eta-Z2-A}, 
whereas $\kappa$ is defined such that
\begin{align}
 & \notag\kappa^{2}\defn\max\Bigg\{\Bigg\langle\int\Big[x\eta^{\prime}\Big(\alpha\vstar+\frac{1}{\sqrt{n}}x\Big)-\frac{\ltwo{\beta}}{\sqrt{n}}\eta^{\second}\Big(\alpha\vstar+\frac{1}{\sqrt{n}}x\Big)\Big]^{2}\varphi_{n}(\dx)\Bigg\rangle,~\\
 & %
\qquad\qquad\qquad\qquad\qquad\qquad\qquad
\bigg\langle\int\Big[\eta^{\prime}\Big(\alpha\vstar+\frac{1}{\sqrt{n}}x\big)\Big]^{2}\varphi_{n}(\dx)\bigg\rangle\Bigg\} . \label{defi:kappa-Z2-D}
\end{align}
We shall also introduce the set $\mathcal{E}$ (resp.~$\Theta$) as in \eqref{eqn:eps-interset} (resp.~\eqref{eq:parameter-set-Z2-A}). 
Once again, it is easily seen that 
$\|\nabla_{\theta} f_{\theta}(Z) \|_2\lesssim n^{100}$ holds for any $Z\in \mathcal{E}$ 
 and $| f_{\theta}(Z) | \lesssim n^{100} \big(\|Z \|_{\mathrm{F}} \big)^{100}$.  
In light of \eqref{eqn:brahms-conc}, 
it is natural to apply Corollary~\ref{cor:Gauss} to obtain uniform control of $f_\theta(\Phi)$ over all $\theta \in \Theta$, 
which we detail as follows.


To begin with, we can take the derivative and use $\|\mu\|_2=\|\beta\|_2=1$ to obtain 
\begin{align}
\label{eqn:dt-mighty-five}
\notag &\lt\|\nabla_{\Phi} f_\theta(\Phi)\rt\|_2 \\
	\notag &\lesssim \frac{2}{\sqrt{n}}\lt\|\sum_{k = -2s}^{t-1} \mu^k\phi_k \circ \eta^{\prime}(v) \circ \eta^{\prime}(v)\rt\|_2 + \frac{2}{\sqrt{n}}\lt\|\sum_{k = -2s}^{t-1} \mu^k\phi_k \circ \sum_{k = -2s}^{t-1} \mu^k\phi_k \circ \eta^{\prime}(v) \circ \eta^{\second}(v)\rt\|_2 + \frac{2}{n^2\sqrt{n}}\lt\|\eta^{\second}(v) \circ \eta^{\third}(v) \rt\|_2 \\
&\qquad + \frac{2}{n\sqrt{n}}\lt\|\eta^{\prime}(v) \circ \eta^{\second}(v)\rt\|_2 + \frac{2}{n\sqrt{n}}\lt\|\sum_{k = -2s}^{t-1} \mu^k\phi_k \circ \eta^{\second}(v) \circ \eta^{\second}(v) \rt\|_2 + \frac{2}{n\sqrt{n}}\lt\|\sum_{k = -2s}^{t-1} \mu^k\phi_k \circ \eta^{\prime}(v) \circ \eta^{\third}(v) \rt\|_2,
\end{align}
as a consequence of \eqref{mu-phi} and~\eqref{eta-phi}. 
With this in place, we can further deduce that
\begin{align}
\label{eqn:scriabin-mys}
	\lt\|\nabla_{\Phi} f_\theta(\Phi)\rt\|_2 &\lesssim \sqrt{\frac{\log n}{n}}, 
\end{align}
whose proof is deferred to the end of this subsection.


Applying Corollary~\ref{cor:Gauss} then reveals that: with probability at least $1-O(n^{-11})$,  
\begin{align}
	&\sup_{\theta\in \Theta} \Big\{f_\theta(\Phi) - \Exs[f_\theta(\Phi)]\Big\}
%
	\lesssim \sqrt{\frac{(t+s)\log^2 n}{n}}. \label{eqn:Dt-concentration}
\end{align}
Taking $\theta = (\mu_t, \alpha_t,\beta_{t-1},\pi_t,\gamma_t)$
in the above inequality~\eqref{eqn:Dt-concentration} and making use of the following observation
\begin{align*}
	\Big\|\sum_{k = -2s}^{t-1} \mu_t^k\phi_k \circ \eta_{t}^{\prime}(v_t) - \frac{1}{n}\sum_{k = -2s}^{t-1} \mu_t^k\beta_{t-1}^k\eta_{t}^{\second}(v_t)\Big\|_2^2 - \kappa_t^2 - \sup_{\theta\in \Theta} \Exs[f_\theta(\Phi)]
	\leq 
	\sup_{\theta\in \Theta} \Big\{f_\theta(\Phi) - \Exs[f_\theta(\Phi)]\Big\} ,
\end{align*}
we arrive at 
\begin{align*}
	\Big\|\sum_{k = -2s}^{t-1} \mu_t^k\phi_k \circ \eta_{t}^{\prime}(v_t) - \frac{1}{n}\sum_{k = -2s}^{t-1} \mu_t^k\beta_{t-1}^k\eta_{t}^{\second}(v_t)\Big\|_2^2 - \kappa_t^2
	- 
	\sup_{\theta\in \Theta} \Exs[f_\theta(\Phi)] \lesssim   
	\sqrt{\frac{(t+s)\log^2 n}{n}}.
\end{align*}

In order to conclude the proof of \eqref{eqn:Z2-Dt}, it suffices to show that for every $\theta\in \Theta$, one has $\Exs \lt[f_{\theta}(\Phi)\rt] \le 0$.
To see this, consider any fixed $\theta$, and use $\varrho$ to denote the angle between the two unit vectors $\mu$ and $\beta$ (so that $\cos\varrho = \inprod{\mu}{\beta}$). 
Hence, one can write 
\begin{align}
	\notag &\Exs \lt[ \lt\|\sum_{k = -2s}^{t-1} \mu^k\phi_k \circ \eta^{\prime}(v) - \frac{1}{n}\sum_{k = -2s}^{t-1} \mu^k\beta^k\eta^{\second}(v)\rt\|_2^2 \rt] \\
\notag &= \Exs _{X, Y \overset{\text{i.i.d.}}{\sim} \mathcal{N}(0, I_n)}\lt[\lt\|\frac{1}{\sqrt{n}}(X\cos\varrho + Y\sin\varrho) \circ \eta^{\prime}\lt(\alpha v^{\star} + \frac{1}{\sqrt{n}}X\rt) - \frac{\cos\varrho}{n}\eta^{\second}\lt(\alpha v^{\star} + \frac{1}{\sqrt{n}}X\rt)\rt\|_2^2 \rt] \\
\notag &= \cos^2\varrho\cdot\Exs _{X \sim \mathcal{N}(0, I_n)}\lt[\lt\|\frac{1}{\sqrt{n}}X \circ \eta^{\prime}\lt(\alpha v^{\star} + \frac{1}{\sqrt{n}}X\rt) - \frac{1}{n}\eta^{\second}\lt(\alpha v^{\star} + \frac{1}{\sqrt{n}}X\rt)\rt\|_2^2 \rt] \\
\notag &\qquad+ \sin^2\varrho\cdot\Exs _{X \sim \mathcal{N}(0, I_n)}\lt[\lt\|\frac{1}{\sqrt{n}}\eta^{\prime}\lt(\alpha v^{\star} + \frac{1}{\sqrt{n}}X\rt)\rt\|_2^2 \rt] \\
&\le \kappa^2, \label{eqn:beethoven-vive}
\end{align}
where the last line follows directly from the definition of $\kappa.$ This in turn implies that $\Exs \lt[f_{\theta}(\Phi)\rt] \le 0$.

Putting the above pieces together justifies the desired inequality~\eqref{eqn:Z2-Dt}, provided that \eqref{eqn:scriabin-mys} is valid.

\paragraph{Proof of inequality~\eqref{eqn:scriabin-mys}.} 

In the sequel, let us first carry out the calculations for the dominant term --- namely, the second term of expression~\eqref{eqn:dt-mighty-five};
Note that the  $(t+s)$-th largest entry (in magnitude) of $\sum_{k = -2s}^{t-1} \mu^k\phi_k$ obeys 
\begin{subequations}
	\label{eqn:muphi-rank}
\begin{align}
	(t+s) \left| \sum_{k = -2s}^{t-1} \mu^k\phi_k\right|_{(t+1)}^2 \leq 
	\sum_{i=1}^{t+s} \left| \sum_{k = -2s}^{t-1} \mu^k\phi_k\right| _{(i)}^2 \lesssim 
	\frac{(t+s)\log n}{n},
\end{align}
which follows from the definition of the event $\mathcal{E}$ (cf.~\eqref{eqn:eps-interset}). 
This implies that
\begin{align}
	\left| \sum_{k = -2s}^{t-1} \mu^k\phi_k\right|_{(t+s)}  \leq 
	C_7\sqrt{ \frac{\log n}{n} }
\end{align}
\end{subequations}
for some large enough constant $C_7>0$. 
By virtue of \eqref{eqn:chocolate-At-Z2}, it holds that 
\begin{align*}
&\bigg\|\sum_{k = -2s}^{t-1} \mu^k\phi_k \circ \sum_{k = -2s}^{t-1} \mu^k\phi_k \circ \eta^{\prime}(v) \circ \eta^{\second}(v)\bigg\|_2 
 \lesssim \alpha\sqrt{n}\,\bigg\|\sum_{k = -2s}^{t-1} \mu^k\phi_k \circ \sum_{k = -2s}^{t-1} \mu^k\phi_k\bigg\|_2 \\
	&\qquad \lesssim \sqrt{n}\,\bigg\|\sum_{k = -2s}^{t-1} \mu^k\phi_k\bigg\|_{\infty}\bigg( \sum_{i=1}^t\Big|\sum_{k = -2s}^{t-1} \mu^k\phi_k\Big|_{(i)}^2 \bigg)^{1/2} 
	+\sqrt{n} \left|\sum_{k = -2s}^{t-1} \mu^k\phi_k\right|_{(t+1)}\bigg\|\sum_{k = -2s}^{t-1} \mu^k\phi_k\bigg\|_2 \\
	&\qquad \lesssim \sqrt{n} \sqrt{\frac{(t+s)\log n}{n}} \cdot \sqrt{\frac{(t+s)\log n}{n}} +
	\sqrt{n} \sqrt{\frac{\log n}{n}} \bigg( 1 + \sqrt{\frac{(t+s)\log n}{n}} \bigg)\\
&\qquad \lesssim \sqrt{\log n}.
\end{align*}
This leads to the desired bound for the second term of \eqref{eqn:dt-mighty-five}. 
The other terms can be bounded in a similar manner, which we omit here for brevity.

\subsection{Establishing the induction hypotheses for the next iteration}
\label{sec:xi-z2}

In this subsection, we move on to establish the induction hypotheses \eqref{Z2-induction} for the $(t+1)$-th iteration, 
in addition to controlling several intermediate quantities.  
For this purpose, Theorem~\ref{thm:main} offers a general recipe to control the residual terms $\ltwo{\xi_{t}}$ and $|\Delta_{\alpha,t}|$ 
by means of the key quantities $A_t,B_t,D_t$ that have been analyzed in Section~\ref{sec:z2-key}. 
Direct application of Theorem~\ref{thm:main} or Corollary~\ref{cor:recursion-spectral} already leads to non-asymptotic performance bounds. 
It turns out that for the problem of $\mathbb{Z}_2$ synchronization, 
we might be able to obtain tighter error bounds (i.e., $\sqrt{t/n}$ vs.~$\sqrt{t^2/n}$) 
if we slightly refine the analysis of Theorem~\ref{thm:main} by exploiting the problem-specific structure, 
which we shall detail as follows.


\subsubsection{Induction step for bounding $\ltwo{\xi_{t}}$}
\label{sec:z2-xi-t}

In this subsection, we aim to establish the induction hypothesis \eqref{eqn:Z2-induction-st} for the next iteration (namely, showing that 
$\|\xi_t\|_2\leq S_{t+1}$. %
In view of Theorem~\ref{thm:recursion-spectral} and \eqref{defi:A}, the residual term $\xi_{t}$ obeys
\begin{align}
\label{eqn:xi-z2-again}
	\|\xi_{t}\|_2 \leq \Big\langle \sum_{k = -2s}^{t-1} \mu_t^{k}\phi_k, \delta_{t}\Big\rangle - \langle\delta_{t}^{\prime}\rangle \sum_{k = -2s}^{t-1} \mu_t^{k}\beta_{t-1}^k + A_t + O\Big(\sqrt{\frac{(t+s)\log n}{n}} \Big),
\end{align} 
where $\delta_{t}$ and $\delta_{t}^{\prime}$ are defined as
\begin{align*}
	\delta_{t} &\defn \eta_{t}\Big(\alpha_t\vstar + \sum_{k = -2s}^{t-1} \beta_{t-1}^k\phi_k + \xi_{t-1}\Big) - \eta_{t}\Big(\alpha_t\vstar + \sum_{k = -2s}^{t-1} \beta_{t-1}^k\phi_k\Big), \\
 	\delta_{t}^{\prime} &\defn \eta_{t}^{\prime}\Big(\alpha_t\vstar + \sum_{k = -2s}^{t-1} \beta_{t-1}^k\phi_k + \xi_{t-1}\Big) - \eta_{t}^{\prime}\Big(\alpha_t\vstar + \sum_{k = -2s}^{t-1} \beta_{t-1}^k\phi_k\Big).
\end{align*}
We have already bounded $A_t$ in Section~\ref{sec:control-A-Z2}. As a result, it comes down to bounding $\delta_{t}$ and $\delta_{t}^{\prime}$.


As alluded to previously, we can obtain slightly tighter bounds than directly invoking Theorem~\ref{thm:main} or Corollary~\ref{cor:recursion-spectral}, 
by improving the proof of Theorem~\ref{thm:main} a little a bit with the aid of the special structure of $\mathbb{Z}_2$ synchronization. 
Specifically, recall from \eqref{eqn:chocolate} that
$$
|\eta_t(w)| \lesssim \frac{1}{\alpha_t\sqrt{n}}
\qquad 
\text{and}
\qquad |\eta_t^{\prime}(w)| \lesssim 1 \qquad \text{for any }w\in \real.
$$
These two basic bounds allow us to strengthen the \eqref{eqn:shostakovich-delta-t} and \eqref{eqn:shostakovich-delta-t-prime} as follows in the proof of Theorem~\ref{thm:main}:
\begin{subequations}
\label{eqn:delta-z2}
\begin{align}
\Big|\delta_{t}-\eta_{t}^{\prime}\Big(\alpha_{t}\vstar+\sum_{k=1}^{t-1}\beta_{t-1}^{k}\phi_{k}\Big)\circ\xi_{t-1}\Big| & \leq\rho_{1}\big|\xi_{t-1}\big|^{2},\\
\|\delta_{t}\|_{\infty} & \lesssim\frac{1}{\alpha_{t}\sqrt{n}},\\
\bigg|\delta_{t}^{\prime}-\eta_{t}^{\second}\Big(\alpha_{t}\vstar+\sum_{k=1}^{t-1}\beta_{t-1}^{k}\phi_{k}\Big)\circ\xi_{t-1}\bigg| & \leq\rho_{2}\big|\xi_{t-1}\big|^{2},\\
\|\delta_{t}^{\prime}\|_{\infty} & \lesssim1,
\end{align}
\end{subequations}
where we recall that $\Gamma = 0$ in $\mathbb{Z}_2$ synchronization (as there is no discontinuous point in $\tanh(\cdot)$).

To help further bound~\eqref{eqn:delta-z2}, we make note of some preliminary facts below. Let us introduce the following index set:
\begin{align*}
\mathcal{I} \defn \lt\{i : \bigg|\sum_{k = -2s}^{t-1} \mu_t^k\phi_{k, i}\bigg| > C_7 \sqrt{\frac{\log n}{n}}\rt\},
\end{align*}
where $C_7>0$ is a large enough constant employed in \eqref{eqn:muphi-rank}. 
By virtue of \eqref{eqn:muphi-rank}, one has 
\begin{align*}
	|\mathcal{I}|\leq t+s. 
\end{align*}
For notational simplicity, we overload the notation by introducing two vectors: 
\[
	\ind_{\mathcal{I}} \coloneqq \big[ \ind_{\mathcal{I}}(i) \big]_{1\leq i\leq n} \in \real^n
	\qquad \text{and} \qquad
	\ind_{\mathcal{I}^{\mathrm{c}}} \coloneqq \big[ \ind_{\mathcal{I}^{\mathrm{c}}}(i) \big]_{1\leq i\leq n} \in \real^n.
\]
In addition, let us define 
$$
\widehat{\xi}_{t-1} \defn \xi_{t-1} \circ \ind_{\mathcal{I}^{\mathrm{c}}}. 
$$
Based on this set of notation, we can readily derive from \eqref{eqn:delta-z2} that
\begin{subequations}
	\label{eqn:z2-delta-t}
\begin{align}
\Big|\delta_{t}-\eta_{t}^{\prime}\Big(\alpha_{t}\vstar+\sum_{k=1}^{t-1}\beta_{t-1}^{k}\phi_{k}\Big)\circ\widehat{\xi}_{t-1}\Big| & \lesssim\rho_{1}\big|\widehat{\xi}_{t-1}\big|^{2}+\frac{1}{\alpha_{t}\sqrt{n}}\ind_{\mathcal{I}},\\
\bigg|\delta_{t}^{\prime}-\eta_{t}^{\second}\Big(\alpha_{t}\vstar+\sum_{k=1}^{t-1}\beta_{t-1}^{k}\phi_{k}\Big)\circ\xi_{t-1}\bigg| & \leq\rho_{2}\big|\xi_{t-1}\big|^{2}.
\end{align}
\end{subequations}

We aim to control the right-hand side of expression~\eqref{eqn:xi-z2-again}, which boils down to bounding $\Big\langle \sum_{k = -2s}^{t-1} \mu_t^k\phi_k, \delta_{t}\Big\rangle - \langle\delta_{t}^{\prime}\rangle \sum_{k = -2s}^{t-1} \mu_t^k\beta_{t-1}^k$. 
Substitution of \eqref{eqn:z2-delta-t} into \eqref{eqn:xi-z2-again} leads to 
\begin{align}
\notag\|\xi_{t}\|_{2} & \leq\bigg\langle\bigg|\sum_{k=-2s}^{t-1}\mu_{t}^{k}\phi_{k}\bigg|,\rho_{1}\widehat{\xi}_{t-1}^{2}+\frac{1}{\alpha_{t}\sqrt{n}}\ind_{\mathcal{I}}\bigg\rangle+\bigg\langle\sum_{k=-2s}^{t-1}\mu_{t}^{k}\phi_{k},\eta_{t}^{\prime}\Big(\alpha_{t}\vstar+\sum_{k=-2s}^{t-1}\beta_{t-1}^{k}\phi_{k}\Big)\circ\widehat{\xi}_{t-1}\bigg\rangle\\
 & -\bigg\langle\eta_{t}^{\second}\Big(\alpha_{t}\vstar+\sum_{k=-2s}^{t-1}\beta_{t-1}^{k}\phi_{k}\Big)\circ\xi_{t-1}\bigg\rangle\sum_{k=-2s}^{t-1}\mu_{t}^{k}\beta_{t-1}^{k}+\lt\langle\rho_{2}\xi_{t-1}^{2}\rt\rangle\bigg|\sum_{k=-2s}^{t-1}\mu_{t}^{k}\beta_{t-1}^{k}\bigg|+A_{t}+O\Big(\sqrt{\frac{(t+s)\log n}{n}}\Big).
	\label{eqn:papageno}
\end{align}
Next, we shall control each term in \eqref{eqn:papageno} separately. 
\begin{itemize}
\item 
We begin with the first term in \eqref{eqn:papageno}. Recalling the definition of set $\mathcal{I}$ and using $\rho_1\lesssim \sqrt{n}$, we have
\begin{align*}
\Bigg\langle\bigg|\sum_{k=-2s}^{t-1}\mu_{t}^{k}\phi_{k}\bigg|,\rho_{1}\widehat{\xi}_{t-1}^{2}+\frac{1}{\alpha_{t}\sqrt{n}}\ind_{\mathcal{I}}\Bigg\rangle & \lesssim\frac{1}{\alpha_{t}\sqrt{n}}\sum_{i\in\mathcal{I}}\Big|\sum_{k=-2s}^{t-1}\mu_{t}^{k}\phi_{k,i}\Big|+ \rho_1 \sum_{i\notin\mathcal{I}}\Big|\sum_{k=-2s}^{t-1}\mu_{t}^{k}\phi_{k,i}\xi_{t-1,i}^{2}\Big|\\
 & \leq\frac{1}{\alpha_{t}\sqrt{n}}\sqrt{\big|\mathcal{I}\big|\sum_{i\in\mathcal{I}}\Big|\sum_{k=-2s}^{t-1}\mu_{t}^{k}\phi_{k,i}\Big|^{2}}+(\sqrt{n})C_{7}\sqrt{\frac{\log n}{n}}\sum_{i\notin\mathcal{I}}\xi_{t-1,i}^{2}\\
 & \lesssim\frac{1}{\alpha_{t}\sqrt{n}}\sqrt{\big|\mathcal{I}\big|\sum_{i\in\mathcal{I}}\Big|\sum_{k=-2s}^{t-1}\mu_{t}^{k}\phi_{k,i}\Big|^{2}}+ \sqrt{\log n} \,\|\xi_{t-1}\|_{2}^{2},
\end{align*}
where the second line follows from Cauchy-Schwarz and the definition
of $\mathcal{I}$. 
Recalling that $\{\phi_k\}$ fall within the set $\mathcal{E}$ (cf.~\eqref{eq:eps-set}) with high probability and using $|\mathcal{I}|\lesssim t+s$, we can further derive 
\begin{align}
\Bigg\langle \bigg| \sum_{k = -2s}^{t-1} \mu_t^k\phi_k \bigg|, \rho_1\widehat{\xi}_{t-1}^2 + \frac{1}{\alpha_t\sqrt{n}} \ind_{\mathcal{I}}\Bigg\rangle 
	&\lesssim \frac{(t+s)\sqrt{\log n}}{\alpha_tn} + \sqrt{\log n} \,\|\xi_{t-1}\|_2^2 \label{eqn:z2-part-1}
\end{align}
with probability at least $1 - O(n^{-11})$. 

\item Regarding the fourth term in \eqref{eqn:papageno}, one can use $|\mu_t^{\top}\beta_{t-1}|\leq \|\mu_t\|_2\|\beta_{t-1}\|_2=1$ and $\rho_2\lesssim n $ to get
\begin{align}
\lt\langle\rho_2\xi_{t-1}^2\rt  \rangle \bigg| \sum_{k = -2s}^{t-1} \mu_t^k\beta_{t-1}^k \bigg|
	\lesssim  \|\xi_{t-1}\|_2^2. 
\label{eqn:z2-part-2}
\end{align}
%

\item With regards to the second and third term in \eqref{eqn:papageno}, direct calculations yield  
\begin{align}
\notag &\bigg| \Big\langle \sum_{k = -2s}^{t-1} \mu_t^k\phi_k, \eta_{t}^{\prime}\Big(\alpha_t\vstar + \sum_{k = -2s}^{t-1} \beta_{t-1}^k\phi_k\Big) \circ \widehat{\xi}_{t-1}\Big\rangle - \Big\langle\eta_{t}^{\second}\Big(\alpha_t\vstar + \sum_{k = -2s}^{t-1} \beta_{t-1}^k\phi_k\Big) \circ \xi_{t-1}\Big\rangle \sum_{k = -2s}^{t-1} \mu_t^k\beta_{t-1}^k \bigg| \\
\notag &= \bigg| \Big\langle \sum_{k = -2s}^{t-1} \mu_t^k\phi_k \circ \eta_{t}^{\prime}\Big(\alpha_t\vstar + \sum_{k = -2s}^{t-1} \beta_{t-1}^k\phi_k\Big) \circ \ind_{\mathcal{I}^{\mathrm{c}}} - \frac{1}{n}\sum_{k = -2s}^{t-1} \mu_t^k\beta_{t-1}^k\eta_{t}^{\second}\Big(\alpha_t\vstar + \sum_{k = -2s}^{t-1} \beta_{t-1}^k\phi_k\Big), \xi_{t-1}\Big\rangle \bigg|  \\
\notag &\le \Big\| \sum_{k = -2s}^{t-1} \mu_t^k\phi_k \circ \eta_{t}^{\prime}\Big(\alpha_t\vstar + \sum_{k = -2s}^{t-1} \beta_{t-1}^k\phi_k\Big) - \frac{1}{n}\sum_{k = -2s}^{t-1} \mu_t^k\beta_{t-1}^k\eta_{t}^{\second}\Big(\alpha_t\vstar + \sum_{k = -2s}^{t-1} \beta_{t-1}^k\phi_k\Big)\Big\|_2 \|\xi_{t-1}\|_2 \\
\notag &\qquad+ \Big\| \sum_{k = -2s}^{t-1} \mu_t^k\phi_k \circ \eta_{t}^{\prime}\Big(\alpha_t\vstar + \sum_{k = -2s}^{t-1} \beta_{t-1}^k\phi_k\Big) \circ \ind_{\mathcal{I}}\Big\|_2 \|\xi_{t-1}\|_2 \\
	&\leq \sqrt{\kappa_t^2 + D_t}\,\|\xi_{t-1}\|_2 + O\bigg(\sqrt{\frac{(t+s)\log n}{n}}\bigg)\|\xi_{t-1}\|_2 ,
	\label{eqn:z2-part-2}
\end{align}
where we invoke the assumption~\eqref{defi:D}, the inequality \eqref{eqn:spect-brahms} and the bound~\eqref{eqn:chocolate}. 
\end{itemize}

Substituting the above bounds into inequality~\eqref{eqn:papageno} gives 
\begin{align}
\|\xi_{t}\|_2 &\le 
	\Bigg(\sqrt{\kappa_t^2 + D_t} + O\bigg(\sqrt{\frac{(t+s)\log n}{n}}\bigg)\Bigg)\|\xi_{t-1}\|_2 + O\Bigg(\sqrt{\frac{(t+s)\log n}{n}} + A_t + \frac{(t+s)\sqrt{\log n}}{\alpha_tn} + \sqrt{\log n}\,\|\xi_{t-1}\|_2^2\Bigg) \notag\\
	&\le \lt(\sqrt{\kappa_t^2 + O\bigg(\sqrt{\frac{(t+s)\log^2 n}{n}} \bigg)} + O\lt(\sqrt{\frac{(t+s)\log n}{n}} + \sqrt{\log n} \, S_t\rt)\rt)\|\xi_{t-1}\|_2 + O\lt(\sqrt{\frac{(t+s)\log n}{(\lambda- 1)n}}\rt) \notag\\
	&\le \lt( 1- \frac{1}{40}(\lambda-1) \rt)\|\xi_{t-1}\|_2 + O\lt(\sqrt{\frac{(t+s)\log n}{(\lambda - 1) n}}\rt) ; \label{eq:xit-UB-recurrence-Z2}
%
\end{align}
here, the penultimate step follows from the inequalities \eqref{eqn:Z2-At}, \eqref{eqn:Z2-Dt} and induction hypothesis \eqref{Z2-induction} for $\|\xi_{t-1}\|_2$, 
while the last line makes use of the induction hypothesis  and is valid if $\sqrt{\frac{(t+s)\log^2 n}{n}} \ll \lambda-1$ and if
\begin{equation}
	\kappa_t \leq 1 - \frac{1}{15}(\lambda - 1).
	\label{eq:kappat-UB-claim-Z2}
\end{equation}
The proof of this inequality \eqref{eq:kappat-UB-claim-Z2} is postponed to Section~\ref{sec:bound-kappat-Z2}. 
Invoking the induction the hypothesis (\ref{eqn:Z2-induction-st}) for $\|\xi_{t-1}\|_2$ 
in the above inequality \eqref{eq:xit-UB-recurrence-Z2}, we arrive at
\begin{align*}
\|\xi_{t}\|_{2} & \le\left(1-\frac{1}{40}(\lambda-1)\right)\|\xi_{t-1}\|_{2}+C_{3}\sqrt{\frac{(t+s)\log n}{(\lambda-1)n}}\\
 & \leq\left(1-\frac{1}{40}(\lambda-1)\right)\left\{ C_{1}\sqrt{\frac{(t+s)\log n}{(\lambda-1)^{3}n}}+C_{1}\left(1-\frac{1}{40}(\lambda-1)\right)^{t-1}\frac{\log^{3.5}n}{\sqrt{(\lambda-1)^{9}n}}\right\} +C_{3}\sqrt{\frac{(t+s)\log n}{(\lambda-1)n}}\\
 & =C_{1}\left(1-\frac{1}{15}(\lambda-1)\right)^{t}\frac{\log^{3.5}n}{\sqrt{(\lambda-1)^{9}n}}+\left\{ C_{1}\left(1-\frac{1}{40}(\lambda-1)\right)\sqrt{\frac{(t+s)\log n}{(\lambda-1)^{3}n}}+C_{3}(\lambda-1)\sqrt{\frac{(t+s)\log n}{(\lambda-1)^{3}n}}\right\} \\
 & \leq C_{1}\left(1-\frac{1}{15}(\lambda-1)\right)^{t}\frac{\log^{3.5}n}{\sqrt{(\lambda-1)^{9}n}}+C_{1}\sqrt{\frac{(t+s)\log n}{(\lambda-1)^{3}n}},
\end{align*}
provided that the ratio $C_{1}/C_{3}$ is sufficiently large. 
This validates the induction hypothesis (\ref{eqn:Z2-induction-st}) for $\|\xi_{t}\|_2$, 
thereby completing the induction step for $\|\xi_{t}\|_2.$

%
%
%
%
%
%
%

\subsubsection{Bounding the residual term $\Delta_{\alpha,t}$}
\label{sec:z2-delta-alpha}

Note that the denoising function in $\mathbb{Z}_2$ synchronization is smooth everywhere, 
and hence $E_t = 0$ (see \eqref{defi:E}).  
The bound \eqref{eqn:delta-alpha-general} then yields
\begin{align*}
|\Delta_{\alpha,t}| &\lesssim B_t + \lt(\rho + \rho_1\|\vstar\|_{\infty}\|\xi_{t-1}\|_2 \rt)\cdot \|\xi_{t-1}\|_2 \\
&\lesssim
	B_t + \ltwo{\xi_{t-1}}
	+
	\|\xi_{t-1}\|_2^2,
\end{align*}
where the last line follows by relation~\eqref{eqn:chocolate}. 
Recall that our induction hypothesis says $\|\xi_{t-1}\|_2 \leq S_t$, 
and that we have bounded $B_{t}$ in \eqref{eqn:Z2-Bt}.
These taken together imply that
\begin{align}
\label{eqn:beethoven-sonata}
	|\Delta_{\alpha,t}| &\lesssim  \|\xi_{t-1}\|_2 + \sqrt{\frac{(t+s)\log n}{n}} \lesssim S_t,
\end{align}
where the last inequality comes from \eqref{eqn:Z2-induction-st}.


\subsubsection{Bounding $\alpha_{t}$ and understanding state evolution}
\label{sec:main-recursion-z2}

Next, we turn to the induction step for establishing \eqref{eq:Z2-induction-alphat} and \eqref{eqn:z2-delta-alpha-final} regarding $\alpha_{t}$. 
More precisely, under the induction hypothesis~\eqref{Z2-induction} for the $t$-th iteration, 
we would like to show that \eqref{eq:Z2-induction-alphat} and \eqref{eqn:z2-delta-alpha-final} hold for the $(t+1)$-th iteration w.r.t.~$\alpha_{t+1}$. 


\begin{figure}[t]
\centering
\includegraphics[width=0.9\textwidth]{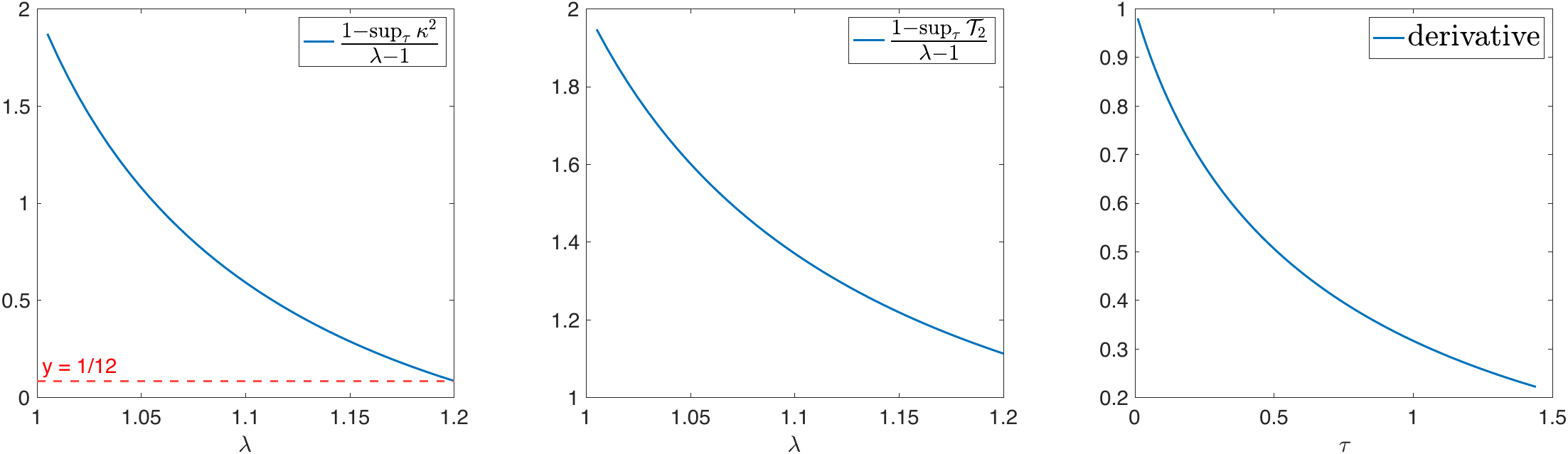}
	\caption{
	Numerical illustrations for quantities regarding $\kappa_t$ (the left panel) and $\alpha_t$ (the middle and the right panel). 
	Left panel: quantity associated with \eqref{eqn:kappa-ze-left} as a function of $\lambda$ within the range $(1,1.2]$;
	middle panel:  the quantity \eqref{eqn:middle-figure} as a function of $\lambda$ within  the range $(1,1.2]$; 
	right panel: the derivative of \eqref{eq:g-lambda-tau-defi} as a function of $\tau$ within  the range of $[0,1.44]$. 
	}
\label{fig:numerics}
\end{figure}

From the definition of $\alpha_{t+1}$ (see \eqref{eq:alpha-beta-recursion-spect}), we have
\begin{align}
	\alpha_{t+1}^2 &= \lambda^2 \langle \vstar, \eta_t(x_t)\rangle^2 = \frac{\lambda^2 \langle \vstar, \tanh(\pi_tx_t)\rangle^2}{\|\tanh(\pi_tx_t)\|_2^2}.
	\label{eq:alphat-defn-ind}
\end{align}
To understand the dynamics of $\alpha_{t}$, 
let us look at the state evolution recursion --- namely, a sequence of scalars $\{\tau_t\}$ defined recursively as follows:
\begin{subequations}
	\label{eq:SE}
\begin{align}
	\tau_{1} &= \lambda^{2}-1 \\
	\tau_{t+1} &\defn \frac{\lambda^2 \lt[\int \tanh(\tau_t + \sqrt{\tau_t}x)\varphi(\dx)\rt]^2}{\int \tanh^2(\tau_t + \sqrt{\tau_t}x)\varphi(\dx)} = \lambda^2 \int \tanh(\tau_t + \sqrt{\tau_t}x)\varphi(\dx). 
	\label{eq:SE-t-t}
\end{align}
\end{subequations}
Here, the last line comes from \citet[Appendix B.2]{deshpande2017asymptotic}. 
As it turns out, this scalar sequence \eqref{eq:SE} converges monotonically to a fixed point $\tau^{\star}$ of the recursion \eqref{eq:SE}, 
namely,   
\begin{align}
 \tau_t \nearrow \tau^{\star},
	\qquad \tau_t,\tau^{\star} \in (\lambda^2 - 1 , \lambda^2),
	\qquad\text{ where } \tau^{\star} \text{ obeys }
\tau^{\star} = \lambda^2 \int \tanh(\tau^{\star} + \sqrt{\tau^{\star}}x)\varphi(\dx).
	\label{eq:tau-t-monotone}
\end{align}
This claim can be established as follows by studying the property of the function 
\begin{equation}
	g_{\lambda}(\tau) = \frac{\lambda^2 \lt[\int \tanh(\tau + \sqrt{\tau}x)\varphi(\dx)\rt]^2}{\int \tanh^2(\tau + \sqrt{\tau}x)\varphi(\dx)}
	 = \lambda^2 \int \tanh^2(\tau + \sqrt{\tau}x)\varphi(\dx).
	 \label{eq:definition:glambda-tau}
\end{equation}
\begin{itemize}
\item[(i)] We first observe that, for any $1<\lambda \leq 1.2$, the following derivative 
\begin{align}
	\label{eq:g-lambda-tau-defi}
	\frac{1}{\lambda^2}\cdot\frac{\mathrm{d} g_{\lambda}(\tau)}{\mathrm{d}\tau} =
	\frac{\mathrm{d}}{\mathrm{d} \tau} \int \tanh(\tau + \sqrt{\tau}x)\varphi(\dx) = \int \lt(1 + \frac{x}{2\sqrt{\tau}}\rt)\lt(1 - \tanh^2(\tau + \sqrt{\tau}x)\rt)\varphi(\dx)
\end{align}
always obeys 
\begin{equation}
	\frac{\mathrm{d} g_{\lambda}(\tau)}{\mathrm{d}\tau} \in (0,\lambda^2)
	\label{eq:g-lambda-tau-01}
\end{equation}
and is decreasing in $\tau$ within the interval $\tau \in [\lambda^2-1, \lambda^2] \subseteq [0, 1.44]$ 
(given our assumption that $1<\lambda \leq 1.2$); this is numerically validated in the right panel of Figure~\ref{fig:numerics}.

\item[(ii)] Secondly, we observe that $g_{\lambda}(\tau_1)>\tau_1$, where $\tau_1 = \lambda^2 - 1$. 
	To prove this, consider the problem of estimating a Bernoulli random variable $X \sim \mathsf{Bernoulli}(1/2) \in \mathcal\{1,-1\}$ 
	based on the observation $Y=\sqrt{\tau} X + Z$, where $Z\sim \mathcal{N}(0,1)$ is independent from $X$. 
		It is well known that $\tanh(\sqrt{\tau}Y)=\tanh(\tau X+ \sqrt{\tau}Z) = \mathbb{E}[X\mid Y]$ 
		is the minimum mean square error (MMSE) estimator \citep[Appendix B.2]{deshpande2017asymptotic}. 
		In addition, the MMSE estimator $\mathbb{E}[X\mid Y]$ is known to be the projection of $Y$ onto the space of functions of $Y$, 
		and as a result, it achieves the largest correlation with $X$ among all estimators based on $Y$. 
		This implies that the estimator $\tanh(\tau X+ \sqrt{\tau}Z)$ enjoys higher correlation with $X$ compared to the other estimator 
		$\tau X+ \sqrt{\tau}Z$, thus leading to
\begin{align*}
	\frac{\lt[\int \tanh(\tau_1 + \sqrt{\tau_1}x)\varphi(\dx)\rt]^2}{\int \tanh^2(\tau_1 + \sqrt{\tau_1}x)\varphi(\dx)} 
	\geq
	\frac{\lt[\int (\tau_1 + \sqrt{\tau_1}x)\varphi(\dx)\rt]^2}{\int (\tau_1 + \sqrt{\tau_1}x)^2\varphi(\dx)} = \frac{\tau_1^2}{\tau_1^2 + \tau_1} = \frac{\lambda^2 - 1}{\lambda^2}.
\end{align*}
Given that the left-hand side of the above relation is given by $\frac{g_{\lambda}(\tau_1)}{\lambda^2}$ (cf.~\eqref{eq:SE-t-t}), we conclude that
\[
	\frac{g_{\lambda}(\tau_1)}{\lambda^2} \geq \frac{\lambda^2 - 1}{\lambda^2} 
		\qquad \Longrightarrow \qquad
		g_{\lambda}(\tau_1) \geq \lambda^2 - 1 = \tau_1. 
\]

\item[(iii)] Thirdly, it is seen that $g_{\lambda}(\lambda^2) < \lambda^2$, which follows from \eqref{eq:definition:glambda-tau} and the fact that $|\tanh(w)|<1$ $(w\in \real)$. 

\item[(iv)] The above three properties immediately reveal that:  
	\begin{itemize}
		\item[(a)] There exists a unique fixed point $\tau^{\star}$ of $g_{\lambda}(\cdot)$ within $(\lambda^2 - 1, \lambda^2)$; 
		\item[(b)] Starting from $\tau_1=\lambda^2 - 1$, $\tau_t$ is monotonically increasing in $t$ and keeps moving closer to (but remains below) $\tau^{\star}$. To see this, note that for any $\tau_t< \tau^{\star}$, one has $\tau_{t+1}=g_{\lambda}(\tau_t) \leq g_{\lambda}(\tau^{\star}) = \tau^{\star}$ and $\tau_{t+1}=g_{\lambda}(\tau_t) > \tau_t$ (as $\tau_t<\tau^{\star}$ and $\tau_1<g_{\lambda}(\tau_1)$). 
	\end{itemize}

%
\end{itemize}
%

With the state evolution sequence $\{\tau_t\}$ in place, we claim that for every $t$, it satisfies  
\begin{align}
	\alpha_{t+1}^2 = \left(1 + O\Big(\frac{\widetilde{S}_{t+1}}{(\lambda-1)^{2.5}} \Big)\right)\tau_{t+1} 
	= \big( 1 + o(1) \big) \tau_{t+1}, 
	\label{eq:SE-induction}
\end{align}
where $\widetilde{S}_t$ is defined in \eqref{eqn:crude-st}. 
If the claim \eqref{eq:SE-induction} were valid, then one could readily conclude that 
\begin{align*}
 	(1 + o(1))\lambda^2
	&\geq \alpha_{t+1}^2 = \left(1 + o(1)\right)\tau_{t+1} 
 	\geq 
 	(1 + o(1)) \tau_{1} = (1 + o(1)) (\lambda^2 - 1), \\
	\alpha_{t+1}^{2}&=\left(1+O\bigg(\sqrt{\frac{t\log n}{(\lambda-1)^{8}n}}+\frac{\log^{3.5}n}{\sqrt{(\lambda-1)^{14}n}}\bigg)\right)\tau_{t+1},
 \end{align*} 
where we have used \eqref{eq:tau-t-monotone} and the definition \eqref{eqn:crude-st} of $\widetilde{S}_t$. 
Consequently, if we can establish inequality~\eqref{eq:SE-induction}, we can finish the inductive step with respect to $\alpha_{t}$.

\paragraph{Proof of claim \eqref{eq:SE-induction}.}
We intend to accomplish this via an induction argument. 
Assuming that \eqref{eq:SE-induction} is valid for the $t$-th iteration, 
we would like to establish \eqref{eq:SE-induction} for the $(t+1)$-th iteration as well.  
Let
\begin{align}
	\varsigma_{t} &\defn \alpha_{t}^2 - \tau_{t}, \qquad t\geq 1,
\end{align}
then it is equivalent to proving that 
\begin{equation}
	|\varsigma_{t+1}| \leq \left(\frac{C_{6}\widetilde{S}_{t+1}}{(\lambda-1)^{2.5}}\right)\tau_{t+1}
	\label{eq:varsigma-t1-Z2}
\end{equation}
for some constant $C_6>0$ large enough, 
provided  that 
\begin{equation}
	|\varsigma_{t}| \leq \lt(\frac{C_{6}\widetilde{S}_{t}}{(\lambda-1)^{2.5}}\rt)\tau_{t}.
	\label{eq:varsigma-t-Z2}
\end{equation}

Towards this end, let us define the following quantity: 
\begin{align*}
\mathcal{T}_1 \defn \frac{\lambda^2 \langle \vstar, \tanh(\pi_tx_t)\rangle^2}{\|\tanh(\pi_tx_t)\|_2^2} - \frac{\lambda^2 \lt[\int \tanh(\alpha_t^2 + \alpha_tx)\varphi(\dx)\rt]^2}{\int \tanh^2(\alpha_t^2 + \alpha_tx)\varphi(\dx)}.
\end{align*}
We can then employ \eqref{eq:alphat-defn-ind} and \eqref{eq:SE} to derive
\begin{align}
\varsigma_{t+1} \defn \alpha_{t+1}^2 - \tau_{t+1} 
&= \lambda^2 \int \tanh(\alpha_t^2 + \alpha_tx)\varphi(\dx) - \lambda^2 \int \tanh(\tau_t + \sqrt{\tau_t}x)\varphi(\dx) + \mathcal{T}_1 \notag\\
&= 
	\varsigma_t\, \underbrace{\lambda^2  \int \lt(1 - \tanh^2(\tau_t + \sqrt{\tau_t}x)\rt)\lt(1 + \frac{1}{2\sqrt{\tau_t}}x\rt)\varphi(\dx)}_{=:\mathcal{T}_2}  + O\lt(\frac{\varsigma_t^2}{\tau_t^{3/2}}\rt) + \mathcal{T}_1,
	\label{eq:vartau-identity}
\end{align}
where the last identity shall be established towards the end of this subsection. 
In what follows, we shall look at $\mathcal{T}_1$ and $\mathcal{T}_2$ separately.

\begin{itemize}
\item 
To control $\mathcal{T}_{2}$, 
we first observe that $\mathcal{T}_2 \geq 0$, 
a direct consequence of \eqref{eq:g-lambda-tau-01} and \eqref{eq:g-lambda-tau-defi}. 
In addition,  we claim that, 
for any $\tau \ge \lambda^2 - 1$ and any $\lambda \in [1,1.2]$,  
\begin{align}
\label{eqn:middle}
	0\leq \mathcal{T}_2 (\lambda,\tau) \defn \lambda^2 \int \lt(1 - \tanh^2(\tau + \sqrt{\tau}x)\rt)\lt(1 + \frac{1}{2\sqrt{\tau}}x\rt)\varphi(\dx) \le 1 - (\lambda-1).
\end{align}
To see this, 
we resort to the numerical verification. To be specific,   
the middle panel of Figure~\ref{fig:numerics} plots the following quantity 
\begin{align}
\label{eqn:middle-figure}
	\frac{1 - \sup_{\tau: \lambda^2 -1 \leq \tau \leq \lambda^2} \mathcal{T}_2(\lambda, \tau)}{\lambda -1}
\end{align}
as a function of $\lambda$; it is clearly seen from  Figure~\ref{fig:numerics} that this ratio is strictly above 1 for any $\lambda \in [1,1.2].$ 
All this indicates that
\[
	0 \leq \mathcal{T}_2\leq 1 - (\lambda - 1). 
\]

\item Next, we turn attention to $\mathcal{T}_1$.  Repeating the same argument as in \eqref{eq:tanh2-diff-123} and \eqref{eq:tanh2-diff-456} 
	and recognizing that $|\tanh^{\prime}(w)|\leq 1$, we have 
\[
	\bigg|{\displaystyle \int}\tanh\Big(\frac{\pi_{t}}{\sqrt{n}}(\alpha_{t}+x)\Big)\varphi(\mathrm{d}x)-{\displaystyle \int}\tanh\big(\alpha_{t}(\alpha_{t}+x)\big)\varphi(\mathrm{d}x)\bigg|\lesssim \frac{S_t}{\alpha_t} .
\]
This taken together with the definition of $\Delta_{\alpha, t}$ (cf.~\eqref{eqn:alpha-t-genearl}),  the definition of $\alpha_{t+1}$ (cf.~\eqref{eq:alpha-beta-recursion-spect}) and the fact that $\sqrt{n}v^{\star}_i\in \{1,-1\}$ ($1\leq i\leq n$) gives 
\begin{align*}
\langle \vstar, \tanh(\pi_tx_t)\rangle &= \sqrt{n}\int \tanh\lt(\frac{\pi_t}{\sqrt{n}}(\alpha_t + x)\rt)\varphi(\dx) + O\lt(\gamma_t^{-1} \big|\Delta_{\alpha, t}\big| \rt) \\
&= \sqrt{n}\int \tanh(\alpha_t^2 + \alpha_tx)\varphi(\dx) + O\lt( \frac{S_t}{\alpha_t} \sqrt{n} + \alpha_t S_t \sqrt{n} \rt) \\
	&= \sqrt{n}\int \tanh(\alpha_t^2 + \alpha_tx)\varphi(\dx) + O\lt( \frac{S_t}{\alpha_t} \sqrt{n}  \rt), 
\end{align*}
where the second line is due to \eqref{eqn:beethoven-sonata} and \eqref{eqn:giovanni}, 
and the last line is valid since $\alpha_t\lesssim \lambda \leq 1$ (cf.~\eqref{eq:Z2-induction-alphat}). 
Additionally,  \eqref{eqn:giovanni-2} tells us that
\begin{align*}
	\|\tanh(\pi_tx_t)\|_2^2 = n\int \tanh^2(\alpha_t^2 + \alpha_tx)\varphi(\dx) + O\left(\alpha_{t}^{2}n\bigg(\frac{S_{t}}{\alpha_{t}^{3}}\bigg)\right).
\end{align*}

Taking these two relations collectively with \eqref{eqn:tanh-basic-alpha} and \eqref{eq:int-alpha2-tanh-123} ensures that
\begin{align*}
	|\mathcal{T}_1| \lesssim \lambda^2 \alpha_t^2 \lt(  \frac{S_{t}}{\alpha_{t}^{3}} \rt) \asymp \lambda^2   \frac{S_{t}}{\alpha_{t}} .
\end{align*}
\end{itemize}

Putting the above bounds together, we arrive at
\[
|\varsigma_{t+1}| \leq\big(1-(\lambda-1)\big) |\varsigma_{t}| + 
O\bigg(\frac{\varsigma_{t}^{2}}{\tau_{t}^{3/2}}\bigg) + O\lt( \lambda^2   \frac{S_{t}}{\alpha_{t}} \rt). 
\]
Given that $\tau_t$ is increasing in $t$ (see \eqref{eq:tau-t-monotone}), there exists some large enough constant $C_8>0$ such that
\begin{align*}
\frac{|\varsigma_{t+1}|}{\tau_{t+1}} & \leq\big(1-(\lambda-1)\big)\frac{|\varsigma_{t}|}{\tau_{t}}+O\bigg(\frac{\varsigma_{t}^{2}}{\tau_{t}^{5/2}}\bigg)+O\lt(\lambda^{2}\frac{S_{t}}{\alpha_{t}\tau_{t}}\rt)\\
 & \leq\big(1-(\lambda-1)\big)\frac{|\varsigma_{t}|}{\tau_{t}}+\frac{C_{8}}{\sqrt{\lambda-1}}\frac{\varsigma_{t}^{2}}{\tau_{t}^{2}}+\frac{C_{8}S_{t}}{(\lambda-1)^{1.5}}\\
 & \leq\big(1-(\lambda-1)\big)\left\{ \frac{C_{6}\widetilde{S}_{t}}{(\lambda-1)^{2.5}}\right\} +\left(C_{8}\frac{C_{6}\widetilde{S}_{t}}{(\lambda-1)^{3}}\right)\frac{C_{6}\widetilde{S}_{t}}{(\lambda-1)^{2.5}}+\frac{C_{8}\widetilde{S}_{t}}{(\lambda-1)^{1.5}}\\
 & \leq\frac{C_{6}\widetilde{S}_{t}}{(\lambda-1)^{2.5}}\leq\frac{C_{6}\widetilde{S}_{t+1}}{(\lambda-1)^{2.5}},
\end{align*}
where the second line holds since $\alpha_t^2=(1+o(1))\tau_t$ and $\tau_t\geq \lambda^2 - 1 \asymp \lambda - 1$ (cf.~\eqref{eq:tau-t-monotone}), 
the third line relies on \eqref{eq:varsigma-t-Z2} and $\tau_t\gtrsim \lambda - 1$, 
and the last line is valid provided that $\frac{\widetilde{S}_{t}}{(\lambda-1)^{4}}\ll1$. 
This in turn establishes \eqref{eq:varsigma-t1-Z2} for the $(t+1)$-th iteration. 

\paragraph{Proof of relation \eqref{eq:vartau-identity}.}
We first make the observation that
\begin{align}
 & \Big|\tanh(\alpha_{t}^{2}+\alpha_{t}x)-\tanh(\tau_{t}+\sqrt{\tau_{t}}x)-\big(1-\tanh^{2}(\tau_{t}+\sqrt{\tau_{t}}x)\big)\Big(\alpha_{t}^{2}+\alpha_{t}x-\tau_{t}-\sqrt{\tau_{t}}x\Big)\Big|\nonumber \\
 & \qquad\leq\frac{1}{2}(\alpha_{t}^{2}+\alpha_{t}x-\tau_{t}-\sqrt{\tau_{t}}x)^{2},\label{eq:tanh-diff-135}
\end{align}
which follows due to Taylor expansion and the fact that $\tanh^{\prime}(w)=1-\tanh^{2}(w)\in[0,1]$
for any $w\in\mathbb{R}$. In addition, one has
\begin{align*}
\alpha_{t}^{2}+\alpha_{t}x-\tau_{t}-\sqrt{\tau_{t}}x & =(\alpha_{t}^{2}-\tau_{t})+\frac{\alpha_{t}^{2}-\tau_{t}}{\alpha_{t}+\sqrt{\tau_{t}}}x=\varsigma_{t}\Big(1+\frac{1}{2\sqrt{\tau_{t}}}\Big)x+\varsigma_{t}\Big(\frac{1}{\alpha_{t}+\sqrt{\tau_{t}}}-\frac{1}{2\sqrt{\tau_{t}}}\Big)x\\
 & =\varsigma_{t}\Big(1+\frac{1}{2\sqrt{\tau_{t}}}\Big)x-\Big(\frac{\varsigma_{t}^{2}}{2(\alpha_{t}+\sqrt{\tau_{t}})^{2}\sqrt{\tau_{t}}}\Big)x,
\end{align*}
which further implies that
\begin{align*}
	&\qquad \Big|\big(\alpha_{t}^{2}+\alpha_{t}x-\tau_{t}-\sqrt{\tau_{t}}x\big)-\varsigma_{t}\Big(1+\frac{1}{2\sqrt{\tau_{t}}}\Big)x\Big|  \leq\frac{\varsigma_{t}^{2}}{2\tau_{t}^{3/2}}|x| \\
	&\Longrightarrow \qquad \big(\alpha_{t}^{2}+\alpha_{t}x-\tau_{t}-\sqrt{\tau_{t}}x\big)^{2}  \leq2\varsigma_{t}^{2}\Big(1+\frac{1}{2\sqrt{\tau_{t}}}\Big)^{2}x^{2}+\frac{\varsigma_{t}^{4}}{2\tau_{t}^{3}}x^{2}.
\end{align*}
Substituting the preceding two inequalities into (\ref{eq:tanh-diff-135}) yields
\begin{align*}
 & \Big|\tanh(\alpha_{t}^{2}+\alpha_{t}x)-\tanh(\tau_{t}+\sqrt{\tau_{t}}x)-\big(1-\tanh^{2}(\tau_{t}+\sqrt{\tau_{t}}x)\big)\varsigma_{t}\Big(1+\frac{1}{2\sqrt{\tau_{t}}}\Big)x\Big|\\
 & \qquad\leq\frac{\varsigma_{t}^{2}}{2\tau_{t}^{3/2}}|x|+\varsigma_{t}^{2}\Big(1+\frac{1}{2\sqrt{\tau_{t}}}\Big)^{2}x^{2}+\frac{\varsigma_{t}^{4}}{4\tau_{t}^{3}}x^{2}.
\end{align*}
Taking the integral and using the facts that $\tau_t\leq \lambda^2\lesssim 1$ (cf.~\eqref{eq:tau-t-monotone}) and the induction hypothesis $|\varsigma_t|\lesssim \tau_t$ then give
\begin{align*}
 & \left|{\displaystyle \int}\tanh(\alpha_{t}^{2}+\alpha_{t}x)\varphi(\mathrm{d}x)-{\displaystyle \int}\tanh(\tau_{t}+\sqrt{\tau_{t}}x)\varphi(\mathrm{d}x)-\varsigma_{t}\Big(1+\frac{1}{2\sqrt{\tau_{t}}}\Big){\displaystyle \int}\big(1-\tanh^{2}(\tau_{t}+\sqrt{\tau_{t}}x)\big)\varphi(\mathrm{d}x)\right|\\
 & \quad\leq{\displaystyle \int}\left\{ \frac{\varsigma_{t}^{2}}{2\tau_{t}^{3/2}}|x|+\varsigma_{t}^{2}\Big(1+\frac{1}{2\sqrt{\tau_{t}}}\Big)^{2}x^{2}+\frac{\varsigma_{t}^{4}}{4\tau_{t}^{3}}x^{2}\right\} \varphi(\mathrm{d}x)
 \lesssim\frac{\varsigma_{t}^{2}}{\tau_{t}^{3/2}}+\frac{\varsigma_{t}^{2}}{\tau_{t}}+\frac{\varsigma_{t}^{4}}{\tau_{t}^{3}} 
	\asymp \frac{\varsigma_{t}^{2}}{\tau_{t}^{3/2}}. 
\end{align*}

\subsubsection{Bounding quantity $\kappa_t$} \label{sec:bound-kappat-Z2}

Recall that the analysis in Section~\ref{sec:z2-xi-t} requires bounding  $\kappa_t$, which shall be done in this subsection with the assistance of expression \eqref{eq:SE-induction}.
First, combining~\eqref{eqn:don} and~\eqref{eqn:giovanni} leads to
\begin{align}
\gamma_t^2\pi_t^2 = \frac{(1 + o(\lambda - 1))\alpha_t^2}{\int \tanh(\alpha_t^2 + \alpha_tx)\varphi(\dx)} 
	= \frac{(1 + o(\lambda -1)) \alpha_t^2}{(1 + o(\lambda -1)) \tau_{t+1} / \lambda^2}
	\leq \frac{(1 + o(\lambda -1)) \lambda^2 \alpha_t^2}{(1 + o(\lambda -1)) \tau_{t} }
	= (1 + o(\lambda -1))\lambda^2,
	\label{eq:prod-gamma-pi-lambda}
\end{align}
provided that $\frac{\widetilde{S}_t}{(\lambda-1)^{3.5}}\ll 1$. 
Here, the second identity holds due to~\eqref{eq:SE}, \eqref{eq:varsigma-t1-Z2} and \eqref{eq:vartau-identity}, 
 the inequality is valid since $\tau_t$ is increasing in $t$ (see \eqref{eq:tau-t-monotone}),  
and the last identity comes from \eqref{eq:SE-induction}. 
Recalling the definition of $\kappa_{t}$ (cf.~\eqref{defi:kappa}) and the fact that $\ltwo{\beta_{t-1}} = 1$ gives 
\begin{align}
 \kappa_t^2 = 
\max\Bigg\{
\Bigg\langle\int\Bigg[x\eta_{t}^{\prime}\Big(\alpha_t\vstar + \frac{1}{\sqrt{n}}x\Big)  - \frac{1}{\sqrt{n}}\eta_{t}^{\second}\Big(\alpha_t\vstar &+ \frac{1}{\sqrt{n}}x\Big)\Bigg]^2 \varphi_n(\dx)\Bigg\rangle, ~
\bigg\langle
\int\Big[\eta_{t}^{\prime}\Big(\alpha_t\vstar + \frac{1}{\sqrt{n}}x\Big)\Big]^2\varphi_n(\dx)
\bigg\rangle\Bigg\}.
\label{eq:kappat-square-135}
\end{align}
In what follows, let us control each term in \eqref{eq:kappat-square-135} separately. 
\begin{itemize}
\item
To begin with, in view of the relations~\eqref{eqn:super-basic}, 
we obtain 
\begin{align}
\notag & \Bigg\langle\int\Bigg[x\eta_{t}^{\prime}\Big(\alpha_{t}\vstar+\frac{1}{\sqrt{n}}x\Big)-\frac{1}{\sqrt{n}}\eta_{t}^{\second}\Big(\alpha_{t}\vstar+\frac{1}{\sqrt{n}}x\Big)\Bigg]^{2}\varphi_{n}(\dx)\Bigg\rangle\\
	\notag & =\frac{1}{n}\int\lt[\lt(\gamma_{t}\pi_{t}x+\frac{2}{\sqrt{n}}\gamma_{t}\pi_{t}^{2}\tanh\Big(\pi_{t}\Big(\alpha_{t}\vstar+\frac{1}{\sqrt{n}}x\Big)\Big)\rt)\cdot\Big(1-\tanh^{2}\Big(\pi_{t}\Big(\alpha_{t}\vstar+\frac{1}{\sqrt{n}}x\Big)\Big)\Big)\rt]^{2}\varphi_{n}(\dx)\\
 & =\int\lt[\lt(\gamma_{t}\pi_{t}x+\frac{2}{\sqrt{n}}\gamma_{t}\pi_{t}^{2}\tanh\Big(\frac{\pi_{t}}{\sqrt{n}}\big(\alpha_{t}+x\big)\Big)\rt)\cdot\Big(1-\tanh^{2}\Big(\frac{\pi_{t}}{\sqrt{n}}\big(\alpha_{t}+x\big)\Big)\Big)\rt]^{2}\varphi(\dx), 
	\label{eq:w-inner-prod-identity}
\end{align}
where the last step follows from $\sqrt{n}\vstar_{i} \in \{+1, -1\}$ and the symmetry of $\varphi(\cdot).$
Reorganizing terms and recalling that $\pi_t = (1 + o(\lambda -1))\alpha_t\sqrt{n}$ (cf.~\eqref{eqn:don}) and $\gamma_t^2\pi_t^2 \le (1 + o(\lambda -1))\lambda^2$ (cf.~\eqref{eq:prod-gamma-pi-lambda}), we arrive at
\begin{align*}
\eqref{eq:w-inner-prod-identity} &= (1+o(\lambda -1))\gamma^2_t \pi^2_t \int \lt[\lt(x + 2\alpha_t\tanh(\alpha_t^2+\alpha_t x)\rt)\lt(1-\tanh^2(\alpha_t^2+\alpha_t x)\rt)\rt]^2 \varphi(\dx) \\
&\le (1+o(\lambda -1))\lambda^2\int \lt[\lt(x + 2\sqrt{\tau_t}\tanh(\tau_t+\sqrt{\tau_t} x)\rt)\lt(1-\tanh^2(\tau_t+\sqrt{\tau_t} x)\rt)\rt]^2 \varphi(\dx),
\end{align*}
where the last line also relies on the relation \eqref{eq:SE-induction}.
As a result, we reach
\begin{align}
	\notag &\Bigg\langle\int\Bigg[x\eta_{t}^{\prime}\big(\alpha_t\vstar + \frac{1}{\sqrt{n}}x\big)  - \frac{1}{\sqrt{n}}\eta_{t}^{\second}\big(\alpha_t\vstar + \frac{1}{\sqrt{n}}x\big)\Bigg]^2 \varphi_n(\dx)\Bigg\rangle \\
	&\qquad \leq 
	(1+o(\lambda -1))\lambda^2\int \lt[\lt(x + 2\sqrt{\tau_t}\tanh(\tau_t+\sqrt{\tau_t} x)\rt)\lt(1-\tanh^2(\tau_t+\sqrt{\tau_t} x)\rt)\rt]^2 \varphi(\dx). \label{eqn:vienna}
\end{align}

\item 
	Through similar calculations (for which we omit the details here), one can deduce that
\begin{align}
\label{eqn:Musikverein}
	&\Big\langle
\int\Big[\eta_{t}^{\prime}\Big(\alpha_t\vstar + \frac{1}{\sqrt{n}}x\big)\Big]^2\varphi_n(\dx) \Big\rangle
\le (1+o(\lambda -1)) \lambda^2\int \lt[1-\tanh^2(\tau_t+\sqrt{\tau_t} x)\rt]^2 \varphi(\dx). 
\end{align}

\end{itemize}

Finally, 
let us look at the following function:
\begin{align}
&{\kappa}^2(\lambda, \tau) \defn \lambda^2 \max \bigg\{\int \Big[\lt(x + 2\sqrt{\tau}\tanh(\tau+\sqrt{\tau} x)\rt)\lt(1-\tanh^2(\tau+\sqrt{\tau} x)\rt)\Big]^2 \varphi(\dx), \int \lt[1-\tanh^2(\tau+\sqrt{\tau} x)\rt]^2 \varphi(\dx)\bigg\}.
	\label{eq:defn-kappa2-lambda-tau}
\end{align}
For any $1 < \lambda < 1.2$, we observe that 
\begin{align}
\label{eqn:kappa-ze-left}
&\sup_{\tau: \lambda^2-1 \le \tau \le \lambda^2}\kappa(\lambda, \tau) \le 1 - \frac{\lambda-1}{12} ,
\end{align}
which has been numerically validated in the left panel of Figure~\ref{fig:numerics}. 
Thus, putting the above results together, we have demonstrated the advertised bound for $\kappa_{t}$:   
\begin{align}
\label{eqn:kappa-Z2}
\kappa_t^2 \leq (1+o(1))\sup_{\tau: \lambda^2-1 \le \tau \le \lambda^2}\kappa(\lambda, \tau)  \le 1 - \frac{\lambda-1}{15}.
\end{align}

\section{Initialization for sparse PCA}
\label{sec:init-sparse-pca}

This section is dedicated to the study of AMP with two data-driven initialization schemes that achieve non-trivial correlation with the truth, with a focus on the scenario where the SNR is above the computational limit.

\subsection{AMP with data-dependent initialization: strong SNR regime}
\label{sec:result-sparse-init}

Let us begin by considering the strong SNR regime where
\begin{align}
\label{eqn:init-1}
	\lambda\|v^{\star}\|_{\infty} \gtrsim \sqrt{\frac{k\log n}{n}}.
\end{align}
For instance, if $\|\vstar\|_\infty = O(\frac{1}{\sqrt{k}})$, 
then \eqref{eqn:init-1} imposes a constraint on $\lambda$ as $\lambda \gtrsim
 k\sqrt{\frac{\log n}{n}}$.

\paragraph{Initialization scheme \#1: diagonal maximization.}

Set $\eta_0(x_{0}) = 0$, and take
\begin{align}
	x_1 = e_{\hat{s}},\qquad\text{with }~ {\hat{s}} \defn \arg\max_i \lt|M_{ii}\rt|, 
	\label{defi:v1}
\end{align}
where $e_i\in \real^n$ denotes the $i$-th standard basis vector. 
In words, this initialization simply identifies the largest diagonal entry of $M$, and forms a standard basis vector w.r.t.~this entry. 
Given the ambiguity of the global sign (i.e., one can only hope to recover $\vstar$ up to global sign), 
we shall assume $\vstar_{\hat{s}} \ge 0$ without loss of generality.

As it turns out, in the strong SNR regime \eqref{eqn:init-1}, 
the algorithm \eqref{defi:v1} is guaranteed to find an index within the following index subset: 
\begin{align}
\label{eqn:sparse-set-s0}
	\mathcal{S}_0 \defn \Big\{s \in [n] ~\mid~ |\vstar_s| \geq \frac{1}{2} \|\vstar\|_\infty \Big\}.
\end{align}
Moreover, executing one iteration of AMP from $x_1=e_{\hat{s}}$ is able to yield a nontrivial correlation with the truth $\vstar$.  
These two facts are formally stated in the following proposition, with its proof deferred to Section~\ref{sec:pf-sparse-ini}.
\begin{props} 
\label{thm:sparse-init}
Suppose the signal strength satisfies~\eqref{eqn:init-1}. 
With probability at least $1-O(n^{-11})$, one has:
\begin{enumerate}

\item the index ${\hat{s}}$ (cf.~\eqref{defi:v1}) satisfies $\hat{s} \in \mathcal{S}_{0}$; 

\item for every $s \in \mathcal{S}_{0}$, the AMP updates~\eqref{eqn:AMP-updates} initialized with $x_1=e_{s}$   
and $\eta_0(x_0)=0$ obey 
\begin{align}
\label{eqn:vstar-s}
	\big|\big\langle v^{\star}, \eta_2(x_2) \big\rangle \big| \asymp 1. 
\end{align}

\end{enumerate}
%
\end{props}

\paragraph{Non-asymptotic theory of AMP when initialized by \eqref{defi:v1}.} 
Despite the statistical dependency between $\hat{s}$ and $W$, 
Proposition~\ref{thm:sparse-init} guarantees that it always comes from a fixed and small index subset. 
Consequently, basic union bounding suffices in helping us analyze  the subsequent AMP iterates. 
This is summarized in the result below;  
the proof can be found in Section~\ref{sec:pf-sparse}.
\begin{cors}
\label{cor:sparse-init-1}
If the signal strength satisfies \eqref{eqn:init-1}, 
	then with probability at least $1-O(n^{-10})$, the AMP iterates~\eqref{eqn:AMP-updates} with initialization \eqref{defi:v1} obey \eqref{eqn:sparse-decomp-dvorak} - \eqref{eqn:soccer} for $2\leq t\lesssim \frac{n\lambda^2}{\log^3 n}$, 
where $\alpha_3^{\star} \asymp \lambda$. 
\end{cors}



\subsection{AMP with data-dependent initialization:  weak signal regime}
\label{sec:result-sparse-init-weak}

We now move on to the following regime that violates  the condition~\eqref{eqn:init-1}: 
\begin{align}
\label{eqn:init-2}
	\lambda \gtrsim \frac{k}{\sqrt{n}}
~~\text{ and }~~\|v^{\star}\|_{\infty} = o\Big(\sqrt{\frac{\log n}{k}}\Big).
\end{align}
It is noteworthy that $\lambda$ cannot be further reduced, 
as a computational barrier has been widely conjectured that asserts that no polynomial algorithm can achieve consistent estimation if $\lambda = o(k/{\sqrt{n}})$ 
\citep{berthet2013computational,cai2015optimal,wang2016statistical,hopkins2017power}.

\paragraph{Initialization scheme \#2.}  
Before describing our next initialization scheme, 
we give two remarks below.
\begin{itemize}
	\item As shown in the prior literature, there exists a computationally feasible algorithm that allows one to find an estimate $\widehat{v}_{\mathsf{oracle}}$ 
		that obeys $\|\widehat{v}_{\mathsf{oracle}}\|_2=1$ and 
		\begin{equation}
			\big|\big\langle \widehat{v}_{\mathsf{oracle}}, \vstar \big\rangle \big| \asymp 1 
			\label{eq:correlation-oracle}
		\end{equation}
		with probability exceeding $1-O(n^{-10})$, 
		as long as $\lambda \gtrsim k / \sqrt{n}$ in the model \eqref{eqn:wigner-sparse}. 
		An example of this kind is the one based on covariance thresholding studied in \citet{deshpande2014sparse,krauthgamer2015semidefinite}.\footnote{While \citet{deshpande2014sparse} focused primarily on the spiked Wishart model, it is fairly easy to transfer the Wigner model \eqref{eqn:wigner-sparse} into the model therein, by using a simple Gaussian lifting trick to asymmetrize $M$.} 
		In what follows, we shall call this algorithm as an oracle algorithm. 

	\item The estimate returned by the above oracle algorithm, however, exhibits complicated statistical dependency on $W$, 
		thus precluding us from directly invoking our AMP analysis framework. 
\end{itemize}
\noindent
In light of the above observations, we propose an initialization scheme based on sample splitting, 
which repeats the following steps for $N \asymp \log n$ rounds. In each round $j$:  
\begin{itemize}
	\item[1)] Randomly sample an index subset $\mathcal{I}_{j}$, independent of $M$, with mean size $|\mathcal{I}_{j}|=np$ (each $i \in [n]$ is included in $\mathcal{I}_{j}$ with probability $p$)
		and partition $M$ into four independent blocks, namely, $M_{\mathcal{I}_{j}, \mathcal{I}_{j}}$,  
$M_{\mathcal{I}_{j}^{c}, \mathcal{I}_{j}}$, 
$M_{\mathcal{I}_{j}, \mathcal{I}_{j}^{c}}$,  
$M_{\mathcal{I}_{j}^{c}, \mathcal{I}_{j}^{c}}$. 
	Here and below,  $M_{\mathcal{I}, \mathcal{J}}$ denotes the submatrix of $M$ with rows (resp.~columns) coming from those with indices in $\mathcal{I}$ (resp.~$\mathcal{J}$). 

	\item[2)] Apply the oracle algorithm mentioned above with a little follow-up step to obtain a unit-norm estimate $x^j \in \real^{|\mathcal{I}_j^{\mathrm{c}}|}$ (see   Algorithm~\ref{alg:split} for details).

	\item[3)] Run AMP on a smaller-dimensional (but independent) submatrix $M_{\mathcal{I}_{j}^{c}, \mathcal{I}_{j}^c}$;  the size of $\mathcal{I}_{j}$ is chosen to be $o(n)$, 
		so that the efficiency of the AMP will not degrade much. 

\end{itemize} 
Finally, we select an index set $\mathcal{I}_{\widehat{j}}$ based on the following criterion:  
\begin{align*}
	\widehat{j} \defn \argmax_{1\leq j \leq N} ~\Big\{ x^{j\top} M_{\mathcal{I}_j^{c},\mathcal{I}_j^{c}} x^{j}\Big\}. 
\end{align*}
In other words, we pick an index set such that its initial estimate has the largest correlation with the complement diagonal block. 
The fact that $x^j$ is statistically independent from  $W_{\mathcal{I}_{\widehat{j}}^{c}, \mathcal{I}_{\widehat{j}}^{c}}$ 
plays a crucial role in the subsequent analysis. 
The whole initialization scheme is summarized in Algorithm~\ref{alg:split}.

We are then positioned to derive some key properties of the above initialization scheme. 
For ease of exposition, let us define an index subset 
\begin{align}
	\mathcal{S}_1 \defn \bigg\{j \in N : \frac{\big\langle v_{\mathcal{I}_{j}^{c}}^{\star}, x^j \big\rangle}{\big\|v_{\mathcal{I}_{j}^{c}}^{\star}\big\|_2} \asymp 1\bigg\}. 
\end{align}
%
We immediately make note of the following property, whose proof is provided in Section~\ref{sec:pf-prop-sparse-split}.

\begin{props} 
\label{prop:sparse-split}
Consider the regime \eqref{eqn:init-2} with $k \gg \log n$, and set $p = C_p \frac{\log n}{k}$ for some large constant $C_p$.  
%
With probability at least $1 - O(n^{-10})$, the vector  $v^{\widehat{j}}$ computed in \eqref{eqn:vj} --- with $\widehat{j}$ chosen in \eqref{eqn:sparse-choose-j} --- satisfies 
\begin{align} 
\label{eq:spec-init}
v^{\widehat{j}} \defn M_{\mathcal{I}_{\widehat{j}}^{c}, \mathcal{I}_{\widehat{j}}}v_{\mathcal{I}_{\widehat{j}}} = \alpha_1 v_{\mathcal{I}_{\widehat{j}}^{c}}^{\star} + \phi_0
	\qquad \text{with }
\alpha_1 \asymp \lambda\sqrt{p}, 
\end{align}
where $\phi_0 \sim \mathcal{N}(0, \frac{1}{n}I_{|\mathcal{I}^c_{\widehat{j}}|})$ is independent from $W_{\mathcal{I}_{\widehat{j}}^{c}, \mathcal{I}_{\widehat{j}}^{c}}$ conditional on $\mathcal{I}_{\widehat{j}}$.
Moreover, one has $\widehat{j} \in \mathcal{S}_{1}$ with probability at least $1 - O(n^{-10}).$
\end{props}

\paragraph{Non-asymptotic theory of AMP as initialized in Algorithm~\ref{alg:split}.}

As revealed by Proposition~\ref{prop:sparse-split}, the aforementioned initialization scheme provides an almost independent estimate that enjoys  non-vanishing correlation with the truth. 
We can then execute the AMP update rule \eqref{eqn:AMP-updates-sparse} on the submatrix $M_{\mathcal{I}_{\widehat{j}}^{c}, \mathcal{I}_{\widehat{j}}^{c}}$ 
in order to obtain an estimate $x_t$ for the subvector of $\vstar$ from the index subset $\mathcal{I}_{\widehat{j}}^{c}$; 
details are summarized in Algorithm~\ref{alg:split}. 
With this in mind, our theory developed so far readily leads to finite-sample characterizations of this estimate $x_t$. 
More specifically, Theorem~\ref{thm:sparse} together with some basic union bounds reveals that with probability at least $1 - O(n^{-10})$, 
the estimate $x_t$ returned by Algorithm~\ref{alg:split} satisfies
\begin{align}
\label{eqn:decomp-1-sparse45}
	x_{t+1,i} &= \alpha_{t+1} \vstar_{i} 
	+ \sum_{j = 1}^{t} \beta_{t}^j\phi_{j, i} 
	+ \xi_{t, i}, \qquad \text{for all } i \in \mathcal{I}^c_{\widehat{j}} , 
\end{align}
where again the $\phi_j$'s are i.i.d.~drawn from $\mathcal{N}(0, \frac{1}{n} I_{|\mathcal{I}^c_{\widehat{j}}|})$, with the coefficients $\beta_{t}$, $\alpha_{t+1}$ and $\|\xi_{t}\|_2$ satisfying the predictions of Theorem~\ref{thm:sparse} (except that $\alpha_{t+1}$ should be rescaled by $\sqrt{1-p}$ to account for the reduced signal size). 
%

\begin{remark}
	The careful reader might remark that Algorithm~\ref{alg:split} only returns an estimate over the index subset $\mathcal{I}_{\widehat{j}}^{c}$. 
	One still needs to estimate the remaining entries of $\vstar$. 
	To do so, we can simply rerun the algorithm to generate different sampling sets, in the hope of producing another estimate that covers the remaining subvector (which is likely to happen given that $\mathcal{I}_{\widehat{j}}$ is vanishingly small). 
	The two AMP outputs can then be merged easily to estimate the whole vector $\vstar$. 
	Details are omitted here as they are not the focus of the current paper. 

\end{remark}

\begin{algorithm}[t]
\DontPrintSemicolon
{\bf Input:} data matrix $M$; an oracle algorithm as described in \eqref{eq:correlation-oracle}; $p \asymp \frac{\log n}{k}$; $\tau_1 \asymp \sqrt{\frac{\log n}{n}}$.\\

{\bf Initialization:}
\begin{enumerate}
\item Set $N \asymp \log n$. For every $j \in [N]$, sample an index subset $\mathcal{I}_j \subset [n]$ such that each
 $i \in [n]$ belongs to $\mathcal{I}_{j}$ independently with probability $p$. 

\item For each $j\in [N]$, partition $M$ into four sub-matrices $M_{\mathcal{I}_j, \mathcal{I}_j}, M_{\mathcal{I}_j, \mathcal{I}_j^{c}}, M_{\mathcal{I}_j^{c}, \mathcal{I}_j}$ and $M_{\mathcal{I}_j^{c}, \mathcal{I}_j^{c}}$. 
Run the oracle algorithm to obtain a unit-norm estimate $v_{\mathcal{I}_j}$ for $v^{\star}_{\mathcal{I}_j}$ --- the subvector of $\vstar$ in the index set $\mathcal{I}_j$ --- based on $M_{\mathcal{I}_j, \mathcal{I}_j}$,
which satisfies $\langle v^{\star}_{\mathcal{I}}, v_{\mathcal{I}}\rangle \asymp \big\|v^{\star}_{\mathcal{I}}\big\|_2$ with high probability.
Compute
\begin{align}
\label{eqn:vj}
	v^j &\defn M_{\mathcal{I}_j^{c}, \mathcal{I}_j}\cdot v_{\mathcal{I}_j}, \\
\label{eqn:sparse-init-2-j}
	x^j &\defn \frac{\mathsf{ST}_{\tau_1}(v^{j})}{\ltwo{\mathsf{ST}_{\tau_1}(v^{j})}} .
\end{align}

\item Compute 
\begin{align}
\label{eqn:sparse-choose-j}
	\widehat{j} \defn \argmax_{1\leq j \leq N} ~\Big\{ x^{j\top} M_{\mathcal{I}_j^{c},\mathcal{I}_j^{c}} x^{j}\Big\}
	\qquad \text{and} \qquad
	x_1 = x^{\widehat{j}}. 
\end{align}

\end{enumerate}

	{\bf AMP:} run AMP \eqref{eqn:AMP-updates-sparse} on $M_{\mathcal{I}_{\widehat{j}}^{c}, \mathcal{I}_{\widehat{j}}^{c}}$ with initialization $x_1$ and  $\eta_0(x_{0}) = 0$ to obtain $x_{t} \in \real^{|\mathcal{I}_{\widehat{j}}^c|}$.



\caption{AMP with sample-split initialization.}
 \label{alg:split}
\end{algorithm}
\medskip


\section{Sparse PCA: Proofs of Theorem~\ref{thm:sparse} and Corollary~\ref{cor:sparse-init-1}}
\label{sec:pf-sparse}

Akin to the problem of $\mathbb{Z}_{2}$ synchronization, 
we always have $\|\beta_{t-1}\|_2 = \ltwo{\eta_{t-1}(x_{t-1})} = 1$ 
given our choices of the denoising functions~\eqref{eqn:eta-sparse}. 
As a result, we shall focus attention on tracking $\alpha_{t}.$

The proofs of Theorem~\ref{thm:sparse} and Corollary~\ref{cor:sparse-init-1} mainly follow from Theorem~\ref{thm:main}, 
with the assistance of an induction argument. 
Specifically, our induction hypotheses for the $t$-th iteration are
\begin{align}
\label{eqn:sparse-induction}
	\alpha_{t} \asymp \lambda \qquad \text{and} \qquad \ltwo{\xi_{t-1}} \lesssim 
	\sqrt{\frac{(t-1)\log^3 n+ k}{n}}.
\end{align}
In what follows, we shall assume the induction hypotheses \eqref{eqn:sparse-induction}  are valid for the $t$-th iteration, 
and demonstrate their validity for the $(t+1)$-th iteration; the base case will be validated in Section~\ref{sec:recursion}. 
The only difference between Theorem~\ref{thm:sparse} and Corollary~\ref{cor:sparse-init-1} lies in the initialization step which 
is detailed in Section~\ref{sec:pf-initialization-sparse}. 


\subsection{Preliminary facts}
\label{sec:prelim-sparse}

Before delving into the details of the main proof, we collect several preliminary results that shall be used repeatedly throughout this section.

\subsubsection{Properties about the denoising functions}
Recall that we adopt the following denoising functions:  for any $x\in \real^n$, 
\begin{align}
	\eta_{t}(x) \defn \gamma_t \cdot \mathrm{sign}(x) \circ (|x| - \tau_t 1)_{+} 
	\qquad 
	\text{where }\gamma_t \defn \big\|\mathrm{sign}(x_t) \circ (|x_t| - \tau_t 1)_{+}\big\|_2^{-1}
	~\text{ and }~\tau_t \asymp \sqrt{\frac{\log n}{n}}, 
\end{align}
Here and throughout, 
we abuse the notation to use it in an entrywise manner when applied to vectors, i.e.,   
\begin{equation}
	\mathrm{sign}(x)=\big[\mathrm{sign}(x_i) \big]_{1\leq i\leq n} 
	\qquad \text{and} \qquad (|x| - \tau 1)_{+} = \big[ (|x_i|-\tau )_+ \big]_{1\leq i\leq n}. 
\end{equation}
for any $x=[x_i]_{1\leq i\leq n}\in \real^n$. 
The entrywise derivative of $\eta_{t}(x)$ w.r.t.~$x$ is given by
\begin{equation}
	\eta_{t}^{\prime}(x) = \gamma_t \mathds{1}(|x| > \tau_t 1) \defn \gamma_t \big[ \mathds{1}(|x_i| > \tau_t) \big]_{1\leq i\leq n}.    
\end{equation}
Here, $\eta_{t}^{\prime}(w)$ is well-defined for all differentiable points $w\in \real$, with its value for the non-differentiable points (i.e., $w=\pm\tau_t$) taken to be 0; 
this works for our purpose given that the non-differentiable part are accounted for separately in Theorem~\ref{thm:main}.

Next, consider a set of parameters $\mu = [\mu^j]_{1\leq j\leq t-1} \in \mathcal{S}^{t-2}$, $\alpha \in \real$ and $\beta = [\beta^j]_{1\leq j\leq t-1} \in \real^{t-1}$ independent of $\{\phi_j\}$, 
and define the following vector (which is a function of $\alpha$ and $\beta$): 
\[
	v = \alpha \vstar + \sum_{j = 1}^{t-1} \beta^j\phi_j \in \real^n.
\]
We also define, for any positive numbers $\tau,\gamma \in \real$,  
\begin{equation}
	\eta(x; \tau) \defn \gamma \cdot \mathrm{sign}(x) \circ (|x| - \tau 1)_{+}. 
\end{equation}
Elementary calculations together with $\|\mu\|_2=\|\beta\|_2=1$ yield 
\begin{subequations}
\label{eqn:sparse-derivative}
\begin{align}
\Big\|\nabla_{\phi_j} \Big\langle \sum_{j = 1}^{t-1} \mu^j\phi_j, a\Big\rangle\Big\|_2 &\le |\mu^j|\|a\|_2,\qquad 
	&&\text{for any given } a \in \mathbb{R}^n \\
	\Big\|\nabla_{\phi_j} \big\langle \eta(v; \tau), a\big\rangle\big\|_2 &\le |\beta^{j}|\|\eta^{\prime}(v; \tau) \circ a\|_2\qquad &&\text{for any given } a \in \mathbb{R}^n \\
	\Big\|\nabla_{\mu} \Big\langle \sum_{j = 1}^{t-1} \mu^j\phi_j, a\Big\rangle\Big\|_2 &\le \|a\|_2\sum_{j = 1}^{t-1} \|\phi_j\|_2\qquad &&\text{for any given } a \in \mathbb{R}^n \\
	\Big\|\nabla_{\mu, \beta} \Big( \sum_{j = 1}^{t-1} \mu^j\beta^j \Big)\Big\|_2 &\le 2 \\
	\Big\|\nabla_{\alpha, \beta, \tau} \big\langle \eta(v; \tau), a\big\rangle\Big\|_2 &\le \|a\|_2\|\eta^{\prime}(v; \tau)\|_2
	\Big( 1+ \sum_j\|\phi_j\|_2 \Big) \qquad &&\text{for any given } a \in \mathbb{R}^n. 
\end{align}
\end{subequations}


\subsubsection{Basic concentration results}

Next, we collect some basic concentration results. 
Similar to \eqref{eq:eps-set} in Section~\ref{sec:basic-concentration-buble}, 
we define  
\begin{align}
	&\notag \mathcal{E}_s \defn 
	\left\{ \{\phi_j\} : \max_{1\leq j\leq t-1} \|\phi_j\|_2  < 1+ C_5\sqrt{\frac{\log \frac{n}{\delta}}{n}}\right\} 
	\bigcap \left\{ \{\phi_j\} : \sup_{a \in \mathcal{S}^{t-2}} \Big\|\sum_{j = 1}^{t-1} a_k\phi_j\Big\|_2 < 1 + C_5 \sqrt{\frac{t\log \frac{n}{\delta}}{n}} \right\} \\
	&\hspace{2cm} \bigcap 
	\left\{\{\phi_j\} : \sup_{a \in \mathcal{S}^{t-2}}\sum_{i = 1}^s \Big|\sum_{j = 1}^{t-1} a_k\phi_j\Big|_{(i)}^2 < \frac{C_5(t + s)\log \frac{n}{\delta}}{n}\right\} 
	\label{eq:eps-set-2}
\end{align}
for some sufficiently large constant $C_5>0$. 
As discussed in~\eqref{eqn:eps-interset}, the convex set $\mathcal{E}_{s}$ satisfies 
\begin{align*}
	\mprob\lt(\{\phi_j\} \in\bigcap_{s = 1}^n \mathcal{E}_s \rt) \geq 1 -  \delta. 
\end{align*}
In addition, let us introduce an additional collection of convex sets: for any $1\leq s\leq n$, 
\begin{align}
\label{eqn:set-E-tilde}
	\widetilde{\mathcal{E}}_s \defn \lt\{\{\phi_j\} : \|\Phi_{s, :}\|_2 \leq  C_5 \sqrt{(t-1)\log \frac{n}{\delta}}\rt\},
\end{align}
where $\Phi_{s, :}$ denotes the $s$-th row of matrix $\Phi = \sqrt{n}[\phi_1,\cdots,\phi_{t-1}] \in \real^{n \times (t-1)}$.  
Standard Gaussian concentration results \citep[Chapter 4.4]{vershynin2018high} reveal that $\{\phi_j\}$ falls within $\bigcap_{s = 1}^n \widetilde{\mathcal{E}}_s$ 
with probability at least $1-\delta$, provided that $C_5$ is large enough. 
As a result, it is readily seen that
\begin{align}
\label{eqn:eps-interset-sparse}
\mprob( \{\phi_j\} \in \mathcal{E} ) \geq 1 - 2 \delta,
	\qquad \text{for }\mathcal{E} \defn \bigcap_{s = 1}^n \big( \mathcal{E}_s \cap \widetilde{\mathcal{E}}_s \big). 
\end{align}
Throughout the rest of the proof, we shall take $\delta$ to be sufficiently small, say, $\delta \asymp n^{-300}$ (similar to Section~\ref{sec:basic-concentration-buble}).

\subsubsection{Bounding the size of $\eta^{\prime}_t$}
As studied in the case of $\mathbb{Z}_{2}$ synchronization (see \eqref{eqn:muphi-rank}), 
we know that conditional on the event $\mathcal{E}$, 
the $t$-th largest entry of $\sum_{j = 1}^{t-1} \beta^j \phi_j$ for an arbitrary unit vector $\beta=[\beta^1,\cdots,\beta^{t-1}]\in \mathcal{S}^{t-2}$ obeys
\begin{align*}
	t \bigg| \sum_{j = 1}^{t-1} \beta^j\phi_j\bigg|_{(t)}^2 \leq 
	\sum_{i=1}^t \bigg|\sum_{j = 1}^{t-1} \beta^j\phi_j\bigg|_{(i)}^2 \lesssim 
	\frac{(t + t)\log n}{n} \asymp \frac{t\log n}{n},
\end{align*}
where the last inequality uses the definition of $\mathcal{E}$. 
It therefore implies that for every $l \geq t$, one has 
\begin{align*}
	\Big|\sum_{j = 1}^{t-1} \beta^j\phi_j\Big|_{(l)} 
	\leq \bigg| \sum_{j = 1}^{t-1} \beta^j\phi_j\bigg|_{(t)}\lesssim \sqrt{\frac{\log n}{n}} 
\end{align*}
with probability at least $1-O(n^{-11})$. 
Now consider the vector $v \defn \alpha_t\vstar + \sum_{j = 1}^{t-1} \beta^j\phi_j$.
Since $|v^{\star}|_{(k +1)} = 0$ (given that $v^{\star}$ is $k$-sparse) and that $\tau_{t}\geq C_3\sqrt{\frac{\log n}{n}}$ 
for some constant $C_3>0$ large enough,
 we can show that
\begin{align}
\label{eqn:gradient-sparse}
	\big|\eta_t^{\prime}(v)\big|_{(i)}
	= \big| |v|-\tau 1 \big|_{(i)} 
	= 0,\qquad\text{for }i > k + t 
\end{align}
with probability exceeding $1-O(n^{-11})$. 
As a direct consequence of \eqref{eqn:gradient-sparse}, for any vector $a \in \real^n$ one has  
\begin{align}
\label{eqn:sparse-eta-a-prod}
	\lt\|\eta_t^{\prime}(v) \circ a\rt\|_2 = \gamma_t \sqrt{ \sum_{i=1}^{k+t} |a|_{(i)}^2 } \lesssim \lambda^{-1}\sqrt{ \sum_{i=1}^{k+t} |a|_{(i)}^2 },
\end{align}
where the last relation comes from~\eqref{eqn:gamma-t-evolution}. 
%


\subsection{Tight estimate of $\gamma_t$}
\label{sec:tight-estimate-gamma-sparse}

In this subsection, 
our goal is to show that under the induction hypotheses \eqref{eqn:sparse-induction}, we have
%
%
\begin{align}
\label{eqn:gamma-t-evolution}
\gamma_t \defn \big\|\mathrm{sign}(x_t)(|x_t| - \tau_t 1)_{+} \big\|_2^{-1} 
	=  \big\|(|x_t| - \tau_t 1)_{+} \big\|_2^{-1}
	\asymp \lambda^{-1},
\end{align}
which would then imply that (see Assumption~\ref{assump:eta}) 
\begin{align}
\label{eqn:rho-sparse}
	\rho = \lambda^{-1} \qquad\text{and}\qquad \rho_1 = 0.
\end{align}

In order to show this, we resort to Corollary~\ref{cor:Gauss}. 
Let us define
\begin{align*}
	\Phi = \sqrt{n} \big[ \phi_1,\cdots,\phi_{t-1} \big]
	\qquad \text{and}\qquad
	\theta = (\alpha, \beta, \tau) \in \mathbb{R} \times \mathbb{R}^{t-1} \times \mathbb{R}
	\quad \text{with }\beta = [\beta^1,\cdots,\beta^{t-1}],
\end{align*}
and consider the following function:
\begin{align}
f_{\theta}(\Phi) \defn \big\|\mathrm{sign}(v) \circ (|v| - \tau 1)_{+}\big\|_2^2,
\qquad
\text{with } v \defn \alpha \vstar + \sum_{j = 1}^{t-1} \beta^j\phi_j. 
\label{eq:f-Theta-v-sparse-1}
\end{align}
Let us also introduce the following set of parameters:
\begin{align*}
	\Theta \defn \lt\{ \theta = (\alpha, \beta, \tau) \in \mathbb{R} \times \mathbb{R}^{t-1} \times \mathbb{R} \,\Big| \, 
	\alpha \asymp \lambda, \|\beta\|_2=1, C_5 \sqrt{\frac{\log n}{n}}\leq \tau \asymp \sqrt{\frac{\log n}{n}}\rt \}
\end{align*}
for some large enough constant $C_5>0$. 
Consequently, $\gamma_t$ (cf.~\eqref{eqn:gamma-t-evolution}) 
can be viewed as $f_{\theta}(\Phi)$ with $\theta = (\alpha_t, \beta_{t-1}, \tau_t)$, 
and hence it suffices to develop a uniform bound on $f_{\theta}(\Phi)$ over all $\theta \in \Theta$.

%

It is easily seen that $|f_{\theta}(\Phi)|\lesssim n^{100} \big(\max_j \|\phi_j\|_2\big)^{100}$, 
and that $\|\nabla_{\theta} f(Z) \|_2 \lesssim n^{100}$ for all $Z\in \mathcal{E}$ (cf.~\eqref{eqn:eps-interset-sparse}). 
In addition, given that $\log\frac{1}{\delta} \asymp \log n$, it follows from \eqref{eq:f-Theta-v-sparse-1} and \eqref{eq:eps-set-2} that 
\begin{align}
\label{eqn:vbound-sparse}
	\|v\|_2 \leq |\alpha|\ltwo{\vstar} +  1 + C_5\sqrt{\frac{t\log \frac{n}{\delta}}{n}}  \asymp \lambda + 1 \asymp 1
\end{align}
over the set $\mathcal{E}$, 
where the penultimate step is valid as long as $\frac{t\log n}{n} \lesssim 1 $. 
Moreover, we observe that
\begin{align*}
\big\|(|v|-\tau1)_{+}\big\|_{2}^{2} & \leq\sum_{i:\,v_{i}^{\star}\neq0}\Big|\alpha v_{i}^{\star}+\sum_{j=1}^{t-1}\beta^{j}\phi_{j,i}\Big|^{2}+\sum_{i:\,v_{i}^{\star}=0}\bigg(\Big|\sum_{j=1}^{t-1}\beta^{j}\phi_{j,i}\Big|-\tau\bigg)^{2}\\
 & \lesssim\alpha^{2}\|v^{\star}\|_{2}^{2}+\sum_{i:\,v_{i}^{\star}\neq0}\Big|\sum_{j=1}^{t-1}\beta^{j}\phi_{j,i}\Big|^{2}+\sum_{i=1}^{n}\bigg(\Big|\sum_{j=1}^{t-1}\beta^{j}\phi_{j,i}\Big|-\tau\bigg)^{2}\\
 & \lesssim\alpha^{2}+\sum_{i=1}^{2k+t}\Big|\sum_{j=1}^{t-1}\beta^{j}\phi_{j}\Big|_{(i)}^{2}\lesssim\alpha^{2}+\frac{(k+t)\log n}{n},
\end{align*}
where the last line comes from (\ref{eqn:gradient-sparse}) and the
definition (\ref{eq:eps-set-2}) of $\mathcal{E}$. 
This in turn allows us to calculate
\begin{align*}
\lt\|\nabla_{\Phi} f_{\theta}(\Phi)\rt\|_2 &\le 2\frac{\|\beta\|_2}{\sqrt{n}}
	\big\|\mathrm{sign}(v) \circ (|v| - \tau 1)_{+} \circ \ind\lt(|v| > \tau 1\rt)\big\|_2 \le 2\frac{\big\|(|v|-\tau1)_{+}\big\|_{2}}{\sqrt{n}} 
	\lesssim \frac{ \alpha + \sqrt{\frac{(k+t)\log n}{n}} }{\sqrt{n}} .
%
\end{align*}
Therefore, Corollary~\ref{cor:Gauss} and \eqref{eqn:brahms-conc} tell us that,  with probability at least $1-O(n^{-11})$, 
%
%
\begin{align*}
\lt| \big\| (|v| - \tau 1)_{+} \big\|_2^2 - \int\Big\|\Big(\Big|\alpha \vstar + \frac{1}{\sqrt{n}}x\Big| - \tau 1\Big)_{+}\Big\|_2^2\varphi_n(\dx)\rt| 
	&\lesssim \bigg( \alpha + \sqrt{\frac{(k+t)\log n}{n}} \bigg) \sqrt{\frac{t\log n}{n}}
\end{align*}
holds simultaneously for all $\theta\in\Theta$.
This in turn implies that
%
%
%
\begin{align}
\label{eqn:summer}
	\bigg| \big\| (|v_t| - \tau_t 1)_{+} \big\|_2^2 - \int\Big\|\lt(\Big|\alpha_t\vstar + \frac{1}{\sqrt{n}}x \Big| - \tau_t 1\rt)_{+}\Big\|_2^2\varphi_n(\dx)\bigg| 
	&\lesssim \bigg( \alpha_t + \sqrt{\frac{(k+t)\log n}{n}} \bigg)  \sqrt{\frac{t\log n}{n}}. 
\end{align}

Next, let us assess the size of the quantity $\int\|(|\alpha_tv^{\star} + \frac{1}{\sqrt{n}}x| - \tau_t 1)_{+}\|_2^2\varphi_n(\dx)$. 
For those indices $i$ obeying $|\alpha_tv_i^{\star}|  \geq 2 \tau_t$, 
it is easily seen from basic Gaussian properties that 
\begin{align}
\label{eqn:tmp-integral}
\int\Big(\Big|\alpha_tv_i^{\star} + \frac{1}{\sqrt{n}}x\Big| - \tau_t\Big)_{+}^2\varphi(\dx) 
	\asymp \int_{-\sqrt{\log n}}^{\sqrt{\log n}} \big(\alpha_tv_i^{\star}\big)^2\varphi(\dx) \asymp \lt(\alpha_tv_i^{\star}\rt)^2,  
\end{align}
which together with the induction hypothesis $\alpha_t\asymp \lambda$ gives 
\begin{align}
\label{eqn:quartet1579-large}
	 \sum_{i:\, |\alpha_t v_i^{\star}|  \geq 2\tau_t \asymp \sqrt{\frac{\log n}{n}} }\int\Big(\Big|\alpha_tv_i^{\star} + \frac{1}{\sqrt{n}}x\Big| - \tau_t\Big)_{+}^2\varphi(\dx) 
	\asymp \lambda^2 \sum_{i:\, |\alpha_t v_i^{\star}|  \geq 2\tau_t \asymp \sqrt{\frac{\log n}{n}} }\lt(v_i^{\star}\rt)^2.
\end{align}
Additionally, it is observed that
\begin{align}
1\ge\sum_{i}\big(v_{i}^{\star}\big)^{2}\ind\Big(|\alpha_t v_{i}^{\star}|\geq 2\tau_t \Big) & \overset{(\mathrm{i})}{\geq}\sum_{i}\big(v_{i}^{\star}\big)^{2}\ind\Big(|v_{i}^{\star}|\geq\sqrt{\frac{1}{2k}}\Big)=1-\sum_{i}\big(v_{i}^{\star}\big)^{2}\ind\Big(0<|v_{i}^{\star}|<\sqrt{\frac{1}{2k}}\Big) \notag\\
 & \geq1-k\cdot\left(\sqrt{\frac{1}{2k}}\right)^{2}1=\frac{1}{2} ,
	\label{eqn:basics-vstar-2}
\end{align}
where (i) holds since for any $i$ with $|v_{i}^{\star}|\geq\sqrt{\frac{1}{2k}}$,
one necessarily has $|\alpha_t v_{i}^{\star}| \asymp |\lambda v_{i}^{\star}|\geq 2\tau_t \asymp \sqrt{\frac{\log n}{n}}$
as long as $\lambda \geq C_2 \sqrt{\frac{k\log n}{n}}$ for some large enough constant $C_2>0$. 
Substitution into \eqref{eqn:quartet1579-large} yields
\begin{align}
\label{eqn:quartet1579-large-135}
	 \sum_{i:\, |\alpha_t v_i^{\star}|  \geq 2\tau_t \asymp \sqrt{\frac{\log n}{n}} }\int\Big(\Big|\alpha_tv_i^{\star} + \frac{1}{\sqrt{n}}x\Big| - \tau_t\Big)_{+}^2\varphi(\dx) 
	\asymp \lambda^2 .
\end{align}
Moreover, when it comes to those indices $i$ obeying $|\alpha_tv_i^{\star}|<2\tau_t \asymp \sqrt{\frac{\log n}{n}}$, one has
\[
\sum_{i:\,|\alpha_{t}v_{i}^{\star}|<2\tau_{t}\asymp\sqrt{\frac{\log n}{n}}}\int\Big(\Big|\alpha_{t}v_{i}^{\star}+\frac{1}{\sqrt{n}}x\Big|-\tau_{t}\Big)_{+}^{2}\varphi(\mathrm{d}x)\leq\sum_{i:\,|\alpha_{t}v_{i}^{\star}|<2\tau_{t}\asymp\sqrt{\frac{\log n}{n}}}\big(\alpha_{t}v_{i}^{\star}\big)^{2}\lesssim k\cdot\frac{\log n}{n}\lesssim\lambda^{2}, 
\]
provided that $\lambda^2 \gtrsim \frac{k \log n}{n}$. 
This combined with \eqref{eqn:quartet1579-large-135} leads to
%
%
%
%
%
%
\begin{align}
\label{eqn:quartet15}
	\int \Big\| \Big(\Big|\alpha_tv^{\star} + \frac{1}{\sqrt{n}}x \Big| - \tau_t 1 \Big)_{+} \Big\|_2^2\varphi_n(\dx)
	&\asymp  \lambda^2. 
\end{align}
Taking this collectively with \eqref{eqn:summer} gives
\begin{align}
 & \bigg|\big\|(|v_{t}|-\tau_{t} 1)_{+}\big\|_{2}-\bigg(\int\Big\|\lt(\Big|\alpha_{t}\vstar+\frac{1}{\sqrt{n}}x\Big|-\tau_{t} 1\rt)_{+}\Big\|_{2}^{2}\varphi_{n}(\dx)\bigg)^{\frac{1}{2}}\bigg|\notag\\
 & \quad\quad\leq\frac{\Big|\big\|(|v_{t}|-\tau_{t} 1)_{+}\big\|_{2}^{2}-\int\big\|\lt(\big|\alpha_{t}\vstar+\frac{1}{\sqrt{n}}x\big|-\tau_{t} 1\rt)_{+}\big\|_{2}^{2}\varphi_{n}(\dx)\Big|}{\Big(\int\big\|\big(\big|\alpha_{t}\vstar+\frac{1}{\sqrt{n}}x\big|-\tau_{t} 1\big)_{+}\big\|_{2}^{2}\varphi_{n}(\dx)\Big)^{\frac{1}{2}}}
 \lesssim\frac{ \alpha_t + \sqrt{\frac{(k+t)\log n}{n}} }{\lambda}\sqrt{\frac{t\log n}{n}}.
	\label{eq:summer-twice}
\end{align}

Now in order to control $\gamma_t$, we still need to establish a connection between $\|\mathrm{sign}(x_t)\circ(|x_t| - \tau_t)_{+}\|_2$ and $\|\mathrm{sign}(v_t)\circ(|v_t| - \tau_t)_{+}\|_2.$ 
Recognizing that $x_t = v_t + \xi_{t-1}$,  
we can invoke the triangle inequality to obtain 
\begin{align}
\notag\gamma_{t}^{-1} & =\big\|(|x_{t}|-\tau_{t} 1)_{+}\big\|_{2}
 =\big\|(|v_{t}|-\tau_{t} 1)_{+}\big\|_{2}+O(\|\xi_{t-1}\big\|_{2})\\
\notag & =\sqrt{\int\Big\|\Big(\Big|\alpha_{t}v^{\star}+\frac{1}{\sqrt{n}}x\Big|-\tau_{t} 1\Big)_{+}\Big\|_{2}^{2}\varphi_{n}(\dx)}+O\bigg(\frac{1}{\lambda}\sqrt{\frac{t\log n}{n}}+\|\xi_{t-1}\|_{2}\bigg) \\
 & =\lt( 1+ O\Bigg(\frac{\alpha_t + \sqrt{\frac{(k+t)\log n}{n}}}{\lambda^2}\sqrt{\frac{t\log n}{n}}+\frac{\|\xi_{t-1}\|_{2}}{\lambda}\Bigg) \rt) 
	\sqrt{\int\Big\|\Big(\Big|\alpha_{t}v^{\star}+\frac{1}{\sqrt{n}}x\Big|-\tau_{t} 1\Big)_{+}\Big\|_{2}^{2}\varphi_{n}(\dx)}
	\asymp \lambda,
	\label{eqn:beethoven}
\end{align}
where the penultimate step follows from inequality~\eqref{eq:summer-twice}, 
and the last line makes use of \eqref{eqn:quartet15}.  
This establishes the claimed result in \eqref{eqn:gamma-t-evolution}.


\subsection{Controlling key quantities $A_t,B_t,D_t, E_t$ and $\kappa_t$}
\label{sec:control-sparse}

In order to apply Theorem~\ref{thm:main} for the sparse spiked Wigner model, 
a key step lies in bounding the multiple key quantities $A_t,\ldots,G_t$ (see \eqref{defi:A}-\eqref{defi:G}) as specified in Assumption~\ref{assump:A-H-eta}, 
which we aim to accomplish in this subsection.  Note that we do not need to bound $C_t, F_t$ and $G_t$ as they only appear in the bound on $\Delta_{\beta,t}$, which is irrelevant in this case. 
The rest of the section is dedicated to bounding $A_t,B_t,D_t, E_t$. 
Along the way, we shall also control $\kappa_t$, which is needed when calculating $D_t$.

\subsubsection{Quantity $A_t$ in \eqref{defi:A}}

Unlike the case of $\mathbb{Z}_{2}$ synchronization where the denoising functions are smooth everywhere, caution needs to be exercised when handling discontinuity points in sparse spiked Wigner models. 
Consider any given $\mu=[\mu^1,\cdots,\mu^{t-1}]$, $\alpha \in \real$, $\beta \in [\beta^1,\cdots,\beta^{t-1}]$ and $\tau,\gamma\in \real$, 
and let us take  
\begin{subequations}
\label{eq:defn-Phi-Theta-theta-sparse}
\begin{align}
	&\Phi \defn \sqrt{n}(\phi_1,\ldots,\phi_{t-1}),
	\qquad
	\theta \defn \big[ \mu, \alpha, \beta, \tau, \gamma \big] \in \mathcal{S}^{t-2} \times \real \times \mathcal{S}^{t-2} \times \real \times \real,
	  \\
	  v \defn \alpha \vstar + &\sum_{j = 1}^{t-1} \beta^j\phi_j \in \real^n
	\quad 
	\Theta \defn \lt\{ \theta = \big[ \mu, \alpha, \beta, \tau \big] \,\Big|\, 
	\alpha \asymp \gamma^{-1}\asymp \lambda, \|\mu\|_2=\|\beta\|_2=1, \tau \asymp  \sqrt{\frac{\log n}{n}}\rt\}.
\end{align}
\end{subequations}

Recall that $A_t$ consists of two parts: 
$\big\langle \sum_{j = 1}^{t-1} \mu^j\phi_j, \eta_{t}(v_t) \big\rangle$ and $\big\langle\eta_t^{\prime}\big\rangle \sum_{j = 1}^{t-1} \mu^j\beta_{t-1}^j$. 
In order to bound the first part of $A_t$, we intend to first derive a uniform control of the following function:   
\begin{align*}
	f_{\theta}(\Phi) \defn \Big\langle \sum_{j = 1}^{t-1} \mu^j\phi_j, \eta(v)\Big\rangle
\end{align*}
over all $\theta \in \Theta$, 
where we define (with its dependency on $\theta$ suppressed in the notation)
\begin{align}
	\eta(x) \coloneqq \gamma\, \mathrm{sign}(x) \circ (|x|-\tau 1)_+. 
	\label{eq:eta-notation-suppressed-sparse}
\end{align}
Towards this end, we first repeat the analysis in Section~\ref{sec:tight-estimate-gamma-sparse} (in particular, \eqref{eqn:summer} and \eqref{eqn:quartet15}) to derive
\begin{equation}
	\big\| (|v|-\tau 1)_+ \big\|_2 \asymp \lambda 
	\qquad \text{and} \qquad
	\|\eta(v)\|_2 = \gamma\, \big\|  (|v|-\tau 1)_+ \big\|_2 \asymp 1, 
	\label{eq:eta-v-norm-sparse}
\end{equation}
for any $\theta \in \Theta$. 
We can then invoke the derivative calculation in~\eqref{eqn:sparse-derivative} to arrive at 
\begin{align*}
\lt\|\nabla_{\Phi} f_{\theta}(\Phi)\rt\|_2 
	&\le \frac{\|\mu\|_2}{\sqrt{n}}\lt\|\eta(v)\rt\|_2 + \frac{\|\beta\|_2}{\sqrt{n}}\bigg\|\sum_{j = 1}^{t-1} \mu^j\phi_j \circ \eta^{\prime}(v)\bigg\|_2  \\
	&\lesssim \frac{1}{\sqrt{n}} + \frac{1}{\lambda\sqrt{n}}\Bigg(\sum_{l=1}^{k+t}\bigg|\sum_{j=1}^{t-1}\mu^{j}\phi_{j}\bigg|_{(l)}^{2}\Bigg)^{1/2} 
	\lesssim \frac{1}{\sqrt{n}}\bigg(1+\frac{1}{\lambda}\sqrt{\frac{(t+k)\log n}{n}}\bigg)\asymp\frac{1}{\sqrt{n}},  
\end{align*}
where the second inequality applies \eqref{eq:eta-v-norm-sparse} and \eqref{eqn:sparse-eta-a-prod}, 
 the third inequality invokes the property of $\mathcal{E}$ in \eqref{eq:eps-set-2}, 
 and the last relation is valid as long as $\lambda^2 \gtrsim \frac{k\log n}{n}$ and $t\lesssim \frac{\lambda^2 n}{\log n}$. 
 Additionally, it is trivially seen that $ f_{\theta}(\Phi) $  as a function of $\theta$ is $n^{100}$-Lipschitz for any given $\Phi\in \mathcal{E}$ 
 and $|f_{\theta}(\Phi) |\lesssim n^{100} \big( \max_j \|\phi_j\|_2 \big)^{100}$. 
 As a result, invoke Corollary~\ref{cor:Gauss} in conjunction with \eqref{eqn:brahms-conc} to arrive at
\begin{align}
\label{eqn:sparse-At-1}
\sup_{\theta\in \Theta}\Big|f_{\theta}(\Phi) - \mathbb{E}\lt[f_{\theta}(\Phi)\rt]\Big| 
&\lesssim \sqrt{\frac{t\log n}{n}}, 
\end{align}
with probability at least $1-O(n^{-11})$. 



%
%

Next, we move on to consider the second part of $A_t$, namely, 
\begin{align*}
\big\langle\eta_t^{\prime}(v_t)\big\rangle \cdot \sum_{j = 1}^{t-1} \mu_t^j  \beta_{t-1}^j
\qquad
\text{where }
	\lt\langle\eta_t^{\prime}(v_t)\rt\rangle 
	&= \frac{\gamma_t}{n} \sum_{i = 1}^n \ind\Big(\big|\alpha_tv_i^{\star} + \sum_j\beta_{t-1}^j\phi_{j, i}\big| > \tau_t\Big).
\end{align*}
Given that the indicator function is not Lipschitz continuous, 
we resort to Corollary~\ref{cor:Gauss-jump} to control it.  
For any given $\theta \in \Theta$, define
\begin{align}
\label{eqn:hfunction}
	h_{i, \theta}(\Phi_{i, :}) \defn \bigg| \alpha v_i^{\star} + \sum_{j=1}^{t-1}\beta^j\phi_{j, i} \bigg|,
	\qquad 1\leq i\leq n,  
\end{align}
where $\Phi_{i,:}$ denotes the $i$-th row of $\Phi$. 
Clearly, for any $\theta,\widetilde{\theta}\in \Theta,$ one can easily verify that
\[
	\big| h_{i, \theta}(\Phi_{i, :}) - h_{i, \widetilde{\theta}}(\Phi_{i, :}) \big|
	\leq n^{100} \big\|\theta - \widetilde{\theta}\|_2
\]
for any $\Phi\in \mathcal{E}$; and for any $\theta \in \Theta$ and $\tau \leq n$, one has
\[
	\mathbb{P}\Big( \tau - 400n^{-100} \leq h_{i,\theta}(\Phi_{i,:}) \leq \tau + 400 n^{-100} \Big) \lesssim n^{-1}. 
\]
Therefore, Corollary~\ref{cor:Gauss-jump} together with \eqref{eqn:brahms-conc}  reveals that with probability at least $1-O(n^{-11})$, 
\begin{align}
\sup_{\theta\in \Theta}
\lt|\sum_{i = 1}^n  \ind\lt(h_{i, \theta} > \tau\rt) - \sum_{i = 1}^n  \mathbb{P}\lt(h_{i, \theta} > \tau\rt)\rt| 
&\lesssim \sup_{\theta\in \Theta} \sqrt{\sum_{i = 1}^n \mathbb{P}\lt(h_{i, \theta} > \tau\rt)t\log n} + t\log n \notag\\
	& \lesssim \sqrt{(k + n\cdot O(n^{-11})) t\log n} + t\log n 
	\lesssim  \sqrt{t(t+k)} \log n \label{eq:sum-h-tau-sparse-UB}
\end{align}
holds simultaneously for all $\theta\in \Theta$, 
where the last inequality comes from \eqref{eqn:gradient-sparse} given that $\tau \asymp \sqrt{\frac{\log n}{n}}$. 
Recognizing that $|\mu^{\top}\beta| \leq 1$, we further have
\begin{align}
\sup_{\theta\in \Theta}
	\lt|  \mu^{\top} \beta \sum_{i = 1}^n  \ind\lt(h_{i, \theta} > \tau\rt) - \mu^{\top} \beta\sum_{i = 1}^n  \mathbb{P}\lt(h_{i, \theta} > \tau\rt)\rt| 
	&\lesssim  |\mu^{\top}\beta| \sqrt{t(t+k)\log n} \lesssim  \sqrt{t(t+k)} \log n.
	\label{eqn:sparse-At-3}
\end{align}
%


To summarize, let us decompose the quantity of interest in \eqref{defi:A} as follows:
\begin{align*}
&\Bigg|\lt\langle \sum_{j = 1}^{t-1} \mu^j\phi_j, \eta_{t}(v_t)\rt\rangle - \big\langle\eta_t^{\prime}(v_t)\big\rangle \sum_{j = 1}^{t-1} \mu^j\beta_{t-1}^j\Bigg| \\
&\qquad \qquad \leq 
\sup_{\theta\in \Theta} \big|f_{\theta}(\Phi) - \Exs [f_{\theta}(\Phi)]\big|
+
\sup_{\theta\in \Theta} \Big|\Exs [f_{\theta}(\Phi)] - \big\langle\eta^{\prime}(v)\big\rangle \cdot \sum_{j = 1}^{t-1} \mu^j  \beta^j\Big|\\
&\qquad \qquad = 
\sup_{\theta\in \Theta} \big|f_{\theta}(\Phi) - \Exs [f_{\theta}(\Phi)]\big|
+
\sup_{\theta\in \Theta} \Big|\frac{\gamma}{n}\sum_{i = 1}^n \mu^\top \beta \cdot\mathbb{P}\lt(|h_{i, \theta}| > \tau\rt) - \big\langle\eta^{\prime}(v)\big\rangle  
	\mu^{\top} \beta \Big|,
\end{align*}
where the last equality follows from Stein's lemma, that is,
\begin{align*}
 \Exs [f_{\theta}(\Phi)] = 
\mathbb{E}\lt[\lt\langle \sum_{j = 1}^{t-1} \mu^j\phi_j, \eta(v)\rt\rangle\rt] =
\Exs\lt[\lt\langle\eta^{\prime}(v)\rt\rangle \sum_{j = 1}^{t-1} \mu^j\beta^j\rt] = 
\frac{\gamma}{n}\sum_{i = 1}^n  \mu^\top \beta \cdot\mathbb{P}\lt(|h_{i, \theta}| > \tau\rt).
\end{align*}
Taking the decomposition above collectively with \eqref{eqn:sparse-At-1} and \eqref{eqn:sparse-At-3} yields 
\begin{align}
\Bigg|\lt\langle \sum_{j = 1}^{t-1} \mu^j\phi_j, \eta_{t}(v_t)\rt\rangle - \lt\langle\eta_t^{\prime}\rt\rangle \sum_{j = 1}^{t-1} \mu^j\beta_{t-1}^j\Bigg| 
&\lesssim \sqrt{\frac{t\log n}{n}} + \frac{\sqrt{t(t+k)}\log n}{n} \asymp \sqrt{\frac{t\log n}{n}}  =: A_t, \label{eqn:sparse-At}
\end{align}
where the last relation is valid under Assumption \eqref{cond:t-k}.


\subsubsection{Quantity $B_t$ in \eqref{defi:B}}
\label{sec:control-Bt-sparse}

Recall that quantity $B_{t}$ is concerned with bounding $v^{\star\top}\eta_{t}(v_t)$.  
To do so, let us again adopt the definitions of $\Phi, \theta, \Theta, v$ as in \eqref{eq:defn-Phi-Theta-theta-sparse}, 
and definte the following function parameterized by $\theta$: 
\begin{align*}
	f_{\theta}(\Phi) \defn v^{\star\top}\eta(v),
\end{align*}
with the function $\eta$ defined in \eqref{eq:eta-notation-suppressed-sparse}. 
In order to bound $v^{\star\top}\eta_{t}(v_t)$, 
we first develop a valid bound on $f_{\theta}(\Phi)$ that is valid simultaneously for all $\theta\in \Theta$.

Towards this end, consider any fixed parameter $\theta \in \Theta$, and apply \eqref{eqn:sparse-derivative} to reach
\begin{align}
\lt\|\nabla_{\Phi} f_{\theta}(\Phi)\rt\|_2 &\le 
	\frac{\|\beta\|_2}{\sqrt{n}}\lt\|v^{\star} \circ \eta^{\prime}(v)\rt\|_2 \lesssim \frac{1}{\sqrt{n}} \cdot \frac{1}{\lambda } \|\vstar\|_2
	\lesssim \sqrt{\frac{1}{n\lambda^2}},
	\label{eq:grad-f-Bt-sparse}
\end{align}
where we have used property~\eqref{eqn:sparse-eta-a-prod} as well as the fact that $\|\beta\|_2=1$. 
Additionally, it is easily seen that $\lt\|\nabla_{\theta} f_{\theta}(\Phi)\rt\|_2 \lesssim n^{100}$ for any $\Phi\in \mathcal{E}$ and 
$|f_{\theta}(\Phi)|\lesssim n^{100} \max_j \|\phi_j\|_2^{100}$. 
As a consequence,  Corollary~\ref{cor:Gauss-jump} taken together with \eqref{eqn:brahms-conc} indicates that, with probability at least $1-O(n^{-11})$, 
\begin{align}
\notag \lt|v^{\star\top}\eta_{t}(v_t) - v^{\star\top}\int\eta_t\lt(\alpha_tv^{\star} + \frac{1}{\sqrt{n}}x\rt)\varphi_n(\dx)\rt| &\le \sup_{\theta} \lt|v^{\star\top}\eta - v^{\star\top}\int\eta\lt(\alpha v^{\star} + \frac{1}{\sqrt{n}}x\rt)\varphi_n(\dx)\rt| \\
&\lesssim \sqrt{\frac{t\log n}{n\lambda^2}} =: B_t. \label{eqn:sparse-Bt}
\end{align}

\subsubsection{Bounding quantity $\kappa_t$}

This subsection develops an upper bound on the quantity $\kappa_t^2$
defined in \eqref{defi:kappa}, which is crucial in controlling $D_t$. 
From the choices of the denoising functions, $\eta_{t}^{\second}$ is well-defined and equal to $0$ except at two non-differentiable points. 
To bound $\kappa_t^2$, it is thus sufficient to control quantities 
$\langle \int[\eta_{t}^{\prime}(\alpha_t\vstar + \frac{1}{\sqrt{n}}x\big)]^2\varphi_n(\dx) \rangle$ and $\langle\int[x\eta_{t}^{\prime}\big(\alpha_t\vstar + \frac{1}{\sqrt{n}}x\big)]^2 \varphi_n(\dx)\rangle$ separately, given that $\|\beta_{t-1}\|_2=1$. 



Let us first consider the term $\langle \int[\eta_{t}^{\prime}(\alpha_t\vstar + \frac{1}{\sqrt{n}}x\big)]^2\varphi(\dx) \rangle$.
Recall our induction hypothesis $\alpha_t \asymp \lambda$ as well as our assumptions 
$\lambda \gtrsim \sqrt{\frac{k\log n}{n}}$ and $\tau_t \asymp \sqrt{\frac{\log n}{n}}$. 
We shall divide the index set $[n]$ into two parts and look at each part separately. 
For those indices $i$ obeying $v_i^{\star}\neq 0$, one has the trivial upper bound
\begin{align}
\int\ind\Big(\Big|\alpha_{t}v_{i}^{\star}+\frac{1}{\sqrt{n}}x\Big|>\tau_{t}\Big)\varphi(\dx) \leq 1. 
	\label{eq:indicator-large-sparse-123}
\end{align}
Otherwise, for those entries with $v_i^{\star}=0$, we find that 
\begin{align*}
	\int\ind\Big(\Big|\alpha_{t}v_{i}^{\star}+\frac{1}{\sqrt{n}}x\Big|>\tau_{t}\Big)\varphi(\dx)
	= 	\int\ind\Big(\Big|\frac{1}{\sqrt{n}}x\Big|>\tau_{t}\Big)\varphi(\dx)
	\leq
	2\int_{\sqrt{n}\tau_t}^{\infty} \varphi(\dx) \lesssim \frac{1}{n},
\end{align*}
provided that $\tau_t \geq 2 \sqrt{\frac{\log n}{n}}$. 
Putting these two cases together gives 
\begin{align}
\notag \lt\langle\int\Big[\eta_{t}^{\prime}\Big(\alpha_tv + \frac{\|\beta_{t-1}\|_2}{\sqrt{n}}x\Big)\Big]^2\varphi_n(\dx)\rt\rangle 
&= \gamma_t^2\lt\langle\int \ind\Big(\Big|\alpha_tv^{\star} + \frac{1}{\sqrt{n}}x\Big| > \tau_t 1\Big)\varphi_n(\dx)\rt\rangle  \\
	& \lesssim \frac{1}{n\lambda^2} \lt( k\cdot 1 + (n-k)\cdot \frac{1}{n} \rt) 
\asymp \frac{k}{n\lambda^2}. \label{eqn:kt-1}
\end{align}

Similarly, it can also be established that 
\begin{align}
\label{eqn:kt-2}
\lt\langle\int\lt[x\eta_{t}^{\prime}\lt(\alpha_tv^{\star} + \frac{1}{\sqrt{n}}x\rt)\rt]^2\varphi_n(\dx)\rt\rangle 
&= \gamma_t^2\lt\langle\int x^2\ind\lt(\lt|\alpha_tv^{\star} + \frac{1}{\sqrt{n}}x\rt| > \tau_t 1\rt)\varphi_n(\dx)\rt\rangle
\lesssim \frac{k}{n\lambda^2}. 
\end{align}
Consequently, putting the above two cases together with  the definition \eqref{defi:kappa} yields
\begin{align}
\label{eqn:sparse-kappa}
	\kappa_t^2 \lesssim \frac{k}{n\lambda^2}. 
\end{align}

\subsubsection{Quantity $D_t$ in \eqref{defi:D}}

We now turn to the analysis of $D_{t}$. 
Note that $\eta_{t}^{\second}$ is well-defined and equal to $0$ except at two non-differentiable points. 
Hence, to control $D_{t}$, it is sufficient to consider the following function:
\begin{align}
	\Big\|\sum_{j = 1}^{t-1} \mu^j_t\phi_j \circ \eta_{t}^{\prime}(v_t) \Big\|_2^2 
	&= 
	\gamma_t^2 \sum_{i = 1}^n \Big(\sum_{j = 1}^{t-1} \mu^j_t \phi_{j, i}\Big)^2 \ind\Big(\Big|\alpha_tv_i^{\star} + \sum_j\beta_{t-1}^j\phi_{j, i}\Big| > \tau_t\Big).
\end{align}
%

Setting the stage, let us define $\Phi, \theta, \Theta, v, \eta$ as in \eqref{eq:defn-Phi-Theta-theta-sparse} and \eqref{eq:eta-notation-suppressed-sparse}, 
and introduce the following functions:  
\begin{align*}
	f_{i, \theta}(\Phi_{i, :}) \defn \bigg(\sum_{j = 1}^{t-1} \mu^j\phi_{j, i}\bigg)^2,
\qquad \text{ and } \qquad
	h_{i, \theta}(\Phi_{i, :}) \defn \bigg| \alpha v_i^{\star} + \sum_{j=1}^{t-1}\beta^j\phi_{j, i} \bigg| .
\end{align*}
For every fixed $\mu\in \mathcal{S}^{t-2}$, $\sum_{j = 1}^{t-1} \mu^j\phi_{j, i}$ 
is Gaussian with mean zero and variance $1/n$; 
therefore, $f_{i,\theta} \ge 0$ is $\frac{1}{n}$-subexponential with $\mathbb{E}[f_{i,\theta}] = 1/n$ (see 
\citet[Lemma 2.7]{vershynin2018high}).  
In addition, it can be straightforwardly checked that 
(i) $\|\nabla_{\theta} f_{i, \theta}(\Phi_{i, :})\|_2\lesssim n^{100}$ for any $\Phi\in \mathcal{E}$; 
(ii) $|f_{i, \theta}(\Phi_{i, :})|\lesssim n^{100} \|\Phi\|_{\mathrm{F}}^{100}$; 
and (iii) $\mathbb{P}\big(\tau - 400n^{-100} \leq h_{i, \theta}(\Phi_{i, :}) \leq \tau + 400n^{-100}\big)\lesssim n^{-1}$ for any $\tau\in \real$ and any $\theta \in \Theta$.   
By virtue of Corollary~\ref{cor:Gauss-jump} and \eqref{eqn:brahms-conc}, 
we can readily see that, with probability at least $1-O(n^{-11})$,
\begin{align*}
	&\sup_{\theta \in \Theta}\lt|\sum_{i = 1}^n \Big(\sum_{j = 1}^{t-1} \mu^j\phi_{j, i}\Big)^2\ind\Big( \Big| \alpha v_i^{\star} + \sum_{j=1}^{t-1}\beta^j\phi_{j, i} \Big| > \tau\Big) - \mathbb{E}\bigg[\Big\|\sum_{j = 1}^{t-1} \mu^j\phi_j \circ \ind\Big( \Big| \alpha v^{\star} + \sum_{j=1}^{t-1}\beta^j\phi_{j} \Big| > \tau 1\Big)\Big\|_2^2\bigg]\rt| \\
	&\qquad \lesssim \sup_{\theta\in \Theta}\frac{1}{n}\sqrt{\sum_{i = 1}^n \mathbb{P}\lt( \Big| \alpha v_i^{\star} + \sum_{j = 1}^{t-1}\beta^j\phi_{j, i} \Big| > \tau\rt)t\log^3 n} + \frac{t\log^2 n}{n} \lesssim \sqrt{\frac{t(t+k)\log^4 n}{n^2}}
\end{align*}
holds simultaneously for all $\theta\in \Theta$, 
where the last inequality follows from the same argument as in \eqref{eq:sum-h-tau-sparse-UB}.  
Additionally, recalling that for general denoising functions, we have established relation~\eqref{eqn:beethoven-vive}. 
When specialized to the current setting, it asserts that 
\begin{align*}
	\mathbb{E}\Bigg[\Big\|\sum_{j = 1}^{t-1} \mu^j\phi_j \circ \eta^{\prime}(v)\Big\|_2^2\Bigg] 
	- 
	\max\lt\{
	\Big\langle\int\Big[x\eta^{\prime}\Big(\alpha v^{\star} + \frac{1}{\sqrt{n}}x\Big)\Big]^2\varphi_n(\dx)\Big\rangle,
	\Big\langle\int\Big[\eta^{\prime}\Big(\alpha v^{\star} + \frac{1}{\sqrt{n}}x\Big)\Big]^2\varphi_n(\dx)\Big\rangle\rt\} 
	\leq 
	0. 
\end{align*}

Putting the above bounds together, using the definition \eqref{defi:kappa} of $\kappa_t$, 
and recognizing that $(\mu_t, \alpha_t, \beta_{t-1}, \tau_t, \gamma_t)\in \Theta$, 
we can obtain
\begin{align}
 & \bigg\|\sum_{j=1}^{t-1}\mu^{j}\phi_{j}\circ\eta_{t}^{\prime}(v_{t})\bigg\|_{2}^{2}-\kappa_{t}^{2} \notag\\
 & \lesssim\gamma_{t}^{2}\sup_{\theta\in\Theta}\Bigg|\sum_{i=1}^{n}\Big(\sum_{j=1}^{t-1}\mu^{j}\phi_{j,i}\Big)^{2}\ind\Big(\Big|\alpha v_{i}^{\star}+\sum_{j=1}^{t-1}\beta^{j}\phi_{j,i}\Big|>\tau\Big)-\mathbb{E}\bigg[\Big\|\sum_{j=1}^{t-1}\mu^{j}\phi_{j}\circ\ind\Big(\Big|\alpha v^{\star}+\sum_{j=1}^{t-1}\beta^{j}\phi_{j}\Big|>\tau\Big)\Big\|_{2}^{2}\bigg]\Bigg|\notag\\
 & \quad+\sup_{\theta\in\Theta}\left\{ \mathbb{E}\Bigg[\Big\|\sum_{j=1}^{t-1}\mu^{j}\phi_{j}\circ\eta^{\prime}(v)\Big\|_{2}^{2}\Bigg]-\max\lt\{\Big\langle\int\Big[x\eta^{\prime}\Big(\alpha v^{\star}+\frac{1}{\sqrt{n}}x\Big)\Big]^{2}\varphi_{n}(\dx)\Big\rangle,\Big\langle\int\Big[\eta^{\prime}\Big(\alpha v^{\star}+\frac{1}{\sqrt{n}}x\Big)\Big]^{2}\varphi_{n}(\dx)\Big\rangle\rt\}\right\} \notag\\
 & \lesssim\gamma_{t}^{2}\sqrt{\frac{t(t+k)\log^{4}n}{n^{2}}}
	\asymp \frac{1}{\lambda^{2}}\sqrt{\frac{t(t+k)\log^{4}n}{n^{2}}}\eqqcolon D_{t},
	\label{eqn:sparse-dt}
\end{align}
where we remind the reader that $\gamma_t\asymp \lambda^{-1}$ (see \eqref{eqn:gamma-t-evolution}).

%
%
%


\subsubsection{Quantity $E_t$ in \eqref{defi:E}}

We now turn attention to quantity $E_t$, 
which requires us to work with non-differentiable points. 
Note that the denosing function $\eta'_{t}$ is only non-differentiable at two points: $-\tau_{t}$ and $\tau_{t}$. 
The goal of this subsection to prove that: with probability at least $1-O(n^{-11})$,  
\begin{align}
\label{eqn:sparse-Et}
	\sum_{m\in\{\tau_{t},-\tau_{t}\}}\sum_{i=1}^{n}\ind\bigg(\bigg|\alpha_{t}v_{i}^{\star}+\sum_{j=1}^{t-1}\beta_{t-1}^{j}\phi_{j,i}-m\bigg|\le\theta(m)\bigg)\lesssim k+t\log^3 n+n\|\xi_{t-1}\|_{2}^{2}\eqqcolon E_{t}	
\end{align}
holds for any choice $\theta(m)$ satisfying 
\begin{align*}
\sum_{m\in\{\tau_{t},-\tau_{t}\}}\sum_{i=1}^{n}\bigg|\alpha_{t}v_{i}^{\star}+\sum_{j=1}^{t-1}\beta_{t-1}^{j}\phi_{j,i}-m\bigg|^{2}\ind\bigg(\bigg|\alpha_{t}v_{i}^{\star}+\sum_{j=1}^{t-1}\beta_{t-1}^{j}\phi_{j,i}-m\bigg|\le\theta(m)\bigg)\le\|\xi_{t-1}\|_{2}^{2}.	
\end{align*}

Towards this, let us adopt the definitions of $\Phi, \theta, \Theta, v$ in \eqref{eq:defn-Phi-Theta-theta-sparse} as before, 
and generate a Gaussian random variable $z \sim \mathcal{N}(0, 1/n)$.  
As shall be seen momentarily, the following two relations hold true uniformly over all $\theta \in \Theta$ and all $\omega\in \real$ obeying $\omega\lesssim n$: 
\begin{subequations}
\begin{align}
\label{eqn:simon}
\notag &\Bigg\|\bigg|\alpha v^{\star} + \sum_{j = 1}^{t-1} \beta^j\phi_j - \tau 1 \bigg|\circ
\ind\bigg(\Big|\alpha v^{\star} + \sum_{j = 1}^{t-1} \beta^j\phi_j - \tau 1\Big| \le \omega 1\bigg)\Bigg\|_2^2 \\
&\quad \ge (n-k)\mathbb{E}\lt[\lt|z - \tau\rt|^2\ind\lt(\lt|z - \tau 1\rt| \le \omega\rt)\rt]  - \frac{\log n}{n} \cdot O\lt(\sqrt{(n-k)\mathbb{P}\lt(\lt|z - \tau\rt| \le \omega\rt)t\log n} + t\log n\rt), \\
\label{eqn:garfunkel}
& \bigg\|\ind\Big(\Big|\alpha v^{\star} + \sum_{j = 1}^{t-1} \beta^j\phi_j - \tau 1\Big| \le \omega 1\Big)\bigg\|_0 \lesssim k + (n-k) \mathbb{P}\lt(\lt|z - \tau\rt| \le \omega\rt) + t\log n.
\end{align}
\end{subequations}
Taking these two inequalities~\eqref{eqn:simon} and \eqref{eqn:garfunkel} as given for the moment (which we shall return to prove shortly), 
we proceed to justify the following claim: 
for any point $\omega\in \real$ obeying $\omega \lesssim n$ and 
\begin{align}
	\Bigg\|\bigg|\alpha_t v^{\star} + \sum_{j = 1}^{t-1} \beta^j_{t-1} \phi_j - \tau_t 1\bigg|
	\circ \ind\bigg(\Big|\alpha_t v^{\star} + \sum_{j = 1}^{t-1} \beta^j_{t-1} \phi_j - \tau_t 1 \Big| \le \omega 1 \bigg)\Bigg\|_2^2 \leq 
	\|\xi_{t-1}\|_2^2, 
	\label{eqn:sparse-Et-condition-1}
\end{align}
one necessarily satisfies 
\begin{align}
\label{eqn:sparse-Et-temp}
	\sum_{m\in\{\tau_{t},-\tau_{t}\}}\sum_{i=1}^{n}\ind\bigg(\bigg|\alpha_{t}v_{i}^{\star}+\sum_{j=1}^{t-1}\beta_{t-1}^{j}\phi_{j,i}-m\bigg|\le \omega \bigg)\lesssim k+t\log^3 n+n\|\xi_{t-1}\|_{2}^{2}.	
\end{align}
If this were valid, then one could immediately establish~\eqref{eqn:sparse-Et}, thus completing the control of $E_{t}$.

In what follows, let us prove this claim \eqref{eqn:sparse-Et-temp}. 
\begin{itemize}
\item 
	Suppose the point $\omega$ satisfies $\mathbb{P}\lt(\lt|z - \tau\rt| \le \omega\rt) \lesssim \frac{k+t\log^3 n}{n}$.
	Then in view of \eqref{eqn:garfunkel}, one has 
\begin{align*}
	\bigg\|\ind\Big(\Big|\alpha v^{\star} + \sum_{j = 1}^{t-1} \beta^j\phi_j - \tau\Big| \le \omega\Big)\bigg\|_0 \lesssim k + t\log^3 n, 
\end{align*}
which holds uniformly over all $\theta \in \Theta$. 
If this is the case for our choice $(\alpha, \beta,\tau) = (\alpha_t, \beta_{t-1}, \pm \tau_t)$, 
then we have established inequality~\eqref{eqn:sparse-Et}.

\item 

Consider now the complement case where $\omega$ satisfies
\begin{equation}
	\mathbb{P}\lt(\lt|z - \tau\rt| \leq \omega\rt) \gg \frac{k+t\log^3 n}{n}.
	\label{eq:complement-z-omega-tau-sparse}
\end{equation}
We first make note of the fact that $\omega$ needs to satisfy $\omega \geq \sqrt{8/n}$ in this case; 
otherwise one must have 
$$
		\mathbb{P}\lt(\lt|z - \tau\rt| \le \omega\rt) 
		=\mathbb{P}\lt( \tau - \omega \le  z \le \tau + \omega\rt) \le 
		\mathbb{P}\lt(z \ge \tau - \omega\rt)   \le \mathbb{P}\lt(z \geq \sqrt{\frac{2\log n}{n}}\rt) \le \frac{k+t\log^3 n}{n},
$$
which belongs to the previous case.  
Based on this simple observation, direct calculations lead to 
\begin{align*}
	& \mathbb{E}\lt[\lt|z-\tau\rt|^{2}\ind\lt(\lt|z-\tau\rt|\le\omega\rt)\rt]  =\mathbb{E}\lt[\lt|z-\tau\rt|^{2}\,\big|\,\lt|z-\tau\rt|\le\omega\rt]\mathbb{P}\lt(\lt|z-\tau\rt|\le\omega\rt)\\
 & \qquad \geq\mathbb{E}\lt[\lt|z-\tau\rt|^{2}\ind\{z\in[\tau-\omega,\tau-\omega/2]\}\,\big|\,\lt|z-\tau\rt|\le\omega\rt]\mathbb{P}\lt(\lt|z-\tau\rt|\le\omega\rt)\\
 & \qquad \geq \bigg(\frac{\omega}{2}\bigg)^{2}\mathbb{P}\Big( z\in[\tau-\omega,\tau-\omega/2]\,\big|\,\lt|z-\tau\rt|\le\omega\Big)\mathbb{P}\lt(\lt|z-\tau\rt|\le\omega\rt) \\
 & \qquad \geq \frac{\omega^2}{8} \mathbb{P}\lt(\lt|z-\tau\rt|\le\omega\rt)  \geq \frac{1}{n}\mathbb{P}\lt(\lt|z-\tau\rt|\le\omega\rt),
\end{align*}
which combined with \eqref{eqn:simon} and \eqref{eq:complement-z-omega-tau-sparse} gives
\begin{align*}
	& \bigg\|\Big|\alpha v^{\star}+\sum_{j=1}^{t-1}\beta^{j}\phi_{j}-\tau 1\Big| \circ \ind\Big(\Big|\alpha v^{\star}+\sum_{j=1}^{t-1}\beta^{j}\phi_{j}-\tau1\Big|\le\omega1\Big)\bigg\|_{2}^{2}\\
 & \qquad \geq(n-k)\mathbb{E}\lt[\lt|z-\tau\rt|^{2}\ind\lt(\lt|z-\tau1\rt|\le\omega\rt)\rt]-\frac{\log n}{n}\cdot O\lt(\sqrt{(n-k)\mathbb{P}\lt(\lt|z-\tau\rt|\le\omega\rt)t\log n}+t\log n\rt)\\
 & \qquad \geq\mathbb{P}\lt(\lt|z-\tau\rt|\le\omega\rt)-\frac{\log n}{n}\cdot O\lt(\sqrt{(n-k)\mathbb{P}\lt(\lt|z-\tau\rt|\le\omega\rt)t\log n}+t\log n\rt)\gtrsim\mathbb{P}\lt(\lt|z-\tau\rt|\le\omega\rt).
\end{align*}
Setting $\alpha=\alpha_t$, $\beta=\beta_{t-1}$ and $\tau= \pm \tau_t$ and utilizing \eqref{eqn:sparse-Et-condition-1}, 
we arrive at
\[
	\big\| \xi_{t-1}\big\|_2^2 \gtrsim \mathbb{P}\lt(\lt|z-\tau_t\rt|\le\omega\rt).
\]
Taking this and expression~\eqref{eqn:garfunkel} collectively  yields our advertised bound \eqref{eqn:sparse-Et}, given that $\theta(m)$ is trivially below $n$ with high probability.
\end{itemize}

With the above arguments in mind, 
everything comes down to establishing the inequalities~\eqref{eqn:simon} and \eqref{eqn:garfunkel}, which shall be done in the following.

\paragraph{Proof of inequality~\eqref{eqn:simon}.}

First, given any $\theta\in \Theta$ and any $\omega \lesssim n$, we find it useful to develop the following lower bound:  
\begin{align}
\label{eqn:simple-lb}
\notag &\Bigg\|\Big|\alpha_tv^{\star} + \sum_{j = 1}^{t-1} \beta^j\phi_j - \tau 1\Big|\ind\bigg(\Big|\alpha_tv^{\star} + \sum_{j = 1}^{t-1} \beta^j\phi_j - \tau 1\Big| \le \omega 1\bigg)\Bigg\|_2^2 \\
\notag &\qquad \qquad = \sum_{i = 1}^n \Big|\alpha v_i^{\star} + \sum_{j = 1}^{t-1} \beta^j\phi_{j, i} - \tau \Big|^2 
\ind\Big(\Big|\alpha v_i^{\star} + \sum_{j = 1}^{t-1} \beta^j\phi_{j, i} - \tau \Big| \le \omega\Big) \\
&\qquad \qquad \ge \sum_{i :\, v_i^{\star} = 0} \Big|\sum_{j = 1}^{t-1} \beta^j\phi_{j, i} - \tau \Big|^2\ind\Big(\Big|\sum_{j = 1}^{t-1} \beta^j\phi_{j, i} - \tau \Big| \le \omega\Big).
\end{align}
Next, we aim to further bound the right-hand side of \eqref{eqn:simple-lb} by means of Corollary~\ref{cor:Gauss-jump}. 
%
%
%

Towards this, let us define the following functions:
\begin{align*}
	f_{i, \theta}(\Phi_{i, :}) &\defn \Big(\sum_{j = 1}^{t-1} \beta^j\phi_{j, i} - \tau\Big)^2,
	\qquad
	\text{ and }
	\qquad
	h_{i,\theta}(\Phi_{i, :}) \defn \bigg| \sum_{j = 1}^{t-1} \beta^j\phi_{j, i} - \tau \bigg|.
\end{align*}
Note that for any fixed $\beta \in \mathcal{S}^{t-2}$, $\sum_{j = 1}^{t-1} \beta^j\phi_{j, i}$ follows a Gaussian distribution with variance $1/n$;   
therefore, $f_{i, \theta}(\Phi_{i, :})$ is a $\big(\tau^2 + \frac{1}{n}\big)$-subexponential random variable with mean $\mathbb{E}[f_{i, \theta}] = \tau^2 + \frac{1}{n}$. 
We make the observation that 
(i) $|f_{i, \theta}(\Phi_{i, :})|\lesssim n^{100} \max_j \|\phi_j\|_2^{100}$; 
(ii) $\|\nabla_{\theta} f_{i, \theta} (\Phi_{i, :})\|_2 \lesssim n^{100}$  for any $\Phi\in \mathcal{E}$; 
(iii) $\mathbb{P}\big( \tau - 400n^{-100}  \leq h_{i,\theta}(\Phi_{i, :}) \leq \tau + 400n^{-100} \big) \lesssim 1/n$ for any $\tau\in \real$. 
Therefore, applying Corollary~\ref{cor:Gauss-jump} together with \eqref{eqn:brahms-conc} yields: with probability exceeding $1-O(n^{-11})$, 
%
%
%
\begin{align}
\label{eqn:hawaii}
\notag &\sum_{i : v_i^{\star} = 0} \Big|\sum_{j = 1}^{t-1} \beta^j\phi_{j, i} - \tau\Big|^2\ind\Big(\Big|\sum_{j = 1}^{t-1} \beta^j\phi_{j, i} - \tau\Big| \le \omega\Big) - (n-k)\mathbb{E}\lt[\lt|z - \tau\rt|^2\ind\lt(\lt|z - \tau\rt| \le \omega\rt)\rt] \\
&\qquad \qquad \ge  - \frac{\log n}{n} \cdot O\lt(\sqrt{(n-k)\mathbb{P}\lt(\lt|z - \tau\rt| \le \omega\rt)t\log n} + t\log n\rt),
\end{align}
holds simultaneously for all $\theta \in \Theta$ and all $\omega \lesssim n$, 
where we denote $z\sim \mathcal{N}(0,1/n)$. 
Careful readers might already notice that: when applying Corollary~\ref{cor:Gauss-jump}, instead of considering the indicator function $\ind\lt(h_{i, \theta}(X_i) > \tau\rt)$ as in the original form, our result above is concerned with a different kind of indicator function $\ind\lt( h_{i, \theta}(X_i) < \tau\rt)$;  
fortunately, the proof of this version of indicator functions follows verbatim as that of Lemma~\ref{lem:Gauss-jump} and Corollary~\ref{cor:Gauss-jump}, 
and hence we omit the details here. 
Putting \eqref{eqn:simple-lb} and \eqref{eqn:hawaii} together, we have validated relation~\eqref{eqn:simon}.

\paragraph{Proof of inequality~\eqref{eqn:garfunkel}.} 

Using exactly the same analysis as above and taking $f_{i, \theta}(\Phi_{i, :}) = 1$ and $h_{i,\theta}(\Phi_{i, :}) = \big| \sum_{j = 1}^{t-1} \beta^j\phi_j - \tau \big|$, Corollary~\ref{cor:Gauss-jump} ensures that 
\begin{align*}
	\bigg\|\ind\bigg(\Big|\alpha v^{\star} + \sum_{j = 1}^{t-1} \beta^j\phi_j - \tau\Big| \le \omega\bigg)\bigg\|_0 
	&\le k + \sum_{i :\, v_i^{\star} = 0} \ind\Big(\Big|\sum_{j = 1}^{t-1} \beta^j\phi_{j, i} - \tau\Big| \le \omega\Big) \\
	&\lesssim k + \sqrt{(n-k) \mathbb{P}\lt(\lt|z - \tau\rt| \le \omega\rt)t\log n} + t\log n \\
	&\lesssim k + (n-k) \mathbb{P}\lt(\lt|z - \tau\rt| \le \omega\rt) + t\log n
\end{align*}
holds simultaneously  for all $\theta \in \Theta$ and all $\omega \lesssim n$.
This completes the proof of \eqref{eqn:garfunkel}. 




\subsection{Establishing the induction hypotheses via recursion}
\label{sec:recursion}

The goal of this subsection is to finish the induction-based proof of \eqref{eqn:sparse-induction}. 
We shall start by establishing \eqref{eqn:sparse-induction} for the $(t+1)$-th iteration, assuming that it holds for the $t$-th iteration. 
We will then return to verify the base case for the two types of initialization methods. 
Before proceeding, we remind the readers of our assumptions: 
\begin{align}
\label{eqn:don-giovani}
	\frac{t\log^3 n}{n\lambda^2} \ll 1,\qquad \frac{k\log n}{n\lambda^2} \ll 1,
\end{align}
and we shall always take $\tau_{t}$ to be on the order of $\sqrt{\frac{\log n}{n}}$ with some sufficiently large preconstant.  
In addition, 
\begin{equation}
	\label{eq:rho-sparse}
	|\eta_t^{\prime}(w)| \lesssim \frac{1}{\lambda} \eqqcolon \rho ,
	\quad
	|\eta_t^{\second}(w)| = 0 \eqqcolon \rho_1,
	\quad 
	|\eta_t^{\third}(w)| = 0 \eqqcolon \rho_2
	\qquad 
	\text{for any differentiable point }w\in \real, 
\end{equation}
where the calculation of $\rho$ has made use of \eqref{eqn:gamma-t-evolution}. 



\subsubsection{Inductive step for \eqref{eqn:sparse-induction} regarding $\xi_{t}$, $\Delta_{\alpha,t}$ and $\alpha_t$ }

Assuming the induction hypotheses~\eqref{eqn:sparse-induction} hold at $t$, 
we intend to prove their validity for $t+1.$

\paragraph{Bounding $\xi_{t}$.}
In terms of $\|\xi_{t}\|_2$, the result~\eqref{eqn:xi-t-general} of Theorem~\ref{thm:main} together with \eqref{eq:rho-sparse} and $\|\beta_t\|_2=1$ gives 
\begin{align*}
\|\xi_{t}\|_2 &\le \sqrt{\kappa_t^2 + D_t} \,\|\xi_{t-1}\|_2 \\
&\qquad+ O\lt(\sqrt{\frac{t\log n}{n}}\|\beta_{t}\|_2 + A_t + \lt(\sqrt{\frac{t}{n}}\rho_1 + \frac{\rho_2\|\beta_{t-1}\|_2}{n}\rt)\|\xi_{t-1}\|_2^2 + \rho\sqrt{\frac{E_t + t\log n}{n}}\|\xi_{t-1}\|_2 + \frac{\rho E_t\|\beta_{t-1}\|_2}{n}\rt) \\
&\le \sqrt{\kappa_t^2 + D_t} \, \|\xi_{t-1}\|_2 
+ O\lt(\sqrt{\frac{t\log n}{n}} + A_t + \sqrt{\frac{E_t + t\log n}{\lambda^2 n}}\|\xi_{t-1}\|_2 + \frac{ E_t}{\lambda n}\rt).
\end{align*}
Making use of the bounds \eqref{eqn:sparse-kappa},  \eqref{eqn:sparse-At}, \eqref{eqn:sparse-dt} and \eqref{eqn:sparse-Et}, we can further derive 
\begin{align}
\|\xi_{t}\|_2	&\le c \sqrt{\frac{k}{n\lambda^2} + \frac{\sqrt{t(t+k)}\log^2 n}{n\lambda^2}}
\|\xi_{t-1}\|_2 + O\lt(\sqrt{\frac{t\log n}{n}} + \frac{1}{\lambda}\sqrt{\frac{E_t + t\log n}{n}}\|\xi_{t-1}\|_2 + \frac{E_t}{n\lambda}\rt)  \notag\\
	&\le c^{\prime}\|\xi_{t-1}\|_2 + C_7\bigg(\sqrt{\frac{t\log n}{n}} 
	+ \sqrt{\frac{k+t \log^3 n}{n \lambda^2}}\|\xi_{t-1}\|_2 + \frac{1}{\lambda}\|\xi_{t-1}\|_2^2 + \frac{k+t\log^3 n}{n\lambda}\bigg) 
	\label{eq:xi-recurse-sparse-1}
\end{align}
for some large enough constant $C_7>0$, where we take 
$$
	c' \defn c \sqrt{\frac{k}{n\lambda^2} + \frac{\sqrt{t(t+k)}\log^2 n}{n\lambda^2}} \ll 1
$$
under the assumptions \eqref{eqn:don-giovani}. 
Supposing that $$\ltwo{\xi_{t-1}} \leq C_8 \sqrt{\frac{(t-1)\log^{3}n+k}{n}}$$ 
for some constant $C_8>0$ large enough, we can invoke \eqref{eq:xi-recurse-sparse-1} to reach
\begin{align}
\|\xi_{t}\|_{2} & \le c^{\prime}C_{8}\sqrt{\frac{(t-1)\log^{3}n+k}{n}}\\
 & \quad+C_{7}\bigg(\sqrt{\frac{t\log n}{n}}+C_{8}\sqrt{\frac{k+t\log^{3}n}{n\lambda^{2}}}\sqrt{\frac{(t-1)\log^{3}n+k}{n}}+\frac{C_{8}^{2}\big((t-1)\log^{3}n+k\big)}{\lambda n}+\frac{k+t\log^{3}n}{n\lambda}\bigg)\\
 & \leq C_{8}\sqrt{\frac{t\log^{3}n+k}{n}}
\end{align}
%
where we have made use of the relation~\eqref{eqn:don-giovani} and the condition $c'\ll 1$. 
This in turn finishes the inductive step for bounding $\|\xi_{t}\|_2$.

%

\paragraph{Bounding $\Delta_{\alpha,t}$.} 

A direct application of inequality~\eqref{eqn:delta-alpha-general} in Theorem~\ref{thm:main} together with \eqref{eq:rho-sparse} gives 
\begin{align*}
|\Delta_{\alpha,t}| \,\lesssim\, B_t + \rho \|\xi_{t-1}\|_2.
\end{align*}
%
Replacing $B_{t}$ and $\rho$ with their corresponding bounds in \eqref{eqn:sparse-Bt} and \eqref{eqn:rho-sparse}, 
and invoking \eqref{eqn:sparse-induction}, we arrive at 
\begin{align}
|\Delta_{t, \alpha}| &\lesssim \sqrt{\frac{t\log n}{n\lambda^2}} + \frac{1}{\lambda} \|\xi_{t-1}\|_2 \lesssim \sqrt{\frac{k + t\log^3 n}{n\lambda^2}}.
	\label{eq:Delta-t-UB-sparse}
\end{align}


\paragraph{Controlling $\alpha_t$.} 

Equipped with the control of $|\Delta_{t, \alpha}|$ in \eqref{eq:Delta-t-UB-sparse}, we can now prove that $\alpha_{t+1} \asymp \lambda$ for all $t \geq 1$; 
in fact, we intend to prove a stronger result, namely, if $\alpha_{t} \asymp \lambda$, 
\begin{align}
\label{eqn:stable-alphat-sparse}
	\alpha_{t+1} = \lambda + O\bigg(\sqrt{\frac{k\log n + (t+1)\log^3 n}{n}}\bigg).
\end{align}
%
Taking the definition \eqref{eqn:alpha-t-genearl} of $\alpha_{t+1}$ collectively with expressions~\eqref{eqn:beethoven} and \eqref{eq:residual-sparse} 
as well as the property $\lambda \asymp \alpha_t \gtrsim \sqrt{\frac{(k+t)\log n}{n}}$ gives 
\begin{align}
\alpha_{t+1} &= \lambda  v^{\star \top} \int \gamma_t \mathsf{ST}_{\tau_t}\left(\alpha_{t} \vstar + \frac{x}{\sqrt{n}} \right)\varphi_n(\dx) + \lambda\Delta_{\alpha,t} \notag\\
&= \frac{\lambda v^{\star \top} \int \mathsf{ST}_{\tau_t}\left(\alpha_{t} \vstar + \frac{x}{\sqrt{n}} \right)\varphi_n(\dx)}{\sqrt{\int\Big\|\mathsf{ST}_{\tau_t}\left(\alpha_{t} \vstar + \frac{x}{\sqrt{n}} \right)\Big\|_{2}^{2}\varphi_{n}(\dx)} + O\Big(\sqrt{\frac{t\log n}{n\lambda^{2}}}+ \|\xi_{t-1}\|_2 \Big)} + \lambda\Delta_{\alpha,t} \notag\\
&= \lt( 1+ O\Bigg(\frac{1}{\lambda}\sqrt{\frac{t\log n}{n}}+\frac{\|\xi_{t-1}\|_{2}}{\lambda}\Bigg) \rt)
\frac{\lambda v^{\star \top} \int \mathsf{ST}_{\tau_t}\left(\alpha_{t} \vstar + \frac{x}{\sqrt{n}} \right)\varphi_n(\dx)}{\sqrt{\int\Big\|\mathsf{ST}_{\tau_t}\left(\alpha_{t} \vstar + \frac{x}{\sqrt{n}} \right)\Big\|_{2}^{2}\varphi_{n}(\dx)} } + O\Big(\sqrt{\frac{k + t\log^3 n}{n}}\Big). 
\label{eqn:sparse-alpha-recursion}
\end{align}


In order to bound \eqref{eqn:sparse-alpha-recursion}, 
we first observe that for every $\alpha$ obeying $ \sqrt{\frac{(k+t) \log n}{n}}\lesssim \alpha \lesssim 1$ and $\tau_{t}\asymp \sqrt{\frac{\log n}{n}}$, it holds that 
\begin{align*}
	\left\|\mathsf{ST}_{\tau_t}\left(\alpha \vstar + \frac{x}{\sqrt{n}} \right) - \alpha \vstar\right\|_2 &\le \left\|\mathsf{ST}_{\tau_t}\left(\alpha \vstar \right) - \alpha \vstar\right\|_2 + \left\|\mathsf{ST}_{\tau_t}\left(\alpha \vstar + \frac{x}{\sqrt{n}} \right) - \mathsf{ST}_{\tau_t}\left(\alpha \vstar \right)\right\|_2 \\
	&\le \left\|\tau_t 1 \circ \ind\left(\left|\vstar\right| > 0 \right)\right\|_2 + \left\|\frac{x}{\sqrt{n}} \circ \left(\ind\left(\left|\alpha \vstar\right| > \tau_t 1 \right) + \ind\left(\left|\alpha \vstar + \frac{x}{\sqrt{n}}\right| > \tau_t 1 \right)\right)\right\|_2, 
\end{align*}
where the last step invokes relation $\big||\alpha \vstar + \frac{x}{\sqrt{n}}| - |\alpha\vstar|\big| \leq \frac{|x|}{\sqrt{n}}.$
In addition, note that 
\begin{align*}
	\ind\left(\left|\alpha \vstar\right| > \tau_t 1 \right) + 
	\ind\left(\left|\alpha \vstar + \frac{x}{\sqrt{n}}\right| > \tau_t 1 \right)
	&\leq 
	\ind\left(\left|\alpha \vstar\right| > \tau_t 1 \right) + 
	\ind\left(\left|\alpha \vstar \right| > \frac{\tau_t}{2} 1 \right)
	+
	\ind\left(\left|\frac{x}{\sqrt{n}}\right| > \frac{\tau_t}{2} 1 \right) \\
	&\leq 2\ind(|\vstar| > 0) + \ind\left(\frac{|x|}{\sqrt{n}} > \frac{\tau_t}{2} 1 \right). 
\end{align*}
Taking the above two relations together yields 
\begin{align*}
		\left\|\mathsf{ST}_{\tau_t}\left(\alpha \vstar + \frac{x}{\sqrt{n}} \right) - \alpha \vstar\right\|_2 
		&\lesssim \left\|\left(\tau_t + \frac{|x|}{\sqrt{n}}\right) \circ \ind\left(\left|\vstar\right| > 0 \right)\right\|_2 + \left\|\frac{x}{\sqrt{n}} \circ \ind\left(\frac{|x|}{\sqrt{n}} > \frac{\tau_t}{2} 1 \right)\right\|_2,
\end{align*}
which further implies
\begin{align}
\label{eqn:cello-sonata}
	\notag \displaystyle \int \left\|\mathsf{ST}_{\tau_t}\left(\alpha \vstar + \frac{x}{\sqrt{n}} \right) - \alpha \vstar\right\|_2^2 \varphi_n(\dx) 
	&\lesssim  
	\displaystyle\int \sum_{i \in\{\vstar_i\neq 0\}} \left(\tau^2_t + \frac{x_i^2}{n}\right) \varphi(\dx)
	+
	\frac{1}{n}\sum_{i=1}^n 2\displaystyle\int^{\infty}_{2\sqrt{\log n}} x_i^2 \varphi(\dx)\\
	%
	&\lesssim \frac{k\log n}{n}.
\end{align}
In words, relation~\eqref{eqn:cello-sonata} ensures that $\mathsf{ST}_{\tau_t}\left(\alpha \vstar + \frac{x}{\sqrt{n}} \right)$ lies close to $ \alpha \vstar$.

Next, we would like to employ the above relation to show that
\begin{align}
\label{eqn:stable-inner-sparse}
	\frac{ v^{\star \top} \displaystyle\int \mathsf{ST}_{\tau_t}\left(\alpha \vstar + \frac{x}{\sqrt{n}} \right)\varphi_n(\dx)}{\sqrt{\displaystyle\int\Big\|\mathsf{ST}_{\tau_t}\left(\alpha \vstar + \frac{x}{\sqrt{n}} \right)\Big\|_{2}^{2}\varphi_{n}(\dx)}}
	=
	1 + O\Big(\frac{1}{\alpha}\sqrt{\frac{k\log n}{n}}\Big);
\end{align}
if this were true, then combining it with~\eqref{eqn:sparse-alpha-recursion}, \eqref{eq:residual-sparse} and $\alpha_t\asymp \lambda$ would justify the bound stated in \eqref{eqn:stable-alphat-sparse}. 
To prove~\eqref{eqn:stable-inner-sparse}, we find it helpful to consider the inner product between 
$\mathsf{ST}_{\tau_t}\left(\alpha \vstar + \frac{x}{\sqrt{n}} \right)$ and $ \alpha \vstar$ as follows
\begin{align}
\label{eqn:inner-product-11}
\notag	&\alpha v^{\star \top} \int \mathsf{ST}_{\tau_t}\left(\alpha \vstar + \frac{x}{\sqrt{n}} \right)\varphi_n(\dx)\\
\notag	&\quad =
	\displaystyle \int \Big\|\mathsf{ST}_{\tau_t}\left(\alpha \vstar + \frac{x}{\sqrt{n}} \right)\Big\|^2\varphi_n(\dx)
	+
	\displaystyle\int \Big(\alpha v^{\star} - \mathsf{ST}_{\tau_t}\Big(\alpha \vstar + \frac{x}{\sqrt{n}} \Big)\Big)^\top \mathsf{ST}_{\tau_t}\left(\alpha \vstar + \frac{x}{\sqrt{n}} \right)\varphi_n(\dx) \\
\notag	&\quad =
	\displaystyle\int\Big\|\mathsf{ST}_{\tau_t}\left(\alpha \vstar + \frac{x}{\sqrt{n}} \right)\Big\|_{2}^{2}\varphi_{n}(\dx)
	+
	O\left(\sqrt{\displaystyle\int \Big\|\alpha v^{\star} - \mathsf{ST}_{\tau_t}\Big(\alpha \vstar + \frac{x}{\sqrt{n}}\Big) \Big\|_2^2\varphi_n(x)} \sqrt{\displaystyle\int\Big\|\mathsf{ST}_{\tau_t}\left(\alpha \vstar + \frac{x}{\sqrt{n}} \right)\Big\|_{2}^{2}\varphi_{n}(\dx)}\right)\\
\notag	&\quad =
	\displaystyle\int\Big\|\mathsf{ST}_{\tau_t}\left(\alpha \vstar + \frac{x}{\sqrt{n}} \right)\Big\|_{2}^{2}\varphi_{n}(\dx)
	+
	O\Big( \alpha \sqrt{\frac{k\log n}{n}} + \frac{k\log n}{n} \Big) \\
	&\quad =
	\displaystyle\int\Big\|\mathsf{ST}_{\tau_t}\left(\alpha \vstar + \frac{x}{\sqrt{n}} \right)\Big\|_{2}^{2}\varphi_{n}(\dx)
	+
	O\Big( \alpha \sqrt{\frac{k\log n}{n}}  \Big),	
\end{align}
where the last line is valid since $\alpha \gtrsim \sqrt{\frac{k \log n}{n}}$, and the penultimate step uses inequality~\eqref{eqn:cello-sonata} and the following crude bound:  
\begin{align*}
	\displaystyle\int \Big\|\mathsf{ST}_{\tau_t}\Big(\alpha \vstar + \frac{x}{\sqrt{n}}\Big) \Big\|_2^2\varphi_n(x)
	\lesssim
	\displaystyle\int \Big\|\alpha v^{\star} - \mathsf{ST}_{\tau_t}\Big(\alpha \vstar + \frac{x}{\sqrt{n}}\Big) \Big\|_2^2\varphi_n(x)
	+
	\alpha^2 \leq  \frac{k\log n}{n} + \alpha^2.
\end{align*}
Similarly, we can also write 
\begin{align}
\label{eqn:big-ben}
\notag	&\alpha v^{\star \top} \int \mathsf{ST}_{\tau_t}\left(\alpha \vstar + \frac{x}{\sqrt{n}} \right)\varphi_n(\dx)\\
\notag &\quad =
	\displaystyle \int \Big\|\mathsf{ST}_{\tau_t}\left(\alpha \vstar + \frac{x}{\sqrt{n}} \right)\Big\|^2\varphi_n(\dx)
	+
	\displaystyle\int \Big(\alpha v^{\star} - \mathsf{ST}_{\tau_t}\Big(\alpha \vstar + \frac{x}{\sqrt{n}} \Big)\Big)^\top \mathsf{ST}_{\tau_t}\left(\alpha \vstar + \frac{x}{\sqrt{n}} \right)\varphi_n(\dx) \\
\notag	&\quad =
	\displaystyle\int \ltwo{\alpha\vstar}^2\varphi_{n}(\dx)
	+
	O\left(\sqrt{\displaystyle\int \Big\|\alpha v^{\star} - \mathsf{ST}_{\tau_t}\Big(\alpha \vstar + \frac{x}{\sqrt{n}}\Big) \Big\|_2^2\varphi_n(x)} \sqrt{\displaystyle\int \ltwo{\alpha\vstar}^2\varphi_{n}(\dx)}\right)\\
	&\quad =
	\alpha^2
	+
	O\Big(\alpha \sqrt{\frac{k\log n}{n}}\Big). 
\end{align}
Putting the above two relations together and recalling that $\alpha\gtrsim \sqrt{\frac{k\log n}{n}}$  yields
\begin{align}
\label{eqn:st-is-bounded}
	\displaystyle\int\Big\|\mathsf{ST}_{\tau_t}\left(\alpha \vstar + \frac{x}{\sqrt{n}} \right)\Big\|_{2}^{2}\varphi_{n}(\dx) =
	\alpha^2
	+
	O\Big(\alpha \sqrt{\frac{k\log n}{n}}\Big) \asymp \alpha^2,
\end{align}
and hence
\begin{align*}
	\displaystyle\int\Big\|\mathsf{ST}_{\tau_t}\left(\alpha \vstar + \frac{x}{\sqrt{n}} \right)\Big\|_{2}^{2}\varphi_{n}(\dx) &=
	\alpha^2 \bigg(1+ 
	O\bigg( \frac{1}{\alpha} \sqrt{\frac{k\log n}{n}}\bigg) \bigg);\\
	\alpha v^{\star \top} \int \mathsf{ST}_{\tau_t}\left(\alpha \vstar + \frac{x}{\sqrt{n}} \right)\varphi_n(\dx)
	&= \alpha^2 \bigg(1+ 
	O\bigg( \frac{1}{\alpha} \sqrt{\frac{k\log n}{n}}\bigg) \bigg).
\end{align*}
This implies \eqref{eqn:stable-inner-sparse}, thus completing the proof of \eqref{eqn:stable-alphat-sparse}. 
Consequently, we complete the inductive step.

\bigskip\noindent
It remains to verify the base case for \eqref{eqn:sparse-induction}, 
which is postponed to Section~\ref{sec:pf-initialization-sparse}.

\subsubsection{Bounding $\alpha_{t+1}-\alpha_{t+1}^{\star}$}

Another condition claimed in Theorem~\ref{thm:sparse} is the bound \eqref{eqn:se-alpha-sparse} on the difference between $\alpha_{t+1}$ and $\alpha_{t+1}^{\star}$, 
which we study in this subsection. 
In order to understand the dynamics of $\alpha_t$, it remains to understand the property of the function $f(\cdot)$ defined in \eqref{eqn:franck}.
A little algebra yields 
\begin{align*}
	\frac{\mathrm{d} f(\alpha)}{\partial {\alpha}}  
= \frac{\lambda v^{\star \top} \displaystyle \int \vstar \circ \ind\left(\Big|\alpha \vstar + \frac{x}{\sqrt{n}} \Big| > \tau_t 1 \right)\varphi_n(\dx)}{\left(\displaystyle \int\Big\|\mathsf{ST}_{\tau_t}\left(\alpha \vstar + \frac{x}{\sqrt{n}} \right)\Big\|_{2}^{2}\varphi_{n}(\dx)\right)^{1/2}} 
- \frac{\lambda \left(v^{\star \top} \displaystyle \int \mathsf{ST}_{\tau_t}\left(\alpha \vstar + \frac{x}{\sqrt{n}} \right)\varphi_n(\dx)\right)^2}{\left(\displaystyle \int\Big\|\mathsf{ST}_{\tau_t}\left(\alpha \vstar + \frac{x}{\sqrt{n}} \right)\Big\|_{2}^{2}\varphi_{n}(\dx)\right)^{3/2}} \ge 0,
\end{align*}
where the last inequality follows from the elementary relation $\mathbb{E}[X^2]\cdot\mathbb{E}[Y^2] \ge \big(\mathbb{E}[XY]\big)^2$.

In addition, we make the observation that  every $\alpha = \lambda + O\Big(\sqrt{\frac{t\log^3 n +k\log n}{n}}\Big)$ obeys 
\begin{align*}
\frac{\mathrm{d} f(\alpha)}{\partial {\alpha}}
&\leq 
\frac{\lambda\left(\displaystyle \int\Big\|\mathsf{ST}_{\tau_t}\left(\alpha \vstar + \frac{x}{\sqrt{n}} \right)\Big\|_{2}^{2}\varphi_{n}(\dx)\right)-\lambda \left(v^{\star \top} \displaystyle \int \mathsf{ST}_{\tau_t}\left(\alpha \vstar + \frac{x}{\sqrt{n}} \right)\varphi_n(\dx)\right)^2}{\left(\displaystyle \int\Big\|\mathsf{ST}_{\tau_t}\left(\alpha \vstar + \frac{x}{\sqrt{n}} \right)\Big\|_{2}^{2}\varphi_{n}(\dx)\right)^{3/2}}\\
&\leq \frac{\lambda \left(\lambda^2 + O\Big(\lambda \sqrt{\frac{k\log n + t\log^3 n}{n}}\Big)\right) - \lambda \left(\lambda - O\Big(\lambda \sqrt{\frac{k\log n + t\log^3 n}{n}}\Big)\right)^2}{\left(\lambda^2 - O\Big(\lambda \sqrt{\frac{k\log n + t\log^3 n}{n}}\Big)\right)^{3/2}} \\
&\lesssim \sqrt{\frac{k \log n + t\log^3 n}{n\lambda^2}} \le \frac{1}{2},
\end{align*}
where the second line follows from \eqref{eqn:inner-product-11} and \eqref{eqn:big-ben}, 
and the last line is valid under the assumptions \eqref{eqn:don-giovani}.

Based on the above properties, we further claim that $f(\alpha) = \alpha$  has one solution --- denoted by $\alpha^{\star}$ --- within the range $[\lambda/10, \lambda]$.  
In order to see this, recall that in \eqref{eqn:stable-inner-sparse}, we have shown that for any given  $\alpha \asymp \lambda$, 
\begin{align}
	f(\alpha) = 
	\frac{\lambda v^{\star \top} \displaystyle\int \mathsf{ST}_{\tau_t}\left(\alpha \vstar + \frac{x}{\sqrt{n}} \right)\varphi_n(\dx)}{\sqrt{\displaystyle\int\Big\|\mathsf{ST}_{\tau_t}\left(\alpha \vstar + \frac{x}{\sqrt{n}} \right)\Big\|_{2}^{2}\varphi_{n}(\dx)}}
	=
	 \bigg(1 + O\bigg(\sqrt{\frac{k \log n+t\log^3 n}{n\lambda^2}}\bigg) \bigg) \lambda,
	\label{eq:f-alpha-property}
\end{align}
where we invoke the assumption that $\frac{k \log n}{n\lambda^2} \ll 1$. In particular, by taking $\alpha = \frac{1}{10}\lambda$, we can deduce that 
\begin{align*}
	f\Big(\frac{1}{10}\lambda\Big) = \big(1 + o(1) \big) \lambda
	> \frac{1}{10} \lambda.
\end{align*}
In addition, it is easily seen that
\begin{align*}
	f(\lambda) \leq 
	\frac{\lambda \| v^{\star }\|_2  \sqrt{\displaystyle\int\Big\|\mathsf{ST}_{\tau_t}\left(\alpha \vstar + \frac{x}{\sqrt{n}} \right)\Big\|_{2}^{2}\varphi_{n}(\dx)}}{\sqrt{\displaystyle\int\Big\|\mathsf{ST}_{\tau_t}\left(\alpha \vstar + \frac{x}{\sqrt{n}} \right)\Big\|_{2}^{2}\varphi_{n}(\dx)}} = \lambda.
\end{align*}
Given that $\frac{\mathrm{d} f(\alpha)}{\partial {\alpha}} \in [0,1/2] $ for any $\alpha\in [\lambda/10, \lambda]$, 
we conclude that there exists a unique point within $[\lambda/10, \lambda]$ obeying $f({\alpha})=\alpha$.

Hence, for any $t$ such that $\alpha_t = \lambda + O\Big(\sqrt{\frac{t\log^3 n +k\log n}{n}}\Big)$ obeying
\begin{align*}
	\big|\alpha_{t} - \alpha_{t}^{\star}\big| \leq C_6 \sqrt{\frac{k \log n + t\log^3 n}{n}}
\end{align*}
for some large enough constant $C_6>0$, one can invoke $\alpha^\star_{t+1} = f(\alpha^\star_{t})$ to deduce that 
\begin{align*}
\big|\alpha_{t+1} - \alpha_{t+1}^{\star}\big| &\le \big|f(\alpha_{t}) - f(\alpha_{t}^{\star})\big| + C_7 \sqrt{\frac{k \log n + t\log^3 n}{n}} \\
&\le \frac{1}{2}\big|\alpha_{t} - \alpha_{t}^{\star}\big| + C_7 \sqrt{\frac{k \log n + t\log^3 n}{n}}
	\leq  C_6 \sqrt{\frac{k \log n + t\log^3 n}{n}} 
\end{align*}
provided that $C_6>2C_7$, 
where we recall that each $\alpha_{t}$ satisfies~\eqref{eqn:stable-alphat-sparse} shortly after initialization. 
Invoking the above relation recursively leads  to
\begin{align*}
	\big|\alpha_{t+1} - \alpha_{t+1}^{\star}\big| \lesssim \sqrt{\frac{k \log n + t\log^3 n}{n}},
\end{align*}
thus establishing the advertised bound~\eqref{eqn:se-alpha-sparse}.

\subsubsection{Initial condition for \eqref{eqn:sparse-induction} with two initialization paradigms}
\label{sec:pf-initialization-sparse}

In order to conclude the proof, 
we still need to verify whether the induction hypotheses \eqref{eqn:sparse-induction} hold at the initial stage of the algorithm. 
In what follows, we shall look at two types of initialization schemes separately.

\begin{itemize}
\item 
Suppose now that we have access to an initialization point $x_{1}$ independent of $W$ such that $\inprod{\vstar}{\eta_1(x_1)}\asymp 1$, as assumed in Theorem~\ref{thm:sparse}. 
In this case, one can see that 
\begin{align*}
	x_2 = M\eta_1(x_1) &= \big(\lambda\vstar v^{\star\top} + W \big) \eta_1(x_1)
	 = \lambda\inprod{\vstar}{\eta_1(x_1)} \cdot \vstar + W\eta_1(x_1)\\
	 &= \alpha_2 \vstar +  \phi_1 + \xi_1,
\end{align*}
where 
\begin{align*}
	\alpha_{2} & \defn \lambda\inprod{\vstar}{\eta_1(x_1)}  \asymp \lambda, \\
	\phi_1 &\defn W\eta_1(x_{1}) + \zeta_1, \\
	-\xi_1 = \zeta_1 &\defn \Big(\frac{\sqrt{2}}{2}-1\Big) \Big( \eta_1(x_{1}) ^\top W \eta_1(x_{1}) \Big) \eta_1(x_{1}).
\end{align*}
		As discussed in Lemma~\ref{lem:distribution} and in \eqref{defn:zWz}, we have $\phi_1\sim \mathcal{N}(0,\frac{1}{n}I_n)$ and $\eta_1(x_{1})^\top W \eta_1(x_{1}) \sim \mathcal{N}(0, \frac{2}{n})$, where we have used the fact that $\|\eta_1(x_1)\|_2=1$. Therefore, 
it holds that $\ltwo{\xi_1} \leq \sqrt{\frac{\log n}{n}}$ with probability at least $1 - O(n^{-11})$. 
As a result, we have established  the induction hypotheses \eqref{eqn:sparse-induction} 
for $t=2$, as required by Theorem~\ref{thm:sparse}.

\item Another type of initialization schemes considered in this paper is~\eqref{defi:v1}, which concerns Corollary~\ref{cor:sparse-init-1}.
By definition, index $\hat{s}$ is selected by maximizing the diagonal entries of $M$, resulting in statistical dependence between $x_{1}$ and $W$. 
As a consequence, Theorem~\ref{thm:main} is not directly applicable. 
To cope with the statistical dependency, let us generate an AMP sequence starting from $e_s$ for each given $s\in \mathcal{S}_{0}$. 
For each of these AMP sequences, it turns out that the initial condition for \eqref{eqn:sparse-induction} when $t=3$ is satisfied,
which is formulated in the result below. The proof of this lemma can be found in Section~\ref{sec:pf-ken-sparse-ini}. 

\begin{lems}
\label{lem:sparse-ini-1-ini}
	Consider the AMP procedure initialized with $\eta_0(x_{0}) = 0$ and $x_{1} = e_{s}$ for any given $s\in \mathcal{S}_{0}$ (cf.~\eqref{eqn:sparse-set-s0}).
With probability at least $1 - O(n^{-11})$, the iterate $x_3$ admits the following decomposition:  
\begin{align*}
	x_3 = \alpha_3 \vstar + \beta_2^1 \phi_1 + \beta_2^2 \phi_2 + \xi_2,
\end{align*}
where $\phi_1$ and $\phi_2$ are i.i.d.~drawn from $\mathcal{N}\big(0,\frac{1}{n}I_n\big)$, and 
\begin{align}
	\alpha_3 = \lambda v^{\star\top} \eta_2(x_2) \asymp \lambda
	~~\text{ and }~~
	\|\xi_2\|_2 \lesssim \sqrt{\frac{\log n}{n}}. 
\end{align}
%
\end{lems}

Taking a simple union bound over all $s \in \mathcal{S}_{0}$, we conclude that with probability at least $1 - O(n^{-10})$, the initial condition \eqref{eqn:sparse-induction} 
is satisfied for $t=3$ if AMP is initialized at $e_{s}$ for any $s \in \mathcal{S}_{0}$, thus making Theorem~\ref{thm:sparse}) applicable.  
Further,  Proposition~\ref{thm:sparse-init} guarantees that  $\mprob(\hat{s} \in \mathcal{S}_0) \geq 1 - O(n^{-11}).$
Putting these together, we can guarantee that the AMP initialized at $e_{\hat{s}}$ yields the required decomposition with probability at least $1 - O(n^{-10})$, 
as claimed in Corollary~\ref{cor:sparse-init-1}. 

\end{itemize}

\noindent 
In summary, 
putting the above two initial conditions together with the previous inductive steps finishes the proof of 
Theorem~\ref{thm:sparse} and Corollary~\ref{cor:sparse-init-1}.

\subsubsection{Implications for $\ell_2$ estimation accuracy}
%

Our theory reveals the $\ell_2$ estimation accuracy of AMP, which we briefly discuss in this subsection. 
Consider the estimator $\frac{1}{\lambda}\mathsf{ST}_{\tau_t}\left(x_t \right)$ and look at its $\ell_2$ estimation error. 
The triangle inequality directly implies
\begin{align*}
 	\left\|\mathsf{ST}_{\tau_t}\left(x_t \right) - \lambda\vstar\right\|_2 
 	&\leq 
 	\left\|\mathsf{ST}_{\tau_t}\left(x_t \right) - \mathsf{ST}_{\tau_t}\left(v_t \right)\right\|_2 
 	+
 	\left\|\mathsf{ST}_{\tau_t}\left(v_t \right) - \alpha_t \vstar\right\|_2 
 	+
 	\left\|\alpha_t \vstar - \lambda \vstar\right\|_2 \\
 	&\leq \ltwo{\xi_{t-1}} + \left\|\mathsf{ST}_{\tau_t}\left(v_t \right) - \alpha_t \vstar\right\|_2  + |\alpha_t - \lambda|\\
 	&\leq \left\|\mathsf{ST}_{\tau_t}\left(v_t \right) - \alpha_t \vstar\right\|_2 + O\left(\sqrt{\frac{k\log n + t\log^3 n}{n}}\right)
 \end{align*} 
with probability at least $1-O(n^{-11})$, where the last step follows from Theorem~\ref{thm:sparse} and inequality~\eqref{eqn:stable-alphat-sparse}.

Now consider the Lipschitz function $f(x) \defn \ltwo{\mathsf{ST}_{\tau_t}(\alpha_t\vstar + x) - \alpha_t\vstar}, ~x \in \real^{n}$. 
Inequality~\eqref{eq:sup-f-beta-phi-Ef} yields 
\begin{align*}
	\Big\|\mathsf{ST}_{\tau_t}\Big(\alpha_t \vstar + \sum_{j=1}^{t-1} \beta_{t-1}^j \phi_j \Big) - \alpha_t \vstar\Big\|_2 
	-
	\Exs_{g\sim \mathcal{N}(0,\frac{1}{n}I_n)}\Big[ \left\|\mathsf{ST}_{\tau_t}\big(\alpha_t \vstar + g \big) - \alpha_t \vstar\right\|_2 \Big]
	\lesssim
	\sqrt{\frac{t\log n}{n}}
\end{align*}
holds with probability at least $1 - O(n^{-11})$. 
In addition, in view of \eqref{eqn:cello-sonata} (taking $\alpha = \alpha_t$), we can deduce 
\begin{align*}
	\Exs_{g\sim \mathcal{N}(0,\frac{1}{n}I_n)}\Big[ \left\|\mathsf{ST}_{\tau_t}\big(\alpha_t \vstar + g \big) - \alpha_t \vstar\right\|_2 \Big]
	\lesssim 
	\sqrt{\frac{k\log n}{n}}.
\end{align*}
Putting these pieces together ensures that, with probability at least $1 - O(n^{-11})$, 
\begin{align}
\label{eqn:l2-error-hp}
	\Big\|\frac{1}{\lambda}\mathsf{ST}_{\tau_t}\left(x_t \right) - \vstar\Big\|_2 
	\lesssim \sqrt{\frac{k\log n + t\log^3 n}{n\lambda^2}}.
\end{align}

\subsection{Proof of Proposition~\ref{thm:sparse-init}}
\label{sec:pf-sparse-ini}

Let us begin by considering the magnitude of $\langle v^{\star}, \eta_2(x_2)\rangle$, 
which is the focus of the claim \eqref{eqn:vstar-s}. 
From the AMP iteration~\eqref{eqn:AMP-updates}, it is seen that
\begin{align}
\label{eqn:intermezzi}
	x_2 = M\eta_1(x_1) = \Big(\lambda\vstar\big(v^{\star}\big)^{\top} + W\Big) e_s
	 = \lambda\vstar_s\vstar + We_s,
\end{align}
where we note that $\eta_{1}(x_1) = e_{s}$ as a consequence of the denoising function \eqref{eqn:eta-sparse}. 
Recognizing that each $W_{ii}$ (resp.~$W_{ij}$ with $i\neq j$) is an independent Gaussian random variable with variance $2/n$ (resp.~$1/n$), 
one can invoke standard Gaussian concentration results to obtain \citep[Chapter 2.6]{vershynin2018high}
\begin{align}
\label{eqn:simple-diagonal}
	\max_{1\leq i, j\leq n} \lt|\big(We_j\big)_{i}\rt| \leq 6 \sqrt{\frac{\log n}{n}}
\end{align}
with probability at least $1 - O(n^{-11})$. 
Therefore, if $\lambda|\vstar_s| \geq 12 \sqrt{\frac{k \log n}{n}}$,   
then it follows from \eqref{eqn:simple-diagonal} that
\begin{align*}
	|x_{2, i}| \geq |\lambda \vstar_s|\cdot  |v_i^{\star} | - \big|\big(We_j\big)_{i} \big|  &> 6 \sqrt{\frac{\log n}{n}},\qquad\text{if }|\vstar_i| \ge \frac{1}{2\sqrt{k}}; \\
	|x_{2, i}| \leq  \big|\big(We_j\big)_{i}\big|  &\leq 6 \sqrt{\frac{\log n}{n}},\qquad\text{if }\vstar_i = 0. 
\end{align*}
In the meantime, the above argument also reveals that 
\begin{align*}
	\mathrm{sign}(x_{2, i}) &= \mathrm{sign}(\vstar_i),\qquad&&\text{if }|x_{2, i}| > 6 \sqrt{\frac{\log n}{n}}
	; \\
	\big(|x_{2,i}|-\tau_2\big)_+ &= 0, \qquad &&\text{if }|x_{2, i}| \leq 6 \sqrt{\frac{\log n}{n}}.
\end{align*}

With the preceding observations in place, we can demonstrate that
\begin{align*}
\notag \Big\|\mathrm{sign}(x_2)\circ (|x_2| - \tau_2 1)_{+}\Big\|_2^2 
	&\le \sum_{i = 1}^n x_{2, i}^2\ind\Big(|x_{2, i}| \gtrsim \sqrt{\frac{\log n}{n}}\Big) 
 \lesssim \sum_{i : \vstar_i \ne 0} \Big(|\lambda\vstar_s\vstar_i| + \sqrt{\frac{\log n}{n}}\Big)^2 \\
&\lesssim \lt(\lambda\vstar_s\rt)^2 + \frac{k\log n}{n} 
\asymp \lt(\lambda\vstar_s\rt)^2 
\end{align*}
under the assumption that $\lambda|\vstar_s| \geq 12 \sqrt{\frac{k \log n}{n}}$. 
In addition, the above properties also reveal that: 
\begin{align}
\label{eqn:qianqian}
\notag \big|\langle\vstar, \mathrm{sign}(x_2)\circ (|x_2| - \tau_2 1)_{+}\rangle\big| & = 
\sum_{ i : |x_{2, i}| > 6 \sqrt{\frac{\log n}{n}}} \lt(|x_{2, i}| - \tau_2\rt)|\vstar_i| 
\geq \sum_{i : |\vstar_i| \ge \frac{1}{2\sqrt{k}}}  \lt(|x_{2, i}| - \tau_2\rt)|\vstar_i| \\
	&\geq \frac{1}{2} \sum_{i : |\vstar_i| \ge \frac{1}{2\sqrt{k}}}  |x_{2, i}|\cdot |\vstar_i|  
	\gtrsim \sum_{i : |\vstar_i| \ge \frac{1}{2\sqrt{k}}}  |\lambda\vstar_s|\, v_i^{\star 2} \gtrsim \lambda |\vstar_s|,
\end{align}
where the last line follows from the assumption $\lambda|\vstar_s| \geq C_5 \sqrt{\frac{k \log n}{n}}$ for some large enough constant $C_5>0$ (so that $|x_{2, i}| \geq 2 \tau_2$), as well as the following condition (using the $k$-sparse property of $\vstar$)
\begin{align}
\label{eqn:vstar-basic-1}
	\sum_{i : |\vstar_i| \geq \frac{1}{2\sqrt{k}}} v_i^{\star 2} = 1 - 
	\sum_{i : |\vstar_i| < \frac{1}{2\sqrt{k}}} v_i^{\star 2} 
	\geq 1 - k\cdot  \bigg( \frac{1}{2\sqrt{k}} \bigg)^2
	= \frac{3}{4}.
\end{align}
Putting the above pieces together and recalling that $\|\eta_(x_2)\|_2=1$, we conclude that 
\begin{align*}
	\big|\inprod{\vstar}{\eta_2(x_2)}\big|
	= \frac{\big|\langle\vstar, \mathrm{sign}(x_2)\circ (|x_2| - \tau_2 1)_{+}\rangle\big|}{\lt\|\mathrm{sign}(x_2)\circ (|x_2| - \tau_2 1)_{+}\rt\|_2} \asymp 1. 
\end{align*}

Next, we turn to proving the second claim of Proposition~\ref{thm:sparse-init}, 
towards which we would like to show that with probability at least $1 - O(n^{-11})$,  
\begin{align}
\label{eqn:vs-is-large}
	|\langle x_{1}, \vstar\rangle| = |\vstar_{\hat{s}}| \ge \frac{1}{2}\|v^{\star}\|_{\infty},
\end{align}
where $\hat{s} \defn \arg\max_i \lt|M_{ii}\rt|$ and $x_{1} = e_{\hat{s}}$. 
Indeed, it follows from \eqref{eqn:simple-diagonal} that
\begin{align*}
	\lambda\big(v^{\star}_{i}\big)^2 - 6\sqrt{\frac{\log n}{n}} \leq 
	\Big| \big(\lambda v^{\star}v^{\star\top} + W\big)_{ii}  \Big|
	\leq 
	\lambda\big(v^{\star}_{i}\big)^2 + 6\sqrt{\frac{\log n}{n}}
\end{align*}
with probability at least $1 - O(n^{-11})$. 
Given that $\hat{s}$ maximizes $|(\lambda v^{\star}v^{\star\top} + W)_{ii}|$ over all $i\in [n]$, we have 
\begin{align*}
\lambda\big(v^{\star}_{\hat{s}}\big)^2 + 6\sqrt{\frac{\log n}{n}} \ge 
\Big| \lt(\lambda v^{\star}v^{\star\top} + W\rt)_{\hat{s} \hat{s}} \Big|
\geq 
\Big| \big(\lambda v^{\star}v^{\star\top} + W\big)_{jj} \Big|
\geq
\lambda\big(v^{\star}_{\max}\big)^2 -6 \sqrt{\frac{\log n}{n}}, \quad \text{with } j \defn \argmax_i |\vstar_i|.
\end{align*}
In the case when $\lambda\|v^{\star}\|_{\infty} \gtrsim \sqrt{\frac{k\log n}{n}}$, 
the above inequality immediately establishes~\eqref{eqn:vs-is-large} by observing that
$
 	\|v^{\star}\|_{\infty} \geq \sqrt{\frac{\ltwo{\vstar}^{2}}{k}} = \frac{1}{\sqrt{k}}.
$
We have thus completed the proof of Proposition~\ref{thm:sparse-init}.

\subsection{Proof of Proposition~\ref{prop:sparse-split}}
\label{sec:pf-prop-sparse-split}

For notational convenience, 
we omit the index of $j$ in $\mathcal{I}_{j}$ and write $\mathcal{I}$ instead throughout this proof,
as long as it is clear from the context. 

\paragraph{Step 1: analysis for a single round.} 
Let us first state some concentration properties regarding $v^{\star}_{\mathcal{I}}$. 
By construction, $\|v^{\star}_{\mathcal{I}}\|_0$ can be viewed as the sum of $k$ independent Bernoulli random variables, each of which has mean $p$. 
The Bernstein inequality tells us that
\begin{align*}
	\Big|\big\|v^{\star}_{\mathcal{I}}\big\|_0 - kp\Big| \leq \sqrt{2kp(1-p)\log \frac{2}{\delta}} + 2\log \frac{2}{\delta}
\end{align*}
holds with probability at least $1 - \delta$.
It thus implies that 
$$
	\big\|v^{\star}_{\mathcal{I}}\big\|_0 - pk = o(pk)
$$
under the assumption that $kp \gtrsim \log n \gg \log\frac{2}{\delta}$ 
and $\delta \asymp 1$ (e.g., $\delta = \frac{1}{10}$). 
In addition, $\|v^{\star}_{\mathcal{I}}\big\|_2^2 = \sum_{i\in \text{supp}(\vstar)}  v^{\star 2}_{i} \ind(i \in \mathcal{I})$ is the sum of $k$ independent bounded random variables, with total variance bounded above by
\[
	\mathsf{Var}\big( \|v^{\star}_{\mathcal{I}}\big\|_2^2 \big) 
	\leq  p  \big\|v^{\star}_{\mathcal{I}}\big\|_{\infty}^2 \sum_{i\in \text{supp}(\vstar)}  v^{\star 2}_{i} 
	= p  \big\|v^{\star}_{\mathcal{I}}\big\|_{\infty}^2. 
\]
Invoking Bernstein's inequality again gives
\begin{align*}
	\lt|\big\|v^{\star}_{\mathcal{I}}\big\|_2^2 - p\rt| \leq \big\|v^{\star}_{\mathcal{I}}\big\|_{\infty}\sqrt{2 p\log \frac{2}{\delta}} + 2\big\|v^{\star}_{\mathcal{I}}\big\|_{\infty}^2\log \frac{2}{\delta},
\end{align*}
with probability at least $1 - \delta$.
Consequently, it guarantees that with probability $1 - \delta$, 
\begin{align}
\label{eqn:sparse-vc-norm-tmp}
	\|v^{\star}_{\mathcal{I}}\big\|_2^2 - p = o(p) ~\text{ and }~
	\|v^{\star}_{\mathcal{I}^c}\big\|_2^2 = 1 - p + o(p),
\end{align}
under the assumption that $\|v^{\star}\|_{\infty} = o\big(\sqrt{\frac{\log n}{k}}\big)$ and $p \gtrsim \frac{\log n}{k}$.
Combining the above two relations gives
\begin{align}
\label{eqn:vstar-i}
	\lambda \big\|v^{\star}_{\mathcal{I}}\big\|_2^2 ~\gtrsim~ \frac{\big\|v^{\star}_{\mathcal{I}}\big\|_0}{\sqrt{n}},
\end{align}
with the proviso that $\lambda \gtrsim k/\sqrt{n}.$


Suppose now that there exists an oracle algorithm (as in \eqref{eq:correlation-oracle}) whose returned solution $v_{\mathcal{I}} = \mathsf{Oracle}(M_{\mathcal{I}, \mathcal{I}})$ (computed solely based on $M_{\mathcal{I}, \mathcal{I}}$) satisfies 
\begin{align*}
	\langle v^{\star}_{\mathcal{I}}, v_{\mathcal{I}}\rangle \asymp \big\|v^{\star}_{\mathcal{I}}\big\|_2 \asymp \sqrt{p}
\end{align*}
with probability at least $1-\delta$ for some small constant $\delta$ (note that the randomness comes from the sampling process). 
Taking $v \defn M_{\mathcal{I}^{c}, \mathcal{I}}  v_{\mathcal{I}}$ yields 
\begin{align}
\label{eqn:x1-sparse-split}
v 
= M_{\mathcal{I}^{c}, \mathcal{I}} v_{\mathcal{I}} 
= (\lambda \vstar_{\mathcal{I}^{c}} v_{\mathcal{I}}^{\star\top} + W_{\mathcal{I}^{c}, \mathcal{I}}) v_{\mathcal{I}} 
= \alpha_1 v_{\mathcal{I}^{c}}^{\star} + \phi_0,
\end{align}
where $\alpha_1 = \lambda \langle v^{\star}_{\mathcal{I}}, v_{\mathcal{I}}\rangle \asymp \lambda\sqrt{p}$ and $\phi_0 \sim \mathcal{N}(0, \frac{\ltwo{ v_{\mathcal{I}}}^2}{n}I)$. 
Importantly, both $\alpha_1$ and $\phi_0$ are independent of $W_{\mathcal{I}^{c}, \mathcal{I}^{c}}$.
Based on the construction \eqref{eqn:sparse-choose-j} of $x_{1}$, it suffices to verify  
\begin{align}
\label{eqn:initial-const-corr}
\frac{\big\langle v_{\mathcal{I}^{c}}^{\star}, \mathsf{ST}_{\tau_1}(v) \big\rangle}{\big\|v_{\mathcal{I}^{c}}^{\star}\big\|_2\big\|\mathsf{ST}_{\tau_1}(v)\big\|_2} 
\asymp 1.
\end{align}
%

In order to validate~\eqref{eqn:initial-const-corr}, we look at the distributional of~\eqref{eqn:x1-sparse-split}.
First, given the fact that $\|\phi_0\|_{\infty} \leq 6 \sqrt{\frac{\log n}{n}}$ with probability at least $1 - O(n^{-11})$ and  
	$\alpha_1 \asymp \lambda \sqrt{p} 
		\gtrsim \sqrt{\frac{k\log n \log \frac{1}{\delta}}{n}},$
we can use \eqref{eqn:x1-sparse-split} to get 
\begin{align*}
	\mathrm{sign}\big(\mathsf{ST}_{\tau_1}(v_i) \big) = \mathrm{sign} \big(v^{\star}_i \big)
	\qquad \text{and} \qquad
	\begin{cases}
		\Big| \mathsf{ST}_{\tau_1}(v_i)\Big|\geq \frac{1}{2}\alpha_1 |v_{i}^{\star}|  \mathds{1} 
	\Big(\lambda\sqrt{p} |v_{i}^{\star}| \geq C_8 \sqrt{\frac{\log n}{n}} 1\Big) \\
		\Big| \mathsf{ST}_{\tau_1}(v_i)\Big|\leq 2\alpha_1 |v_{i}^{\star}|  \mathds{1} 
	\Big(\lambda\sqrt{p} |v_{i}^{\star}| \geq C_7 \sqrt{\frac{\log n}{n}} 1\Big)
	\end{cases}
	\qquad i\in \mathcal{I}^{c}
\end{align*}
for some suitable constants $C_7,C_8>0$. 
As a result, we can see that 
\begin{align}
	\Big\langle v_{\mathcal{I}^{c}}^{\star}, \frac{\mathsf{ST}_{\tau_1}(v)}{\ltwo{\mathsf{ST}_{\tau_1}(v)}} \Big\rangle
	\gtrsim 
	\frac{\alpha_1\sum_{i\in \mathcal{B}} v_{i}^{\star 2}}{\alpha_1 \|v^{\star}\|_2}
	\qquad
	\text{with }~\mathcal{B} \defn \Big\{i \mid i \in \mathcal{I}^{c}; |v_{i}^{\star}| \geq C_8 \frac{\sqrt{\log n/n}}{\lambda\sqrt{p}} \Big\}.
	\label{eq:ST-inner-product-v-LB}
\end{align}
%
Moreover, recalling the concentration results $\|v^{\star}_{\mathcal{I}}\big\|_2^2 - p = o(p)$ and $\big\|v^{\star}_{\mathcal{I}}\big\|_0 - pk = o(pk)$,  we find that 
\begin{align*}
\sum_{i\in\mathcal{B}^{c}\cap\mathcal{I}^{c}}v_{i}^{\star2} & =\Big\| v_{\mathcal{I}^{c}}\circ\ind\Big(v_{\mathcal{I}^{c}}<\frac{C_{8}\sqrt{\log n/n}}{\lambda\sqrt{p}}1\Big)\Big\|_{2}^{2}\leq\big(k-kp+o(kp)\big)\cdot\frac{C_{8}^{2}\log n}{\lambda^{2}pn}\\
 & \leq\big(k-kp+o(kp)\big)\cdot\frac{C_{8}^{2}\log n}{10C_{8}^{2}k\log n\log\frac{1}{\delta}}\leq\frac{1-p}{10\log(1/\delta)}(1+o(1))\leq\frac{1}{10}\ltwo{v_{\mathcal{I}^{c}}^{\star}}^{2},
\end{align*}
provided that $\lambda^2 p \geq 10C_8^2 \frac{k\log n \log \frac{1}{\delta}}{n}$.
Substitution into \eqref{eq:ST-inner-product-v-LB} gives
\[
	\Big\langle v_{\mathcal{I}^{c}}^{\star}, \frac{\mathsf{ST}_{\tau_1}(v)}{\ltwo{\mathsf{ST}_{\tau_1}(v)}} \Big\rangle
	\gtrsim 1, 
\]
which together with $\|v_{\mathcal{I}^{c}}\|_2\asymp 1$ establishes the required inequality~\eqref{eqn:initial-const-corr}.

\paragraph{Step 2: repeating the procedure for $N$ times.} 
Thus far, we have proved that inequalities~\eqref{eqn:sparse-vc-norm-tmp}  and \eqref{eqn:initial-const-corr} are satisfied  with probability $1 - \delta$, for $\delta$ being some small constant (e.g., $\delta = 1/10$); 
here, the uncertainty comes from the random sampling process.
In order to boost the success probability, we --- as detailed in Step 1 --- repeat the sampling procedure for $N \defn 10\log n/\log (1/\delta)$ times, 
outputing $N$ independent subsets $\mathcal{I}_1, \ldots, \mathcal{I}_{N} \subseteq [n]$
and $N$ corresponding estimators (denoted by $v^{j}$ for $1\leq j\leq N$). 
Then there exists at least one subset $\mathcal{I}_j$ such that, with probability at least $1 - \delta^{N} = 1 - O(n^{-10})$, 
\begin{align}
\label{eqn:tzigane}
	\|v^{\star}_{\mathcal{I}^c}\big\|_2^2 = 1 - p + o(p)
	\qquad \text{and} \qquad
	\frac{\big\langle v_{\mathcal{I}^{c}}^{\star}, \mathsf{ST}_{\tau_1}(v^{j}) \big\rangle}{\big\|v_{\mathcal{I}^{c}}^{\star}\big\|_2\big\|\mathsf{ST}_{\tau_1}(v^{j})\big\|_2} \asymp 1
\end{align}
for $v^j \defn M_{\mathcal{I}_j^{c}, \mathcal{I}_j}  v_{\mathcal{I}_j}$. 
Based on these two relations, we arrive at  
\begin{align}
\notag	\frac{\mathsf{ST}_{\tau_1}(v^j)^\top M_{\mathcal{I}_j^{c}, \mathcal{I}_j^c}  \mathsf{ST}_{\tau_1}(v^j)}{\|\mathsf{ST}_{\tau_1}(v^j)\|_2^2}
	&=
	\lambda 
	\frac{\big\langle v_{\mathcal{I}_j^{c}}^{\star}, \mathsf{ST}_{\tau_1}(v^{j}) \big\rangle^2}{\big\|\mathsf{ST}_{\tau_1}(v^{j})\big\|^2_2}
	+
	\underbrace{\frac{\mathsf{ST}_{\tau_1}(v^j)^\top W_{\mathcal{I}_j^{c}\mathcal{I}_j^{c}} \mathsf{ST}_{\tau_1}(v^j)}{\big\|\mathsf{ST}_{\tau_1}(v^j)\big\|^2_2}}_{=: \varepsilon_j}\\
\notag	&= \lambda 
	\frac{\big\langle v_{\mathcal{I}_j^{c}}^{\star}, \mathsf{ST}_{\tau_1}(v^j) \big\rangle^2}{\ltwo{v_{\mathcal{I}_j^{c}}^{\star}}^2 \big\|\mathsf{ST}_{\tau_1}(v^j)\big\|^2_2} \cdot \ltwo{v_{\mathcal{I}_j^{c}}^{\star}}^2 + \varepsilon_j \\
\notag	&\gtrsim \lambda (1 - p) + \varepsilon_j, \qquad \text{where } \varepsilon_j \sim \mathcal{N}\Big(0,\frac{1}{n}\Big)\\
	&\asymp \lambda 
\end{align}
Here, we remind the readers that since $v^{j}$ is independent of
 $W_{\mathcal{I}_j^{c}\mathcal{I}_j^{c}}$, and
$\varepsilon_{j}$ follows a Gaussian distribution $\mathcal{N}(0,\frac{1}{n}).$
We also make note of the relation that 
$\lambda(1-p) \gtrsim \frac{k(1-p)}{\sqrt{n}}  \gtrsim \frac{\log n}{\sqrt{n}}$ and 
$\max_{1\leq j\leq N} |\varepsilon_{j}| \leq 5\sqrt{\frac{\log N}{n}}$ with probability $1-O(n^{-11})$.

Therefore, if we select $\widehat{j}$ according to \eqref{eqn:sparse-choose-j} as follows:
\begin{align*}
	\widehat{j} \defn \arg\max_j ~\Bigg\{\frac{\mathsf{ST}_{\tau_1}(v^j)^\top M_{\mathcal{I}_j^{c}, \mathcal{I}_j^c}  \mathsf{ST}_{\tau_1}(v^j)}{\|\mathsf{ST}_{\tau_1}(v^j)\|_2^2}\Bigg\},
\end{align*}
then we necessarily have
\begin{align}
\label{eqn:sparse-concl}
	\frac{\big\langle v_{\mathcal{I}_{\widehat{j}}^{c}}^{\star}, \mathsf{ST}_{\tau_1}(v^{\widehat{j}}) \big\rangle}{\big\|v_{\mathcal{I}_{\widehat{j}}^{c}}^{\star}\big\|_2\big\|\mathsf{ST}_{\tau_1}(v^{\widehat{j}})\big\|_2} \asymp 1.
\end{align}
Otherwise, we would end up with  
\begin{align*}
	\frac{\mathsf{ST}_{\tau_1}(v^{\widehat{j}})^\top M_{\mathcal{I}^{c}, \mathcal{I}^c}  \mathsf{ST}_{\tau_1}(v^{\widehat{j}})}{\|\mathsf{ST}_{\tau_1}(v^{\widehat{j}})\|_2^2}
	&=
	\lambda\frac{\big\langle v_{\mathcal{I}_{\widehat{j}}^{c}}^{\star}, \mathsf{ST}_{\tau_1}(v^{\widehat{j}}) \big\rangle^2}{\big\|\mathsf{ST}_{\tau_1}(v^{\widehat{j}})\big\|^2_2} + \frac{\mathsf{ST}_{\tau_1}(v^{\widehat{j}})^\top W_{\mathcal{I}_{\widehat{j}}^{c}\mathcal{I}_{\widehat{j}}^{c}} \mathsf{ST}_{\tau_1}(v^{\widehat{j}})}{\big\|\mathsf{ST}_{\tau_1}(v^{\widehat{j}})\big\|^2_2}\\
	&\ll \lambda + \sqrt{\frac{\log n}{n}} \\
	&\lesssim
	\frac{\mathsf{ST}_{\tau_1}(v^j)^\top M_{\mathcal{I}^{c}, \mathcal{I}^c}  \mathsf{ST}_{\tau_1}(x^j_1)}{\|\mathsf{ST}_{\tau_1}(x^j_1)\|_2^2},
\end{align*}
where $j$ corresponds to the one obeying \eqref{eqn:tzigane}. 
This, however, contradicts the definition of $\widehat{j}.$
We have therefore concluded the proof of Proposition~\ref{prop:sparse-split}.




\subsection{Proof of Lemma~\ref{lem:sparse-ini-1-ini}}
\label{sec:pf-ken-sparse-ini}

In this subsection, we analyze the first three AMP iterates when initialized at $\eta_0(x_{0}) = 0$ and $x_1 = e_{s}$ 
for any given $s \in \mathcal{S}_0$, where $\mathcal{S}_0 \defn \{s \in [n] \mid |\vstar_s| \geq \frac{1}{2} \|\vstar\|_\infty \}$ (defined in expression \eqref{eqn:sparse-set-s0}). 
Without loss of generality, we shall assume $v_{s}^{\star} >0$ throughout this subsection.

\paragraph{The 2nd iterate.} It follows from the AMP update rule that
\begin{align*}
	x_2 = M\eta_1(x_1) = \Big(\lambda\vstar\big(v^{\star}\big)^{\top} + W\Big) e_{s}
	 = \lambda\vstar_s\vstar + We_{s},
\end{align*}
where we use $\eta_{1}(x_1) = e_{s}$ in view of the definition of the denoising function in \eqref{eqn:eta-sparse}. 
Recalling that in the proof of Theorem~\ref{thm:main}, we establish the decomposition~\eqref{def:dynamics} with $\phi_{j}$ and $\beta_t^j$ defined in \eqref{def:phi_k} and \eqref{eqn:eta-decomposition} respectively. 
Instantiating this to the current case, we have 
$$
		\beta_1^1 = 1
		\qquad \text{and}\qquad
		\phi_{1} = We_{s} + \Big(\frac{\sqrt{2}}{2}-1 \Big)e_{s}^{\top}We_{s} \sim \mathcal{N}\Big(0, \frac{1}{n}I_n\Big).
$$
Hence, $x_{2}$ can be expressed as 
\begin{align}
	x_2 &= \alpha_2\vstar + \phi_1 + \xi_1,
\end{align}
		where $\alpha_2 = \lambda\vstar_s \gtrsim \lambda$ obeying $|\alpha_2| \geq \|\vstar\|_{\infty}$, and $\xi_1 = (1 - \frac{\sqrt{2}}{2}) (e_{s}^{\top}We_{s}) e_{s}$.

To proceed, we find it helpful to make note of the following two properties. 
First, from expression~\eqref{eqn:qianqian}, we know that  
\begin{align}
\label{eqn:cadenza}
\gamma_2^{-1} = \big\|\mathrm{sign}(x_2)\circ (|x_2| - \tau_2 1)_{+} \big\|_2
\geq \big|\langle\vstar, \mathrm{sign}(x_2)\circ (|x_2| - \tau_2 1)_{+} \rangle\big|  \gtrsim \lambda |\vstar_s|.  
\end{align}
%
In addition, with probability at least $1 - O(n^{-11})$, one can derive from \eqref{eqn:simple-diagonal} that 
\begin{align}
\label{eqn:finale}
	|\xi_{1,s}| = \Big|\Big(1 - \frac{\sqrt{2}}{2} \Big) e_{s}^{\top}We_{s} \Big|
	\leq \max_{1\leq i\leq n} |W_{ii}| \lesssim \sqrt{\frac{\log n}{n}}. 
\end{align}

\paragraph{The 3rd iterate.} 
In view of decomposition~\eqref{def:dynamics}, we write 
\begin{align*}
	x_3 &= M\eta_2(x_2) - \langle\eta_2^{\prime}(x_2)\rangle\eta_1(x_1) 
	= \alpha_3 \vstar + \beta_2^1 \phi_1 + \beta_2^2 \phi_2 + \xi_2,
\end{align*}
where (see\eqref{eq:xi-expression})
$$
	\alpha_3 = \lambda v^{\star\top}\eta_2(x_2)
	\qquad \text{and} \qquad 
	\xi_{2} \in \mathsf{span}\{e_{s}, \eta_2(x_2)\}.
$$
Proposition~\ref{thm:sparse-init} tells us that $\alpha_3 \asymp \lambda.$

Next, we look at the size of $\ltwo{\xi_2}.$ 
To begin with, by definition, we have $\mu_2^1 = \frac{\xi_1^\top e_{s}}{\ltwo{\xi_1}} = 1$ for $\xi_1 = (1 - \frac{\sqrt{2}}{2}) (e_{s}^{\top}We_{s}) e_{s}$.
Combining this with relation~\eqref{eq:xi_bound} for $t=2$, we arrive at  
\begin{align}
\label{eqn:xi-2-vincent}
	\|\xi_2\|_2 = \langle \phi_1, \delta_{2}\rangle - \langle\delta_{2}^{\prime}\rangle 
	+ \Delta_2 
	+ O\Big(\sqrt{\frac{\log n}{n}}\Big), 
\end{align}
where $\delta_{2}$ takes the following form 
\begin{align*}
	\delta_{2} = \eta_2(\alpha_2\vstar + \phi_1 + \xi_1)
	- \eta_2(\alpha_2\vstar + \phi_1). 
\end{align*}
To control $\|\xi_2\|_2$, it then suffices to upper bound $\Delta_{2}$ as well as 
$\langle \phi_1, \delta_{2}\rangle - \langle\delta_{2}^{\prime}\rangle$.
Let us start with the quantity $\Delta_{2}$.
Recall that $|\Delta_{2}| \leq A_2 \lesssim \sqrt{\frac{2\log n}{n}}$ (in view of \eqref{eqn:sparse-At}), where
we remind the readers that the proof of \eqref{eqn:sparse-At} is built upon the assumptions $\alpha_2 \lesssim \lambda$ and $\tau_2 \asymp \sqrt{\frac{\log n}{n}}.$

We then move on to consider the quantity $\langle \phi_1, \delta_{2}\rangle - \langle\delta_{2}^{\prime}\rangle$. 
First, given that $\xi_{1}$ is along the direction of $e_{s}$, the entries of $\delta_{2}$
are all zero except for $\delta_{2,s}.$
With this observation in mind, the term of interest can be written as   
\begin{align}
\label{eqn:opera}
	\big| \langle \phi_1, \delta_{2}\rangle - \langle\delta_{2}^{\prime}\rangle \big| &= \Big| \phi_{1,s}\delta_{2, s} - \frac{1}{n}\delta_{2,s}^{\prime} \Big|
	\lesssim \sqrt{\frac{\log n}{n}} \cdot \frac{|\xi_{1,s}|}{\lambda v_{s}^{\star}} + \frac{1}{n\lambda v_{s}^{\star}}.
\end{align}
To see why the last inequality is valid, we use inequality~\eqref{eqn:cadenza} to obtain 
\begin{align*}
|\delta_{2,s}| = \big|\eta_2(\alpha_2\vstar_s + \phi_{1,s} + \xi_{1,s})
	- \eta_2(\alpha_2\vstar_{s} + \phi_{1,s})|
	\leq \gamma_2 |\xi_{1,s} \big| 
	\lesssim 
	\frac{|\xi_{1,s}|}{\lambda\vstar_s},
\end{align*}
as a result of the Lipschitz property of $\eta_t$, and in addition, 
\begin{align*}
	|\delta_{2,s}'| \leq 2 \gamma_2 \lesssim \frac{1}{\lambda\vstar_{s}}.
\end{align*}
Taking inequality~\eqref{eqn:opera} together with \eqref{eqn:finale} and recalling $\vstar_{s} \gtrsim \frac{1}{\sqrt{k}}$ (cf.~\eqref{eqn:vstar-s}) lead to 
\begin{align*}
\big| \langle \phi_1, \delta_{2}\rangle - \langle\delta_{2}^{\prime}\rangle \big| &
\lesssim \sqrt{\frac{\log n}{n}}\sqrt{\frac{k\log n}{n\lambda^2}} 
+ \frac{\sqrt{k}}{n\lambda} \lesssim \sqrt{\frac{\log n}{n}}.
\end{align*}
Substitution back into \eqref{eqn:xi-2-vincent} gives 
\begin{align*}
	\|\xi_2\|_2 \lesssim \sqrt{\frac{\log n}{n}}.
\end{align*}
We have thus established Lemma~\ref{lem:sparse-ini-1-ini}.

\section*{Acknowledgment}

This work was partially supported by NSF grants DMS 2147546/2015447 and the NSF CAREER award DMS-2143215. 
Part of this work was done while G.~Li and Y.~Wei were visiting the Simons Institute for the Theory of Computing.

\bibliographystyle{apalike}
\bibliography{../reference-amp}

\begin{thebibliography}{}

\bibitem[Amini and Wainwright, 2008]{amini2008high}
Amini, A.~A. and Wainwright, M.~J. (2008).
\newblock High-dimensional analysis of semidefinite relaxations for sparse
  principal components.
\newblock In {\em 2008 IEEE international symposium on information theory},
  pages 2454--2458. IEEE.

\bibitem[Aubin et~al., 2020]{aubin2020exact}
Aubin, B., Loureiro, B., Baker, A., Krzakala, F., and Zdeborov{\'a}, L. (2020).
\newblock Exact asymptotics for phase retrieval and compressed sensing with
  random generative priors.
\newblock In {\em Mathematical and Scientific Machine Learning}, pages 55--73.
  PMLR.

\bibitem[Bai and Yao, 2008]{bai2008central}
Bai, Z. and Yao, J.-f. (2008).
\newblock Central limit theorems for eigenvalues in a spiked population model.
\newblock In {\em Annales de l'IHP Probabilit{\'e}s et statistiques},
  volume~44, pages 447--474.

\bibitem[Baik et~al., 2005]{baik2005phase}
Baik, J., Arous, G.~B., and P{\'e}ch{\'e}, S. (2005).
\newblock Phase transition of the largest eigenvalue for nonnull complex sample
  covariance matrices.
\newblock {\em The Annals of Probability}, 33(5):1643--1697.

\bibitem[Bandeira et~al., 2019]{bandeira2019computational}
Bandeira, A.~S., Kunisky, D., and Wein, A.~S. (2019).
\newblock Computational hardness of certifying bounds on constrained pca
  problems.
\newblock {\em arXiv preprint arXiv:1902.07324}.

\bibitem[Bandeira et~al., 2018]{bandeira2018notes}
Bandeira, A.~S., Perry, A., and Wein, A.~S. (2018).
\newblock Notes on computational-to-statistical gaps: predictions using
  statistical physics.
\newblock {\em Portugaliae Mathematica}, 75(2):159--186.

\bibitem[Bandeira and Van~Handel, 2016]{bandeira2016sharp}
Bandeira, A.~S. and Van~Handel, R. (2016).
\newblock Sharp nonasymptotic bounds on the norm of random matrices with
  independent entries.
\newblock {\em The Annals of Probability}, 44(4):2479--2506.

\bibitem[Bao et~al., 2021]{bao2021singular}
Bao, Z., Ding, X., and Wang, K. (2021).
\newblock Singular vector and singular subspace distribution for the matrix
  denoising model.
\newblock {\em The Annals of Statistics}, 49(1):370--392.

\bibitem[Barbier et~al., 2016]{barbier2016mutual}
Barbier, J., Dia, M., Macris, N., and Krzakala, F. (2016).
\newblock The mutual information in random linear estimation.
\newblock In {\em 2016 54th Annual Allerton Conference on Communication,
  Control, and Computing (Allerton)}, pages 625--632. IEEE.

\bibitem[Barbier et~al., 2019]{barbier2019optimal}
Barbier, J., Krzakala, F., Macris, N., Miolane, L., and Zdeborov{\'a}, L.
  (2019).
\newblock Optimal errors and phase transitions in high-dimensional generalized
  linear models.
\newblock {\em Proceedings of the National Academy of Sciences},
  116(12):5451--5460.

\bibitem[Bayati et~al., 2015]{bayati2015universality}
Bayati, M., Lelarge, M., and Montanari, A. (2015).
\newblock Universality in polytope phase transitions and message passing
  algorithms.
\newblock {\em The Annals of Applied Probability}, 25(2):753--822.

\bibitem[Bayati and Montanari, 2011a]{bayati2011dynamics}
Bayati, M. and Montanari, A. (2011a).
\newblock The dynamics of message passing on dense graphs, with applications to
  compressed sensing.
\newblock {\em IEEE Transactions on Information Theory}, 57(2):764--785.

\bibitem[Bayati and Montanari, 2011b]{bayati2011lasso}
Bayati, M. and Montanari, A. (2011b).
\newblock The {LASSO} risk for gaussian matrices.
\newblock {\em IEEE Transactions on Information Theory}, 58(4):1997--2017.

\bibitem[Berthet and Rigollet, 2013a]{berthet2013computational}
Berthet, Q. and Rigollet, P. (2013a).
\newblock Computational lower bounds for sparse {PCA}.
\newblock {\em arXiv preprint arXiv:1304.0828}.

\bibitem[Berthet and Rigollet, 2013b]{berthet2013optimal}
Berthet, Q. and Rigollet, P. (2013b).
\newblock Optimal detection of sparse principal components in high dimension.
\newblock {\em The Annals of Statistics}, 41(4):1780--1815.

\bibitem[Bolthausen, 2009]{bolthausen2009high}
Bolthausen, E. (2009).
\newblock On the high-temperature phase of the sherrington-kirkpatrick model.
\newblock In {\em Seminar at EURANDOM, Eindhoven}.

\bibitem[Borgerding and Schniter, 2016]{borgerding2016onsager}
Borgerding, M. and Schniter, P. (2016).
\newblock Onsager-corrected deep learning for sparse linear inverse problems.
\newblock In {\em 2016 IEEE Global Conference on Signal and Information
  Processing (GlobalSIP)}, pages 227--231. IEEE.

\bibitem[Bu et~al., 2020]{bu2020algorithmic}
Bu, Z., Klusowski, J.~M., Rush, C., and Su, W.~J. (2020).
\newblock Algorithmic analysis and statistical estimation of {SLOPE} via
  approximate message passing.
\newblock {\em IEEE Transactions on Information Theory}, 67(1):506--537.

\bibitem[Cai et~al., 2015]{cai2015optimal}
Cai, T., Ma, Z., and Wu, Y. (2015).
\newblock Optimal estimation and rank detection for sparse spiked covariance
  matrices.
\newblock {\em Probability theory and related fields}, 161(3):781--815.

\bibitem[Candes et~al., 2015]{candes2015phase}
Candes, E.~J., Li, X., and Soltanolkotabi, M. (2015).
\newblock Phase retrieval via {W}irtinger flow: Theory and algorithms.
\newblock {\em IEEE Transactions on Information Theory}, 61(4):1985--2007.

\bibitem[Capitaine et~al., 2009]{capitaine2009largest}
Capitaine, M., Donati-Martin, C., and F{\'e}ral, D. (2009).
\newblock The largest eigenvalues of finite rank deformation of large wigner
  matrices: convergence and nonuniversality of the fluctuations.
\newblock {\em The Annals of Probability}, 37(1):1--47.

\bibitem[Celentano et~al., 2021]{celentano2021local}
Celentano, M., Fan, Z., and Mei, S. (2021).
\newblock Local convexity of the {TAP} free energy and {AMP} convergence for
  $z_2$-synchronization.
\newblock {\em arXiv preprint arXiv:2106.11428}.

\bibitem[Celentano and Montanari, 2022]{celentano2022fundamental}
Celentano, M. and Montanari, A. (2022).
\newblock Fundamental barriers to high-dimensional regression with convex
  penalties.
\newblock {\em The Annals of Statistics}, 50(1):170--196.

\bibitem[Celentano et~al., 2020]{celentano2020lasso}
Celentano, M., Montanari, A., and Wei, Y. (2020).
\newblock The {L}asso with general {G}aussian designs with applications to
  hypothesis testing.
\newblock {\em arXiv preprint arXiv:2007.13716}.

\bibitem[Chen and Lam, 2021]{chen2021universality}
Chen, W.-K. and Lam, W.-K. (2021).
\newblock Universality of approximate message passing algorithms.
\newblock {\em Electronic Journal of Probability}, 26:1--44.

\bibitem[Chen et~al., 2021a]{chen2021asymmetry}
Chen, Y., Cheng, C., and Fan, J. (2021a).
\newblock Asymmetry helps: Eigenvalue and eigenvector analyses of
  asymmetrically perturbed low-rank matrices.
\newblock {\em The Annals of statistics}, 49(1):435.

\bibitem[Chen et~al., 2021b]{chen2021spectral}
Chen, Y., Chi, Y., Fan, J., and Ma, C. (2021b).
\newblock Spectral methods for data science: A statistical perspective.
\newblock {\em Foundations and Trends{\textregistered} in Machine Learning},
  14(5):566--806.

\bibitem[Cheng et~al., 2021]{cheng2021tackling}
Cheng, C., Wei, Y., and Chen, Y. (2021).
\newblock Tackling small eigen-gaps: Fine-grained eigenvector estimation and
  inference under heteroscedastic noise.
\newblock {\em IEEE Transactions on Information Theory}, 67(11):7380--7419.

\bibitem[Chi et~al., 2019]{chi2019nonconvex}
Chi, Y., Lu, Y.~M., and Chen, Y. (2019).
\newblock Nonconvex optimization meets low-rank matrix factorization: An
  overview.
\newblock {\em IEEE Transactions on Signal Processing}, 67(20):5239--5269.

\bibitem[d'Aspremont et~al., 2004]{d2004direct}
d'Aspremont, A., Ghaoui, L., Jordan, M., and Lanckriet, G. (2004).
\newblock A direct formulation for sparse {PCA} using semidefinite programming.
\newblock {\em Advances in neural information processing systems}, 17.

\bibitem[Deshpande et~al., 2017]{deshpande2017asymptotic}
Deshpande, Y., Abbe, E., and Montanari, A. (2017).
\newblock Asymptotic mutual information for the balanced binary stochastic
  block model.
\newblock {\em Information and Inference: A Journal of the IMA}, 6(2):125--170.

\bibitem[Deshpande and Montanari, 2014a]{deshpande2014information}
Deshpande, Y. and Montanari, A. (2014a).
\newblock Information-theoretically optimal sparse {PCA}.
\newblock In {\em 2014 IEEE International Symposium on Information Theory},
  pages 2197--2201. IEEE.

\bibitem[Deshpande and Montanari, 2014b]{deshpande2014sparse}
Deshpande, Y. and Montanari, A. (2014b).
\newblock Sparse {PCA} via covariance thresholding.
\newblock {\em Advances in Neural Information Processing Systems}, 27.

\bibitem[Deshpande et~al., 2014]{deshpande2014cone}
Deshpande, Y., Montanari, A., and Richard, E. (2014).
\newblock Cone-constrained principal component analysis.
\newblock {\em Advances in Neural Information Processing Systems}, 27.

\bibitem[Ding et~al., 2019]{ding2019subexponential}
Ding, Y., Kunisky, D., Wein, A.~S., and Bandeira, A.~S. (2019).
\newblock Subexponential-time algorithms for sparse {PCA}.
\newblock {\em arXiv preprint arXiv:1907.11635}.

\bibitem[Donoho and Montanari, 2016]{donoho2016high}
Donoho, D. and Montanari, A. (2016).
\newblock High dimensional robust m-estimation: Asymptotic variance via
  approximate message passing.
\newblock {\em Probability Theory and Related Fields}, 166(3):935--969.

\bibitem[Donoho et~al., 2013]{donoho2013information}
Donoho, D.~L., Javanmard, A., and Montanari, A. (2013).
\newblock Information-theoretically optimal compressed sensing via spatial
  coupling and approximate message passing.
\newblock {\em IEEE transactions on information theory}, 59(11):7434--7464.

\bibitem[Donoho et~al., 2009]{donoho2009message}
Donoho, D.~L., Maleki, A., and Montanari, A. (2009).
\newblock Message-passing algorithms for compressed sensing.
\newblock {\em Proceedings of the National Academy of Sciences},
  106(45):18914--18919.

\bibitem[Donoho and Montanari, 2015]{donoho2015variance}
Donoho, D.~L. and Montanari, A. (2015).
\newblock Variance breakdown of huber (m)-estimators: $ n/p \rightarrow
  (1,\infty)$.
\newblock {\em arXiv preprint arXiv:1503.02106}.

\bibitem[Dudeja et~al., 2022]{dudeja2022universality}
Dudeja, R., Lu, Y.~M., and Sen, S. (2022).
\newblock Universality of approximate message passing with semi-random
  matrices.
\newblock {\em arXiv preprint arXiv:2204.04281}.

\bibitem[El~Alaoui et~al., 2020]{el2020fundamental}
El~Alaoui, A., Krzakala, F., and Jordan, M. (2020).
\newblock Fundamental limits of detection in the spiked wigner model.
\newblock {\em The Annals of Statistics}, 48(2):863--885.

\bibitem[El~Karoui, 2018]{el2018impact}
El~Karoui, N. (2018).
\newblock On the impact of predictor geometry on the performance on
  high-dimensional ridge-regularized generalized robust regression estimators.
\newblock {\em Probability Theory and Related Fields}, 170(1):95--175.

\bibitem[Fan et~al., 2022a]{fan2022asymptotic}
Fan, J., Fan, Y., Han, X., and Lv, J. (2022a).
\newblock Asymptotic theory of eigenvectors for random matrices with diverging
  spikes.
\newblock {\em Journal of the American Statistical Association},
  117(538):996--1009.

\bibitem[Fan, 2022]{fan2022approximate}
Fan, Z. (2022).
\newblock Approximate message passing algorithms for rotationally invariant
  matrices.
\newblock {\em The Annals of Statistics}, 50(1):197--224.

\bibitem[Fan et~al., 2022b]{fan2022tap}
Fan, Z., Li, Y., and Sen, S. (2022b).
\newblock {TAP} equations for orthogonally invariant spin glasses at high
  temperature.
\newblock {\em arXiv preprint arXiv:2202.09325}.

\bibitem[Fan et~al., 2021]{fan2021tap}
Fan, Z., Mei, S., and Montanari, A. (2021).
\newblock {TAP} free energy, spin glasses and variational inference.
\newblock {\em The Annals of Probability}, 49(1):1--45.

\bibitem[Fan and Wu, 2021]{fan2021replica}
Fan, Z. and Wu, Y. (2021).
\newblock The replica-symmetric free energy for {I}sing spin glasses with
  orthogonally invariant couplings.
\newblock {\em arXiv preprint arXiv:2105.02797}.

\bibitem[Feng et~al., 2022]{feng2021unifying}
Feng, O.~Y., Venkataramanan, R., Rush, C., Samworth, R.~J., et~al. (2022).
\newblock A unifying tutorial on approximate message passing.
\newblock {\em Foundations and Trends{\textregistered} in Machine Learning},
  15(4):335--536.

\bibitem[F{\'e}ral and P{\'e}ch{\'e}, 2007]{feral2007largest}
F{\'e}ral, D. and P{\'e}ch{\'e}, S. (2007).
\newblock The largest eigenvalue of rank one deformation of large wigner
  matrices.
\newblock {\em Communications in mathematical physics}, 272(1):185--228.

\bibitem[Fletcher and Rangan, 2014]{fletcher2014scalable}
Fletcher, A.~K. and Rangan, S. (2014).
\newblock Scalable inference for neuronal connectivity from calcium imaging.
\newblock {\em Advances in neural information processing systems}, 27.

\bibitem[Fletcher and Rangan, 2018]{fletcher2018iterative}
Fletcher, A.~K. and Rangan, S. (2018).
\newblock Iterative reconstruction of rank-one matrices in noise.
\newblock {\em Information and Inference: A Journal of the IMA}, 7(3):531--562.

\bibitem[Gao and Zhang, 2022]{gao2022sdp}
Gao, C. and Zhang, A.~Y. (2022).
\newblock {SDP} achieves exact minimax optimality in phase synchronization.
\newblock {\em IEEE Transactions on Information Theory}.

\bibitem[Hopkins et~al., 2017]{hopkins2017power}
Hopkins, S.~B., Kothari, P.~K., Potechin, A., Raghavendra, P., Schramm, T., and
  Steurer, D. (2017).
\newblock The power of sum-of-squares for detecting hidden structures.
\newblock In {\em 2017 IEEE 58th Annual Symposium on Foundations of Computer
  Science (FOCS)}, pages 720--731. IEEE.

\bibitem[Hsu et~al., 2012]{hsu2012tail}
Hsu, D., Kakade, S., and Zhang, T. (2012).
\newblock A tail inequality for quadratic forms of subgaussian random vectors.
\newblock {\em Electronic Communications in Probability}, 17:1--6.

\bibitem[Hu and Lu, 2020]{hu2020universality}
Hu, H. and Lu, Y.~M. (2020).
\newblock Universality laws for high-dimensional learning with random features.
\newblock {\em arXiv preprint arXiv:2009.07669}.

\bibitem[Javanmard and Montanari, 2013]{javanmard2013state}
Javanmard, A. and Montanari, A. (2013).
\newblock State evolution for general approximate message passing algorithms,
  with applications to spatial coupling.
\newblock {\em Information and Inference: A Journal of the IMA}, 2(2):115--144.

\bibitem[Javanmard et~al., 2016]{javanmard2016phase}
Javanmard, A., Montanari, A., and Ricci-Tersenghi, F. (2016).
\newblock Phase transitions in semidefinite relaxations.
\newblock {\em Proceedings of the National Academy of Sciences},
  113(16):E2218--E2223.

\bibitem[Johnstone, 2001]{johnstone2001distribution}
Johnstone, I.~M. (2001).
\newblock On the distribution of the largest eigenvalue in principal components
  analysis.
\newblock {\em The Annals of statistics}, 29(2):295--327.

\bibitem[Johnstone and Lu, 2009]{johnstone2009consistency}
Johnstone, I.~M. and Lu, A.~Y. (2009).
\newblock On consistency and sparsity for principal components analysis in high
  dimensions.
\newblock {\em Journal of the American Statistical Association},
  104(486):682--693.

\bibitem[Johnstone and Paul, 2018]{johnstone2018pca}
Johnstone, I.~M. and Paul, D. (2018).
\newblock Pca in high dimensions: An orientation.
\newblock {\em Proceedings of the IEEE}, 106(8):1277--1292.

\bibitem[Keshavan et~al., 2010]{keshavan2010matrix}
Keshavan, R.~H., Montanari, A., and Oh, S. (2010).
\newblock Matrix completion from noisy entries.
\newblock {\em Journal of Machine Learning Research}, 11:2057--2078.

\bibitem[Knowles and Yin, 2013]{knowles2013isotropic}
Knowles, A. and Yin, J. (2013).
\newblock The isotropic semicircle law and deformation of {W}igner matrices.
\newblock {\em Communications on Pure and Applied Mathematics},
  66(11):1663--1749.

\bibitem[Krauthgamer et~al., 2015]{krauthgamer2015semidefinite}
Krauthgamer, R., Nadler, B., and Vilenchik, D. (2015).
\newblock Do semidefinite relaxations solve sparse {PCA} up to the information
  limit?
\newblock {\em The Annals of Statistics}, 43(3):1300--1322.

\bibitem[Krzakala et~al., 2016]{krzakala2016mutual}
Krzakala, F., Xu, J., and Zdeborov{\'a}, L. (2016).
\newblock Mutual information in rank-one matrix estimation.
\newblock In {\em 2016 IEEE Information Theory Workshop (ITW)}, pages 71--75.
  IEEE.

\bibitem[Lee et~al., 2016]{lee2016bulk}
Lee, J.~O., Schnelli, K., Stetler, B., and Yau, H.-T. (2016).
\newblock Bulk universality for deformed {W}igner matrices.
\newblock {\em The Annals of Probability}, 44(3):2349--2425.

\bibitem[Lelarge and Miolane, 2019]{lelarge2019fundamental}
Lelarge, M. and Miolane, L. (2019).
\newblock Fundamental limits of symmetric low-rank matrix estimation.
\newblock {\em Probability Theory and Related Fields}, 173(3):859--929.

\bibitem[Lesieur et~al., 2015]{lesieur2015phase}
Lesieur, T., Krzakala, F., and Zdeborov{\'a}, L. (2015).
\newblock Phase transitions in sparse {PCA}.
\newblock In {\em 2015 IEEE International Symposium on Information Theory
  (ISIT)}, pages 1635--1639. IEEE.

\bibitem[Lesieur et~al., 2017]{lesieur2017constrained}
Lesieur, T., Krzakala, F., and Zdeborov{\'a}, L. (2017).
\newblock Constrained low-rank matrix estimation: Phase transitions,
  approximate message passing and applications.
\newblock {\em Journal of Statistical Mechanics: Theory and Experiment},
  2017(7):073403.

\bibitem[Li et~al., 2021]{li2021minimax}
Li, G., Cai, C., Poor, H.~V., and Chen, Y. (2021).
\newblock Minimax estimation of linear functions of eigenvectors in the face of
  small eigen-gaps.
\newblock {\em arXiv preprint arXiv:2104.03298}.

\bibitem[Li et~al., 2023]{li2023approximate}
Li, G., Fan, W., and Wei, Y. (2023).
\newblock Approximate message passing from random initialization with
  applications to $\mathbb{Z}_2$ synchronization.
\newblock {\em arXiv preprint arXiv:2302.03682}.

\bibitem[Li and Wei, 2021]{li2021minimum}
Li, Y. and Wei, Y. (2021).
\newblock Minimum $\ell_1$-norm interpolators: Precise asymptotics and multiple
  descent.
\newblock {\em arXiv preprint arXiv:2110.09502}.

\bibitem[Ma et~al., 2020]{ma2020implicit}
Ma, C., Wang, K., Chi, Y., and Chen, Y. (2020).
\newblock Implicit regularization in nonconvex statistical estimation: Gradient
  descent converges linearly for phase retrieval, matrix completion, and blind
  deconvolution.
\newblock {\em Foundations of Computational Mathematics}, 20(3):451--632.

\bibitem[Ma et~al., 2018]{ma2018optimization}
Ma, J., Xu, J., and Maleki, A. (2018).
\newblock Optimization-based {AMP} for phase retrieval: The impact of
  initialization and $\ell_2$-regularization.
\newblock {\em arXiv preprint arXiv:1801.01170}.

\bibitem[Ma, 2013]{ma2013sparse}
Ma, Z. (2013).
\newblock Sparse principal component analysis and iterative thresholding.
\newblock {\em The Annals of Statistics}, 41(2):772--801.

\bibitem[Ma and Nandy, 2021]{ma2021community}
Ma, Z. and Nandy, S. (2021).
\newblock Community detection with contextual multilayer networks.
\newblock {\em arXiv preprint arXiv:2104.02960}.

\bibitem[Macris et~al., 2020]{macris2020all}
Macris, N., Rush, C., et~al. (2020).
\newblock All-or-nothing statistical and computational phase transitions in
  sparse spiked matrix estimation.
\newblock {\em Advances in Neural Information Processing Systems},
  33:14915--14926.

\bibitem[Massart, 2007]{massart2007concentration}
Massart, P. (2007).
\newblock {\em Concentration inequalities and model selection: Ecole d'Et{\'e}
  de Probabilit{\'e}s de Saint-Flour XXXIII-2003}.
\newblock Springer.

\bibitem[Mondelli and Venkataramanan, 2021]{mondelli2021pca}
Mondelli, M. and Venkataramanan, R. (2021).
\newblock {PCA} initialization for approximate message passing in rotationally
  invariant models.
\newblock {\em Advances in Neural Information Processing Systems},
  34:29616--29629.

\bibitem[Montanari and Richard, 2015]{montanari2015non}
Montanari, A. and Richard, E. (2015).
\newblock Non-negative principal component analysis: Message passing algorithms
  and sharp asymptotics.
\newblock {\em IEEE Transactions on Information Theory}, 62(3):1458--1484.

\bibitem[Montanari and Sen, 2016]{montanari2016semidefinite}
Montanari, A. and Sen, S. (2016).
\newblock Semidefinite programs on sparse random graphs and their application
  to community detection.
\newblock In {\em Proceedings of the forty-eighth annual ACM symposium on
  Theory of Computing}, pages 814--827.

\bibitem[Montanari and Venkataramanan, 2021]{montanari2021estimation}
Montanari, A. and Venkataramanan, R. (2021).
\newblock Estimation of low-rank matrices via approximate message passing.
\newblock {\em The Annals of Statistics}, 49(1):321--345.

\bibitem[Oymak and Tropp, 2018]{oymak2018universality}
Oymak, S. and Tropp, J.~A. (2018).
\newblock Universality laws for randomized dimension reduction, with
  applications.
\newblock {\em Information and Inference: A Journal of the IMA}, 7(3):337--446.

\bibitem[P{\'e}ch{\'e}, 2006]{peche2006largest}
P{\'e}ch{\'e}, S. (2006).
\newblock The largest eigenvalue of small rank perturbations of hermitian
  random matrices.
\newblock {\em Probability Theory and Related Fields}, 134(1):127--173.

\bibitem[Peng, 2012]{peng2012eigenvalues}
Peng, M. (2012).
\newblock Eigenvalues of deformed random matrices.
\newblock {\em arXiv preprint arXiv:1205.0572}.

\bibitem[Perry et~al., 2018a]{perry2018message}
Perry, A., Wein, A.~S., Bandeira, A.~S., and Moitra, A. (2018a).
\newblock Message-passing algorithms for synchronization problems over compact
  groups.
\newblock {\em Communications on Pure and Applied Mathematics},
  71(11):2275--2322.

\bibitem[Perry et~al., 2018b]{perry2018optimality}
Perry, A., Wein, A.~S., Bandeira, A.~S., and Moitra, A. (2018b).
\newblock Optimality and sub-optimality of {PCA} i: Spiked random matrix
  models.
\newblock {\em The Annals of Statistics}, 46(5):2416--2451.

\bibitem[Rangan and Fletcher, 2012]{rangan2012iterative}
Rangan, S. and Fletcher, A.~K. (2012).
\newblock Iterative estimation of constrained rank-one matrices in noise.
\newblock In {\em 2012 IEEE International Symposium on Information Theory
  Proceedings}, pages 1246--1250. IEEE.

\bibitem[Reeves and Pfister, 2019]{reeves2019replica}
Reeves, G. and Pfister, H.~D. (2019).
\newblock The replica-symmetric prediction for random linear estimation with
  gaussian matrices is exact.
\newblock {\em IEEE Transactions on Information Theory}, 65(4):2252--2283.

\bibitem[Rudelson and Vershynin, 2013]{rudelson2013hanson}
Rudelson, M. and Vershynin, R. (2013).
\newblock Hanson-{W}right inequality and sub-gaussian concentration.
\newblock {\em Electronic Communications in Probability}, 18:1--9.

\bibitem[Rush et~al., 2017]{rush2017capacity}
Rush, C., Greig, A., and Venkataramanan, R. (2017).
\newblock Capacity-achieving sparse superposition codes via approximate message
  passing decoding.
\newblock {\em IEEE Transactions on Information Theory}, 63(3):1476--1500.

\bibitem[Rush and Venkataramanan, 2018]{rush2018finite}
Rush, C. and Venkataramanan, R. (2018).
\newblock Finite sample analysis of approximate message passing algorithms.
\newblock {\em IEEE Transactions on Information Theory}, 64(11):7264--7286.

\bibitem[Schniter, 2011]{schniter2011message}
Schniter, P. (2011).
\newblock A message-passing receiver for bicm-ofdm over unknown
  clustered-sparse channels.
\newblock {\em IEEE Journal of Selected Topics in Signal Processing},
  5(8):1462--1474.

\bibitem[Schniter and Rangan, 2014]{schniter2014compressive}
Schniter, P. and Rangan, S. (2014).
\newblock Compressive phase retrieval via generalized approximate message
  passing.
\newblock {\em IEEE Transactions on Signal Processing}, 63(4):1043--1055.

\bibitem[Sellke, 2021]{sellke2021optimizing}
Sellke, M. (2021).
\newblock Optimizing mean field spin glasses with external field.
\newblock {\em arXiv preprint arXiv:2105.03506}.

\bibitem[Simchowitz et~al., 2018]{simchowitz2018tight}
Simchowitz, M., El~Alaoui, A., and Recht, B. (2018).
\newblock Tight query complexity lower bounds for {PCA} via finite sample
  deformed wigner law.
\newblock In {\em Proceedings of the 50th Annual ACM SIGACT Symposium on Theory
  of Computing}, pages 1249--1259.

\bibitem[Singer, 2011]{singer2011angular}
Singer, A. (2011).
\newblock Angular synchronization by eigenvectors and semidefinite programming.
\newblock {\em Applied and computational harmonic analysis}, 30(1):20--36.

\bibitem[Sur and Cand{\`e}s, 2019]{sur2019modern}
Sur, P. and Cand{\`e}s, E.~J. (2019).
\newblock A modern maximum-likelihood theory for high-dimensional logistic
  regression.
\newblock {\em Proceedings of the National Academy of Sciences},
  116(29):14516--14525.

\bibitem[Sur et~al., 2019]{sur2019likelihood}
Sur, P., Chen, Y., and Cand{\`e}s, E.~J. (2019).
\newblock The likelihood ratio test in high-dimensional logistic regression is
  asymptotically a rescaled chi-square.
\newblock {\em Probability Theory and Related Fields}, 175(1-2):487--558.

\bibitem[Thrampoulidis et~al., 2018]{thrampoulidis2018precise}
Thrampoulidis, C., Abbasi, E., and Hassibi, B. (2018).
\newblock Precise error analysis of regularized $ m $-estimators in high
  dimensions.
\newblock {\em IEEE Transactions on Information Theory}, 64(8):5592--5628.

\bibitem[Venkataramanan et~al., 2021]{venkataramanan2021estimation}
Venkataramanan, R., K{\"o}gler, K., and Mondelli, M. (2021).
\newblock Estimation in rotationally invariant generalized linear models via
  approximate message passing.
\newblock {\em arXiv preprint arXiv:2112.04330}.

\bibitem[Vershynin, 2018]{vershynin2018high}
Vershynin, R. (2018).
\newblock {\em High-dimensional probability: An introduction with applications
  in data science}, volume~47.
\newblock Cambridge university press.

\bibitem[Vu and Lei, 2012]{vu2012minimax}
Vu, V. and Lei, J. (2012).
\newblock Minimax rates of estimation for sparse {PCA} in high dimensions.
\newblock In {\em Artificial intelligence and statistics}, pages 1278--1286.
  PMLR.

\bibitem[Vu et~al., 2013]{vu2013fantope}
Vu, V.~Q., Cho, J., Lei, J., and Rohe, K. (2013).
\newblock Fantope projection and selection: A near-optimal convex relaxation of
  sparse {PCA}.
\newblock {\em Advances in neural information processing systems}, 26.

\bibitem[Wainwright, 2019]{wainwright2019high}
Wainwright, M.~J. (2019).
\newblock {\em High-dimensional statistics: A non-asymptotic viewpoint},
  volume~48.
\newblock Cambridge University Press.

\bibitem[Wang et~al., 2016]{wang2016statistical}
Wang, T., Berthet, Q., and Samworth, R.~J. (2016).
\newblock Statistical and computational trade-offs in estimation of sparse
  principal components.
\newblock {\em The Annals of Statistics}, 44(5):1896--1930.

\bibitem[Wang et~al., 2022]{wang2022universality}
Wang, T., Zhong, X., and Fan, Z. (2022).
\newblock Universality of approximate message passing algorithms and tensor
  networks.
\newblock {\em arXiv preprint arXiv:2206.13037}.

\bibitem[Wei et~al., 2019]{wei2019geometry}
Wei, Y., Wainwright, M.~J., and Guntuboyina, A. (2019).
\newblock The geometry of hypothesis testing over convex cones: Generalized
  likelihood ratio tests and minimax radii.
\newblock {\em The Annals of Statistics}, 47(2):994--1024.

\bibitem[Yan et~al., 2021]{yan2021inference}
Yan, Y., Chen, Y., and Fan, J. (2021).
\newblock Inference for heteroskedastic {PCA} with missing data.
\newblock {\em arXiv preprint arXiv:2107.12365}.

\bibitem[Zdeborov{\'a} and Krzakala, 2016]{zdeborova2016statistical}
Zdeborov{\'a}, L. and Krzakala, F. (2016).
\newblock Statistical physics of inference: Thresholds and algorithms.
\newblock {\em Advances in Physics}, 65(5):453--552.

\bibitem[Zhong and Boumal, 2018]{zhong2018near}
Zhong, Y. and Boumal, N. (2018).
\newblock Near-optimal bounds for phase synchronization.
\newblock {\em SIAM Journal on Optimization}, 28(2):989--1016.

\bibitem[Zhou and Chen, 2023]{zhou2023deflated}
Zhou, Y. and Chen, Y. (2023).
\newblock Deflated heteropca: Overcoming the curse of ill-conditioning in
  heteroskedastic {PCA}.
\newblock {\em arXiv preprint arXiv:2303.06198}.

\bibitem[Zou et~al., 2006]{zou2006sparse}
Zou, H., Hastie, T., and Tibshirani, R. (2006).
\newblock Sparse principal component analysis.
\newblock {\em Journal of computational and graphical statistics},
  15(2):265--286.

\end{thebibliography}


\end{document}